\documentclass{amsart}
\usepackage{amssymb}
\usepackage{amsfonts}

\newtheorem{theorem}{Theorem}
\theoremstyle{plain}
\newtheorem{acknowledgement}{Acknowledgement}

\newtheorem{corollary}{Corollary}

\newtheorem{definition}{Definition}
\newtheorem{example}{Example}

\newtheorem{lemma}{Lemma}
\newtheorem{notation}{Notation}
\newtheorem{problem}{Problem}
\newtheorem{proposition}{Proposition}
\newtheorem{remark}{Remark}

\numberwithin{equation}{section}
\input{tcilatex}

\begin{document}
\title[Two weight boundedness]{A two weight theorem for $\alpha $-fractional
singular integrals with an energy side condition and quasicube testing}
\author[E.T. Sawyer]{Eric T. Sawyer}
\address{ Department of Mathematics \& Statistics, McMaster University, 1280
Main Street West, Hamilton, Ontario, Canada L8S 4K1 }
\email{sawyer@mcmaster.ca}
\thanks{Research supported in part by NSERC}
\author[C.-Y. Shen]{Chun-Yen Shen}
\address{ Department of Mathematics \\
National Central University \\
Chungli, 32054, Taiwan }
\email{chunyshen@gmail.com}
\thanks{C.-Y. Shen supported in part by the NSC, through grant
NSC102-2115-M-008-015-MY2}
\author[I. Uriarte-Tuero]{Ignacio Uriarte-Tuero}
\address{ Department of Mathematics \\
Michigan State University \\
East Lansing MI }
\email{ignacio@math.msu.edu}
\thanks{ I. Uriarte-Tuero has been partially supported by grants DMS-1056965
(US NSF), MTM2010-16232, MTM2009-14694-C02-01 (Spain), and a Sloan
Foundation Fellowship. }
\date{May 25, 2015}

\begin{abstract}
Let $\sigma $ and $\omega $ be locally finite positive Borel measures on $%
\mathbb{R}^{n}$ with no common point masses, and let $T^{\alpha }$\ be a
standard $\alpha $-fractional Calder\'{o}n-Zygmund operator on $\mathbb{R}%
^{n}$ with $0\leq \alpha <n$. Suppose that $\Omega :\mathbb{R}%
^{n}\rightarrow \mathbb{R}^{n}$ is a globally biLipschitz map, and refer to
the images $\Omega Q$ of cubes $Q$ as \emph{quasicubes}. Furthermore, assume
as side conditions the $\mathcal{A}_{2}^{\alpha }$ conditions and certain $%
\alpha $\emph{-energy conditions} taken over quasicubes. Then we show that $%
T^{\alpha }$ is bounded from $L^{2}\left( \sigma \right) $ to $L^{2}\left(
\omega \right) $ if the quasicube testing conditions hold for $T^{\alpha }$%
\textbf{\ }and its dual, and if the quasiweak boundedness property holds for 
$T^{\alpha }$.

Conversely, if $T^{\alpha }$ is bounded from $L^{2}\left( \sigma \right) $
to $L^{2}\left( \omega \right) $, then the quasitesting conditions hold, and
the quasiweak boundedness condition holds. If the vector of $\alpha $%
-fractional Riesz transforms $\mathbf{R}_{\sigma }^{\alpha }$ (or more
generally a strongly elliptic vector of transforms) is bounded from $%
L^{2}\left( \sigma \right) $ to $L^{2}\left( \omega \right) $, then the $%
\mathcal{A}_{2}^{\alpha }$ conditions hold. We do not know if our
quasienergy conditions are necessary when $n\geq 2$.
\end{abstract}

\maketitle
\tableofcontents

\section{Introduction}

In this paper we prove a two weight inequality for standard $\alpha $%
-fractional Calder\'{o}n-Zygmund operators $T^{\alpha }$ in Euclidean space $%
\mathbb{R}^{n}$, where we assume $n$-dimensional $\mathcal{A}_{2}^{\alpha }$
conditions and certain $\alpha $\emph{-energy conditions} as side conditions
(in higher dimensions the Poisson kernels used in these two conditions
differ). We state and prove our theorem in the more general setting of \emph{%
quasicubes}\footnote{%
The previous version \texttt{arXiv:1302.5093v9 }of this paper, with some
corrections and without quasicubes, will appear in Revista Mat. \cite%
{SaShUr5}. The current paper is not intended for publication.}. This
extension to quasicubes is for the most part a cosmetic modification of the
proof for usual cubes, but definitions need to be carefully made, and the
extension of cube arguments to quasicubes must be checked in detail. The
quasicube testing conditions enhance the flexibility of the $T1$ theorem,
since any fixed set of testing conditions is highly unstable - for this see
the examples of weight pairs in \cite{LaSaUr2}.

We begin by describing the notion of quasicube used in this paper - a
special case of the classical notion used in quasiconformal theory.

\begin{definition}
We say that a homeomorphism $\Omega :\mathbb{R}^{n}\rightarrow \mathbb{R}%
^{n} $ is a globally biLipschitz map if%
\begin{equation}
\left\Vert \Omega \right\Vert _{Lip}\equiv \sup_{x,y\in \mathbb{R}^{n}}\frac{%
\left\Vert \Omega \left( x\right) -\Omega \left( y\right) \right\Vert }{%
\left\Vert x-y\right\Vert }<\infty ,  \label{rigid}
\end{equation}%
and $\left\Vert \Omega ^{-1}\right\Vert _{Lip}<\infty $.
\end{definition}

Note that a globally biLipschitz map $\Omega $ is differentiable almost
everywhere, and that there are constants $c,C>0$ such that%
\begin{equation*}
c\leq J_{\Omega }\left( x\right) \equiv \left\vert \det D\Omega \left(
x\right) \right\vert \leq C,\ \ \ \ \ x\in \mathbb{R}^{n}.
\end{equation*}

\begin{example}
\label{wild}Quasicubes can be wildly shaped, as illustrated by the standard
example of a logarithmic spiral in the plane $f_{\varepsilon }\left(
z\right) =z\left\vert z\right\vert ^{2\varepsilon i}=ze^{i\varepsilon \ln
\left( z\overline{z}\right) }$. Indeed, $f_{\varepsilon }:\mathbb{%
C\rightarrow C}$ is a globally biLipschitz map with Lipschitz constant $%
1+C\varepsilon $ since $f_{\varepsilon }^{-1}\left( w\right) =w\left\vert
w\right\vert ^{-2\varepsilon i}$ and%
\begin{equation*}
\nabla f_{\varepsilon }=\left( \frac{\partial f_{\varepsilon }}{\partial z},%
\frac{\partial f_{\varepsilon }}{\partial \overline{z}}\right) =\left(
\left\vert z\right\vert ^{2\varepsilon i}+i\varepsilon \left\vert
z\right\vert ^{2\varepsilon i},i\varepsilon \frac{z}{\overline{z}}\left\vert
z\right\vert ^{2\varepsilon i}\right) .
\end{equation*}%
On the other hand, $f_{\varepsilon }$ behaves wildly at the origin since the
image of the closed unit interval on the real line under $f_{\varepsilon }$
is an infinite logarithmic spiral.
\end{example}

\begin{definition}
Suppose that $\Omega :\mathbb{R}^{n}\rightarrow \mathbb{R}^{n}$ is a
globally biLipschitz map.

\begin{enumerate}
\item If $E$ is a measurable subset of $\mathbb{R}^{n}$, we define $\Omega
E\equiv \left\{ \Omega \left( x\right) :x\in E\right\} $ to be the image of $%
E$ under the homeomorphism $\Omega $.

\begin{enumerate}
\item In the special case that $E=Q$ is a cube in $\mathbb{R}^{n}$, we will
refer to $\Omega Q$ as a quasicube (or $\Omega $-quasicube if $\Omega $ is
not clear from the context).

\item We define the center $c_{\Omega Q}=c\left( \Omega Q\right) $ of the
quasicube $\Omega Q$ to be the point $\Omega c_{Q}$ where $c_{Q}=c\left(
Q\right) $ is the center of $Q$.

\item We define the side length $\ell \left( \Omega Q\right) $ of the
quasicube $\Omega Q$ to be the sidelength $\ell \left( Q\right) $ of the
cube $Q$.

\item For $r>0$ we define the `dilation' $r\Omega Q$ of a quasicube $\Omega
Q $ to be $\Omega rQ$ where $rQ$ is the usual `dilation' of a cube in $%
\mathbb{R}^{n}$ that is concentric with $Q$ and having side length $r\ell
\left( Q\right) $.
\end{enumerate}

\item If $\mathcal{K}$ is a collection of cubes in $\mathbb{R}^{n}$, we
define $\Omega \mathcal{K}\equiv \left\{ \Omega Q:Q\in \mathcal{K}\right\} $
to be the collection of quasicubes $\Omega Q$ as $Q$ ranges over $\mathcal{K}
$.

\item If $\mathcal{F}$ is a grid of cubes in $\mathbb{R}^{n}$, we define the
inherited quasigrid structure on $\Omega \mathcal{F}$ by declaring that $%
\Omega Q$ is a child of $\Omega Q^{\prime }$ in $\Omega \mathcal{F}$ if $Q$
is a child of $Q^{\prime }$ in the grid $\mathcal{F}$.
\end{enumerate}
\end{definition}

Note that if $\Omega Q$ is a quasicube, then $\left\vert \Omega Q\right\vert
^{\frac{1}{n}}\approx \left\vert Q\right\vert ^{\frac{1}{n}}=\ell \left(
Q\right) =\ell \left( \Omega Q\right) $ shows that the measure of $\Omega Q$
is approximately its sidelength to the power $n$, more precisely there are
positive constants $c,C$ such that $c\left\vert J\right\vert ^{\frac{1}{n}%
}\leq \ell \left( J\right) \leq C\left\vert J\right\vert ^{\frac{1}{n}}$ for
any quasicube $J=\Omega Q$. We will generally use the expression $\left\vert
J\right\vert ^{\frac{1}{n}}$ in the various estimates arising in the proofs
below, but will often use $\ell \left( J\right) $ when defining collections
of quasicubes. Moreover, there are constants $R_{big}$ and $R_{small}$ such
that we have the comparability containments%
\begin{equation*}
Q+\Omega x_{Q}\subset R_{big}\Omega Q\text{ and }R_{small}\Omega Q\subset
Q+\Omega x_{Q}\ .
\end{equation*}

Given a fixed globally biLipschitz map $\Omega $ on $\mathbb{R}^{n}$, we
will define below the $n$-dimensional $\mathcal{A}_{2}^{\alpha }$
conditions, testing conditions, and energy conditions using $\Omega $%
-quasicubes in place of cubes, and we will refer to these new conditions as
quasi-$\mathcal{A}_{2}^{\alpha }$ (-testing, -energy) conditions. It turns
out that the $\mathcal{A}_{2}^{\alpha }$ conditions are equivalent to the
quasi-$\mathcal{A}_{2}^{\alpha }$ conditions, and so we can simply use the $%
\mathcal{A}_{2}^{\alpha }$ conditions throughout the paper. We will then
prove a $T1$ theorem with quasitesting and with a quasienergy side
condition. We now describe a particular case informally, and later explain
the full theorem in detail.

We show that for positive locally finite Borel measures $\sigma $ and $%
\omega $ without common point masses, and \emph{assuming} the quasienergy
conditions in the Theorem below, a strongly elliptic collection of standard $%
\alpha $-fractional Calder\'{o}n-Zygmund operators $\mathbf{T}^{\alpha }$ is
bounded from $L^{2}\left( \sigma \right) $ to $L^{2}\left( \omega \right) $,%
\begin{equation}
\left\Vert \mathbf{T}^{\alpha }\left( f\sigma \right) \right\Vert
_{L^{2}\left( \omega \right) }\lesssim \left\Vert f\right\Vert _{L^{2}\left(
\sigma \right) },  \label{2 weight}
\end{equation}%
(with $0\leq \alpha <n$) if and only if the $\mathcal{A}_{2}^{\alpha }$
conditions hold, the quasicube testing conditions for $\mathbf{T}^{\alpha }$
and its dual hold, and the quasiweak boundedness property holds. This
identifies the culprit in higher dimensions as the pair of quasienergy
conditions. We point out that these quasienergy conditions are implied by
higher dimensional analogues of essentially all the other side conditions
used previously in two weight theory, in particular doubling conditions, the
Energy Hypothesis (1.16) in \cite{LaSaUr2}, and the uniformly full dimension
assumption on the weights in version 3 of \cite{LaWi}.

The final argument by M. Lacey (\cite{Lac}) in the proof of the
Nazarov-Treil-Volberg conjecture for the Hilbert transform is the
culmination of a large body of work on two-weighted inequalities beginning
with the work of Nazarov, Treil and Volberg (\cite{NaVo}, \cite{NTV1}, \cite%
{NTV2}, \cite{NTV3} and \cite{Vol}) and continuing with that of Lacey and
the authors (\cite{LaSaUr1}, \cite{LaSaUr2}, \cite{LaSaShUr} and \cite%
{LaSaShUr2}), just to mention a few. See the references for further work. We
consider standard singular integrals $T$, as well as their $\alpha $%
-fractional counterparts $T^{\alpha }$, and include

\begin{enumerate}
\item the control of the functional quasienergy condition by the quasienergy
condition modulo $\mathcal{A}_{2}^{\alpha }$,

\item a proof of the necessity of the $\mathcal{A}_{2}^{\alpha }$ condition
for the boundedness of the vector of $\alpha $-fractional Riesz transforms $%
\mathbf{R}^{\alpha ,n}$,

\item the extensions of certain one-dimensional arguments to higher
dimension in light of the differing Poisson integrals used in the $\mathcal{A%
}_{2}^{\alpha }$ and energy conditions,

\item the treatment of certain complications arising from the Lacey-Wick
Monotonicity Lemma,

\item and of course the replacement of cubes by an arbitrary collection of
quasicubes throughout.
\end{enumerate}

These are the main ingredients in this paper. The final point is to adapt
the stopping time and recursion arguments of M. Lacey \cite{Lac} to complete
the proof of our theorem, but only after splitting the stopping form into
two sublinear stopping forms dictated by the right hand side of the
Lacey-Wick Monotonicity Lemma.

It turns out that in higher dimensions, there are two natural `Poisson
integrals' $\mathrm{P}^{\alpha }$ and $\mathcal{P}^{\alpha }$\ that arise,
the usual Poisson integral $\mathrm{P}^{\alpha }$ that emerges in connection
with energy considerations, and a different Poisson integral $\mathcal{P}%
^{\alpha }$ that emerges in connection with size considerations - in
dimension $n=1$ these two Poisson integrals coincide. The standard Poisson
integral $\mathrm{P}^{\alpha }$ appears in the energy conditions, and the
reproducing Poisson integral $\mathcal{P}^{\alpha }$ appears in the $%
\mathcal{A}_{2}$ condition. These two kernels coincide in dimension $n=1$
for the case $\alpha =0$ corresponding to the Hilbert transform.

\begin{acknowledgement}
We are grateful to Michael Lacey for pointing out a number of problems with
our arguments and various oversights in the versions of \cite{SaShUr}, \cite%
{SaShUr2} (now withdrawn), \cite{SaShUr3} on the \emph{arXiv}, including the
mistake in our monotonicity lemma, which has been corrected by M. Lacey and
B. Wick in version 1 of \cite{LaWi}, and in our consequent adaptation of the
stopping time and recursion argument in \cite{Lac}. See these preprints for
some of the details.
\end{acknowledgement}

\begin{remark}
There is overlap of the previous versions 1-6 of this paper \cite{SaShUr}
with the subsequent work of M. Lacey and B. Wick in versions 1 and 2 of \cite%
{LaWi}, but the authors there do not acknowledge this overlap. Some results
and some details of arguments in the current paper overlap with those in 
\cite{LaWi}. In particular: the Monotonicity Lemma \ref{mono} here is due to
Lacey and Wick in Lemma 4.2 of \cite{LaWi}; Lemma \ref{standard delta} here
is proved in \cite{LaWi}, but with the larger bound $\mathcal{A}_{2}^{\alpha
}$ there in place of $A_{2}^{\alpha }$; and an argument treating the
additional term in the Lacey-Wick Monotonicity Lemma as it arises in
functional energy is essentially in \cite{LaWi}. We note that the side
condition in \cite{LaWi} - uniformly full dimension - permits a reversal of
energy, something not assumed in this paper, and that reversal of energy
implies our energy conditions.
\end{remark}

\section{Statements of results}

Now we turn to a precise description of our two weight theorem. For this we
fix once and for all a globally biLipschitz map $\Omega :\mathbb{R}%
^{n}\rightarrow \mathbb{R}^{n}$ for use in all of our quasi-notions. In
order to state our theorem precisely, we need to define standard fractional
singular integrals, the two different Poisson kernels, and a quasienergy
condition sufficient for use in the proof of the two weight theorem. These
are introduced in the following subsections.

\subsection{Standard fractional singular integrals}

Let $0\leq \alpha <n$. Consider a kernel function $K^{\alpha }(x,y)$ defined
on $\mathbb{R}^{n}\times \mathbb{R}^{n}$ satisfying the following fractional
size and smoothness conditions of order $1+\delta $ for some $\delta >0$,%
\begin{eqnarray}
\left\vert K^{\alpha }\left( x,y\right) \right\vert &\leq &C_{CZ}\left\vert
x-y\right\vert ^{\alpha -n},  \label{sizeandsmoothness'} \\
\left\vert \nabla K^{\alpha }\left( x,y\right) \right\vert &\leq
&C_{CZ}\left\vert x-y\right\vert ^{\alpha -n-1},  \notag \\
\left\vert \nabla K^{\alpha }\left( x,y\right) -\nabla K^{\alpha }\left(
x^{\prime },y\right) \right\vert &\leq &C_{CZ}\left( \frac{\left\vert
x-x^{\prime }\right\vert }{\left\vert x-y\right\vert }\right) ^{\delta
}\left\vert x-y\right\vert ^{\alpha -n-1},\ \ \ \ \ \frac{\left\vert
x-x^{\prime }\right\vert }{\left\vert x-y\right\vert }\leq \frac{1}{2}, 
\notag \\
\left\vert \nabla K^{\alpha }\left( x,y\right) -\nabla K^{\alpha }\left(
x,y^{\prime }\right) \right\vert &\leq &C_{CZ}\left( \frac{\left\vert
y-y^{\prime }\right\vert }{\left\vert x-y\right\vert }\right) ^{\delta
}\left\vert x-y\right\vert ^{\alpha -n-1},\ \ \ \ \ \frac{\left\vert
y-y^{\prime }\right\vert }{\left\vert x-y\right\vert }\leq \frac{1}{2}. 
\notag
\end{eqnarray}

Then we define a standard $\alpha $-fractional Calder\'{o}n-Zygmund operator
associated with such a kernel as follows.

\begin{definition}
\label{def alpha standard}We say that $T^{\alpha }$ is a \emph{standard }$%
\alpha $\emph{-fractional singular integral operator with kernel $K^{\alpha
} $} if $T^{\alpha }$ is a bounded linear operator from some $L^{p}\left( 
\mathbb{R}^{n}\right) $ to some $L^{q}\left( \mathbb{R}^{n}\right) $ for
some fixed $1<p\leq q<\infty $, if $K^{\alpha }(x,y)$ is defined on $\mathbb{%
R}^{n}\times \mathbb{R}^{n}$ and satisfies (\ref{sizeandsmoothness'}), and
if $T^{\alpha }$ and $K^{\alpha }$ are related by%
\begin{equation*}
T^{\alpha }f(x)=\int K^{\alpha }(x,y)f(y)dy,\ \ \ \ \ \text{a.e.-}x\notin
supp\ f,
\end{equation*}%
whenever $f\in L^{p}\left( \mathbb{R}^{n}\right) $ has compact support in $%
\mathbb{R}^{n}$. We say $K^{\alpha }(x,y)$ is a \emph{standard }$\alpha $%
\emph{-fractional kernel} if it satisfies (\ref{sizeandsmoothness'}).
\end{definition}

We note that a more general definition of kernel has only order of
smoothness $\delta >0$, rather than $1+\delta $, but the use of the
Monotonicity and Energy Lemmas below, which involve first order Taylor
approximations to the kernel functions $K^{\alpha }\left( \cdot ,y\right) $,
requires order of smoothness more than $1$. A \emph{smooth truncation} of $%
T^{\alpha }$ has kernel $\eta _{\delta ,R}\left( \left\vert x-y\right\vert
\right) K^{\alpha }\left( x,y\right) $ for a smooth function $\eta _{\delta
,R}$ compactly supported in $\left( \delta ,R\right) $, $0<\delta <R<\infty $%
, and satisfying standard CZ estimates. A typical example of an $\alpha $%
-fractional transform is the $\alpha $-fractional \emph{Riesz} vector of
operators%
\begin{equation*}
\mathbf{R}^{\alpha ,n}=\left\{ R_{\ell }^{\alpha ,n}:1\leq \ell \leq
n\right\} ,\ \ \ \ \ 0\leq \alpha <n.
\end{equation*}%
The Riesz transforms $R_{\ell }^{n,\alpha }$ are convolution fractional
singular integrals $R_{\ell }^{n,\alpha }f\equiv K_{\ell }^{n,\alpha }\ast f$
with odd kernel defined by%
\begin{equation*}
K_{\ell }^{\alpha ,n}\left( w\right) \equiv \frac{w^{\ell }}{\left\vert
w\right\vert ^{n+1-\alpha }}\equiv \frac{\Omega _{\ell }\left( w\right) }{%
\left\vert w\right\vert ^{n-\alpha }},\ \ \ \ \ w=\left(
w^{1},...,w^{n}\right) .
\end{equation*}

However, in dealing with energy considerations, it is more convenient to use
the \emph{tangent line truncation} of\emph{\ }the Riesz transform $R_{\ell
}^{\alpha ,n}$ whose kernel is defined to be $\Omega _{\ell }\left( w\right)
\psi _{\delta ,R}^{\alpha }\left( \left\vert w\right\vert \right) $ where $%
\psi _{\delta ,R}^{\alpha }$ is continuously differentiable on an interval $%
\left( 0,S\right) $ with $0<\delta <R<S$, and where $\psi _{\delta
,R}^{\alpha }\left( r\right) =r^{\alpha -n}$ if $\delta \leq r\leq R$, and
has constant derivative on both $\left( 0,\delta \right) $ and $\left(
R,S\right) $ where $\psi _{\delta ,R}^{\alpha }\left( S\right) =0$. Note
that the tangent line extension of a $C^{1,\delta }$ function on the line is
again $C^{1,\delta }$ with no increase in the $C^{1,\delta }$ norm. As shown
in the one dimensional case in \cite{LaSaShUr3}, boundedness of $R_{\ell
}^{n,\alpha }$ with one set of appropriate truncations together with the $%
\mathcal{A}_{2}^{\alpha }$ condition below, is equivalent to boundedness of $%
R_{\ell }^{n,\alpha }$ with all truncations. Similar considerations apply to
general\emph{\ }$\alpha $-fractional singular integrals\emph{,} so that we
are free to use the tangent line truncations throughout the proof of our
theorem.

\subsection{Quasicube testing conditions}

Let $\mathcal{P}^{n}$ denote the collection of all half open half closed
cubes $Q=\dprod\limits_{k=1}^{n}\left[ c_{k}-\frac{\ell }{2},c_{k}+\frac{%
\ell }{2}\right) $ in $\mathbb{R}^{n}$ with sides \textbf{p}arallel to the
coordinate axes, and define the \textbf{c}enter of $Q$ to be $c_{Q}=\left(
c_{1},...,c_{n}\right) $ and the side \textbf{l}ength $\ell \left( Q\right) $
of $Q$ to be $\ell $. The following `dual' quasicube testing conditions are
necessary for the boundedness of $T^{\alpha }$ from $L^{2}\left( \sigma
\right) $ to $L^{2}\left( \omega \right) $, where $\mathcal{Q}^{n}$ :%
\begin{eqnarray*}
\mathfrak{T}_{T^{\alpha }}^{2} &\equiv &\sup_{Q\in \Omega \mathcal{P}^{n}}%
\frac{1}{\left\vert Q\right\vert _{\sigma }}\int_{Q}\left\vert T^{\alpha
}\left( \mathbf{1}_{Q}\sigma \right) \right\vert ^{2}\omega <\infty , \\
\left( \mathfrak{T}_{T^{\alpha }}^{\ast }\right) ^{2} &\equiv &\sup_{Q\in
\Omega \mathcal{P}^{n}}\frac{1}{\left\vert Q\right\vert _{\omega }}%
\int_{Q}\left\vert \left( T^{\alpha }\right) ^{\ast }\left( \mathbf{1}%
_{Q}\omega \right) \right\vert ^{2}\sigma <\infty ,
\end{eqnarray*}%
and where we interpret the right sides as holding uniformly over all tangent
line trucations of $T^{\alpha }$.

\begin{remark}
We alert the reader that the symbols $Q,I,J,K$ will all be used to denote
either cubes or quasicubes, and the context will make clear which is the
case. Throughout most of the proof of the main theorem only quasicubes are
considered.
\end{remark}

\subsection{Quasiweak boundedness property}

The quasiweak boundedness property for $T^{\alpha }$ with constant $C$ is
given by 
\begin{eqnarray*}
&&\left\vert \int_{Q}T^{\alpha }\left( 1_{Q^{\prime }}\sigma \right) d\omega
\right\vert \leq \mathcal{WBP}_{T^{\alpha }}\sqrt{\left\vert Q\right\vert
_{\omega }\left\vert Q^{\prime }\right\vert _{\sigma }}, \\
&&\ \ \ \ \ \text{for all quasicubes }Q,Q^{\prime }\text{ with }\frac{1}{C}%
\leq \frac{\left\vert Q\right\vert ^{\frac{1}{n}}}{\left\vert Q^{\prime
}\right\vert ^{\frac{1}{n}}}\leq C, \\
&&\ \ \ \ \ \text{and either }Q\subset 3Q^{\prime }\setminus Q^{\prime }%
\text{ or }Q^{\prime }\subset 3Q\setminus Q,
\end{eqnarray*}%
and where we interpret the left side above as holding uniformly over all
tangent line trucations of $T^{\alpha }$. Note that the quasiweak
boundedness property is implied by either the \emph{tripled} quasicube
testing condition,%
\begin{equation*}
\left\Vert \mathbf{1}_{3Q}\mathbf{T}^{\alpha }\left( \mathbf{1}_{Q}\sigma
\right) \right\Vert _{L^{2}\left( \omega \right) }\leq \mathfrak{T}_{\mathbf{%
T}^{\alpha }}^{\limfunc{triple}}\left\Vert \mathbf{1}_{Q}\right\Vert
_{L^{2}\left( \sigma \right) },\ \ \ \ \ \text{for all quasicubes }Q\text{
in }\mathbb{R}^{n},
\end{equation*}%
or the tripled dual quasicube testing condition defined with $\sigma $ and $%
\omega $ interchanged and the dual operator $\mathbf{T}^{\alpha ,\ast }$ in
place of $\mathbf{T}^{\alpha }$. In turn, the tripled quasicube testing
condition can be obtained from the quasicube testing condition for the
truncated weight pairs $\left( \omega ,\mathbf{1}_{Q}\sigma \right) $. See
also Remark \ref{surgery} below.

\subsection{Poisson integrals and $\mathcal{A}_{2}^{\protect\alpha }$}

Now let $\mu $ be a locally finite positive Borel measure on $\mathbb{R}^{n}$%
, and suppose $Q$ is an $\Omega $-quasicube in $\mathbb{R}^{n}$. Define $\mu
_{\Omega }$ to be the pushforward of the measure $\mu $ under the globally
biLipschitz map $\Omega $ given by $\int fd\mu _{\Omega }=\int \left( f\circ
\Omega \right) d\mu $, and similarly for $\nu _{\Omega ^{-1}}$. Recall that $%
\left\vert Q\right\vert ^{\frac{1}{n}}\approx \ell \left( Q\right) $ for a
quasicube $Q$. The two $\alpha $-fractional Poisson integrals of $\mu $ on a
quasicube $Q$ are given by:%
\begin{eqnarray*}
\mathrm{P}^{\alpha }\left( Q,\mu \right) &\equiv &\int_{\mathbb{R}^{n}}\frac{%
\left\vert Q\right\vert ^{\frac{1}{n}}}{\left( \left\vert Q\right\vert ^{%
\frac{1}{n}}+\left\vert x-x_{Q}\right\vert \right) ^{n+1-\alpha }}d\mu
\left( x\right) , \\
\mathcal{P}^{\alpha }\left( Q,\mu \right) &\equiv &\int_{\mathbb{R}%
^{n}}\left( \frac{\left\vert Q\right\vert ^{\frac{1}{n}}}{\left( \left\vert
Q\right\vert ^{\frac{1}{n}}+\left\vert x-x_{Q}\right\vert \right) ^{2}}%
\right) ^{n-\alpha }d\mu \left( x\right) ,
\end{eqnarray*}%
where we emphasize that $\left\vert x-x_{Q}\right\vert $ denote Euclidean
distance between $x$ and $x_{Q}$ and $\left\vert Q\right\vert $ denotes the
Lebesgue measure of the quasicube $Q$. At this point we observe that (\ref%
{rigid}) implies that if $Q=\Omega K$ where $K$ is a cube in $\mathbb{R}^{n}$%
, then the corresponding centers satisfy $c_{Q}=\Omega c_{K}$ and we have%
\begin{equation*}
\mathrm{P}^{\alpha }\left( K+c_{Q}-c_{K},\mu \right) \approx \mathrm{P}%
^{\alpha }\left( Q,\mu \right) =\mathrm{P}^{\alpha }\left( K,\mu _{\Omega
^{-1}}\right) \approx \mathrm{P}^{\alpha }\left( Q+c_{K}-c_{Q},\mu _{\Omega
^{-1}}\right) ,
\end{equation*}
and simlarly for $\mathcal{P}^{\alpha }$. We refer to $\mathrm{P}^{\alpha }$
as the \emph{standard} Poisson integral and to $\mathcal{P}^{\alpha }$ as
the \emph{reproducing} Poisson integral.

Let $\sigma $ and $\omega $ be locally finite positive Borel measures on $%
\mathbb{R}^{n}$ with no common point masses, and suppose $0\leq \alpha <n$.
First recall that the classical $A_{2}^{\alpha }$ constant is defined by the
following supremum over all usual cubes $K$ in $\mathcal{Q}^{n}$, 
\begin{equation*}
A_{2}^{\alpha }\left( \sigma ,\omega \right) \equiv \sup_{K\in \mathcal{P}%
^{n}}\frac{\left\vert K\right\vert _{\sigma }}{\left\vert K\right\vert ^{1-%
\frac{\alpha }{n}}}\frac{\left\vert K\right\vert _{\omega }}{\left\vert
K\right\vert ^{1-\frac{\alpha }{n}}},
\end{equation*}%
which is easily seen to be equivalent to both the supremum over all
quasicubes $Q$\ in $\Omega \mathcal{Q}$, 
\begin{equation*}
A_{2}^{\alpha }\left( \sigma ,\omega \right) \equiv \sup_{Q\in \Omega 
\mathcal{P}^{n}}\frac{\left\vert Q\right\vert _{\sigma }}{\left\vert
Q\right\vert ^{1-\frac{\alpha }{n}}}\frac{\left\vert Q\right\vert _{\omega }%
}{\left\vert Q\right\vert ^{1-\frac{\alpha }{n}}},
\end{equation*}%
and to the $A_{2}^{\alpha }$ condition for the weight pair $\left( \sigma
_{\Omega ^{-1}},\omega _{\Omega ^{-1}}\right) $ of pushforward measures
under $\Omega ^{-1}$: 
\begin{equation*}
A_{2}^{\alpha }\left( \sigma _{\Omega ^{-1}},\omega _{\Omega ^{-1}}\right)
\equiv \sup_{K\in \mathcal{P}^{n}}\frac{\left\vert K\right\vert _{\sigma
_{\Omega ^{-1}}}}{\left\vert K\right\vert ^{1-\frac{\alpha }{n}}}\frac{%
\left\vert K\right\vert _{\omega _{\Omega ^{-1}}}}{\left\vert K\right\vert
^{1-\frac{\alpha }{n}}}.
\end{equation*}%
We now define the \emph{one-tailed} $\mathcal{A}_{2}^{\alpha }$ constant
using $\mathcal{P}^{\alpha }$. The energy constants $\mathcal{E}_{\alpha }$
introduced in the next subsection will use the standard Poisson integral $%
\mathrm{P}^{\alpha }$.

\begin{notation}
Let $\Omega \mathcal{P}^{n}$ denote the collection of all quasicubes $%
Q=\Omega K$ where $K\in \mathcal{P}^{n}$ is a usual cube in $\mathbb{R}^{n}$
with sides parallel to the axes, and denote by $\Omega \mathcal{D}^{n}$ or
simply $\Omega \mathcal{D}$ a dyadic quasigrid in $\mathbb{R}^{n}$, where $%
\Omega \mathcal{D\subset }\Omega \mathcal{P}^{n}$. We will typically keep $%
\Omega $ in the notation for these global quasigrids $\Omega \mathcal{D}$,
but will usually suppress $\Omega $ from subquasigrids $\mathcal{F}\subset
\Omega \mathcal{D}$ of stopping times, etc.
\end{notation}

\begin{definition}
The one-sided constants $\mathcal{A}_{2}^{\alpha }$ and $\mathcal{A}%
_{2}^{\alpha ,\ast }$ for the weight pair $\left( \sigma ,\omega \right) $
are given by%
\begin{eqnarray*}
\mathcal{A}_{2}^{\alpha } &\equiv &\sup_{Q\in \Omega \mathcal{P}^{n}}%
\mathcal{P}^{\alpha }\left( Q,\sigma \right) \frac{\left\vert Q\right\vert
_{\omega }}{\left\vert Q\right\vert ^{1-\frac{\alpha }{n}}}<\infty , \\
\mathcal{A}_{2}^{\alpha ,\ast } &\equiv &\sup_{Q\in \Omega \mathcal{P}^{n}}%
\mathcal{P}^{\alpha }\left( Q,\omega \right) \frac{\left\vert Q\right\vert
_{\sigma }}{\left\vert Q\right\vert ^{1-\frac{\alpha }{n}}}<\infty .
\end{eqnarray*}
\end{definition}

These definitions are of course equivalent to the analogous definitions
using usual cubes, 
\begin{eqnarray*}
\mathcal{A}_{2}^{\alpha } &\equiv &\sup_{K\in \mathcal{P}^{n}}\mathcal{P}%
^{\alpha }\left( K,\sigma \right) \frac{\left\vert K\right\vert _{\omega }}{%
\left\vert K\right\vert ^{1-\frac{\alpha }{n}}}<\infty , \\
\mathcal{A}_{2}^{\alpha ,\ast } &\equiv &\sup_{K\in \mathcal{P}^{n}}\mathcal{%
P}^{\alpha }\left( K,\omega \right) \frac{\left\vert K\right\vert _{\sigma }%
}{\left\vert K\right\vert ^{1-\frac{\alpha }{n}}}<\infty .
\end{eqnarray*}%
which we can use interchangably. Now we turn to the definition of a
quasiHaar basis of $L^{2}\left( \mu \right) $.

\subsection{A weighted quasiHaar basis}

Recall we have fixed a globally biLipschitz map $\Omega :\mathbb{R}%
^{n}\rightarrow \mathbb{R}^{n}$. We will use a construction of a quasiHaar
basis in $\mathbb{R}^{n}$ that is adapted to a measure $\mu $ (c.f. \cite%
{NTV2} for the nonquasi case). Given a dyadic quasicube $Q\in \Omega 
\mathcal{D}$, let $\bigtriangleup _{Q}^{\mu }$ denote orthogonal projection
onto the finite dimensional subspace $L_{Q}^{2}\left( \mu \right) $ of $%
L^{2}\left( \mu \right) $ that consists of linear combinations of the
indicators of\ the children $\mathfrak{C}\left( Q\right) $ of $Q$ that have $%
\mu $-mean zero over $Q$:%
\begin{equation*}
L_{Q}^{2}\left( \mu \right) \equiv \left\{ f=\dsum\limits_{Q^{\prime }\in 
\mathfrak{C}\left( Q\right) }a_{Q^{\prime }}\mathbf{1}_{Q^{\prime
}}:a_{Q^{\prime }}\in \mathbb{R},\int_{Q}fd\mu =0\right\} .
\end{equation*}%
Then we have the important telescoping property for dyadic quasicubes $%
Q_{1}\subset Q_{2}$:%
\begin{equation}
\mathbf{1}_{Q_{0}}\left( x\right) \left( \dsum\limits_{Q\in \left[
Q_{1},Q_{2}\right] }\bigtriangleup _{Q}^{\mu }f\left( x\right) \right) =%
\mathbf{1}_{Q_{0}}\left( x\right) \left( \mathbb{E}_{Q_{0}}^{\mu }f-\mathbb{E%
}_{Q_{2}}^{\mu }f\right) ,\ \ \ \ \ Q_{0}\in \mathfrak{C}\left( Q_{1}\right)
,\ f\in L^{2}\left( \mu \right) .  \label{telescope}
\end{equation}%
We will at times find it convenient to use a fixed orthonormal basis $%
\left\{ h_{Q}^{\mu ,a}\right\} _{a\in \Gamma _{n}}$ of $L_{Q}^{2}\left( \mu
\right) $ where $\Gamma _{n}\equiv \left\{ 0,1\right\} ^{n}\setminus \left\{ 
\mathbf{1}\right\} $ is a convenient index set with $\mathbf{1}=\left(
1,1,...,1\right) $. Then $\left\{ h_{Q}^{\mu ,a}\right\} _{a\in \Gamma _{n}%
\text{ and }Q\in \Omega \mathcal{D}}$ is an orthonormal basis for $%
L^{2}\left( \mu \right) $, with the understanding that we add the constant
function $\mathbf{1}$ if $\mu $ is a finite measure. In particular we have%
\begin{equation*}
\left\Vert f\right\Vert _{L^{2}\left( \mu \right) }^{2}=\sum_{Q}\left\Vert
\bigtriangleup _{Q}^{\mu }f\right\Vert _{L^{2}\left( \mu \right)
}^{2}=\sum_{Q}\sum_{a\in \Gamma _{n}}\left\vert \widehat{f}\left( Q\right)
\right\vert ^{2},
\end{equation*}%
where%
\begin{equation*}
\left\vert \widehat{f}\left( Q\right) \right\vert ^{2}\equiv \sum_{a\in
\Gamma _{n}}\left\vert \left\langle f,h_{Q}^{\mu ,a}\right\rangle _{\mu
}\right\vert ^{2},
\end{equation*}%
and the measure is suppressed in the notation. Indeed, this follows from (%
\ref{telescope}) and Lebesgue's differentiation theorem for quasicubes. We
also record the following useful estimate. If $I^{\prime }$ is any of the $%
2^{n}$ $\Omega \mathcal{D}$-children of $I$, and $a\in \Gamma _{n}$, then 
\begin{equation}
\left\vert \mathbb{E}_{I^{\prime }}^{\mu }h_{I}^{\mu ,a}\right\vert \leq 
\sqrt{\mathbb{E}_{I^{\prime }}^{\mu }\left( h_{I}^{\mu ,a}\right) ^{2}}\leq 
\frac{1}{\sqrt{\left\vert I^{\prime }\right\vert _{\mu }}}.
\label{useful Haar}
\end{equation}

\subsection{Good quasigrids and quasienergy conditions}

Given a dyadic quasicube $K\in \Omega \mathcal{D}$ and a positive measure $%
\mu $ we define the quasiHaar projection $\mathsf{P}_{K}^{\mu }\equiv
\sum_{_{J\in \Omega \mathcal{D}:\ J\subset K}}\bigtriangleup _{J}^{\mu }$ on 
$K$ by 
\begin{equation*}
\mathsf{P}_{K}^{\mu }f=\sum_{_{J\in \Omega \mathcal{D}:\ J\subset
K}}\dsum\limits_{a\in \Gamma _{n}}\left\langle f,h_{J}^{\mu ,a}\right\rangle
_{\mu }h_{J}^{\mu ,a}\text{ and }\left\Vert \mathsf{P}_{K}^{\mu
}f\right\Vert _{L^{2}\left( \mu \right) }^{2}=\sum_{_{J\in \Omega \mathcal{D}%
:\ J\subset K}}\dsum\limits_{a\in \Gamma _{n}}\left\vert \left\langle
f,h_{J}^{\mu ,a}\right\rangle _{\mu }\right\vert ^{2},
\end{equation*}%
and where a quasiHaar basis $\left\{ h_{J}^{\mu ,a}\right\} _{a\in \Gamma
_{n}\text{ and }J\in \Omega \mathcal{D}}$ adapted to the measure $\mu $ was
defined in the section on a weighted quasiHaar basis above.

Now we define various notions for quasicubes which are inherited from the
same notions for cubes. The main objective here is to use the familiar
notation that one uses for cubes, but now extended to $\Omega $-quasicubes.
We have already introduced quasigrids $\Omega \mathcal{D}$, the notions of
center, sidelength and dyadic associated to quasicubes $Q\in \Omega \mathcal{%
P}^{n}$, as well as quasiHaar functions, and we will continue to extend to
quasicubes the additional familiar notions related to cubes as we come
across them. We begin with the notion of \emph{good}. Fix a quasigrid $%
\Omega \mathcal{D}$. We say that a dyadic quasicube $J$ is $\left( \mathbf{r}%
,\varepsilon \right) $-\emph{deeply embedded} in a (not necessarily dyadic)
quasicube $K$, or simply $\mathbf{r}$\emph{-deeply embedded} in $K$, which
we write as $J\Subset _{\mathbf{r}}K$, when $J\subset K$ and both 
\begin{eqnarray}
\ell \left( J\right) &\leq &2^{-\mathbf{r}}\ell \left( K\right) ,
\label{def deep embed} \\
\limfunc{quasidist}\left( J,\partial K\right) &\geq &\frac{1}{2}\ell \left(
J\right) ^{\varepsilon }\ell \left( K\right) ^{1-\varepsilon },  \notag
\end{eqnarray}%
where we define the quasidistance $\limfunc{quasidist}\left( E,F\right) $
between two sets $E$ and $F$ to be the Euclidean distance $\limfunc{dist}%
\left( \Omega ^{-1}E,\Omega ^{-1}F\right) $ between the preimages $\Omega
^{-1}E$ and $\Omega ^{-1}F$ of $E$ and $F$ under the map $\Omega $, and
where we recall that $\ell \left( J\right) \approx \left\vert J\right\vert ^{%
\frac{1}{n}}$. For the most part we will consider $J\Subset _{\mathbf{r}}K$
when $J$ and $K$ belong to a common quasigrid $\Omega \mathcal{D}$, but an
exception is made when defining the refined energy constant below.

\begin{definition}
Let $\mathbf{r}\in \mathbb{N}$ and $0<\varepsilon <1$. Fix a quasigrid $%
\Omega \mathcal{D}$. A dyadic quasicube $J$ is $\left( \mathbf{r}%
,\varepsilon \right) $\emph{-good}, or simply \emph{good}, if for \emph{every%
} dyadic superquasicube $I$, it is the case that \textbf{either} $J$ has
side length at least $2^{-\mathbf{r}}$ times that of $I$, \textbf{or} $%
J\Subset _{\mathbf{r}}I$ is $\left( \mathbf{r},\varepsilon \right) $-deeply
embedded in $I$.
\end{definition}

Note that this definition simply asserts that a dyadic quasicube $J=\Omega
J^{\prime }$ is $\left( \mathbf{r},\varepsilon \right) $-good if and only if
the cube $J^{\prime }$ is $\left( \mathbf{r},\varepsilon \right) $-good in
the usual sense. Finally, we say that $J$ is $\mathbf{r}$\emph{-nearby} in $%
K $ when $J\subset K$ and%
\begin{equation*}
\ell \left( J\right) >2^{-\mathbf{r}}\ell \left( K\right) .
\end{equation*}%
The parameters $\mathbf{r},\varepsilon $ will be fixed sufficiently large
and small respectively later in the proof, and we denote the set of such
good dyadic quasicubes by $\Omega \mathcal{D}_{\limfunc{good}}$.

Throughout the proof, it will be convenient to also consider pairs of
quasicubes $J,K$ where $J$ is $\mathbf{\rho }$\emph{-deeply embedded} in $K$%
, written $J\Subset _{\mathbf{\rho }}K$ and meaning (\ref{def deep embed})
holds with the same $\varepsilon >0$ but with $\mathbf{\rho }$ in place of $%
\mathbf{r}$; as well as pairs of quasicubes $J,K$ where $J$ is $\mathbf{\rho 
}$\emph{-nearby} in\emph{\ }$K$, $\ell \left( J\right) >2^{-\mathbf{\rho }%
}\ell \left( K\right) $, for a parameter $\mathbf{\rho }\gg \mathbf{r}$ that
will be fixed later. We define the smaller `good' quasiHaar projection $%
\mathsf{P}_{K}^{\limfunc{good},\omega }$ by%
\begin{equation*}
\mathsf{P}_{K}^{\limfunc{good},\mu }f\equiv \sum_{_{J\in \mathcal{G}\left(
K\right) }}\bigtriangleup _{J}^{\mu }f=\sum_{_{J\in \mathcal{G}\left(
K\right) }}\dsum\limits_{a\in \Gamma _{n}}\left\langle f,h_{J}^{\mu
,a}\right\rangle _{\mu }h_{J}^{\mu ,a},
\end{equation*}%
where $\mathcal{G}\left( K\right) $ consists of the good subcubes of $K$:%
\begin{equation*}
\mathcal{G}\left( K\right) \equiv \left\{ J\in \Omega \mathcal{D}_{\limfunc{%
good}}:J\subset K\right\} ,
\end{equation*}%
and also the larger `subgood' quasiHaar projection $\mathsf{P}_{K}^{\limfunc{%
subgood},\mu }$ by%
\begin{equation*}
\mathsf{P}_{K}^{\limfunc{subgood},\mu }f\equiv \sum_{_{J\in \mathcal{M}_{%
\limfunc{good}}\left( K\right) }}\sum_{J^{\prime }\subset J}\bigtriangleup
_{J^{\prime }}^{\mu }f=\sum_{_{J\in \mathcal{M}_{\limfunc{good}}\left(
K\right) }}\sum_{J^{\prime }\subset J}\dsum\limits_{a\in \Gamma
_{n}}\left\langle f,h_{J^{\prime }}^{\mu ,a}\right\rangle _{\mu
}h_{J^{\prime }}^{\mu ,a},
\end{equation*}%
where $\mathcal{M}_{\limfunc{good}}\left( K\right) $ consists of the \emph{%
maximal} good subcubes of $K$. We thus have 
\begin{eqnarray*}
\left\Vert \mathsf{P}_{K}^{\limfunc{good},\mu }\mathbf{x}\right\Vert
_{L^{2}\left( \mu \right) }^{2} &\leq &\left\Vert \mathsf{P}_{K}^{\limfunc{%
subgood},\mu }\mathbf{x}\right\Vert _{L^{2}\left( \mu \right) }^{2} \\
&\leq &\left\Vert \mathsf{P}_{I}^{\mu }\mathbf{x}\right\Vert _{L^{2}\left(
\mu \right) }^{2}=\int_{I}\left\vert \mathbf{x}-\left( \frac{1}{\left\vert
I\right\vert _{\mu }}\int_{I}\mathbf{x}dx\right) \right\vert ^{2}d\mu \left(
x\right) ,\ \ \ \ \ \mathbf{x}=\left( x_{1},...,x_{n}\right) ,
\end{eqnarray*}%
where $\mathsf{P}_{I}^{\mu }\mathbf{x}$ is the orthogonal projection of the
identity function $\mathbf{x}:\mathbb{R}^{n}\rightarrow \mathbb{R}^{n}$ onto
the vector-valued subspace of $\oplus _{k=1}^{n}L^{2}\left( \mu \right) $
consisting of functions supported in $I$ with $\mu $-mean value zero.

At this point we emphasize that in the setting of quasicubes we continue to
use the linear function $\mathbf{x}$ and not the pushforward of $\mathbf{x}$
by $\Omega $. The reason of course is that the quasienergy defined below is
used to capture the first order information in the Taylor expansion of a
singular kernel.

Recall that in dimension $n=1$, and for $\alpha =0$, the energy condition
constant was defined by%
\begin{equation*}
\mathcal{E}^{2}\equiv \sup_{I=\dot{\cup}I_{r}}\frac{1}{\left\vert
I\right\vert _{\sigma }}\sum_{r=1}^{\infty }\left( \frac{\mathrm{P}^{\alpha
}\left( I_{r},\mathbf{1}_{I}\sigma \right) }{\left\vert I_{r}\right\vert }%
\right) ^{2}\left\Vert \mathsf{P}_{I_{r}}^{\omega }\mathbf{x}\right\Vert
_{L^{2}\left( \omega \right) }^{2}\ ,
\end{equation*}%
where $I$, $I_{r}$ and $J$ are intervals in the real line. The extension of
the energy conditions to higher dimensions uses the collection 
\begin{equation*}
\mathcal{M}_{\mathbf{r}-\limfunc{deep}}\left( K\right) \equiv \left\{ \text{%
maximal }J\Subset _{\mathbf{r}}K\right\}
\end{equation*}%
of \emph{maximal} $\mathbf{r}$-deeply embedded dyadic subquasicubes of a
quasicube $K$ (a subquasicube $J$ of $K$ is a \emph{dyadic} subquasicube of $%
K$ if $J\in \Omega \mathcal{D}$ when $\Omega \mathcal{D}$ is a dyadic
quasigrid containing $K$). We let $J^{\ast }=\gamma J$ where $\gamma \geq 2$%
. Then the goodness parameter $\mathbf{r}$ is chosen sufficiently large,
depending on $\varepsilon $ and $\gamma $, that the bounded overlap property 
\begin{equation}
\sum_{J\in \mathcal{M}_{\mathbf{r}-\limfunc{deep}}\left( K\right) }\mathbf{1}%
_{J^{\ast }}\leq \beta \mathbf{1}_{K}\ ,  \label{bounded overlap}
\end{equation}%
holds for some positive constant $\beta $ depending only on $n,\gamma ,%
\mathbf{r}$ and $\varepsilon $. Indeed, the maximal $\mathbf{r}$-deeply
embedded subquasicubes $J$ of $K$ satisfy the condition%
\begin{equation*}
c_{n}\ell \left( J\right) ^{\varepsilon }\ell \left( K\right)
^{1-\varepsilon }\leq \limfunc{quasidist}\left( J,K^{c}\right) \leq
C_{n}\ell \left( J\right) ^{\varepsilon }\ell \left( K\right)
^{1-\varepsilon },
\end{equation*}%
Now with $0<\varepsilon <1$ and $\gamma \geq 2$ fixed, choose $\mathbf{r}$
so large that $2^{-\left( 1-\varepsilon \right) \mathbf{r}}<\frac{c_{n}}{%
2\gamma }$. Let $y\in K$. Then if $y\in \gamma J$, we have%
\begin{eqnarray*}
c_{n}\ell \left( J\right) ^{\varepsilon }\ell \left( K\right)
^{1-\varepsilon } &\leq &\limfunc{quasidist}\left( J,K^{c}\right) \leq
\gamma \ell \left( J\right) +\limfunc{quasidist}\left( \gamma J,K^{c}\right)
\\
&\leq &\gamma \ell \left( J\right) +\limfunc{quasidist}\left( y,K^{c}\right)
,
\end{eqnarray*}%
which implies%
\begin{equation*}
\frac{c_{n}}{2}\ell \left( J\right) ^{\varepsilon }\ell \left( K\right)
^{1-\varepsilon }\leq \limfunc{quasidist}\left( y,K^{c}\right) .
\end{equation*}%
But we also have 
\begin{equation*}
\limfunc{quasidist}\left( y,K^{c}\right) \leq \ell \left( J\right)
+dist\left( J,K^{c}\right) \leq \ell \left( J\right) +C_{n}\ell \left(
J\right) ^{\varepsilon }\ell \left( K\right) ^{1-\varepsilon }\leq \left( 
\frac{c_{n}}{2\gamma }+C_{n}\right) \ell \left( J\right) ^{\varepsilon }\ell
\left( K\right) ^{1-\varepsilon },
\end{equation*}%
and so altogether,%
\begin{equation*}
\frac{1}{\frac{c_{n}}{2\gamma }+C_{n}}\limfunc{quasidist}\left(
y,K^{c}\right) \leq \ell \left( J\right) ^{\varepsilon }\ell \left( K\right)
^{1-\varepsilon }\leq \frac{2}{c_{n}}\limfunc{quasidist}\left(
y,K^{c}\right) ,
\end{equation*}%
which proves (\ref{bounded overlap}) since the number $\beta $ of dyadic
numbers $2^{j}=\ell \left( J\right) $ that satisfy this last inequality is
bounded independent of $K$ and $y$.

For a fixed a dyadic grid $\Omega \mathcal{D}$, we define a refinement of $%
\mathcal{M}_{\mathbf{r}-\limfunc{deep},\Omega \mathcal{D}}\left( K\right) $
for certain $K$ and each $\ell \geq 1$ as follows. First, we say that a
quasicube $K$ (not necessarily in $\Omega \mathcal{D}$) is a \emph{shifted} $%
\Omega \mathcal{D}$-quasicube if it is a union of $2^{n}$ $\Omega \mathcal{D}
$-quasicubes $K^{\prime }$ with side length $\ell \left( K^{\prime }\right) =%
\frac{1}{2}\ell \left( K\right) $. Thus for any $\Omega \mathcal{D}$%
-quasicube $L$ there are exactly $2^{n}$ shifted $\Omega \mathcal{D}$%
-quasicubes of twice the side length that contain $L$, and one of them is of
course the $\Omega \mathcal{D}$-parent of $L$. We denote the collection of
shifted $\Omega \mathcal{D}$-quasicubes by $\mathcal{S}\Omega \mathcal{D}$.
Now for a shifted quasicube $K\in \mathcal{S}\Omega \mathcal{D}$, define $%
\mathcal{M}_{\mathbf{r}-\limfunc{deep},\Omega \mathcal{D}}\left( K\right) $
to consist of the \emph{maximal} $\mathbf{r}$-deeply embedded $\Omega 
\mathcal{D}$-dyadic subquasicubes $J$ of $K$. (In the special case that $K$
itself belongs to $\Omega \mathcal{D}$, then $\mathcal{M}_{\mathbf{r}-%
\limfunc{deep},\Omega \mathcal{D}}\left( K\right) =\mathcal{M}_{\mathbf{r}-%
\limfunc{deep}}\left( K\right) $). Then for $\ell \geq 1$ we define the
refinement%
\begin{eqnarray*}
\mathcal{M}_{\mathbf{r}-\limfunc{deep},\Omega \mathcal{D}}^{\ell }\left(
K\right) &\equiv &\left\{ J\in \mathcal{M}_{\mathbf{r}-\limfunc{deep},\Omega 
\mathcal{D}}\left( \pi ^{\ell }K^{\prime }\right) \text{ for some }K^{\prime
}\in \mathfrak{C}_{\Omega \mathcal{D}}\left( K\right) :\right. \\
&&\ \ \ \ \ \ \ \ \ \ \ \ \ \ \ \ \ \ \ \ \ \ \ \ \ \ \ \ \ \ \left.
J\subset L\text{ for some }L\in \mathcal{M}_{\mathbf{r}-\limfunc{deep}%
}\left( K\right) \right\} ,
\end{eqnarray*}%
where $\mathfrak{C}_{\Omega \mathcal{D}}\left( K\right) $ is the obvious
extension to shifted quasicubes of the set of $\Omega \mathcal{D}$-children.
Thus $\mathcal{M}_{\mathbf{r}-\limfunc{deep},\Omega \mathcal{D}}^{\ell
}\left( K\right) $ is the union, over all children $K^{\prime }$ of $K$, of
those quasicubes in $\mathcal{M}_{\mathbf{r}-\limfunc{deep}}\left( \pi
^{\ell }K^{\prime }\right) $ that happen to be contained in some $L\in 
\mathcal{M}_{\mathbf{r}-\limfunc{deep},\Omega \mathcal{D}}\left( K\right) $.

Since $J\in \mathcal{M}_{\mathbf{r}-\limfunc{deep},\Omega \mathcal{D}}^{\ell
}\left( K\right) $ implies $\gamma J\subset K$, we also have from (\ref%
{bounded overlap}) that%
\begin{equation}
\sum_{J\in \mathcal{M}_{\mathbf{r}-\limfunc{deep},\Omega \mathcal{D}}^{\ell
}\left( K\right) }\mathbf{1}_{J^{\ast }}\leq \beta \mathbf{1}_{K}\ ,\ \ \ \
\ \text{for each }\ell \geq 0.  \label{bounded overlap'}
\end{equation}%
Of course $\mathcal{M}_{\mathbf{r}-\limfunc{deep},\Omega \mathcal{D}%
}^{1}\left( K\right) \supset \mathcal{M}_{\mathbf{r}-\limfunc{deep}}\left(
K\right) $, but $\mathcal{M}_{\mathbf{r}-\limfunc{deep},\Omega \mathcal{D}%
}^{\ell }\left( K\right) $ is in general a finer subdecomposition of $K$ the
larger $\ell $ is, and may in fact be empty.

Now we proceed to the definition of various quasienergy conditions, which
are symbolically identical to those defined for usual cubes.

\begin{definition}
\label{energy condition}Suppose $\sigma $ and $\omega $ are positive locally
finite Borel measures on $\mathbb{R}^{n}$ without common point masses, and
fix $\gamma \geq 2$. Then the\ deep quasienergy condition constant $\mathcal{%
E}_{\alpha }^{\limfunc{deep}}$, the refined quasienergy condition constant $%
\mathcal{E}_{\alpha }^{\limfunc{refined}}$, and finally the quasienergy
condition constant $\mathcal{E}_{\alpha }$ itself, are given by%
\begin{eqnarray*}
\left( \mathcal{E}_{\alpha }^{\limfunc{deep}}\right) ^{2} &\equiv &\sup_{I=%
\dot{\cup}I_{r}}\frac{1}{\left\vert I\right\vert _{\sigma }}%
\sum_{r=1}^{\infty }\sum_{J\in \mathcal{M}_{\mathbf{r}-\limfunc{deep}}\left(
I_{r}\right) }\left( \frac{\mathrm{P}^{\alpha }\left( J,\mathbf{1}%
_{I\setminus \gamma J}\sigma \right) }{\left\vert J\right\vert ^{\frac{1}{n}}%
}\right) ^{2}\left\Vert \mathsf{P}_{J}^{\limfunc{subgood},\omega }\mathbf{x}%
\right\Vert _{L^{2}\left( \omega \right) }^{2}, \\
\left( \mathcal{E}_{\alpha }^{\limfunc{refined}}\right) ^{2} &\equiv
&\sup_{\Omega \mathcal{D}}\sup_{I}\sup_{\ell \geq 1}\frac{1}{\left\vert
I\right\vert _{\sigma }}\sum_{J\in \mathcal{M}_{\mathbf{r}-\limfunc{deep}%
,\Omega \mathcal{D}}^{\ell }\left( I\right) }\left( \frac{\mathrm{P}^{\alpha
}\left( J,\mathbf{1}_{I\setminus \gamma J}\sigma \right) }{\left\vert
J\right\vert ^{\frac{1}{n}}}\right) ^{2}\left\Vert \mathsf{P}_{J}^{\limfunc{%
subgood},\omega }\mathbf{x}\right\Vert _{L^{2}\left( \omega \right) }^{2} \\
\left( \mathcal{E}_{\alpha }\right) ^{2} &\equiv &\left( \mathcal{E}_{\alpha
}^{\limfunc{deep}}\right) ^{2}+\left( \mathcal{E}_{\alpha }^{\limfunc{refined%
}}\right) ^{2}\ .
\end{eqnarray*}%
where $\sup_{\Omega \mathcal{D}}\sup_{I}$ in the second line is taken over
all quasigrids $\Omega \mathcal{D}$ and shifted quasicubes $I\in \mathcal{S}%
\Omega \mathcal{D}$, and $\sup_{I=\dot{\cup}I_{r}}$ in the first line is
taken over

\begin{enumerate}
\item all dyadic quasigrids $\Omega \mathcal{D}$,

\item all $\Omega \mathcal{D}$-dyadic quasicubes $I$,

\item and all subpartitions $\left\{ I_{r}\right\} _{r=1}^{N\text{ or }%
\infty }$ of the quasicube $I$ into $\Omega \mathcal{D}$-dyadic
subquasicubes $I_{r}$.
\end{enumerate}
\end{definition}

\begin{remark}
The refined quasienergy condition appears in only one place in this paper,
namely in the right hand side of the estimate (\ref{B bound}) in Lemma \ref%
{refined lemma} below.
\end{remark}

We emphasize that in the above definitions we take $I$, $I_{r}$ and $J$ to
be quasicubes and $\mathrm{P}^{\alpha }\left( J,\mathbf{1}_{I\setminus \text{
}\gamma J}\sigma \right) $ and $\mathsf{P}_{J}^{\limfunc{subgood},\omega }$
as defined above for quasicubes, and as mentioned earlier $\mathbf{x}$
remains the identity function.

Note that in the refined quasienergy condition there is no outer
decomposition $I=\dot{\cup}I_{r}$. There are similar definitions for the
dual (backward) quasienergy conditions that simply interchange $\sigma $ and 
$\omega $ everywhere. These definitions of\ the quasienergy conditions
depend on the choice of $\gamma $ and the goodness parameters $\mathbf{r}$
and $\varepsilon $. Note that we can `plug the $\gamma $-hole' in the
Poisson integral $\mathrm{P}^{\alpha }\left( J,\mathbf{1}_{I\setminus \gamma
J}\sigma \right) $ for both $\mathcal{E}_{\alpha }^{\limfunc{deep}}$ and $%
\mathcal{E}_{\alpha }^{\limfunc{refined}}$ using the $A_{2}^{\alpha }$
condition and the bounded overlap property (\ref{bounded overlap'}). Indeed,
define 
\begin{eqnarray}
&&  \label{plug} \\
&&\left( \mathcal{E}_{\alpha }^{\limfunc{deep}\limfunc{plug}}\right)
^{2}\equiv \sup_{I=\dot{\cup}I_{r}}\frac{1}{\left\vert I\right\vert _{\sigma
}}\sum_{r=1}^{\infty }\sum_{J\in \mathcal{M}_{\mathbf{r}-\limfunc{deep}%
}\left( I_{r}\right) }\left( \frac{\mathrm{P}^{\alpha }\left( J,\mathbf{1}%
_{I}\sigma \right) }{\left\vert J\right\vert ^{\frac{1}{n}}}\right)
^{2}\left\Vert \mathsf{P}_{J}^{\limfunc{subgood},\omega }\mathbf{x}%
\right\Vert _{L^{2}\left( \omega \right) }^{2}\ ,  \notag \\
&&\left( \mathcal{E}_{\alpha }^{\limfunc{refined}\limfunc{plug}}\right)
^{2}\equiv \sup_{\Omega \mathcal{D}}\sup_{I}\sup_{\ell \geq 0}\frac{1}{%
\left\vert I\right\vert _{\sigma }}\sum_{J\in \mathcal{M}_{\mathbf{r}-%
\limfunc{deep},\Omega \mathcal{D}}^{\ell }\left( I\right) }\left( \frac{%
\mathrm{P}^{\alpha }\left( J,\mathbf{1}_{I}\sigma \right) }{\left\vert
J\right\vert ^{\frac{1}{n}}}\right) ^{2}\left\Vert \mathsf{P}_{J}^{\limfunc{%
subgood},\omega }\mathbf{x}\right\Vert _{L^{2}\left( \omega \right) }^{2}\ .
\notag
\end{eqnarray}%
Recall that we have both%
\begin{eqnarray}
&&\left( \mathcal{E}_{\alpha }^{\limfunc{deep}\limfunc{plug}}\right) ^{2}
\label{plug the hole deep} \\
&\lesssim &\sup_{I=\dot{\cup}I_{r}}\frac{1}{\left\vert I\right\vert _{\sigma
}}\sum_{r=1}^{\infty }\sum_{J\in \mathcal{M}_{\mathbf{r}-\limfunc{deep}%
}\left( I_{r}\right) }\left( \frac{\mathrm{P}^{\alpha }\left( J,\mathbf{1}%
_{I\setminus \gamma J}\sigma \right) }{\left\vert J\right\vert ^{\frac{1}{n}}%
}\right) ^{2}\left\Vert \mathsf{P}_{J}^{\limfunc{subgood},\omega }\mathbf{x}%
\right\Vert _{L^{2}\left( \omega \right) }^{2}  \notag \\
&&+\sup_{I=\dot{\cup}I_{r}}\frac{1}{\left\vert I\right\vert _{\sigma }}%
\sum_{r=1}^{\infty }\sum_{J\in \mathcal{M}_{\mathbf{r}-\limfunc{deep}}\left(
I_{r}\right) }\left( \frac{\mathrm{P}^{\alpha }\left( J,\mathbf{1}_{\gamma
J}\sigma \right) }{\left\vert J\right\vert ^{\frac{1}{n}}}\right)
^{2}\left\Vert \mathsf{P}_{J}^{\limfunc{subgood},\omega }\mathbf{x}%
\right\Vert _{L^{2}\left( \omega \right) }^{2}  \notag \\
&\lesssim &\left( \mathcal{E}_{\alpha }\right) ^{2}+\sup_{I=\dot{\cup}I_{r}}%
\frac{1}{\left\vert I\right\vert _{\sigma }}\sum_{r=1}^{\infty }\sum_{J\in 
\mathcal{M}_{\mathbf{r}-\limfunc{deep}}\left( I_{r}\right) }\left( \frac{%
\left\vert \gamma J\right\vert _{\sigma }}{\left\vert J\right\vert ^{1+\frac{%
1}{n}-\frac{\alpha }{n}}}\right) ^{2}\left\vert J\right\vert ^{\frac{2}{n}%
}\left\vert J\right\vert _{\omega }  \notag \\
&\lesssim &\left( \mathcal{E}_{\alpha }\right) ^{2}+A_{2}^{\alpha }\sup_{I=%
\dot{\cup}I_{r}}\frac{1}{\left\vert I\right\vert _{\sigma }}%
\sum_{r=1}^{\infty }\sum_{J\in \mathcal{M}_{\mathbf{r}-\limfunc{deep}}\left(
I_{r}\right) }\left\vert \gamma J\right\vert _{\sigma }\lesssim \left( 
\mathcal{E}_{\alpha }^{\limfunc{deep}}\right) ^{2}+\beta A_{2}^{\alpha }\ , 
\notag
\end{eqnarray}%
and similarly 
\begin{equation}
\left( \mathcal{E}_{\alpha }^{\limfunc{refined}\limfunc{plug}}\right)
^{2}\lesssim \left( \mathcal{E}_{\alpha }^{\limfunc{refined}}\right)
^{2}+\beta A_{2}^{\alpha }  \label{plug the hole refined}
\end{equation}%
by (\ref{bounded overlap}) and (\ref{bounded overlap'}) respectively.

\begin{notation}
As above, we will typically use the side length $\ell \left( J\right) $ of a 
$\Omega $-quasicube when we are describing collections of quasicubes, and
when we want $\ell \left( J\right) $ to be a dyadic number; while in
estimates we will typically use $\left\vert J\right\vert ^{\frac{1}{n}%
}\approx \ell \left( J\right) $, and when we want to compare powers of
volumes of quasicubes. We will continue to use the prefix `quasi' when
discussing quasicubes, quasiHaar, quasienergy and quasidistance in the text,
but will not use the prefix `quasi' when discussing other notions. Finally,
we will not modify any mathematical symbols to reflect quasinotions, except
for using $\Omega \mathcal{D}$ to denote a quasigrid, and $\limfunc{quasidist%
}\left( E,F\right) \equiv \limfunc{dist}\left( \Omega ^{-1}E,\Omega
^{-1}F\right) $ to denote quasidistance between sets $E$ and $F$, and using $%
\left\vert x-y\right\vert _{\limfunc{quasi}}\equiv \left\vert \Omega
^{-1}x-\Omega ^{-1}y\right\vert $ to denote quasidistance between points $x$
and $y$. This limited use of quasi in the text serves mainly to remind the
reader we are working entirely in the `quasiworld'.
\end{notation}

In the next remark we give a brief description of how and where these
quasienergy conditions will be implemented in the proof.

\begin{remark}
There are two layers of dyadic decomposition in the quasienergy condition;
the outer layer $I=\dot{\cup}I_{r}$ which is essentially arbitrary, and an
inner layer $I_{r}=\overset{\cdot }{\dbigcup\limits_{J\in \mathcal{M}_{%
\mathbf{r}-\limfunc{deep}}\left( I_{r}\right) }}J$ in which the quasicubes $%
J $ are `nicely arranged' within $I_{r}$. Relative to this doubly layered
decomposition we sum the products $\left( \frac{\mathrm{P}^{\alpha }\left( J,%
\mathbf{1}_{I\setminus \gamma J}\sigma \right) }{\left\vert J\right\vert ^{%
\frac{1}{n}}}\right) ^{2}\left\Vert \mathsf{P}_{J}^{\limfunc{subgood},\omega
}\mathbf{x}\right\Vert _{L^{2}\left( \omega \right) }^{2}$, which resemble a
type of $A_{2}^{\alpha }$ expression as defined above. The point of the
outer decomposition is to capture `stopping time quasicubes', which are
essentially arbitrary in this proof, although sometimes restricted to
certain collections of good quasicubes. The point of the\ inner
decomposition is that with $J^{\ast }=\gamma J$ for $J\in \mathcal{M}_{%
\mathbf{r}-\limfunc{deep}}\left( I_{r}\right) $, we have $J^{\ast }\subset
I_{r}$ and we can then write%
\begin{equation*}
\mathrm{P}^{\alpha }\left( J,\mathbf{1}_{I}\sigma \right) =\mathrm{P}%
^{\alpha }\left( J,\mathbf{1}_{J^{\ast }}\sigma \right) +\mathrm{P}^{\alpha
}\left( J,\mathbf{1}_{I\setminus J^{\ast }}\sigma \right) ,
\end{equation*}%
and use that $\left\Vert \mathsf{P}_{J}^{\limfunc{subgood},\omega }\mathbf{x}%
\right\Vert _{L^{2}\left( \omega \right) }^{2}=\left\Vert \mathsf{P}_{J}^{%
\limfunc{subgood},\omega }\left( \mathbf{x}-\mathbf{c}_{J}\right)
\right\Vert _{L^{2}\left( \omega \right) }^{2}\leq \left\vert J\right\vert ^{%
\frac{2}{n}}\left\vert J\right\vert _{\omega }$ to estimate the product
involving $\mathbf{1}_{J^{\ast }}\sigma $ by%
\begin{equation*}
\left( \frac{\mathrm{P}^{\alpha }\left( J,\mathbf{1}_{J^{\ast }}\sigma
\right) }{\left\vert J\right\vert ^{\frac{1}{n}}}\right) ^{2}\left\Vert 
\mathsf{P}_{J}^{\limfunc{subgood},\omega }\mathbf{x}\right\Vert
_{L^{2}\left( \omega \right) }^{2}\lesssim \left( \frac{\left\vert J^{\ast
}\right\vert ^{\frac{\alpha }{n}-1}\left\vert J^{\ast }\right\vert _{\sigma }%
}{\left\vert J\right\vert ^{\frac{1}{n}}}\right) ^{2}\left\vert J\right\vert
^{\frac{2}{n}}\left\vert J\right\vert _{\omega }\lesssim A_{2}^{\alpha
}\left\vert J^{\ast }\right\vert _{\sigma }\ ,
\end{equation*}%
to which we apply the bounded overlap property (\ref{bounded overlap}),
while the remaining product involving $\mathbf{1}_{I\setminus J^{\ast
}}\sigma $,%
\begin{equation*}
\left( \frac{\mathrm{P}^{\alpha }\left( J,\mathbf{1}_{I\setminus J^{\ast
}}\sigma \right) }{\left\vert J\right\vert ^{\frac{1}{n}}}\right)
^{2}\left\Vert \mathsf{P}_{J}^{\limfunc{subgood},\omega }\mathbf{x}%
\right\Vert _{L^{2}\left( \omega \right) }^{2}\ ,
\end{equation*}%
has a `hole' in the support of $\mathbf{1}_{I\setminus J^{\ast }}\sigma $
that contains the support of $\omega $ in the quasicube $J$ well inside the
hole, and moreover these holes are `nicely arranged' within $I_{r}$. Of
particular importance is that for pairwise disjoint subquasicubes $J^{\prime
}\subset J$, the projections $\left\Vert \mathsf{P}_{J^{\prime }}^{\limfunc{%
subgood},\omega }\mathbf{x}\right\Vert _{L^{2}\left( \omega \right) }^{2}$
are additive, and the Poisson ratios are essentially constant $\frac{\mathrm{%
P}^{\alpha }\left( J^{\prime },\mathbf{1}_{I\setminus J^{\ast }}\sigma
\right) }{\left\vert J^{\prime }\right\vert ^{\frac{1}{n}}}\approx \frac{%
\mathrm{P}^{\alpha }\left( J,\mathbf{1}_{I\setminus J^{\ast }}\sigma \right) 
}{\left\vert J\right\vert ^{\frac{1}{n}}}$. The \emph{deep} quasienergy
condition suffices for all arguments in the proof except for bounding the
two testing conditions for the Poisson operator $\mathbb{P}$, in which case
we also use the \emph{refined} quasienergy condition - see Lemma \ref%
{refined lemma} below.
\end{remark}

\subsection{Statement of the Theorem}

We can now state our main quasicube two weight theorem. Recall that $\Omega :%
\mathbb{R}^{n}\rightarrow \mathbb{R}^{n}$ is a globally biLipschitz map, and
that $\Omega \mathcal{P}^{n}$ denotes the collection of all quasicubes in $%
\mathbb{R}^{n}$ whose preimages under $\Omega $ are usual cubes with sides
parallel to the coordinate axes. Denote by $\Omega \mathcal{D}^{n}\subset
\Omega \mathcal{P}^{n}$ a dyadic quasigrid in $\mathbb{R}^{n}$.

\begin{theorem}
\label{T1 theorem}Suppose that $T^{\alpha }$ is a standard $\alpha $%
-fractional Calder\'{o}n-Zygmund operator on $\mathbb{R}^{n}$, and that $%
\omega $ and $\sigma $ are positive Borel measures on $\mathbb{R}^{n}$
without common point masses. Set $T_{\sigma }^{\alpha }f=T^{\alpha }\left(
f\sigma \right) $ for any smooth truncation of $T_{\sigma }^{\alpha }$.

\begin{enumerate}
\item Suppose $0\leq \alpha <n$ and that $\gamma \geq 2$ is given. Then the
operator $T_{\sigma }^{\alpha }$ is bounded from $L^{2}\left( \sigma \right) 
$ to $L^{2}\left( \omega \right) $, i.e. 
\begin{equation}
\left\Vert T_{\sigma }^{\alpha }f\right\Vert _{L^{2}\left( \omega \right)
}\leq \mathfrak{N}_{T_{\sigma }^{\alpha }}\left\Vert f\right\Vert
_{L^{2}\left( \sigma \right) },  \label{two weight}
\end{equation}%
uniformly in smooth truncations of $T^{\alpha }$, and moreover%
\begin{equation*}
\mathfrak{N}_{T_{\sigma }^{\alpha }}\leq C_{\alpha }\left( \sqrt{\mathcal{A}%
_{2}^{\alpha }+\mathcal{A}_{2}^{\alpha ,\ast }}+\mathfrak{T}_{T^{\alpha }}+%
\mathfrak{T}_{T^{\alpha }}^{\ast }+\mathcal{E}_{\alpha }+\mathcal{E}_{\alpha
}^{\ast }+\mathcal{WBP}_{T^{\alpha }}\right) ,
\end{equation*}%
provided that the two dual $\mathcal{A}_{2}^{\alpha }$ conditions hold, and
the two dual quasitesting conditions for $T^{\alpha }$ hold, the quasiweak
boundedness property for $T^{\alpha }$ holds for a sufficiently large
constant $C$ depending on the goodness parameter $\mathbf{r}$, and provided
that the two dual quasienergy conditions $\mathcal{E}_{\alpha }+\mathcal{E}%
_{\alpha }^{\ast }<\infty $ hold uniformly over all dyadic quasigrids $%
\Omega \mathcal{D}\subset \Omega \mathcal{P}^{n}$, and where the goodness
parameters $\mathbf{r}$ and $\varepsilon $ implicit in the definition of $%
\mathcal{M}_{\mathbf{r}-\limfunc{deep},\Omega \mathcal{D}}^{\ell }\left(
K\right) $ are fixed sufficiently large and small respectively depending on $%
n$, $\alpha $ and $\gamma $.

\item Conversely, suppose $0\leq \alpha <n$ and that $\mathbf{T}^{\alpha
}=\left\{ T_{j}^{\alpha }\right\} _{j=1}^{J}$ is a vector of Calder\'{o}%
n-Zygmund operators with standard kernels $\left\{ K_{j}^{\alpha }\right\}
_{j=1}^{J}$. In the range $0\leq \alpha <\frac{n}{2}$, we assume the
following \emph{ellipticity} condition: there is $c>0$ such that for \emph{%
each} unit vector $\mathbf{u}$ there is $j$ satisfying 
\begin{equation}
\left\vert K_{j}^{\alpha }\left( x,x+t\mathbf{u}\right) \right\vert \geq
ct^{\alpha -n},\ \ \ \ \ t\in \mathbb{R}.  \label{Ktalpha}
\end{equation}%
For the range $\frac{n}{2}\leq \alpha <n$, we asume the following \emph{%
strong ellipticity} condition: for each $m\in \left\{ 1,-1\right\} ^{n}$,
there is a sequence of coefficients $\left\{ \lambda _{j}^{m}\right\}
_{j=1}^{J}$ such that%
\begin{equation}
\left\vert \sum_{j=1}^{J}\lambda _{j}^{m}K_{j}^{\alpha }\left( x,x+t\mathbf{u%
}\right) \right\vert \geq ct^{\alpha -n},\ \ \ \ \ t\in \mathbb{R}.
\label{Ktalpha strong}
\end{equation}%
holds for \emph{all} unit vectors $\mathbf{u}$ in the $n$-ant 
\begin{equation*}
V_{m}=\left\{ x\in \mathbb{R}^{n}:m_{i}x_{i}>0\text{ for }1\leq i\leq
n\right\} ,\ \ \ \ \ m\in \left\{ 1,-1\right\} ^{n}.
\end{equation*}%
Furthermore, assume that each operator $T_{j}^{\alpha }$ is bounded from $%
L^{2}\left( \sigma \right) $ to $L^{2}\left( \omega \right) $, 
\begin{equation*}
\left\Vert \left( T_{j}^{\alpha }\right) _{\sigma }f\right\Vert
_{L^{2}\left( \omega \right) }\leq \mathfrak{N}_{T_{j}^{\alpha }}\left\Vert
f\right\Vert _{L^{2}\left( \sigma \right) }.
\end{equation*}%
Then the fractional $\mathcal{A}_{2}^{\alpha }$ condition holds, and
moreover,%
\begin{equation*}
\sqrt{\mathcal{A}_{2}^{\alpha }+\mathcal{A}_{2}^{\alpha ,\ast }}\leq C%
\mathfrak{N}_{\mathbf{T}^{\alpha }}.
\end{equation*}
\end{enumerate}
\end{theorem}

\begin{problem}
Given any strongly elliptic vector $\mathbf{T}^{\alpha }$ of classical $%
\alpha $-fractional Calder\'{o}n-Zygmund operators, it is an open question
whether or not the usual (quasi or not) energy conditions are necessary for
boundedness of $\mathbf{T}^{\alpha }$. See \cite{SaShUr4} for a failure of 
\emph{energy reversal} in higher dimensions - such an energy reversal was
used in dimension $n=1$ to prove the necessity of the energy condition for
the Hilbert transform.
\end{problem}

\begin{remark}
The boundedness of an individual operator $T^{\alpha }$ cannot in general
imply the finiteness of either $A_{2}^{\alpha }$ or $\mathcal{E}_{\alpha }$.
For a trivial example, if $\sigma $ and $\omega $ are supported on the $x$%
-axis in the plane, then the second Riesz tranform $R_{2}$ is the zero
operator from $L^{2}\left( \sigma \right) $ to $L^{2}\left( \omega \right) $%
, simply because the kernel $K_{2}\left( x,y\right) $ of $R_{2}$ satisfies $%
K_{2}\left( \left( x_{1},0\right) ,\left( y_{1},0\right) \right) =\frac{0-0}{%
\left\vert x_{1}-y_{1}\right\vert ^{3-\alpha }}=0$.
\end{remark}

\begin{remark}
\label{surgery}In \cite{LaWi}, in the setting of usual (nonquasi) cubes, M.
Lacey and B. Wick use the NTV technique of surgery to show that the weak
boundedness property for the Riesz transform vector $\mathbf{R}^{\alpha ,n}$
is implied by the $\mathcal{A}_{2}^{\alpha }$ and testing conditions, and
this has the consequence of eliminating the weak boundedness property as a
condition. Their proof of this implication extends to the more general
operators $T^{\alpha }$ and quasicubes considered here, and so the quasiweak
boundedness property can be dropped from the statement of Theorem \ref{T1
theorem}. In any event, the weak boundedness property is necessary for the
norm inequality, and as such can be viewed as a weak cousin of the testing
conditions.
\end{remark}

\section{Proof of Theorem \protect\ref{T1 theorem}}

We now give the proof of Theorem \ref{T1 theorem} in the following 8
sections. Using the analogue for dyadic quasigrids of the good random grids
of Nazarov, Treil and Volberg, a standard argument of NTV, see e.g. \cite%
{Vol}, reduces the two weight inequality (\ref{2 weight}) for $T^{\alpha }$
to proving boundedness of a bilinear form $\mathcal{T}^{\alpha }\left(
f,g\right) $ with uniform constants over dyadic quasigrids, and where the
quasiHaar supports $\limfunc{supp}\widehat{f}$ and $\limfunc{supp}\widehat{g}
$ of the functions $f$ and $g$ are contained in the collection $\Omega 
\mathcal{D}^{\limfunc{good}}$ of good quasicubes, whose children are all
good as well, with goodness parameters $\mathbf{r<\infty }$ and $\varepsilon
>0$ chosen sufficiently large and small respectively. Here the quasiHaar
support of $f$ is $\limfunc{supp}\widehat{f}\equiv \left\{ I\in \Omega 
\mathcal{D}:\bigtriangleup _{I}^{\sigma }f\neq 0\right\} $, and similarly
for $g$. We may also assume that $\left\vert \partial Q\right\vert _{\sigma
+\omega }=0$ for all dyadic quasicubes $Q$ in the grids $\Omega \mathcal{D}$
we consider, since this property holds with probability $1$ for random grids 
$\Omega \mathcal{D}$.

In fact we can assume even more, namely that the quasiHaar supports $%
\limfunc{supp}\widehat{f}$ and $\limfunc{supp}\widehat{g}$ of $f$ and $g$
are contained in the collection of $\mathbf{\tau }$\emph{-good} quasicubes%
\begin{equation}
\Omega \mathcal{D}_{\left( \mathbf{r},\varepsilon \right) -\limfunc{good}}^{%
\mathbf{\tau }}\equiv \left\{ K\in \Omega \mathcal{D}:\mathfrak{C}%
_{K}\subset \Omega \mathcal{D}_{\left( \mathbf{r},\varepsilon \right) -%
\limfunc{good}}\text{ and }\pi _{\Omega \mathcal{D}}^{\ell }K\in \Omega 
\mathcal{D}_{\left( \mathbf{r},\varepsilon \right) -\limfunc{good}}\text{
for all }0\leq \ell \leq \mathbf{\tau }\right\} ,  \label{extended good grid}
\end{equation}%
that are $\left( \mathbf{r},\varepsilon \right) $-$\limfunc{good}$ and whose
children are also $\left( \mathbf{r},\varepsilon \right) $-$\limfunc{good}$,
and whose $\ell $-parents up to level $\mathbf{\tau }$\ are also $\left( 
\mathbf{r},\varepsilon \right) $-$\limfunc{good}$. Here $\mathbf{\tau }>%
\mathbf{r}$ is a parameter to be fixed in Definition \ref{def parameters}
below. We may assume this restriction on the quasiHaar supports of $f$ and $%
g $ by choosing $\left( \mathbf{r},\varepsilon \right) $ appropriately and
using the following lemma.

\begin{lemma}
\label{better good}Given $\mathbf{s}\geq 1$, $\mathbf{t}\geq 2$ and $%
0<\varepsilon <1$, we have 
\begin{equation*}
\Omega \mathcal{D}_{\left( \mathbf{s}+\mathbf{t},\varepsilon \right) -%
\limfunc{good}}^{\mathbf{s}}\subset \Omega \mathcal{D}_{\left( \mathbf{t}%
,\delta \right) -\limfunc{good}}\ ,
\end{equation*}%
provided%
\begin{equation*}
\mathbf{s}\varepsilon <\mathbf{t}\left( 1-\varepsilon \right) -2\text{ and }%
\delta =\varepsilon +\frac{\mathbf{s}\varepsilon +1}{\mathbf{t}}.
\end{equation*}
\end{lemma}

\begin{proof}
Fix goodness parameters $\mathbf{r}=\mathbf{s}+\mathbf{t}$ and $\varepsilon $%
, and suppose that $\mathbf{s}<\mathbf{r}\left( 1-\varepsilon \right) -2$.
Choose a good quasicube $I$ and a superquasicube $K$ with $\ell \left(
I\right) \leq 2^{-\mathbf{r}}\ell \left( K\right) $. Set $J\equiv \pi ^{%
\mathbf{s}}I$. Then we have%
\begin{equation*}
J=\pi ^{s}I\subset K\text{ and }\ell \left( J\right) \leq 2^{-\mathbf{t}%
}\ell \left( K\right) .
\end{equation*}%
Because $I$ is good we have%
\begin{equation*}
\limfunc{dist}\left( I,K^{c}\right) \geq \frac{1}{2}\ell \left( I\right)
^{\varepsilon }\ell \left( K\right) ^{1-\varepsilon },
\end{equation*}%
and hence also%
\begin{eqnarray*}
\limfunc{dist}\left( J,K^{c}\right) &\geq &\limfunc{dist}\left(
I,K^{c}\right) -\ell \left( J\right) \geq \frac{1}{2}\ell \left( I\right)
^{\varepsilon }\ell \left( K\right) ^{1-\varepsilon }-2^{\mathbf{s}}\ell
\left( I\right) \\
&=&\frac{1}{2}\ell \left( I\right) ^{\varepsilon }\ell \left( K\right)
^{1-\varepsilon }\left\{ 1-2^{1+\mathbf{s}}\left( \frac{\ell \left( I\right) 
}{\ell \left( K\right) }\right) ^{1-\varepsilon }\right\} \geq \frac{1}{4}%
\ell \left( I\right) ^{\varepsilon }\ell \left( K\right) ^{1-\varepsilon }
\end{eqnarray*}%
which follows from $\ell \left( I\right) \leq 2^{-\mathbf{r}}\ell \left(
K\right) $ provided we take $2^{1+\mathbf{s}}2^{-\mathbf{r}\left(
1-\varepsilon \right) }\leq \frac{1}{2}$, i.e.%
\begin{equation*}
\mathbf{s}<\mathbf{r}\left( 1-\varepsilon \right) -2.
\end{equation*}%
Finally we choose $\delta >\varepsilon $ so that%
\begin{equation*}
\frac{1}{4}\ell \left( I\right) ^{\varepsilon }\ell \left( K\right)
^{1-\varepsilon }=2^{-\mathbf{s}\varepsilon -2}\ell \left( J\right)
^{\varepsilon }\ell \left( K\right) ^{1-\varepsilon }\geq \frac{1}{2}\ell
\left( J\right) ^{\delta }\ell \left( K\right) ^{1-\delta },\ \ \ \ \ \text{%
when }\ell \left( J\right) \leq 2^{-\mathbf{t}}\ell \left( K\right) ,
\end{equation*}%
which follows if we choose $\delta $ to satisfy%
\begin{eqnarray*}
2^{-\mathbf{s}\varepsilon -2}\left( 2^{-\mathbf{t}}\ell \left( K\right)
\right) ^{\varepsilon }\ell \left( K\right) ^{1-\varepsilon } &=&\frac{1}{2}%
\left( 2^{-\mathbf{t}}\ell \left( K\right) \right) ^{\delta }\ell \left(
K\right) ^{1-\delta }; \\
2^{-\mathbf{s}\varepsilon -2} &=&\frac{1}{2}\left( 2^{-\mathbf{t}}\right)
^{\delta -\varepsilon }; \\
-\mathbf{s}\varepsilon -1 &=&-\mathbf{t}\left( \delta -\varepsilon \right) ;
\\
\delta &=&\varepsilon +\frac{\mathbf{s}\varepsilon +1}{\mathbf{t}}.
\end{eqnarray*}
\end{proof}

For convenience in notation we will sometimes suppress the dependence on $%
\alpha $ in our nonlinear forms, but will retain it in the operators,
Poisson integrals and constants. More precisely, let $\Omega \mathcal{D}%
^{\sigma }=\Omega \mathcal{D}^{\omega }$ be an $\left( \mathbf{r}%
,\varepsilon \right) $-good quasigrid on $\mathbb{R}^{n}$, and let $\left\{
h_{I}^{\sigma ,a}\right\} _{I\in \Omega \mathcal{D}^{\sigma },\ a\in \Gamma
_{n}}$ and $\left\{ h_{J}^{\omega ,b}\right\} _{J\in \Omega \mathcal{D}%
^{\omega },\ b\in \Gamma _{n}}$ be corresponding quasiHaar bases as
described above, so that%
\begin{eqnarray*}
f &=&\sum_{I\in \Omega \mathcal{D}^{\sigma }}\bigtriangleup _{I}^{\sigma
}f=\sum_{I\in \Omega \mathcal{D}^{\sigma },\text{\ }a\in \Gamma _{n}\text{ }%
}\left\langle f,h_{I}^{\sigma ,a}\right\rangle \ h_{I}^{\sigma
,a}=\sum_{I\in \Omega \mathcal{D}^{\sigma },\text{\ }a\in \Gamma _{n}}%
\widehat{f}\left( I;a\right) \ h_{I}^{\sigma ,a}, \\
g &=&\sum_{J\in \Omega \mathcal{D}^{\omega }\text{ }}\bigtriangleup
_{J}^{\omega }g=\sum_{J\in \Omega \mathcal{D}^{\omega },\text{\ }b\in \Gamma
_{n}}\left\langle g,h_{J}^{\omega ,b}\right\rangle \ h_{J}^{\omega
,b}=\sum_{J\in \Omega \mathcal{D}^{\omega },\text{\ }b\in \Gamma _{n}}%
\widehat{g}\left( J;b\right) \ h_{J}^{\omega ,b},
\end{eqnarray*}%
where the appropriate measure is understood in the notation $\widehat{f}%
\left( I;a\right) $ and $\widehat{g}\left( J;b\right) $, and where these
quasiHaar coefficients $\widehat{f}\left( I;a\right) $ and $\widehat{g}%
\left( J;b\right) $ vanish if the quasicubes $I$ and $J$ are not good.
Inequality (\ref{two weight}) is equivalent to boundedness of the bilinear
form%
\begin{equation*}
\mathcal{T}^{\alpha }\left( f,g\right) \equiv \left\langle T_{\sigma
}^{\alpha }\left( f\right) ,g\right\rangle _{\omega }=\sum_{I\in \Omega 
\mathcal{D}^{\sigma }\text{ and }J\in \Omega \mathcal{D}^{\omega
}}\left\langle T_{\sigma }^{\alpha }\left( \bigtriangleup _{I}^{\sigma
}f\right) ,\bigtriangleup _{J}^{\omega }g\right\rangle _{\omega }
\end{equation*}%
on $L^{2}\left( \sigma \right) \times L^{2}\left( \omega \right) $, i.e.%
\begin{equation*}
\left\vert \mathcal{T}^{\alpha }\left( f,g\right) \right\vert \leq \mathfrak{%
N}_{T^{\alpha }}\left\Vert f\right\Vert _{L^{2}\left( \sigma \right)
}\left\Vert g\right\Vert _{L^{2}\left( \omega \right) },
\end{equation*}%
uniformly over all quasigrids. We may assume the two quasigrids $\Omega 
\mathcal{D}^{\sigma }$ and $\Omega \mathcal{D}^{\omega }$ are equal here,
and this we will do throughout the paper, although we sometimes continue to
use the measure as a superscript on $\Omega \mathcal{D}$ for clarity of
exposition. Roughly speaking, we analyze the form $\mathcal{T}^{\alpha
}\left( f,g\right) $ by splitting it in a nonlinear way into three main
pieces, following in part the approach in \cite{LaSaShUr2} and \cite%
{LaSaShUr3}. The first piece consists of quasicubes $I$ and $J$ that are
either disjoint or of comparable side length, and this piece is handled
using the section on preliminaries of NTV type. The second piece consists of
quasicubes $I$ and $J$ that overlap, but are `far apart' in a nonlinear way,
and this piece is handled using the sections on the Intertwining Proposition
and the control of the functional quasienergy condition by the quasienergy
condition. Finally, the remaining local piece where the overlapping
quasicubes are `close' is handled by generalizing methods of NTV as in \cite%
{LaSaShUr}, and then splitting the stopping form into two sublinear stopping
forms, one of which is handled using techniques of \cite{LaSaUr2}, and the
other using the stopping time and recursion of M. Lacey \cite{Lac}. See the
schematic diagram in Subsection 8.4 below.

\section{Necessity of the $\mathcal{A}_{2}^{\protect\alpha }$ conditions}

Because of the remarks made earlier showing the equivalence of various $%
\mathcal{A}_{2}^{\alpha }$ conditions, we can restrict our attention in this
section only to usual cubes, and then the results transfer automatically to
quasicubes. Here we prove in particular the necessity of the fractional $%
\mathcal{A}_{2}^{\alpha }$ condition when $0\leq \alpha <n$, for the $\alpha 
$-fractional Riesz vector transform $\mathbf{R}^{\alpha }$ defined by%
\begin{equation*}
\mathbf{R}^{\alpha }\left( f\sigma \right) \left( x\right) =\int_{\mathbb{R}%
^{n}}R_{j}^{\alpha }(x,y)f\left( y\right) d\sigma \left( y\right) ,\ \ \ \ \
K_{j}^{\alpha }\left( x,y\right) =\frac{x^{j}-y^{j}}{\left\vert
x-y\right\vert ^{n+1-\alpha }},
\end{equation*}%
whose kernel $K_{j}^{\alpha }\left( x,y\right) $ satisfies (\ref%
{sizeandsmoothness'}) for $0\leq \alpha <n$. Parts of the following argument
are taken from unpublished material obtained in joint work with M. Lacey.

\begin{lemma}
\label{necc frac A2}Suppose $0\leq \alpha <n$. Let $T^{\alpha }$ be any
collection of operators with $\alpha $-standard fractional kernel
satisfying\ the ellipticity condition (\ref{Ktalpha}), and in the case $%
\frac{n}{2}\leq \alpha <n$, we also assume the more restrictive condition (%
\ref{Ktalpha strong}). Then for $0\leq \alpha <n$ we have%
\begin{equation*}
\mathcal{A}_{2}^{\alpha }\lesssim \mathfrak{N}_{\alpha }\left( T^{\alpha
}\right) .
\end{equation*}
\end{lemma}

\begin{remark}
Cancellation properties of $T^{\alpha }$ play no role in the proof below.
Indeed the proof shows that $\mathcal{A}_{2}^{\alpha }$ is dominated by the
best constant $C$ in the restricted inequality%
\begin{equation*}
\left\Vert \chi _{E}T^{\alpha }(f\sigma )\right\Vert _{L^{2,\infty }\left(
\omega \right) }\leq C\left\Vert f\right\Vert _{L^{2}\left( \sigma \right)
},\ \ \ \ \ E=\mathbb{R}^{n}\setminus supp\ f.
\end{equation*}
\end{remark}

\begin{proof}
First we give the proof for the case when $T^{\alpha }$ is the $\alpha $%
-fractional Riesz transform $\mathbf{R}^{\alpha }$, whose kernel is $\mathbf{%
K}^{\alpha }\left( x,y\right) =\frac{x-y}{\left\vert x-y\right\vert
^{n+1-\alpha }}$ . Define the $2^{n}$ generalized $n$-ants $\mathcal{Q}_{m}$
for $m\in \left\{ -1,1\right\} ^{n}$, and their translates $\mathcal{Q}%
_{m}\left( w\right) $ for $w\in \mathbb{R}^{n}$ by 
\begin{eqnarray*}
\mathcal{Q}_{m} &=&\left\{ \left( x_{1},...,x_{n}\right)
:m_{k}x_{k}>0\right\} , \\
\mathcal{Q}_{m}\left( w\right) &=&\left\{ z:z-w\in \mathcal{Q}_{m}\right\}
,\ \ \ \ \ w\in \mathbb{R}^{n}.
\end{eqnarray*}%
Fix $m\in \left\{ -1,1\right\} ^{n}$ and a cube $I$. For $a\in \mathbb{R}%
^{n} $ and $r>0$ let 
\begin{eqnarray*}
s_{I}\left( x\right) &=&\frac{\ell \left( I\right) }{\ell \left( I\right)
+\left\vert x-\zeta _{I}\right\vert }, \\
f_{a,r}\left( y\right) &=&\mathbf{1}_{\mathcal{Q}_{-m}\left( a\right) \cap
B\left( 0,r\right) }\left( y\right) s_{I}\left( y\right) ^{n-\alpha },
\end{eqnarray*}%
where $\zeta _{I}$ is the center of the cube $I$. Now%
\begin{eqnarray*}
\ell \left( I\right) \left\vert x-y\right\vert &\leq &\ell \left( I\right)
\left\vert x-\zeta _{I}\right\vert +\ell \left( I\right) \left\vert \zeta
_{I}-y\right\vert \\
&\leq &\left[ \ell \left( I\right) +\left\vert x-\zeta _{I}\right\vert %
\right] \left[ \ell \left( I\right) +\left\vert \zeta _{I}-y\right\vert %
\right]
\end{eqnarray*}%
implies%
\begin{equation*}
\frac{1}{\left\vert x-y\right\vert }\geq \frac{1}{\ell \left( I\right) }%
s_{I}\left( x\right) s_{I}\left( y\right) ,\ \ \ \ \ x,y\in \mathbb{R}^{n}.
\end{equation*}%
Now the key observation is that with $L\zeta \equiv m\cdot \zeta $, we have%
\begin{equation*}
L\left( x-y\right) =m\cdot \left( x-y\right) \geq \left\vert x-y\right\vert
,\ \ \ \ \ x\in \mathcal{Q}_{m}\left( y\right) ,
\end{equation*}%
which yields%
\begin{eqnarray}
L\left( \mathbf{K}^{\alpha }\left( x,y\right) \right) &=&\frac{L\left(
x-y\right) }{\left\vert x-y\right\vert ^{n+1-\alpha }}
\label{key separation} \\
&\geq &\frac{1}{\left\vert x-y\right\vert ^{n-\alpha }}\geq \ell \left(
I\right) ^{\alpha -n}s_{I}\left( x\right) ^{n-\alpha }s_{I}\left( y\right)
^{n-\alpha },  \notag
\end{eqnarray}%
provided $x\in \mathcal{Q}_{+,+}\left( y\right) $. Now we note that $x\in 
\mathcal{Q}_{m}\left( y\right) $ when $x\in \mathcal{Q}_{m}\left( a\right) $
and $y\in \mathcal{Q}_{-m}\left( a\right) $ to obtain that for $x\in 
\mathcal{Q}_{m}\left( a\right) $, 
\begin{eqnarray*}
L\left( T^{\alpha }\left( f_{a,r}\sigma \right) \left( x\right) \right)
&=&\int_{\mathcal{Q}_{-m}\left( a\right) \cap B\left( 0,r\right) }\frac{%
L\left( x-y\right) }{\left\vert x-y\right\vert ^{n+1-\alpha }}s_{I}\left(
y\right) d\sigma \left( y\right) \\
&\geq &\ell \left( I\right) ^{\alpha -n}s_{I}\left( x\right) ^{n-\alpha
}\int_{\mathcal{Q}_{-m}\left( a\right) \cap B\left( 0,r\right) }s_{I}\left(
y\right) ^{2n-2\alpha }d\sigma \left( y\right) .
\end{eqnarray*}

Applying $\left\vert L\zeta \right\vert \leq \sqrt{n}\left\vert \zeta
\right\vert $ and our assumed two weight inequality for the fractional Riesz
transform, we see that for $r>0$ large, 
\begin{align*}
& \ell \left( I\right) ^{2\alpha -2n}\int_{\mathcal{Q}_{m}\left( a\right)
}s_{I}\left( x\right) ^{2n-2\alpha }\left( \int_{\mathcal{Q}_{-m}\left(
a\right) \cap B\left( 0,r\right) }s_{I}\left( y\right) ^{2n-2\alpha }d\sigma
\left( y\right) \right) ^{2}d\omega \left( x\right) \\
& \leq \left\Vert LT(\sigma f_{a,r})\right\Vert _{L^{2}(\omega
)}^{2}\lesssim \mathfrak{N}_{\alpha }\left( \mathbf{R}^{\alpha }\right)
^{2}\left\Vert f_{a,r}\right\Vert _{L^{2}(\sigma )}^{2}=\mathfrak{N}_{\alpha
}\left( \mathbf{R}^{\alpha }\right) ^{2}\int_{\mathcal{Q}_{-m}\left(
a\right) \cap B\left( 0,r\right) }s_{I}\left( y\right) ^{2n-2\alpha }d\sigma
\left( y\right) .
\end{align*}

Rearranging the last inequality, we obtain%
\begin{equation*}
\ell \left( I\right) ^{2\alpha -2n}\int_{\mathcal{Q}_{m}\left( a\right)
}s_{I}\left( x\right) ^{2n-2\alpha }d\omega \left( x\right) \int_{\mathcal{Q}%
_{-m}\left( a\right) \cap B\left( 0,r\right) }s_{I}\left( y\right)
^{2n-2\alpha }d\sigma \left( y\right) \lesssim \mathfrak{N}_{\alpha }\left( 
\mathbf{R}^{\alpha }\right) ^{2},
\end{equation*}%
and upon letting $r\rightarrow \infty $, 
\begin{equation*}
\int_{\mathcal{Q}_{m}\left( a\right) }\frac{\ell \left( I\right) ^{2-\alpha }%
}{\left( \ell \left( I\right) +\left\vert x-\zeta _{I}\right\vert \right)
^{4-2\alpha }}d\omega \left( x\right) \int_{\mathcal{Q}_{-m}\left( a\right) }%
\frac{\ell \left( I\right) ^{2-\alpha }}{\left( \ell \left( I\right)
+\left\vert y-\zeta _{I}\right\vert \right) ^{4-2\alpha }}d\sigma \left(
y\right) \lesssim \mathfrak{N}_{\alpha }\left( \mathbf{R}^{\alpha }\right)
^{2}.
\end{equation*}%
Note that the ranges of integration above are pairs of opposing $n$-ants.

Fix a cube $Q$, which without loss of generality can be taken to be centered
at the origin, $\zeta _{Q}=0$. Then choose $a=\left( 2\ell \left( Q\right)
,2\ell \left( Q\right) \right) $ and $I=Q$ so that we have%
\begin{eqnarray*}
&&\left( \int_{\mathcal{Q}_{m}\left( a\right) }\frac{\ell \left( Q\right)
^{n-\alpha }}{\left( \ell \left( Q\right) +\left\vert x\right\vert \right)
^{2n-2\alpha }}d\omega \left( x\right) \right) \left( \ell \left( Q\right)
^{\alpha -n}\int_{Q}d\sigma \right) \\
&\leq &C_{\alpha }\int_{\mathcal{Q}_{m}\left( a\right) }\frac{\ell \left(
Q\right) ^{n-\alpha }}{\left( \ell \left( Q\right) +\left\vert x\right\vert
\right) ^{2n-2\alpha }}d\omega \left( x\right) \int_{\mathcal{Q}_{-m}\left(
a\right) }\frac{\ell \left( Q\right) ^{n-\alpha }}{\left( \ell \left(
Q\right) +\left\vert y\right\vert \right) ^{2n-2\alpha }}d\sigma \left(
y\right) \lesssim \mathfrak{N}_{\alpha }\left( \mathbf{R}^{\alpha }\right)
^{2}.
\end{eqnarray*}%
Now fix $m=\left( 1,1,...,1\right) $ and note that there is a fixed $N$
(independent of $\ell \left( Q\right) $) and a fixed collection of rotations 
$\left\{ \rho _{k}\right\} _{k=1}^{N}$, such that the rotates $\rho _{k}%
\mathcal{Q}_{m}\left( a\right) $, $1\leq k\leq N$, of the $n$-ant $\mathcal{Q%
}_{m}\left( a\right) $ cover the complement of the ball $B\left( 0,4\sqrt{n}%
\ell \left( Q\right) \right) $: 
\begin{equation*}
B\left( 0,4\sqrt{n}\ell \left( Q\right) \right) ^{c}\subset
\bigcup_{k=1}^{N}\rho _{k}\mathcal{Q}_{m}\left( a\right) .
\end{equation*}%
Then we obtain, upon applying the same argument to these rotated pairs of $n$%
-ants, 
\begin{equation}
\left( \int_{B\left( 0,4\sqrt{n}\ell \left( Q\right) \right) ^{c}}\frac{\ell
\left( Q\right) ^{n-\alpha }}{\left( \ell \left( Q\right) +\left\vert
x\right\vert \right) ^{2n-2\alpha }}d\omega \left( x\right) \right) \left(
\ell \left( Q\right) ^{\alpha -n}\int_{Q}d\sigma \right) \lesssim \mathfrak{N%
}_{\alpha }\left( \mathbf{R}^{\alpha }\right) ^{2}.  \label{prelim A2}
\end{equation}

Now we assume for the moment the tailless $A_{2}^{\alpha }$ condition 
\begin{equation*}
\ell \left( Q^{\prime }\right) ^{2\left( \alpha -n\right) }\left(
\int_{Q^{\prime }}d\omega \right) \left( \int_{Q^{\prime }}d\sigma \right)
\leq A_{2}^{\alpha }.
\end{equation*}%
If we use this with $Q^{\prime }=4\sqrt{n}Q$, together with (\ref{prelim A2}%
), we obtain%
\begin{equation*}
\left( \int \frac{\ell \left( Q\right) ^{n-\alpha }}{\left( \ell \left(
Q\right) +\left\vert x\right\vert \right) ^{2n-2\alpha }}d\omega \left(
x\right) \right) ^{\frac{1}{2}}\left( \ell \left( Q\right) ^{\alpha
-n}\int_{Q}d\sigma \right) ^{\frac{1}{2}}\lesssim \mathfrak{N}_{\alpha
}\left( \mathbf{R}^{\alpha }\right)
\end{equation*}%
or%
\begin{equation*}
\ell \left( Q\right) ^{\alpha }\left( \frac{1}{\left\vert Q\right\vert }\int 
\frac{1}{\left( 1+\frac{\left\vert x-\zeta _{Q}\right\vert }{\ell \left(
Q\right) }\right) ^{2n-2\alpha }}d\omega \left( x\right) \right) ^{\frac{1}{2%
}}\left( \frac{1}{\left\vert Q\right\vert }\int_{Q}d\sigma \right) ^{\frac{1%
}{2}}\lesssim \mathfrak{N}_{\alpha }\left( \mathbf{R}^{\alpha }\right) .
\end{equation*}%
Clearly we can reverse the roles of the measures $\omega $ and $\sigma $ and
obtain 
\begin{equation*}
\mathcal{A}_{2}^{\alpha }\lesssim \mathfrak{N}_{\alpha }\left( \mathbf{R}%
^{\alpha }\right) +A_{2}^{\alpha }
\end{equation*}%
for the kernels $\mathbf{K}^{\alpha }$, $0\leq \alpha <n$.

More generally, to obtain the case when $T^{\alpha }$ is elliptic and the
tailless $A_{2}^{\alpha }$ condition holds, we note that the key estimate (%
\ref{key separation}) above extends to the kernel $\sum_{j=1}^{J}\lambda
_{j}^{m}K_{j}^{\alpha }$ of $\sum_{j=1}^{J}\lambda _{j}^{m}T_{j}^{\alpha }$
in (\ref{Ktalpha strong}) if the $n$-ants above are replaced by thin cones
of sufficently small aperture, and there is in addition sufficient
separation between opposing cones, which in turn may require a larger
constant than $4\sqrt{n}$ in the choice of $Q^{\prime }$ above.

Finally, we turn to showing that the tailless $A_{2}^{\alpha }$ condition is
implied by the norm inequality, i.e.%
\begin{eqnarray*}
&&A_{2}^{\alpha }\equiv \sup_{Q^{\prime }}\ell \left( Q^{\prime }\right)
^{\alpha }\left( \frac{1}{\left\vert Q^{\prime }\right\vert }\int_{Q^{\prime
}}d\omega \right) ^{\frac{1}{2}}\left( \frac{1}{\left\vert Q^{\prime
}\right\vert }\int_{Q^{\prime }}d\sigma \right) ^{\frac{1}{2}}\lesssim 
\mathfrak{N}_{\alpha }\left( \mathbf{R}^{\alpha }\right) ; \\
&&\text{i.e. }\left( \int_{Q^{\prime }}d\omega \right) \left(
\int_{Q^{\prime }}d\sigma \right) \lesssim \mathfrak{N}_{\gamma }\left( 
\mathbf{R}^{\gamma }\right) ^{2}\left\vert Q^{\prime }\right\vert ^{2-\frac{%
2\alpha }{n}}.
\end{eqnarray*}%
In the range $0\leq \alpha <\frac{n}{2}$ where we only assume (\ref{Ktalpha}%
), we invoke the corresponding argument in \cite{LaSaUr1}. Indeed, with
notation as in that proof, and suppressing some of the initial work there,
then $\mathcal{A}_{2}\left( \omega ,\sigma ;\mathsf{Q}\right) =\left\vert 
\mathsf{Q}\right\vert _{\omega \times \sigma }$ where $\omega \times \sigma $
denotes product measure on $\mathbb{R}^{n}\times \mathbb{R}^{n}$, and we have%
\begin{equation*}
\mathcal{A}_{2}\left( \omega ,\sigma ;\mathsf{Q}_{0}\right) =\sum_{\zeta }%
\mathcal{A}_{2}\left( \omega ,\sigma ;\mathsf{Q}_{\zeta }\right)
+\sum_{\beta }\mathcal{A}_{2}\left( \omega ,\sigma ;\mathsf{P}_{\beta
}\right) .
\end{equation*}%
Now we have%
\begin{equation*}
\sum_{\zeta }\mathcal{A}_{2}\left( \omega ,\sigma ;\mathsf{Q}_{\zeta
}\right) =\sum_{\zeta }\left\vert \mathsf{Q}_{\zeta }\right\vert _{\omega
\times \sigma }\leq \sum_{\zeta }\mathfrak{N}_{\alpha }\left( \mathbf{R}%
^{\alpha }\right) ^{2}\left\vert \mathsf{Q}_{\zeta }\right\vert ^{1-\frac{%
\alpha }{n}},
\end{equation*}%
and 
\begin{eqnarray*}
\sum_{\zeta }\left\vert \mathsf{Q}_{\zeta }\right\vert ^{1-\frac{\alpha }{n}%
} &=&\sum_{k\in \mathbb{Z}:\ 2^{k}\leq \ell \left( Q_{0}\right) }\sum_{\zeta
:\ \ell \left( Q_{\zeta }\right) =2^{k}}\left( 2^{2nk}\right) ^{1-\frac{%
\alpha }{n}} \\
&\approx &\sum_{k\in \mathbb{Z}:\ 2^{k}\leq \ell \left( Q_{0}\right) }\left( 
\frac{2^{k}}{\ell \left( Q_{0}\right) }\right) ^{-n}\left( 2^{2nk}\right)
^{1-\frac{\alpha }{n}}\ \ \ \text{(Whitney)} \\
&=&\ell \left( Q_{0}\right) ^{n}\sum_{k\in \mathbb{Z}:\ 2^{k}\leq \ell
\left( Q_{0}\right) }2^{nk\left( -1+2-\frac{2\alpha }{n}\right) } \\
&\leq &C_{\alpha }\ell \left( Q_{0}\right) ^{n}\ell \left( Q_{0}\right)
^{n\left( 1-\frac{2\alpha }{n}\right) }=C_{\alpha }\left\vert Q_{0}\times
Q_{0}\right\vert ^{2-\frac{2\alpha }{n}}=C_{\alpha }\left\vert \mathsf{Q}%
_{0}\right\vert ^{1-\frac{\alpha }{n}},
\end{eqnarray*}%
provided $0\leq \alpha <\frac{n}{2}$. Since $\omega $ and $\sigma $ have no
point masses in common, it is not hard to show, using that the side length
of $\mathsf{P}_{\beta }=P_{\beta }\times P_{\beta }^{\prime }$ is $2^{-N}$
and $dist\left( \mathsf{P}_{\beta },\Omega \mathcal{D}\right) \leq C2^{-N}$,
that we have the following limit,%
\begin{equation*}
\sum_{\beta }\mathcal{A}_{2}\left( \omega ,\sigma ;\mathsf{P}_{\beta
}\right) \rightarrow 0\text{ as }N\rightarrow \infty .
\end{equation*}%
Indeed, if $\sigma $ has no point masses at all, then%
\begin{eqnarray*}
\sum_{\beta }\mathcal{A}_{2}\left( \omega ,\sigma ;\mathsf{P}_{\beta
}\right) &=&\sum_{\beta }\left\vert P_{\beta }\right\vert _{\omega
}\left\vert P_{\beta }^{\prime }\right\vert _{\sigma } \\
&\leq &\left( \sum_{\beta }\left\vert P_{\beta }\right\vert _{\omega
}\right) \sup_{\beta }\left\vert P_{\beta }^{\prime }\right\vert _{\sigma }
\\
&\leq &C\left\vert Q_{0}\right\vert _{\omega }\sup_{\beta }\left\vert
P_{\beta }^{\prime }\right\vert _{\sigma }\rightarrow 0\text{ as }%
N\rightarrow \infty ,
\end{eqnarray*}%
while if $\sigma $ contains a point mass $c\delta _{x}$, then%
\begin{eqnarray*}
\sum_{\beta :\ x\in P_{\beta }^{\prime }}\mathcal{A}_{2}\left( \omega
,\sigma ;\mathsf{P}_{\beta }\right) &\leq &\left( \sum_{\beta :\ x\in
P_{\beta }^{\prime }}\left\vert P_{\beta }\right\vert _{\omega }\right)
\sup_{\beta :\ x\in P_{\beta }^{\prime }}\left\vert P_{\beta }^{\prime
}\right\vert _{\sigma } \\
&\leq &C\left( \sum_{\beta :\ x\in P_{\beta }^{\prime }}\left\vert P_{\beta
}\right\vert _{\omega }\right) \rightarrow 0\text{ as }N\rightarrow \infty
\end{eqnarray*}%
since $\omega $ has no point mass at $x$. This continues to hold if $\sigma $
contains finitely many point masses disjoint from those of $\omega $, and a
limiting argument finally applies. This completes the proof that $%
A_{2}^{\alpha }\lesssim \mathfrak{N}_{\alpha }\left( \mathbf{R}^{\alpha
}\right) $ for the range $0\leq \alpha <\frac{n}{2}$.

Now we turn to proving $A_{2}^{\alpha }\lesssim \mathfrak{N}_{\alpha }\left( 
\mathbf{R}^{\alpha }\right) $ for the range $\frac{n}{2}\leq \alpha <n$,
where we assume the stronger ellipticity condition (\ref{Ktalpha strong}).
So fix a cube $Q=\dprod\limits_{i=1}^{n}Q_{i}$ where $Q_{i}=\left[
a_{i},b_{i}\right] $. Choose $\theta _{1}\in \left[ a_{1},b_{1}\right] $ so
that both 
\begin{equation*}
\left\vert \left[ a_{1},\theta _{1}\right] \times
\dprod\limits_{i=2}^{n}Q_{i}\right\vert _{\omega },\ \ \ \left\vert \left[
\theta _{1},b_{1}\right] \times \dprod\limits_{i=2}^{n}Q_{i}\right\vert
_{\omega }\geq \frac{1}{2}\left\vert Q\right\vert _{\omega }.
\end{equation*}%
Now denote the two intervals $\left[ a_{1},\theta _{1}\right] $ and $\left[
\theta _{1},b_{1}\right] $ by $\left[ a_{1}^{\ast },b_{1}^{\ast }\right] $
and $\left[ a_{1}^{\ast \ast },b_{1}^{\ast \ast }\right] $ where the order
is chosen so that 
\begin{equation*}
\left\vert \left[ a_{1}^{\ast },b_{1}^{\ast }\right] \times
\dprod\limits_{i=2}^{n}Q_{i}\right\vert _{\sigma }\leq \left\vert \left[
a_{1}^{\ast \ast },b_{1}^{\ast \ast }\right] \times
\dprod\limits_{i=2}^{n}Q_{i}\right\vert _{\sigma }.
\end{equation*}%
Then we have both%
\begin{eqnarray*}
\left\vert \left[ a_{1}^{\ast },b_{1}^{\ast }\right] \times
\dprod\limits_{i=2}^{n}Q_{i}\right\vert _{\omega } &\geq &\frac{1}{2}%
\left\vert Q\right\vert _{\omega }\ , \\
\left\vert \left[ a_{1}^{\ast \ast },b_{1}^{\ast \ast }\right] \times
\dprod\limits_{i=2}^{n}Q_{i}\right\vert _{\sigma } &\geq &\frac{1}{2}%
\left\vert Q\right\vert _{\sigma }\ .
\end{eqnarray*}%
Now choose $\theta _{2}\in \left[ a_{2},b_{2}\right] $ so that both%
\begin{equation*}
\left\vert \left[ a_{1}^{\ast },b_{1}^{\ast }\right] \times \left[
a_{2},\theta _{2}\right] \times \dprod\limits_{i=3}^{n}Q_{i}\right\vert
_{\omega },\ \ \ \left\vert \left[ a_{1}^{\ast },b_{1}^{\ast }\right] \times %
\left[ \theta _{2},b_{2}\right] \times
\dprod\limits_{i=3}^{n}Q_{i}\right\vert _{\omega }\geq \frac{1}{4}\left\vert
Q\right\vert _{\omega },
\end{equation*}%
and denote the two intervals $\left[ a_{2},\theta _{2}\right] $ and $\left[
\theta _{2},b_{2}\right] $ by $\left[ a_{2}^{\ast },b_{2}^{\ast }\right] $
and $\left[ a_{2}^{\ast \ast },b_{2}^{\ast \ast }\right] $ where the order
is chosen so that%
\begin{equation*}
\left\vert \left[ a_{1}^{\ast \ast },b_{1}^{\ast \ast }\right] \times \left[
a_{2}^{\ast },b_{2}^{\ast }\right] \times
\dprod\limits_{i=2}^{n}Q_{i}\right\vert _{\sigma }\leq \left\vert \left[
a_{1}^{\ast \ast },b_{1}^{\ast \ast }\right] \times \left[ a_{2}^{\ast \ast
},b_{2}^{\ast \ast }\right] \times \dprod\limits_{i=2}^{n}Q_{i}\right\vert
_{\sigma }.
\end{equation*}%
Then we have both%
\begin{eqnarray*}
\left\vert \left[ a_{1}^{\ast },b_{1}^{\ast }\right] \times \left[
a_{2}^{\ast },b_{2}^{\ast }\right] \times
\dprod\limits_{i=3}^{n}Q_{i}\right\vert _{\omega } &\geq &\frac{1}{4}%
\left\vert Q\right\vert _{\omega }\ , \\
\left\vert \left[ a_{1}^{\ast \ast },b_{1}^{\ast \ast }\right] \times \left[
a_{2}^{\ast \ast },b_{2}^{\ast \ast }\right] \times
\dprod\limits_{i=3}^{n}Q_{i}\right\vert _{\sigma } &\geq &\frac{1}{4}%
\left\vert Q\right\vert _{\sigma }\ .
\end{eqnarray*}%
Then we choose $\theta _{3}\in \left[ a_{3},b_{3}\right] $ so that both%
\begin{eqnarray*}
\left\vert \left[ a_{1}^{\ast },b_{1}^{\ast }\right] \times \left[
a_{2}^{\ast },b_{2}^{\ast }\right] \times \left[ a_{3},\theta _{3}\right]
\times \dprod\limits_{i=4}^{n}Q_{i}\right\vert _{\omega } &\geq &\frac{1}{8}%
\left\vert Q\right\vert _{\omega }, \\
\left\vert \left[ a_{1}^{\ast },b_{1}^{\ast }\right] \times \left[
a_{2}^{\ast },b_{2}^{\ast }\right] \times \left[ \theta _{3},b_{3}\right]
\times \dprod\limits_{i=4}^{n}Q_{i}\right\vert _{\omega } &\geq &\frac{1}{8}%
\left\vert Q\right\vert _{\omega },
\end{eqnarray*}%
and continuing in this way we end up with two rectangles,%
\begin{eqnarray*}
G &\equiv &\left[ a_{1}^{\ast },b_{1}^{\ast }\right] \times \left[
a_{2}^{\ast },b_{2}^{\ast }\right] \times ...\left[ a_{n}^{\ast
},b_{n}^{\ast }\right] , \\
H &\equiv &\left[ a_{1}^{\ast \ast },b_{1}^{\ast \ast }\right] \times \left[
a_{2}^{\ast \ast },b_{2}^{\ast \ast }\right] \times ...\left[ a_{n}^{\ast
\ast },b_{n}^{\ast \ast }\right] ,
\end{eqnarray*}%
that satisfy%
\begin{eqnarray*}
\left\vert G\right\vert _{\omega } &=&\left\vert \left[ a_{1}^{\ast
},b_{1}^{\ast }\right] \times \left[ a_{2}^{\ast },b_{2}^{\ast }\right]
\times ...\left[ a_{n}^{\ast },b_{n}^{\ast }\right] \right\vert _{\omega
}\geq \frac{1}{2^{n}}\left\vert Q\right\vert _{\omega }, \\
\left\vert H\right\vert _{\sigma } &=&\left\vert \left[ a_{1}^{\ast \ast
},b_{1}^{\ast \ast }\right] \times \left[ a_{2}^{\ast \ast },b_{2}^{\ast
\ast }\right] \times ...\left[ a_{n}^{\ast \ast },b_{n}^{\ast \ast }\right]
\right\vert _{\sigma }\geq \frac{1}{2^{n}}\left\vert Q\right\vert _{\sigma }.
\end{eqnarray*}

However, the rectangles $G$ and $H$ lie in opposing $n$-ants at the vertex $%
\theta =\left( \theta _{1},\theta _{2},...,\theta _{n}\right) $, and so we
can apply (\ref{Ktalpha strong}) to obtain that for $x\in G$,%
\begin{eqnarray*}
\left\vert \sum_{j=1}^{J}\lambda _{j}^{m}T_{j}^{\alpha }\left( \mathbf{1}%
_{H}\sigma \right) \left( x\right) \right\vert &=&\left\vert
\int_{H}\sum_{j=1}^{J}\lambda _{j}^{m}K_{j}^{\alpha }\left( x,y\right)
d\sigma \left( y\right) \right\vert \\
&\gtrsim &\int_{H}\left\vert x-y\right\vert ^{\alpha -n}d\sigma \left(
y\right) \gtrsim \left\vert Q\right\vert ^{\frac{\alpha }{n}-1}\left\vert
H\right\vert _{\sigma }.
\end{eqnarray*}%
Then from the norm inequality we get%
\begin{eqnarray*}
\left\vert G\right\vert _{\omega }\left( \left\vert Q\right\vert ^{\frac{%
\alpha }{n}-1}\left\vert H\right\vert _{\sigma }\right) ^{2} &\lesssim
&\int_{G}\left\vert \sum_{j=1}^{J}\lambda _{j}^{m}T_{j}^{\alpha }\left( 
\mathbf{1}_{H}\sigma \right) \right\vert ^{2}d\omega \\
&\lesssim &\mathfrak{N}_{\sum_{j=1}^{J}\lambda _{j}^{m}T_{j}^{\alpha }}\int 
\mathbf{1}_{H}^{2}d\sigma =\mathfrak{N}_{\sum_{j=1}^{J}\lambda
_{j}^{m}T_{j}^{\alpha }}\left\vert H\right\vert _{\sigma },
\end{eqnarray*}%
from which we deduce that%
\begin{equation*}
\left\vert Q\right\vert ^{2\left( \frac{\alpha }{n}-1\right) }\left\vert
Q\right\vert _{\omega }\left\vert Q\right\vert _{\sigma }\lesssim
2^{2n}\left\vert Q\right\vert ^{2\left( \frac{\alpha }{n}-1\right)
}\left\vert G\right\vert _{\omega }\left\vert H\right\vert _{\sigma
}\lesssim 2^{2n}\mathfrak{N}_{\sum_{j=1}^{J}\lambda _{j}^{m}T_{j}^{\alpha }},
\end{equation*}%
and hence%
\begin{equation*}
A_{2}^{\alpha }\lesssim 2^{2n}\mathfrak{N}_{\sum_{j=1}^{J}\lambda
_{j}^{m}T_{j}^{\alpha }}\ .
\end{equation*}%
This completes the proof of Lemma \ref{necc frac A2}.
\end{proof}

\section{Monotonicity lemma and Energy lemma}

The Monotonicity Lemma below will be used to prove the Energy Lemma, which
is then used in several places in the proof of Theorem \ref{T1 theorem}. The
formulation of the Monotonicity Lemma with $m=2$ for cubes is due to M.
Lacey and B. Wick \cite{LaWi}, and corrects that used in early versions of
this paper. We now return to a fixed collection $\Omega \mathcal{P}^{n}$ of
quasicubes for the remainder of the paper.

\subsection{The monotonicity lemma}

For $0\leq \alpha <n$ and $m\in \mathbb{R}_{+}$, we recall the $m$-weighted
fractional Poisson integral%
\begin{equation*}
\mathrm{P}_{m}^{\alpha }\left( J,\mu \right) \equiv \int_{\mathbb{R}^{n}}%
\frac{\left\vert J\right\vert ^{\frac{m}{n}}}{\left( \left\vert J\right\vert
^{\frac{1}{n}}+\left\vert y-c_{J}\right\vert \right) ^{n+m-\alpha }}d\mu
\left( y\right) ,
\end{equation*}%
where $\mathrm{P}_{1}^{\alpha }\left( J,\mu \right) =\mathrm{P}^{\alpha
}\left( J,\mu \right) $ is the standard Poisson integral.

\begin{lemma}[Monotonicity]
\label{mono}Suppose that$\ I$, $J$ and $J^{\ast }$ are quasicubes in $%
\mathbb{R}^{n}$ such that $J\subset 2J\subset I$, and that $\mu $ is a
signed measure on $\mathbb{R}^{n}$ supported outside $I$. Finally suppose
that $T^{\alpha }$ is a standard fractional singular integral on $\mathbb{R}%
^{n}$ as defined in Definition \ref{def alpha standard} with $0<\alpha <n$.
Then we have the estimate%
\begin{equation}
\left\Vert \bigtriangleup _{J}^{\omega }T^{\alpha }\mu \right\Vert
_{L^{2}\left( \omega \right) }\lesssim \Phi ^{\alpha }\left( J,\left\vert
\mu \right\vert \right) ,  \label{estimate}
\end{equation}%
where for a positive measure $\nu $,%
\begin{eqnarray*}
\Phi ^{\alpha }\left( J,\nu \right) ^{2} &\equiv &\left( \frac{\mathrm{P}%
^{\alpha }\left( J,\nu \right) }{\left\vert J\right\vert ^{\frac{1}{n}}}%
\right) ^{2}\left\Vert \bigtriangleup _{J}^{\omega }\mathbf{x}\right\Vert
_{L^{2}\left( \omega \right) }^{2}+\left( \frac{\mathrm{P}_{1+\delta
}^{\alpha }\left( J,\nu \right) }{\left\vert J\right\vert ^{\frac{1}{n}}}%
\right) ^{2}\left\Vert \mathbf{x}-\mathbf{m}_{J}\right\Vert _{L^{2}\left( 
\mathbf{1}_{J}\omega \right) }^{2}\ , \\
\mathbf{m}_{J} &\equiv &\mathbb{E}_{J}^{\omega }\mathbf{x}=\frac{1}{%
\left\vert J\right\vert _{\omega }}\int_{J}\mathbf{x}d\omega .
\end{eqnarray*}
\end{lemma}

\begin{proof}
Let $\left\{ h_{J}^{\omega ,a}\right\} _{a\in \Gamma }$ be an orthonormal
basis of $L_{J}^{2}\left( \mu \right) $ as in a previous subsection. Now we
use the smoothness estimate (\ref{sizeandsmoothness'}), together with
Taylor's formula and the vanishing mean of the quasiHaar functions $%
h_{J}^{\omega ,a}$ and $\mathbf{m}_{J}\equiv \frac{1}{\left\vert
J\right\vert _{\mu }}\int_{J}\mathbf{x}d\mu \left( x\right) \in J$, to
obtain 
\begin{eqnarray*}
\left\vert \left\langle T^{\alpha }\mu ,h_{J}^{\omega ,a}\right\rangle
_{\omega }\right\vert &=&\left\vert \int \left\{ \int K^{\alpha }\left(
x,y\right) h_{J}^{\omega ,a}\left( x\right) d\omega \left( x\right) \right\}
d\mu \left( y\right) \right\vert =\left\vert \int \left\langle K_{y}^{\alpha
},h_{J}^{\omega ,a}\right\rangle _{\omega }d\mu \left( y\right) \right\vert
\\
&=&\left\vert \int \left\langle K_{y}^{\alpha }\left( x\right)
-K_{y}^{\alpha }\left( \mathbf{m}_{J}\right) ,h_{J}^{\omega ,a}\right\rangle
_{\omega }d\mu \left( y\right) \right\vert \\
&\leq &\left\vert \left\langle \left[ \int \nabla K_{y}^{\alpha }\left( 
\mathbf{m}_{J}\right) d\mu \left( y\right) \right] \left( \mathbf{x}-\mathbf{%
m}_{J}\right) ,h_{J}^{\omega ,a}\right\rangle _{\omega }\right\vert \\
&&+\left\langle \left[ \int \sup_{\mathbf{\theta }_{J}\in J}\left\vert
\nabla K_{y}^{\alpha }\left( \mathbf{\theta }_{J}\right) -\nabla
K_{y}^{\alpha }\left( \mathbf{m}_{J}\right) \right\vert d\mu \left( y\right) %
\right] \left\vert \mathbf{x}-\mathbf{m}_{J}\right\vert ,\left\vert
h_{J}^{\omega ,a}\right\vert \right\rangle _{\omega } \\
&\lesssim &C_{CZ}\frac{\mathrm{P}^{\alpha }\left( J,\left\vert \mu
\right\vert \right) }{\left\vert J\right\vert ^{\frac{1}{n}}}\left\Vert
\bigtriangleup _{J}^{\omega }\mathbf{x}\right\Vert _{L^{2}\left( \omega
\right) }+C_{CZ}\frac{\mathrm{P}_{1+\delta }^{\alpha }\left( J,\left\vert
\mu \right\vert \right) }{\left\vert J\right\vert ^{\frac{1}{n}}}\left\Vert 
\mathbf{x}-\mathbf{m}_{J}\right\Vert _{L^{2}\left( \mathbf{1}_{J}\omega
\right) }.
\end{eqnarray*}
\end{proof}

\subsection{The Energy Lemma}

Suppose now we are given a subset $\mathcal{H}$ of the dyadic quasigrid $%
\Omega \mathcal{D}^{\omega }$. Let $\mathsf{P}_{\mathcal{H}}^{\omega
}=\sum_{J\in \mathcal{H}}\bigtriangleup _{J}^{\omega }$ be the $\omega $%
-quasiHaar projection onto $\mathcal{H}$. For $\mu ,\omega $ positive
locally finite Borel measures on $\mathbb{R}^{n}$, and $\mathcal{H}$ a
subset of the dyadic quasigrid $\Omega \mathcal{D}^{\omega }$, we define $%
\mathcal{H}^{\ast }\equiv \dbigcup\limits_{J\in \mathcal{H}}\left\{
J^{\prime }\in \Omega \mathcal{D}^{\omega }:J^{\prime }\subset J\right\} $.

\begin{lemma}[\textbf{Energy Lemma}]
\label{ener}Let $J\ $be a quasicube in $\Omega \mathcal{D}^{\omega }$. Let $%
\Psi _{J}$ be an $L^{2}\left( \omega \right) $ function supported in $J$ and
with $\omega $-integral zero, and denote its quasiHaar support by $\mathcal{H%
}=\limfunc{supp}\widehat{\Psi _{J}}$. Let $\nu $ be a positive measure
supported in $\mathbb{R}^{n}\setminus \gamma J$ with $\gamma \geq 2$, and
for each $J^{\prime }\in \mathcal{H}$, let $d\nu _{J^{\prime }}=\varphi
_{J^{\prime }}d\nu $ with $\left\vert \varphi _{J^{\prime }}\right\vert \leq
1$. Let $T^{\alpha }$ be a standard $\alpha $-fractional Calder\'{o}%
n-Zygmund operator with $0\leq \alpha <n$. Then with $\delta ^{\prime }=%
\frac{\delta }{2}$ we have%
\begin{eqnarray*}
\left\vert \sum_{J^{\prime }\in \mathcal{H}}\left\langle T^{\alpha }\left(
\nu _{J^{\prime }}\right) ,\bigtriangleup _{J^{\prime }}^{\omega }\Psi
_{J}\right\rangle _{\omega }\right\vert &\lesssim &\left\Vert \Psi
_{J}\right\Vert _{L^{2}\left( \omega \right) }\left( \frac{\mathrm{P}%
^{\alpha }\left( J,\nu \right) }{\left\vert J\right\vert ^{\frac{1}{n}}}%
\right) \left\Vert \mathsf{P}_{\mathcal{H}}^{\omega }\mathbf{x}\right\Vert
_{L^{2}\left( \omega \right) } \\
&&+\left\Vert \Psi _{J}\right\Vert _{L^{2}\left( \omega \right) }\frac{1}{%
\gamma ^{\delta ^{\prime }}}\left( \frac{\mathrm{P}_{1+\delta ^{\prime
}}^{\alpha }\left( J,\nu \right) }{\left\vert J\right\vert ^{\frac{1}{n}}}%
\right) \left\Vert \mathsf{P}_{\mathcal{H}^{\ast }}^{\omega }\mathbf{x}%
\right\Vert _{L^{2}\left( \omega \right) } \\
&\lesssim &\left\Vert \Psi _{J}\right\Vert _{L^{2}\left( \omega \right)
}\left( \frac{\mathrm{P}^{\alpha }\left( J,\nu \right) }{\left\vert
J\right\vert ^{\frac{1}{n}}}\right) \left\Vert \mathsf{P}_{\mathcal{H}^{\ast
}}^{\omega }\mathbf{x}\right\Vert _{L^{2}\left( \omega \right) },
\end{eqnarray*}%
and in particular the `pivotal' bound%
\begin{equation*}
\left\vert \left\langle T^{\alpha }\left( \nu \right) ,\Psi
_{J}\right\rangle _{\omega }\right\vert \leq C\left\Vert \Psi
_{J}\right\Vert _{L^{2}\left( \omega \right) }\mathrm{P}^{\alpha }\left(
J,\left\vert \nu \right\vert \right) \sqrt{\left\vert J\right\vert _{\omega }%
}\ .
\end{equation*}
\end{lemma}

\begin{remark}
The first term on the right side of the energy inequality above is the `big'
Poisson integral $\mathrm{P}^{\alpha }$ times the `small' energy term $%
\left\Vert \mathsf{P}_{\mathcal{H}}^{\omega }\mathbf{x}\right\Vert
_{L^{2}\left( \omega \right) }^{2}$ that is additive in $\mathcal{H}$, while
the second term on the right is the `small' Poisson integral $\mathrm{P}%
_{1+\delta ^{\prime }}^{\alpha }$ times the `big' energy term $\left\Vert 
\mathsf{P}_{\mathcal{H}^{\ast }}^{\omega }\mathbf{x}\right\Vert
_{L^{2}\left( \omega \right) }$ that is no longer additive in $\mathcal{H}$.
The first term presents no problems in subsequent analysis due solely to the
additivity of\ the `small' energy term. It is the second term that must be
handled by special methods. For example, in the Intertwining Proposition
below, the interaction of the singular integral occurs with a pair of
quasicubes $J\subset I$ at \emph{highly separated} levels, where the
goodness of $J$ can exploit the decay $\delta ^{\prime }$ in the kernel of
the `small' Poisson integral $\mathrm{P}_{1+\delta ^{\prime }}^{\alpha }$
relative to the `big' Poisson integral $\mathrm{P}^{\alpha }$, and results
in a bound directly by the quasienergy condition. On the other hand, in the
local recursion of M. Lacey at the end of the \ paper, the separation of
levels in the pairs $J\subset I$ can be as \emph{little} as a fixed
parameter $\mathbf{\rho }$, and here we must first separate the stopping
form into two sublinear forms that involve the two estimates respectively.
The form corresponding to the smaller Poisson integral $\mathrm{P}_{1+\delta
^{\prime }}^{\alpha }$ is again handled using goodness and the decay $\delta
^{\prime }$ in the kernel, while the form corresponding to the larger
Poisson integral $\mathrm{P}^{\alpha }$ requires the stopping time and
recursion argument of M. Lacey.
\end{remark}

\begin{proof}
Using the Monotonicity Lemma \ref{mono}, followed by $\left\vert \nu
_{J^{\prime }}\right\vert \leq \nu $ and the Poisson equivalence 
\begin{equation}
\frac{\mathrm{P}_{m}^{\alpha }\left( J^{\prime },\nu \right) }{\left\vert
J^{\prime }\right\vert ^{\frac{m}{n}}}\approx \frac{\mathrm{P}_{m}^{\alpha
}\left( J,\nu \right) }{\left\vert J\right\vert ^{\frac{m}{n}}},\ \ \ \ \
J^{\prime }\subset J\subset 2J,\ \ \ \limfunc{supp}\nu \cap 2J=\emptyset ,
\label{Poisson equiv}
\end{equation}%
we have%
\begin{eqnarray*}
&&\left\vert \sum_{J^{\prime }\in \mathcal{H}}\left\langle T^{\alpha }\left(
\nu _{J^{\prime }}\right) ,\bigtriangleup _{J^{\prime }}^{\omega }\Psi
_{J}\right\rangle _{\omega }\right\vert =\left\vert \sum_{J^{\prime }\in 
\mathcal{H}}\left\langle \bigtriangleup _{J^{\prime }}^{\omega }T^{\alpha
}\left( \nu _{J^{\prime }}\right) ,\bigtriangleup _{J^{\prime }}^{\omega
}\Psi _{J}\right\rangle _{\omega }\right\vert \\
&\lesssim &\sum_{J^{\prime }\in \mathcal{H}}\Phi ^{\alpha }\left( J^{\prime
},\left\vert \nu _{J^{\prime }}\right\vert \right) \left\Vert \bigtriangleup
_{J^{\prime }}^{\omega }\Psi _{J}\right\Vert _{L^{2}\left( \omega \right) }
\\
&\lesssim &\left( \sum_{J^{\prime }\in \mathcal{H}}\left( \frac{\mathrm{P}%
^{\alpha }\left( J^{\prime },\nu \right) }{\left\vert J^{\prime }\right\vert
^{\frac{1}{n}}}\right) ^{2}\left\Vert \bigtriangleup _{J^{\prime }}^{\omega }%
\mathbf{x}\right\Vert _{L^{2}\left( \omega \right) }^{2}\right) ^{\frac{1}{2}%
}\left( \sum_{J^{\prime }\in \mathcal{H}}\left\Vert \bigtriangleup
_{J^{\prime }}^{\omega }\Psi _{J}\right\Vert _{L^{2}\left( \omega \right)
}^{2}\right) ^{\frac{1}{2}} \\
&&+\left( \sum_{J^{\prime }\in \mathcal{H}}\left( \frac{\mathrm{P}_{1+\delta
}^{\alpha }\left( J^{\prime },\nu \right) }{\left\vert J^{\prime
}\right\vert ^{\frac{1}{n}}}\right) ^{2}\sum_{J^{\prime \prime }\subset
J^{\prime }}\left\Vert \bigtriangleup _{J^{\prime \prime }}^{\omega }\mathbf{%
x}\right\Vert _{L^{2}\left( \omega \right) }^{2}\right) ^{\frac{1}{2}}\left(
\sum_{J^{\prime }\in \mathcal{H}}\left\Vert \bigtriangleup _{J^{\prime
}}^{\omega }\Psi _{J}\right\Vert _{L^{2}\left( \omega \right) }^{2}\right) ^{%
\frac{1}{2}} \\
&\lesssim &\left( \frac{\mathrm{P}^{\alpha }\left( J,\nu \right) }{%
\left\vert J\right\vert ^{\frac{1}{n}}}\right) \left\Vert \mathsf{P}_{%
\mathcal{H}}^{\omega }\mathbf{x}\right\Vert _{L^{2}\left( \omega \right)
}\left\Vert \Psi _{J}\right\Vert _{L^{2}\left( \omega \right) }+\frac{1}{%
\gamma ^{\delta ^{\prime }}}\left( \frac{\mathrm{P}_{1+\delta ^{\prime
}}^{\alpha }\left( J,\nu \right) }{\left\vert J\right\vert ^{\frac{1}{n}}}%
\right) \left\Vert \mathsf{P}_{\mathcal{H}^{\ast }}^{\omega }\mathbf{x}%
\right\Vert _{L^{2}\left( \omega \right) }\left\Vert \Psi _{J}\right\Vert
_{L^{2}\left( \omega \right) }\ .
\end{eqnarray*}%
The last inequality follows from

\begin{eqnarray*}
&&\sum_{J^{\prime }\in \mathcal{H}}\left( \frac{\mathrm{P}_{1+\delta
}^{\alpha }\left( J^{\prime },\nu \right) }{\left\vert J^{\prime
}\right\vert ^{\frac{1}{n}}}\right) ^{2}\sum_{J^{\prime \prime }\subset
J^{\prime }}\left\Vert \bigtriangleup _{J^{\prime \prime }}^{\omega }\mathbf{%
x}\right\Vert _{L^{2}\left( \omega \right) }^{2} \\
&=&\sum_{J^{\prime \prime }\subset J}\left\{ \sum_{J^{\prime }:\ J^{\prime
\prime }\subset J^{\prime }\subset J}\left( \frac{\mathrm{P}_{1+\delta
}^{\alpha }\left( J^{\prime },\nu \right) }{\left\vert J^{\prime
}\right\vert ^{\frac{1}{n}}}\right) ^{2}\right\} \left\Vert \bigtriangleup
_{J^{\prime \prime }}^{\omega }\mathbf{x}\right\Vert _{L^{2}\left( \omega
\right) }^{2} \\
&\lesssim &\frac{1}{\gamma ^{2\delta ^{\prime }}}\sum_{J^{\prime \prime }\in 
\mathcal{H}^{\ast }}\left( \frac{\mathrm{P}_{1+\delta ^{\prime }}^{\alpha
}\left( J^{\prime \prime },\nu \right) }{\left\vert J^{\prime \prime
}\right\vert ^{\frac{1}{n}}}\right) ^{2}\left\Vert \bigtriangleup
_{J^{\prime \prime }}^{\omega }\mathbf{x}\right\Vert _{L^{2}\left( \omega
\right) }^{2}\ ,
\end{eqnarray*}

which in turn follows from (recalling $\delta =2\delta ^{\prime }$ and using 
$\left\vert J^{\prime }\right\vert ^{\frac{1}{n}}+\left\vert y-c_{J^{\prime
}}\right\vert \approx \left\vert J\right\vert ^{\frac{1}{n}}+\left\vert
y-c_{J}\right\vert $ and $\frac{\left\vert J\right\vert ^{\frac{1}{n}}}{%
\left\vert J\right\vert ^{\frac{1}{n}}+\left\vert y-c_{J}\right\vert }\leq 
\frac{1}{\gamma }$ for $y\in \mathbb{R}^{n}\setminus \gamma J$)%
\begin{eqnarray*}
&&\sum_{J^{\prime }:\ J^{\prime \prime }\subset J^{\prime }\subset J}\left( 
\frac{\mathrm{P}_{1+\delta }^{\alpha }\left( J^{\prime },\nu \right) }{%
\left\vert J^{\prime }\right\vert ^{\frac{1}{n}}}\right) ^{2} \\
&=&\sum_{J^{\prime }:\ J^{\prime \prime }\subset J^{\prime }\subset
J}\left\vert J^{\prime }\right\vert ^{\frac{2\delta }{n}}\left( \int_{%
\mathbb{R}^{n}\setminus \gamma J}\frac{1}{\left( \left\vert J^{\prime
}\right\vert ^{\frac{1}{n}}+\left\vert y-c_{J^{\prime }}\right\vert \right)
^{n+1+\delta -\alpha }}d\nu \left( y\right) \right) ^{2} \\
&\lesssim &\sum_{J^{\prime }:\ J^{\prime \prime }\subset J^{\prime }\subset
J}\frac{1}{\gamma ^{2\delta ^{\prime }}}\frac{\left\vert J^{\prime
}\right\vert ^{\frac{2\delta }{n}}}{\left\vert J\right\vert ^{\frac{2\delta 
}{n}}}\left( \int_{\mathbb{R}^{n}\setminus \gamma J}\frac{\left\vert
J\right\vert ^{\frac{\delta ^{\prime }}{n}}}{\left( \left\vert J\right\vert
^{\frac{1}{n}}+\left\vert y-c_{J}\right\vert \right) ^{n+1+\delta ^{\prime
}-\alpha }}d\nu \left( y\right) \right) ^{2} \\
&=&\frac{1}{\gamma ^{2\delta ^{\prime }}}\left( \sum_{J^{\prime }:\
J^{\prime \prime }\subset J^{\prime }\subset J}\frac{\left\vert J^{\prime
}\right\vert ^{\frac{2\delta }{n}}}{\left\vert J\right\vert ^{\frac{2\delta 
}{n}}}\right) \left( \frac{\mathrm{P}_{1+\delta ^{\prime }}^{\alpha }\left(
J,\nu \right) }{\left\vert J\right\vert ^{\frac{1}{n}}}\right) ^{2}\lesssim 
\frac{1}{\gamma ^{2\delta ^{\prime }}}\left( \frac{\mathrm{P}_{1+\delta
^{\prime }}^{\alpha }\left( J,\nu \right) }{\left\vert J\right\vert ^{\frac{1%
}{n}}}\right) ^{2}.
\end{eqnarray*}%
Finally we have the `pivotal' bound from (\ref{Poisson equiv}) and%
\begin{equation*}
\sum_{J^{\prime \prime }\subset J}\left\Vert \bigtriangleup _{J^{\prime
\prime }}^{\omega }\mathbf{x}\right\Vert _{L^{2}\left( \omega \right)
}^{2}=\left\Vert \mathbf{x}-\mathbf{m}_{J}\right\Vert _{L^{2}\left( \mathbf{1%
}_{J}\omega \right) }^{2}\leq \left\vert J\right\vert ^{\frac{2}{n}%
}\left\vert J\right\vert _{\omega }\ .
\end{equation*}
\end{proof}

\section{Preliminaries of NTV type}

An important reduction of our theorem is delivered by the following two
lemmas, that in the case of one dimension are due to Nazarov, Treil and
Volberg (see \cite{NTV3} and \cite{Vol}). The proofs given there do not
extend in standard ways to higher dimensions, and we use the Quasiweak
Boundedness Property to handle the case of touching quasicubes, and an
application of Schur's Lemma to handle the case of separated quasicubes. The
first lemma below is Lemmas 8.1 and 8.7 in \cite{LaWi} but with the larger
constant $\mathcal{A}_{2}^{\alpha }$ there in place of $A_{2}^{\alpha }$.

\begin{lemma}
\label{standard delta}Suppose $T^{\alpha }$ is a standard fractional
singular integral with $0\leq \alpha <n$, and that all of the quasicubes $%
I\in \Omega \mathcal{D}^{\sigma },J\in \Omega \mathcal{D}^{\omega }$ below
are good with goodness parameters $\varepsilon $ and $\mathbf{r}$. Fix a
positive integer $\mathbf{\rho }>\mathbf{r}$. For $f\in L^{2}\left( \sigma
\right) $ and $g\in L^{2}\left( \omega \right) $ we have%
\begin{equation}
\sum_{\substack{ \left( I,J\right) \in \Omega \mathcal{D}^{\sigma }\times
\Omega \mathcal{D}^{\omega }  \\ 2^{-\mathbf{\rho }}\ell \left( I\right)
\leq \ell \left( J\right) \leq 2^{\mathbf{\rho }}\ell \left( I\right) }}%
\left\vert \left\langle T_{\sigma }^{\alpha }\left( \bigtriangleup
_{I}^{\sigma }f\right) ,\bigtriangleup _{J}^{\omega }g\right\rangle _{\omega
}\right\vert \lesssim \left( \mathfrak{T}_{\alpha }+\mathfrak{T}_{\alpha
}^{\ast }+\mathcal{WBP}_{T^{\alpha }}+\sqrt{A_{2}^{\alpha }}\right)
\left\Vert f\right\Vert _{L^{2}\left( \sigma \right) }\left\Vert
g\right\Vert _{L^{2}\left( \omega \right) }  \label{delta near}
\end{equation}%
and 
\begin{equation}
\sum_{\substack{ \left( I,J\right) \in \Omega \mathcal{D}^{\sigma }\times
\Omega \mathcal{D}^{\omega }  \\ I\cap J=\emptyset \text{ and }\frac{\ell
\left( J\right) }{\ell \left( I\right) }\notin \left[ 2^{-\mathbf{\rho }},2^{%
\mathbf{\rho }}\right] }}\left\vert \left\langle T_{\sigma }^{\alpha }\left(
\bigtriangleup _{I}^{\sigma }f\right) ,\bigtriangleup _{J}^{\omega
}g\right\rangle _{\omega }\right\vert \lesssim \sqrt{A_{2}^{\alpha }}%
\left\Vert f\right\Vert _{L^{2}\left( \sigma \right) }\left\Vert
g\right\Vert _{L^{2}\left( \omega \right) }.  \label{delta far}
\end{equation}
\end{lemma}

\begin{lemma}
\label{standard indicator}Suppose $T^{\alpha }$ is a standard fractional
singular integral with $0\leq \alpha <n$, that all of the quasicubes $I\in
\Omega \mathcal{D}^{\sigma },J\in \Omega \mathcal{D}^{\omega }$ below are
good, that $\mathbf{\rho }>\mathbf{r}$, that $f\in L^{2}\left( \sigma
\right) $ and $g\in L^{2}\left( \omega \right) $, that $\mathcal{F}\subset
\Omega \mathcal{D}^{\sigma }$ and $\mathcal{G}\subset \Omega \mathcal{D}%
^{\omega }$ are $\sigma $-Carleson and $\omega $-Carleson collections
respectively, i.e.,%
\begin{equation*}
\sum_{F^{\prime }\in \mathcal{F}:\ F^{\prime }\subset F}\left\vert F^{\prime
}\right\vert _{\sigma }\lesssim \left\vert F\right\vert _{\sigma },\ \ \ \ \
F\in \mathcal{F},\text{ and }\sum_{G^{\prime }\in \mathcal{G}:\ G^{\prime
}\subset G}\left\vert G^{\prime }\right\vert _{\omega }\lesssim \left\vert
G\right\vert _{\omega },\ \ \ \ \ G\in \mathcal{G},
\end{equation*}%
that there are numerical sequences $\left\{ \alpha _{\mathcal{F}}\left(
F\right) \right\} _{F\in \mathcal{F}}$ and $\left\{ \beta _{\mathcal{G}%
}\left( G\right) \right\} _{G\in \mathcal{G}}$ such that%
\begin{equation}
\sum_{F\in \mathcal{F}}\alpha _{\mathcal{F}}\left( F\right) ^{2}\left\vert
F\right\vert _{\sigma }\leq \left\Vert f\right\Vert _{L^{2}\left( \sigma
\right) }^{2}\text{ and }\sum_{G\in \mathcal{G}}\beta _{\mathcal{G}}\left(
G\right) ^{2}\left\vert G\right\vert _{\sigma }\leq \left\Vert g\right\Vert
_{L^{2}\left( \sigma \right) }^{2}\ ,  \label{qo}
\end{equation}%
and finally that for each pair of quasicubes $\left( I,J\right) \in \Omega 
\mathcal{D}^{\sigma }\times \Omega \mathcal{D}^{\omega }$, there are bounded
functions $\beta _{I,J}$ and $\gamma _{I,J}$ supported in $I\setminus 2J$
and $J\setminus 2I$ respectively, satisfying%
\begin{equation*}
\left\Vert \beta _{I,J}\right\Vert _{\infty },\left\Vert \gamma
_{I,J}\right\Vert _{\infty }\leq 1.
\end{equation*}%
Then%
\begin{eqnarray}
&&\sum_{\substack{ \left( F,J\right) \in \mathcal{F}\times \Omega \mathcal{D}%
^{\omega }  \\ F\cap J=\emptyset \text{ and }\ell \left( J\right) \leq 2^{-%
\mathbf{\rho }}\ell \left( F\right) }}\left\vert \left\langle T_{\sigma
}^{\alpha }\left( \beta _{F,J}\mathbf{1}_{F}\alpha _{\mathcal{F}}\left(
F\right) \right) ,\bigtriangleup _{J}^{\omega }g\right\rangle _{\omega
}\right\vert  \label{indicator far} \\
&&+\sum_{\substack{ \left( I,G\right) \in \Omega \mathcal{D}^{\sigma }\times 
\mathcal{G}  \\ I\cap G=\emptyset \text{ and }\ell \left( I\right) \leq 2^{-%
\mathbf{\rho }}\ell \left( G\right) }}\left\vert \left\langle T_{\sigma
}^{\alpha }\left( \bigtriangleup _{I}^{\sigma }f\right) ,\gamma _{I,G}%
\mathbf{1}_{G}\beta _{\mathcal{G}}\left( G\right) \right\rangle _{\omega
}\right\vert  \notag \\
&\lesssim &\sqrt{A_{2}^{\alpha }}\left\Vert f\right\Vert _{L^{2}\left(
\sigma \right) }\left\Vert g\right\Vert _{L^{2}\left( \omega \right) }. 
\notag
\end{eqnarray}
\end{lemma}

\begin{remark}
If $\mathcal{F}$ and $\mathcal{G}$ are $\sigma $-Carleson and $\omega $%
-Carleson collections respectively, and if $\alpha _{\mathcal{F}}\left(
F\right) =\mathbb{E}_{F}^{\sigma }\left\vert f\right\vert $ and $\beta _{%
\mathcal{G}}\left( G\right) =\mathbb{E}_{G}^{\omega }\left\vert g\right\vert 
$, then the `quasi' orthogonality condition (\ref{qo}) holds (here`quasi'
has a different meaining than quasi), and this special case of Lemma \ref%
{standard indicator} serves as a basic example.
\end{remark}

\begin{remark}
Lemmas \ref{standard delta} and \ref{standard indicator} differ mainly in
that an orthogonal collection of quasiHaar projections is replaced by a
`quasi' orthogonal collection of indicators $\left\{ \mathbf{1}_{F}\alpha _{%
\mathcal{F}}\left( F\right) \right\} _{F\in \mathcal{F}}$. More precisely,
the main difference between (\ref{delta far}) and (\ref{indicator far}) is
that a quasiHaar projection $\bigtriangleup _{I}^{\sigma }f$ or $%
\bigtriangleup _{J}^{\omega }g$ has been replaced with a constant multiple
of an indicator $\mathbf{1}_{F}\alpha _{\mathcal{F}}\left( F\right) $ or $%
\mathbf{1}_{G}\beta _{\mathcal{G}}\left( G\right) $, and in addition, a
bounded function is permitted to multiply the indicator of the quasicube
having larger sidelength.
\end{remark}

\begin{proof}
Note that in (\ref{delta near}) we have used the parameter $\mathbf{\rho }$
in the exponent rather than $\mathbf{r}$, and this is possible because the
arguments we use here only require that there are finitely many levels of
scale separating $I$ and $J$. To handle this term we first decompose it into%
\begin{eqnarray*}
&&\left\{ \sum_{\substack{ \left( I,J\right) \in \Omega \mathcal{D}^{\sigma
}\times \Omega \mathcal{D}^{\omega }:\ J\subset 3I  \\ 2^{-\mathbf{\rho }%
}\ell \left( I\right) \leq \ell \left( J\right) \leq 2^{\mathbf{\rho }}\ell
\left( I\right) }}+\sum_{\substack{ \left( I,J\right) \in \Omega \mathcal{D}%
^{\sigma }\times \Omega \mathcal{D}^{\omega }:\ I\subset 3J  \\ 2^{-\mathbf{%
\rho }}\ell \left( I\right) \leq \ell \left( J\right) \leq 2^{\mathbf{\rho }%
}\ell \left( I\right) }}+\sum_{\substack{ \left( I,J\right) \in \Omega 
\mathcal{D}^{\sigma }\times \Omega \mathcal{D}^{\omega }  \\ 2^{-\mathbf{%
\rho }}\ell \left( I\right) \leq \ell \left( J\right) \leq 2^{\mathbf{\rho }%
}\ell \left( I\right)  \\ J\not\subset 3I\text{ and }I\not\subset 3J}}%
\right\} \left\vert \left\langle T_{\sigma }^{\alpha }\left( \bigtriangleup
_{I}^{\sigma }f\right) ,\bigtriangleup _{J}^{\omega }g\right\rangle _{\omega
}\right\vert \\
&&\ \ \ \ \ \ \ \ \ \ \equiv A_{1}+A_{2}+A_{3}.
\end{eqnarray*}%
The proof of the bound for term $A_{3}$ is similar to that of the bound for
the left side of (\ref{delta far}), and so we will defer the bound for $%
A_{3} $ until after (\ref{delta far}) has been proved.

We now consider term $A_{1}$ as term $A_{2}$ is symmetric. To handle this
term we will write the quasiHaar functions $h_{I}^{\sigma }$ and $%
h_{J}^{\omega }$ as linear combinations of the indicators of the children of
their supporting quasicubes, denoted $I_{\theta }$ and $J_{\theta ^{\prime
}} $ respectively. Then we use the quasitesting condition on $I_{\theta }$
and $J_{\theta ^{\prime }}$ when they \emph{overlap}, i.e. their interiors
intersect; we use the quasiweak boundedness property on $I_{\theta }$ and $%
J_{\theta ^{\prime }}$ when they \emph{touch}, i.e. their interiors are
disjoint but their closures intersect (even in just a point); and finally we
use the $A_{2}^{\alpha }$ condition when $I_{\theta }$ and $J_{\theta
^{\prime }}$ are \emph{separated}, i.e. their closures are disjoint. We will
suppose initially that the side length of $J$ is at most the side length $I$%
, i.e. $\ell \left( J\right) \leq \ell \left( I\right) $, the proof for $%
J=\pi I$ being similar but for one point mentioned below. So suppose that $%
I_{\theta }$ is a child of $I$ and that $J_{\theta ^{\prime }}$ is a child
of $J$. If $J_{\theta ^{\prime }}\subset I_{\theta }$ we have from (\ref%
{useful Haar}) that, 
\begin{eqnarray*}
\left\vert \left\langle T_{\sigma }^{\alpha }\left( \mathbf{1}_{I_{\theta
}}\bigtriangleup _{I}^{\sigma }f\right) ,\mathbf{1}_{J_{\theta ^{\prime
}}}\bigtriangleup _{J}^{\omega }g\right\rangle _{\omega }\right\vert
&\lesssim &\sup_{a,a^{\prime }\in \Gamma _{n}}\frac{\left\vert \left\langle
f,h_{I}^{\sigma ,a}\right\rangle _{\sigma }\right\vert }{\sqrt{\left\vert
I_{\theta }\right\vert _{\sigma }}}\left\vert \left\langle T_{\sigma
}^{\alpha }\left( \mathbf{1}_{I_{\theta }}\right) ,\mathbf{1}_{J_{\theta
^{\prime }}}\right\rangle _{\omega }\right\vert \frac{\left\vert
\left\langle g,h_{J}^{\omega ,a^{\prime }}\right\rangle _{\omega
}\right\vert }{\sqrt{\left\vert J_{\theta ^{\prime }}\right\vert _{\omega }}}
\\
&\lesssim &\sup_{a,a^{\prime }\in \Gamma _{n}}\frac{\left\vert \left\langle
f,h_{I}^{\sigma ,a}\right\rangle _{\sigma }\right\vert }{\sqrt{\left\vert
I_{\theta }\right\vert _{\sigma }}}\left( \int_{J_{\theta ^{\prime
}}}\left\vert T_{\sigma }^{\alpha }\left( \mathbf{1}_{I_{\theta }}\right)
\right\vert ^{2}d\omega \right) ^{\frac{1}{2}}\left\vert \left\langle
g,h_{J}^{\omega ,a^{\prime }}\right\rangle _{\omega }\right\vert \\
&\lesssim &\sup_{a,a^{\prime }\in \Gamma _{n}}\frac{\left\vert \left\langle
f,h_{I}^{\sigma ,a}\right\rangle _{\sigma }\right\vert }{\sqrt{\left\vert
I_{\theta }\right\vert _{\sigma }}}\mathfrak{T}_{T_{\alpha }}\left\vert
I_{\theta }\right\vert _{\sigma }^{\frac{1}{2}}\left\vert \left\langle
g,h_{J}^{\omega ,a^{\prime }}\right\rangle _{\omega }\right\vert \\
&\lesssim &\sup_{a,a^{\prime }\in \Gamma _{n}}\mathfrak{T}_{T_{\alpha
}}\left\vert \left\langle f,h_{I}^{\sigma ,a}\right\rangle _{\sigma
}\right\vert \left\vert \left\langle g,h_{J}^{\omega ,a^{\prime
}}\right\rangle _{\omega }\right\vert .
\end{eqnarray*}%
The point referred to above is that when $J=\pi I$ we write $\left\langle
T_{\sigma }^{\alpha }\left( \mathbf{1}_{I_{\theta }}\right) ,\mathbf{1}%
_{J_{\theta ^{\prime }}}\right\rangle _{\omega }=\left\langle \mathbf{1}%
_{I_{\theta }},T_{\omega }^{\alpha ,\ast }\left( \mathbf{1}_{J_{\theta
^{\prime }}}\right) \right\rangle _{\sigma }$ and get the dual quasitesting
constant $\mathfrak{T}_{T_{\alpha }}^{\ast }$. If $J_{\theta ^{\prime }}$
and $I_{\theta }$ touch, then $\ell \left( J_{\theta ^{\prime }}\right) \leq
\ell \left( I_{\theta }\right) $ and we have $J_{\theta ^{\prime }}\subset
3I_{\theta }\setminus I_{\theta }$, and so%
\begin{eqnarray*}
\left\vert \left\langle T_{\sigma }^{\alpha }\left( \mathbf{1}_{I_{\theta
}}\bigtriangleup _{I}^{\sigma }f\right) ,\mathbf{1}_{J_{\theta ^{\prime
}}}\bigtriangleup _{J}^{\omega }g\right\rangle _{\omega }\right\vert
&\lesssim &\sup_{a,a^{\prime }\in \Gamma _{n}}\frac{\left\vert \left\langle
f,h_{I}^{\sigma ,a}\right\rangle _{\sigma }\right\vert }{\sqrt{\left\vert
I_{\theta }\right\vert _{\sigma }}}\left\vert \left\langle T_{\sigma
}^{\alpha }\left( \mathbf{1}_{I_{\theta }}\right) ,\mathbf{1}_{J_{\theta
^{\prime }}}\right\rangle _{\omega }\right\vert \frac{\left\vert
\left\langle g,h_{J}^{\omega ,a^{\prime }}\right\rangle _{\omega
}\right\vert }{\sqrt{\left\vert J_{\theta ^{\prime }}\right\vert _{\omega }}}
\\
&\lesssim &\sup_{a,a^{\prime }\in \Gamma _{n}}\frac{\left\vert \left\langle
f,h_{I}^{\sigma ,a}\right\rangle _{\sigma }\right\vert }{\sqrt{\left\vert
I_{\theta }\right\vert _{\sigma }}}\mathcal{WBP}_{T^{\alpha }}\sqrt{%
\left\vert I_{\theta }\right\vert _{\sigma }\left\vert J_{\theta ^{\prime
}}\right\vert _{\omega }}\frac{\left\vert \left\langle g,h_{J}^{\omega
,a^{\prime }}\right\rangle _{\omega }\right\vert }{\sqrt{\left\vert
J_{\theta ^{\prime }}\right\vert _{\omega }}} \\
&=&\sup_{a,a^{\prime }\in \Gamma _{n}}\mathcal{WBP}_{T^{\alpha }}\left\vert
\left\langle f,h_{I}^{\sigma ,a}\right\rangle _{\sigma }\right\vert
\left\vert \left\langle g,h_{J}^{\omega ,a^{\prime }}\right\rangle _{\omega
}\right\vert .
\end{eqnarray*}%
Finally, if $J_{\theta ^{\prime }}$ and $I_{\theta }$ are separated, and if $%
K$ is the smallest (not necessarily dyadic) quasicube containing both $%
J_{\theta ^{\prime }}$ and $I_{\theta }$, then $\limfunc{dist}\left(
I_{\theta },J_{\theta ^{\prime }}\right) \approx \left\vert K\right\vert ^{%
\frac{1}{n}}$ and we have%
\begin{eqnarray*}
\left\vert \left\langle T_{\sigma }^{\alpha }\left( \mathbf{1}_{I_{\theta
}}\bigtriangleup _{I}^{\sigma }f\right) ,\mathbf{1}_{J_{\theta ^{\prime
}}}\bigtriangleup _{J}^{\omega }g\right\rangle _{\omega }\right\vert
&\lesssim &\sup_{a,a^{\prime }\in \Gamma _{n}}\frac{\left\vert \left\langle
f,h_{I}^{\sigma ,a}\right\rangle _{\sigma }\right\vert }{\sqrt{\left\vert
I_{\theta }\right\vert _{\sigma }}}\left\vert \left\langle T_{\sigma
}^{\alpha }\left( \mathbf{1}_{I_{\theta }}\right) ,\mathbf{1}_{J_{\theta
^{\prime }}}\right\rangle _{\omega }\right\vert \frac{\left\vert
\left\langle g,h_{J}^{\omega ,a^{\prime }}\right\rangle _{\omega
}\right\vert }{\sqrt{\left\vert J_{\theta ^{\prime }}\right\vert _{\omega }}}
\\
&\lesssim &\sup_{a,a^{\prime }\in \Gamma _{n}}\frac{\left\vert \left\langle
f,h_{I}^{\sigma ,a}\right\rangle _{\sigma }\right\vert }{\sqrt{\left\vert
I_{\theta }\right\vert _{\sigma }}}\frac{1}{\limfunc{dist}\left( I_{\theta
},J_{\theta ^{\prime }}\right) ^{n-\alpha }}\left\vert I_{\theta
}\right\vert _{\sigma }\left\vert J_{\theta ^{\prime }}\right\vert _{\omega }%
\frac{\left\vert \left\langle g,h_{J}^{\omega ,a^{\prime }}\right\rangle
_{\omega }\right\vert }{\sqrt{\left\vert J_{\theta ^{\prime }}\right\vert
_{\omega }}} \\
&=&\sup_{a,a^{\prime }\in \Gamma _{n}}\frac{\sqrt{\left\vert I_{\theta
}\right\vert _{\sigma }\left\vert J_{\theta ^{\prime }}\right\vert _{\omega }%
}}{\limfunc{dist}\left( I_{\theta },J_{\theta ^{\prime }}\right) ^{n-\alpha }%
}\left\vert \left\langle f,h_{I}^{\sigma ,a}\right\rangle _{\sigma
}\right\vert \left\vert \left\langle g,h_{J}^{\omega ,a^{\prime
}}\right\rangle _{\omega }\right\vert \\
&\lesssim &\sup_{a,a^{\prime }\in \Gamma _{n}}\frac{\sqrt{\left\vert
K\right\vert _{\sigma }\left\vert K\right\vert _{\omega }}}{\left\vert
K\right\vert ^{\frac{1}{n}\left( n-\alpha \right) }}\left\vert \left\langle
f,h_{I}^{\sigma ,a}\right\rangle _{\sigma }\right\vert \left\vert
\left\langle g,h_{J}^{\omega ,a^{\prime }}\right\rangle _{\omega }\right\vert
\\
&\lesssim &\sqrt{A_{2}^{\alpha }}\sup_{a,a^{\prime }\in \Gamma
_{n}}\left\vert \left\langle f,h_{I}^{\sigma ,a}\right\rangle _{\sigma
}\right\vert \left\vert \left\langle g,h_{J}^{\omega ,a^{\prime
}}\right\rangle _{\omega }\right\vert .
\end{eqnarray*}%
Now we sum over all the children of $J$ and $I$ satisfying $2^{-\mathbf{\rho 
}}\ell \left( I\right) \leq \ell \left( J\right) \leq 2^{\mathbf{\rho }}\ell
\left( I\right) $ for which $J\subset 3I$ to obtain that%
\begin{equation*}
A_{1}\lesssim \left( \mathfrak{T}_{T_{\alpha }}+\mathfrak{T}_{T_{\alpha
}}^{\ast }+\mathcal{WBP}_{T^{\alpha }}+\sqrt{A_{2}^{\alpha }}\right)
\sup_{a,a^{\prime }\in \Gamma _{n}}\sum_{\substack{ \left( I,J\right) \in
\Omega \mathcal{D}^{\sigma }\times \Omega \mathcal{D}^{\omega }:\ J\subset
3I  \\ 2^{-\mathbf{\rho }}\ell \left( I\right) \leq \ell \left( J\right)
\leq 2^{\mathbf{\rho }}\ell \left( I\right) }}\left\vert \left\langle
f,h_{I}^{\sigma ,a}\right\rangle _{\sigma }\right\vert \left\vert
\left\langle g,h_{J}^{\omega ,a^{\prime }}\right\rangle _{\omega
}\right\vert \ .
\end{equation*}%
Now Cauchy-Schwarz gives the estimate%
\begin{eqnarray*}
&&\sup_{a,a^{\prime }\in \Gamma _{n}}\sum_{\substack{ \left( I,J\right) \in
\Omega \mathcal{D}^{\sigma }\times \Omega \mathcal{D}^{\omega }:\ J\subset
3I  \\ 2^{-\mathbf{\rho }}\ell \left( I\right) \leq \ell \left( J\right)
\leq 2^{\mathbf{\rho }}\ell \left( I\right) }}\left\vert \left\langle
f,h_{I}^{\sigma ,a}\right\rangle _{\sigma }\right\vert \left\vert
\left\langle g,h_{J}^{\omega ,a^{\prime }}\right\rangle _{\omega }\right\vert
\\
&\leq &\sup_{a,a^{\prime }\in \Gamma _{n}}\left( \sum_{\substack{ \left(
I,J\right) \in \Omega \mathcal{D}^{\sigma }\times \Omega \mathcal{D}^{\omega
}:\ J\subset 3I  \\ 2^{-\mathbf{\rho }}\ell \left( I\right) \leq \ell \left(
J\right) \leq 2^{\mathbf{\rho }}\ell \left( I\right) }}\left\vert
\left\langle f,h_{I}^{\sigma }\right\rangle _{\sigma }\right\vert
^{2}\right) ^{\frac{1}{2}}\left( \sum_{\substack{ \left( I,J\right) \in
\Omega \mathcal{D}^{\sigma }\times \Omega \mathcal{D}^{\omega }:\ J\subset
3I  \\ 2^{-\mathbf{\rho }}\ell \left( I\right) \leq \ell \left( J\right)
\leq 2^{\mathbf{\rho }}\ell \left( I\right) }}\left\vert \left\langle
g,h_{J}^{\omega }\right\rangle _{\omega }\right\vert ^{2}\right) ^{\frac{1}{2%
}} \\
&\lesssim &\left\Vert f\right\Vert _{L^{2}\left( \sigma \right) }\left\Vert
g\right\Vert _{L^{2}\left( \omega \right) }\ ,
\end{eqnarray*}%
This completes our proof of (\ref{delta near}) save for the deferral of term 
$A_{3}$, which we bound below.

\bigskip

Now we turn to the sum of separated cubes in (\ref{delta far}) and (\ref%
{indicator far}). In each of these inequalities we have either orthogonality
or `quasi' orthogonality, due either to the presence of a quasiHaar
projection such as $\bigtriangleup _{I}^{\sigma }f$, or the presence of an
appropriate Carleson indicator such as $\beta _{F,J}\mathbf{1}_{F}\alpha _{%
\mathcal{F}}\left( F\right) $. We will prove below the estimate for the
separated sum corresponding to (\ref{delta far}). The corresponding
estimates for (\ref{indicator far}) are handled in a similar way, the only
difference being that the `quasi' orthogonality of Carleson indicators such
as $\beta _{F,J}\mathbf{1}_{F}\alpha _{\mathcal{F}}\left( F\right) $ is used
in place of the orthogonality of quasiHaar functions such as $\bigtriangleup
_{I}^{\sigma }f$. The bounded functions $\beta _{F,J}$ are replaced with
constants after an application of the energy lemma, and then the arguments
proceed as below.

We split the pairs $\left( I,J\right) \in \Omega \mathcal{D}^{\sigma }\times
\Omega \mathcal{D}^{\omega }$ occurring in (\ref{delta far}) into two
groups, those with side length of $J$ smaller than side length of $I$, and
those with side length of $I$ smaller than side length of $J$, treating only
the former case, the latter being symmetric. Thus we prove the following
bound:%
\begin{eqnarray*}
\mathcal{A}\left( f,g\right) &\equiv &\sum_{\substack{ \left( I,J\right) \in
\Omega \mathcal{D}^{\sigma }\times \Omega \mathcal{D}^{\omega }  \\ I\cap
J=\emptyset \text{ and }\ell \left( J\right) \leq 2^{-\mathbf{\rho }}\ell
\left( I\right) }}\left\vert \left\langle T_{\sigma }^{\alpha }\left(
\bigtriangleup _{I}^{\sigma }f\right) ,\bigtriangleup _{J}^{\omega
}g\right\rangle _{\omega }\right\vert \\
&\lesssim &\sqrt{A_{2}^{\alpha }}\left\Vert f\right\Vert _{L^{2}\left(
\sigma \right) }\left\Vert g\right\Vert _{L^{2}\left( \omega \right) }.
\end{eqnarray*}%
We apply the `pivotal' bound from the Energy Lemma \ref{ener} to estimate
the inner product $\left\langle T_{\sigma }^{\alpha }\left( \bigtriangleup
_{I}^{\sigma }f\right) ,\bigtriangleup _{J}^{\omega }g\right\rangle _{\omega
}$ and obtain,%
\begin{equation*}
\left\vert \left\langle T_{\sigma }^{\alpha }\left( \bigtriangleup
_{I}^{\sigma }f\right) ,\bigtriangleup _{J}^{\omega }g\right\rangle _{\omega
}\right\vert \lesssim \left\Vert \bigtriangleup _{J}^{\omega }g\right\Vert
_{L^{2}\left( \omega \right) }\mathrm{P}^{\alpha }\left( J,\left\vert
\bigtriangleup _{I}^{\sigma }f\right\vert \sigma \right) \sqrt{\left\vert
J\right\vert _{\omega }}\,,
\end{equation*}%
Denote by $\limfunc{dist}$ the $\ell ^{\infty }$ distance in $\mathbb{R}^{n}$%
: $\limfunc{dist}\left( x,y\right) =\max_{1\leq j\leq n}\left\vert
x_{j}-y_{j}\right\vert $, and denote the corresponding quasidistance by $%
\limfunc{quasidist}\left( x,y\right) =\limfunc{dist}\left( \Omega
^{-1}x,\Omega ^{-1}y\right) $. We now estimate separately the long-range and
mid-range cases where $\limfunc{quasidist}\left( J,I\right) \geq \ell \left(
I\right) $ holds or not, and we decompose $\mathcal{A}$ accordingly:%
\begin{equation*}
\mathcal{A}\left( f,g\right) \equiv \mathcal{A}^{\limfunc{long}}\left(
f,g\right) +\mathcal{A}^{\limfunc{mid}}\left( f,g\right) .
\end{equation*}

\bigskip

\textbf{The long-range case}: We begin with the case where $\limfunc{%
quasidist}\left( J,I\right) $ is at least $\ell \left( I\right) $, i.e. $%
J\cap 3I=\emptyset $. Since $J$ and $I$ are separated by at least $\max
\left\{ \ell \left( J\right) ,\ell \left( I\right) \right\} $, we have the
inequality%
\begin{equation*}
\mathrm{P}^{\alpha }\left( J,\left\vert \bigtriangleup _{I}^{\sigma
}f\right\vert \sigma \right) \approx \int_{I}\frac{\ell \left( J\right) }{%
\left\vert y-c_{J}\right\vert ^{n+1-\alpha }}\left\vert \bigtriangleup
_{I}^{\sigma }f\left( y\right) \right\vert d\sigma \left( y\right) \lesssim
\left\Vert \bigtriangleup _{I}^{\sigma }f\right\Vert _{L^{2}\left( \sigma
\right) }\frac{\ell \left( J\right) \sqrt{\left\vert I\right\vert _{\sigma }}%
}{\limfunc{quasidist}\left( I,J\right) ^{n+1-\alpha }},
\end{equation*}%
since $\int_{I}\left\vert \bigtriangleup _{I}^{\sigma }f\left( y\right)
\right\vert d\sigma \left( y\right) \leq \left\Vert \bigtriangleup
_{I}^{\sigma }f\right\Vert _{L^{2}\left( \sigma \right) }\sqrt{\left\vert
I\right\vert _{\sigma }}$. Thus with $A\left( f,g\right) =\mathcal{A}^{%
\limfunc{long}}\left( f,g\right) $ we have%
\begin{eqnarray*}
A\left( f,g\right) &\lesssim &\sum_{I\in \Omega \mathcal{D}}\sum_{J\;:\;\ell
\left( J\right) \leq \ell \left( I\right) :\ \limfunc{quasidist}\left(
I,J\right) \geq \ell \left( I\right) }\left\Vert \bigtriangleup _{I}^{\sigma
}f\right\Vert _{L^{2}\left( \sigma \right) }\left\Vert \bigtriangleup
_{J}^{\omega }g\right\Vert _{L^{2}\left( \omega \right) } \\
&&\ \ \ \ \ \ \ \ \ \ \ \ \ \ \ \times \frac{\ell \left( J\right) }{\limfunc{%
quasidist}\left( I,J\right) ^{n+1-\alpha }}\sqrt{\left\vert I\right\vert
_{\sigma }}\sqrt{\left\vert J\right\vert _{\omega }} \\
&\equiv &\sum_{\left( I,J\right) \in \mathcal{P}}\left\Vert \bigtriangleup
_{I}^{\sigma }f\right\Vert _{L^{2}\left( \sigma \right) }\left\Vert
\bigtriangleup _{J}^{\omega }g\right\Vert _{L^{2}\left( \omega \right)
}A\left( I,J\right) ; \\
\text{with }A\left( I,J\right) &\equiv &\frac{\ell \left( J\right) }{%
\limfunc{quasidist}\left( I,J\right) ^{n+1-\alpha }}\sqrt{\left\vert
I\right\vert _{\sigma }}\sqrt{\left\vert J\right\vert _{\omega }}; \\
\text{ and }\mathcal{P} &\equiv &\left\{ \left( I,J\right) \in \Omega 
\mathcal{D}\times \Omega \mathcal{D}:\ell \left( J\right) \leq \ell \left(
I\right) \text{ and }\limfunc{quasidist}\left( I,J\right) \geq \ell \left(
I\right) \right\} .
\end{eqnarray*}%
Now let $\Omega \mathcal{D}_{N}\equiv \left\{ K\in \Omega \mathcal{D}:\ell
\left( K\right) =2^{N}\right\} $ for each $N\in \mathbb{Z}$. For $N\in 
\mathbb{Z}$ and $s\in \mathbb{Z}_{+}$, we further decompose $A\left(
f,g\right) $ by pigeonholing the side lengths of $I$ and $J$ by $2^{N}$ and $%
2^{N-s}$ respectively: 
\begin{eqnarray*}
A\left( f,g\right) &=&\sum_{s=0}^{\infty }\sum_{N\in \mathbb{Z}%
}A_{N}^{s}\left( f,g\right) ; \\
A_{N}^{s}\left( f,g\right) &\equiv &\sum_{\left( I,J\right) \in \mathcal{P}%
_{N}^{s}}\left\Vert \bigtriangleup _{I}^{\sigma }f\right\Vert _{L^{2}\left(
\sigma \right) }\left\Vert \bigtriangleup _{J}^{\omega }g\right\Vert
_{L^{2}\left( \omega \right) }A\left( I,J\right) \\
\text{where }\mathcal{P}_{N}^{s} &\equiv &\left\{ \left( I,J\right) \in
\Omega \mathcal{D}_{N}\times \Omega \mathcal{D}_{N-s}:\limfunc{quasidist}%
\left( I,J\right) \geq \ell \left( I\right) \right\} .
\end{eqnarray*}%
Now $A_{N}^{s}\left( f,g\right) =A_{N}^{s}\left( \mathsf{P}_{N}^{\sigma }f,%
\mathsf{P}_{N-s}^{\omega }g\right) $ where $\mathsf{P}_{M}^{\mu
}=\dsum\limits_{K\in \Omega \mathcal{D}_{M}}\bigtriangleup _{K}^{\mu }$
denotes quasiHaar projection onto $\limfunc{Span}\left\{ h_{K}^{\mu
,a}\right\} _{K\in \Omega \mathcal{D}_{M},a\in \Gamma _{n}}$, and so by
orthogonality of the projections $\left\{ \mathsf{P}_{M}^{\mu }\right\}
_{M\in \mathbb{Z}}$ we have%
\begin{eqnarray*}
\left\vert \sum_{N\in \mathbb{Z}}A_{N}^{s}\left( f,g\right) \right\vert
&=&\sum_{N\in \mathbb{Z}}\left\vert A_{N}^{s}\left( \mathsf{P}_{N}^{\sigma
}f,\mathsf{P}_{N-s}^{\omega }g\right) \right\vert \leq \sum_{N\in \mathbb{Z}%
}\left\Vert A_{N}^{s}\right\Vert \left\Vert \mathsf{P}_{N}^{\sigma
}f\right\Vert _{L^{2}\left( \sigma \right) }\left\Vert \mathsf{P}%
_{N-s}^{\omega }g\right\Vert _{L^{2}\left( \omega \right) } \\
&\leq &\left\{ \sup_{N\in \mathbb{Z}}\left\Vert A_{N}^{s}\right\Vert
\right\} \left( \sum_{N\in \mathbb{Z}}\left\Vert \mathsf{P}_{N}^{\sigma
}f\right\Vert _{L^{2}\left( \sigma \right) }^{2}\right) ^{\frac{1}{2}}\left(
\sum_{N\in \mathbb{Z}}\left\Vert \mathsf{P}_{N-s}^{\omega }g\right\Vert
_{L^{2}\left( \omega \right) }^{2}\right) ^{\frac{1}{2}} \\
&\leq &\left\{ \sup_{N\in \mathbb{Z}}\left\Vert A_{N}^{s}\right\Vert
\right\} \left\Vert f\right\Vert _{L^{2}\left( \sigma \right) }\left\Vert
g\right\Vert _{L^{2}\left( \omega \right) }.
\end{eqnarray*}%
Thus it suffices to show an estimate uniform in $N$ with geometric decay in $%
s$, and we will show%
\begin{equation}
\left\vert A_{N}^{s}\left( f,g\right) \right\vert \leq C2^{-s}\sqrt{%
A_{2}^{\alpha }}\left\Vert f\right\Vert _{L^{2}\left( \sigma \right)
}\left\Vert g\right\Vert _{L^{2}\left( \omega \right) },\ \ \ \ \ \text{for }%
s\geq 0\text{ and }N\in \mathbb{Z}.  \label{AsN}
\end{equation}

We now pigeonhole the distance between $I$ and $J$:%
\begin{eqnarray*}
A_{N}^{s}\left( f,g\right) &=&\dsum\limits_{\ell =0}^{\infty }A_{N,\ell
}^{s}\left( f,g\right) ; \\
A_{N,\ell }^{s}\left( f,g\right) &\equiv &\sum_{\left( I,J\right) \in 
\mathcal{P}_{N,\ell }^{s}}\left\Vert \bigtriangleup _{I}^{\sigma
}f\right\Vert _{L^{2}\left( \sigma \right) }\left\Vert \bigtriangleup
_{J}^{\omega }g\right\Vert _{L^{2}\left( \omega \right) }A\left( I,J\right)
\\
\text{where }\mathcal{P}_{N,\ell }^{s} &\equiv &\left\{ \left( I,J\right)
\in \Omega \mathcal{D}_{N}\times \Omega \mathcal{D}_{N-s}:\limfunc{quasidist}%
\left( I,J\right) \approx 2^{N+\ell }\right\} .
\end{eqnarray*}%
If we define $\mathcal{H}\left( A_{N,\ell }^{s}\right) $ to be the bilinear
form on $\ell ^{2}\times \ell ^{2}$ with matrix $\left[ A\left( I,J\right) %
\right] _{\left( I,J\right) \in \mathcal{P}_{N,\ell }^{s}}$, then it remains
to show that the norm $\left\Vert \mathcal{H}\left( A_{N,\ell }^{s}\right)
\right\Vert _{\ell ^{2}\rightarrow \ell ^{2}}$ of $\mathcal{H}\left(
A_{N,\ell }^{s}\right) $ on the sequence space $\ell ^{2}$ is bounded by $%
C2^{-s-\ell }\sqrt{A_{2}^{\alpha }}$. In turn, this is equivalent to showing
that the norm $\left\Vert \mathcal{H}\left( B_{N,\ell }^{s}\right)
\right\Vert _{\ell ^{2}\rightarrow \ell ^{2}}$ of the bilinear form $%
\mathcal{H}\left( B_{N,\ell }^{s}\right) \equiv \mathcal{H}\left( A_{N,\ell
}^{s}\right) ^{\limfunc{tr}}\mathcal{H}\left( A_{N,\ell }^{s}\right) $ on
the sequence space $\ell ^{2}$ is bounded by $C^{2}2^{-2s-2\ell
}A_{2}^{\alpha }$. Here $\mathcal{H}\left( B_{N,\ell }^{s}\right) $ is the
quadratic form with matrix kernel $\left[ B_{N,\ell }^{s}\left( J,J^{\prime
}\right) \right] _{J,J^{\prime }\in \Omega \mathcal{D}_{N-s}}$ having
entries:%
\begin{equation*}
B_{N,\ell }^{s}\left( J,J^{\prime }\right) \equiv \sum_{I\in \Omega \mathcal{%
D}_{N}:\ \limfunc{quasidist}\left( I,J\right) \approx \limfunc{quasidist}%
\left( I,J^{\prime }\right) \approx 2^{N+\ell }}A\left( I,J\right) A\left(
I,J^{\prime }\right) ,\ \ \ \ \ \text{for }J,J^{\prime }\in \Omega \mathcal{D%
}_{N-s}.
\end{equation*}

We are reduced to showing,%
\begin{equation*}
\left\Vert \mathcal{H}\left( B_{N,\ell }^{s}\right) \right\Vert _{\ell
^{2}\rightarrow \ell ^{2}}\leq C2^{-2s-2\ell }A_{2}^{\alpha }\ \ \ \ \text{%
for }s\geq 0\text{, }\ell \geq 0\text{ and }N\in \mathbb{Z}.
\end{equation*}%
For this we begin by computing $B_{N,\ell }^{s}\left( J,J^{\prime }\right) $:%
\begin{eqnarray*}
B_{N,\ell }^{s}\left( J,J^{\prime }\right) &=&\sum_{\substack{ I\in \Omega 
\mathcal{D}_{N}  \\ \limfunc{quasidist}\left( I,J\right) \approx \limfunc{%
quasidist}\left( I,J^{\prime }\right) \approx 2^{N+\ell }}}\frac{\ell \left(
J\right) }{\limfunc{quasidist}\left( I,J\right) ^{n+1-\alpha }}\sqrt{%
\left\vert I\right\vert _{\sigma }}\sqrt{\left\vert J\right\vert _{\omega }}
\\
&&\ \ \ \ \ \ \ \ \ \ \times \frac{\ell \left( J^{\prime }\right) }{\limfunc{%
quasidist}\left( I,J^{\prime }\right) ^{n+1-\alpha }}\sqrt{\left\vert
I\right\vert _{\sigma }}\sqrt{\left\vert J^{\prime }\right\vert _{\omega }}
\\
&=&\left\{ \sum_{\substack{ I\in \Omega \mathcal{D}_{N}  \\ \limfunc{%
quasidist}\left( I,J\right) \approx \limfunc{quasidist}\left( I,J^{\prime
}\right) \approx 2^{N+\ell }}}\left\vert I\right\vert _{\sigma }\frac{1}{%
\limfunc{quasidist}\left( I,J\right) ^{n+1-\alpha }\limfunc{quasidist}\left(
I,J^{\prime }\right) ^{n+1-\alpha }}\right\} \\
&&\ \ \ \ \ \ \ \ \ \ \times \ell \left( J\right) \ell \left( J^{\prime
}\right) \sqrt{\left\vert J\right\vert _{\omega }}\sqrt{\left\vert J^{\prime
}\right\vert _{\omega }}.
\end{eqnarray*}%
Now we show that%
\begin{equation}
\left\Vert B_{N,\ell }^{s}\right\Vert _{\ell ^{2}\rightarrow \ell
^{2}}\lesssim 2^{-2s-2\ell }A_{2}^{\alpha }\ ,  \label{Schur s}
\end{equation}%
by applying the proof of Schur's lemma. Fix $\ell \geq 0$ and $s\geq 0$.
Choose the Schur function $\beta \left( K\right) =\frac{1}{\sqrt{\left\vert
K\right\vert _{\omega }}}$. Fix $J\in \Omega \mathcal{D}_{N-s}$. We have%
\begin{eqnarray*}
&&\sum_{J^{\prime }\in \Omega \mathcal{D}_{N-s}}\frac{\beta \left( J\right) 
}{\beta \left( J^{\prime }\right) }B_{N,\ell }^{s}\left( J,J^{\prime }\right)
\\
&\lesssim &\sum_{\substack{ J^{\prime }\in \Omega \mathcal{D}_{N-s}  \\ 
\limfunc{quasidist}\left( J,J^{\prime }\right) \leq 2^{N+\ell +2}}}\left\{
\sum_{\substack{ I\in \Omega \mathcal{D}_{N}  \\ \limfunc{quasidist}\left(
I,J\right) \approx 2^{N+\ell }}}\left\vert I\right\vert _{\sigma }\right\} \ 
\frac{2^{2\left( N-s\right) }}{2^{2\left( \ell +N\right) \left( n+1-\alpha
\right) }}\left\vert J^{\prime }\right\vert _{\omega } \\
&\lesssim &2^{-2s-2\ell }\frac{\left\vert 2^{10+\ell +s}J\right\vert
_{\sigma }}{2^{\left( \ell +N\right) \left( n-\alpha \right) }}\frac{%
\left\vert 2^{12+\ell +s}J\right\vert _{\omega }}{2^{\left( \ell +N\right)
\left( n-\alpha \right) }}\lesssim 2^{-2s-2\ell }A_{2}^{\alpha },
\end{eqnarray*}%
since $I\in \Omega \mathcal{D}_{N}$ and $\limfunc{quasidist}\left(
I,J\right) \approx 2^{N+\ell }$ imply that $I\subset 2^{10+\ell +s}J$ which
has side length comparable to $2^{\left( \ell +N\right) }$, and similarly $%
J^{\prime }\subset 2^{12+\ell +s}J$. Thus we can now apply Schur's argument
with $\sum_{J}\left( a_{J}\right) ^{2}=\sum_{J^{\prime }}\left( b_{J^{\prime
}}\right) ^{2}=1$ to obtain%
\begin{eqnarray*}
&&\sum_{J,J^{\prime }\in \Omega \mathcal{D}_{N-s}}a_{J}b_{J^{\prime
}}B_{N,\ell }^{s}\left( J,J^{\prime }\right) \\
&=&\sum_{J,J^{\prime }\in \Omega \mathcal{D}_{N-s}}a_{J}\beta \left(
J\right) b_{J^{\prime }}\beta \left( J^{\prime }\right) \frac{B_{N,\ell
}^{s}\left( J,J^{\prime }\right) }{\beta \left( J\right) \beta \left(
J^{\prime }\right) } \\
&\leq &\sum_{J}\left( a_{J}\beta \left( J\right) \right) ^{2}\sum_{J^{\prime
}}\frac{B_{N,\ell }^{s}\left( J,J^{\prime }\right) }{\beta \left( J\right)
\beta \left( J^{\prime }\right) }+\sum_{J^{\prime }}\left( b_{J^{\prime
}}\beta \left( J^{\prime }\right) \right) ^{2}\frac{B_{N,\ell }^{s}\left(
J,J^{\prime }\right) }{\beta \left( J\right) \beta \left( J^{\prime }\right) 
} \\
&=&\sum_{J}\left( a_{J}\right) ^{2}\left\{ \sum_{J^{\prime }}\frac{\beta
\left( J\right) }{\beta \left( J^{\prime }\right) }B_{N,\ell }^{s}\left(
J,J^{\prime }\right) \right\} +\sum_{J^{\prime }}\left( b_{J^{\prime
}}\right) ^{2}\left\{ \sum_{J}\frac{\beta \left( J^{\prime }\right) }{\beta
\left( J\right) }B_{N,\ell }^{s}\left( J,J^{\prime }\right) \right\} \\
&\lesssim &2^{-2s-2\ell }A_{2}^{\alpha }\left( \sum_{J}\left( a_{J}\right)
^{2}+\sum_{J^{\prime }}\left( b_{J^{\prime }}\right) ^{2}\right)
=2^{1-2s-2\ell }A_{2}^{\alpha }.
\end{eqnarray*}%
This completes the proof of (\ref{Schur s}). We can now sum in $\ell $ to
get (\ref{AsN}) and we are done. This completes our proof of the long-range
estimate%
\begin{equation*}
\mathcal{A}^{\limfunc{long}}\left( f,g\right) \lesssim \sqrt{A_{2}^{\alpha }}%
\left\Vert f\right\Vert _{L^{2}\left( \sigma \right) }\left\Vert
g\right\Vert _{L^{2}\left( \omega \right) }\ .
\end{equation*}

\bigskip

At this point we pause to complete the proof of (\ref{delta near}). Indeed,
the deferred term $A_{3}$ can be handled using the above argument since $%
3J\cap I=\emptyset =J\cap 3I$ implies that we can use the Energy Lemma \ref%
{ener} as we did above.

\bigskip

\textbf{The mid range case}: Let%
\begin{equation*}
\mathcal{P}\equiv \left\{ \left( I,J\right) \in \Omega \mathcal{D}\times
\Omega \mathcal{D}:J\text{ is good},\ \ell \left( J\right) \leq 2^{-\mathbf{%
\rho }}\ell \left( I\right) ,\text{ }J\subset 3I\setminus I\right\} .
\end{equation*}%
For $\left( I,J\right) \in \mathcal{P}$, the `pivotal' estimate from the
Energy Lemma \ref{ener} gives%
\begin{equation*}
\left\vert \left\langle T_{\sigma }^{\alpha }\left( \bigtriangleup
_{I}^{\sigma }f\right) ,\bigtriangleup _{J}^{\omega }g\right\rangle _{\omega
}\right\vert \lesssim \left\Vert \bigtriangleup _{J}^{\omega }g\right\Vert
_{L^{2}\left( \omega \right) }\mathrm{P}^{\alpha }\left( J,\left\vert
\bigtriangleup _{I}^{\sigma }f\right\vert \sigma \right) \sqrt{\left\vert
J\right\vert _{\omega }}\,.
\end{equation*}%
Now we pigeonhole the lengths of $I$ and $J$ and the distance between them
by defining%
\begin{equation*}
\mathcal{P}_{N,d}^{s}\equiv \left\{ \left( I,J\right) \in \Omega \mathcal{D}%
\times \Omega \mathcal{D}:J\text{ is good},\ \ell \left( I\right) =2^{N},\
\ell \left( J\right) =2^{N-s},\text{ }J\subset 3I\setminus I,\ 2^{d-1}\leq 
\limfunc{quasidist}\left( I,J\right) \leq 2^{d}\right\} .
\end{equation*}%
Note that the closest a good quasicube $J$ can come to $I$ is determined by
the goodness inequality, which gives this bound for $2^{d}\geq \limfunc{%
quasidist}\left( I,J\right) $: 
\begin{eqnarray*}
&&2^{d}\geq \frac{1}{2}\ell \left( I\right) ^{1-\varepsilon }\ell \left(
J\right) ^{\varepsilon }=\frac{1}{2}2^{N\left( 1-\varepsilon \right)
}2^{\left( N-s\right) \varepsilon }=\frac{1}{2}2^{N-\varepsilon s}; \\
&&\text{which implies }N-\varepsilon s-1\leq d\leq N,
\end{eqnarray*}%
where the last inequality holds because we are in the case of the mid-range
term. Thus we have%
\begin{eqnarray*}
&&\dsum\limits_{\left( I,J\right) \in \mathcal{P}}\left\vert \left\langle
T_{\sigma }^{\alpha }\left( \bigtriangleup _{I}^{\sigma }f\right)
,\bigtriangleup _{J}^{\omega }g\right\rangle _{\omega }\right\vert \lesssim
\dsum\limits_{\left( I,J\right) \in \mathcal{P}}\left\Vert \bigtriangleup
_{J}^{\omega }g\right\Vert _{L^{2}\left( \omega \right) }\mathrm{P}^{\alpha
}\left( J,\left\vert \bigtriangleup _{I}^{\sigma }f\right\vert \sigma
\right) \sqrt{\left\vert J\right\vert _{\omega }} \\
&&\ \ \ \ \ =\dsum\limits_{s=\mathbf{\rho }}^{\infty }\ \sum_{N\in \mathbb{Z}%
}\ \sum_{d=N-\varepsilon s-1}^{N}\ \sum_{\left( I,J\right) \in \mathcal{P}%
_{N,d}^{s}}\ \left\Vert \bigtriangleup _{J}^{\omega }g\right\Vert
_{L^{2}\left( \omega \right) }\mathrm{P}^{\alpha }\left( J,\left\vert
\bigtriangleup _{I}^{\sigma }f\right\vert \sigma \right) \sqrt{\left\vert
J\right\vert _{\omega }}.
\end{eqnarray*}%
Now we use%
\begin{eqnarray*}
\mathrm{P}^{\alpha }\left( J,\left\vert \bigtriangleup _{I}^{\sigma
}f\right\vert \sigma \right) &=&\int_{I}\frac{\ell \left( J\right) }{\left(
\ell \left( J\right) +\left\vert y-c_{J}\right\vert \right) ^{n+1-\alpha }}%
\left\vert \bigtriangleup _{I}^{\sigma }f\left( y\right) \right\vert d\sigma
\left( y\right) \\
&\lesssim &\frac{2^{N-s}}{2^{d\left( n+1-\alpha \right) }}\left\Vert
\bigtriangleup _{I}^{\sigma }f\right\Vert _{L^{2}\left( \sigma \right) }%
\sqrt{\left\vert I\right\vert _{\sigma }}
\end{eqnarray*}%
and apply Cauchy-Schwarz in $J$ and use $J\subset 3I$ to get%
\begin{eqnarray*}
&&\dsum\limits_{\left( I,J\right) \in \mathcal{P}}\left\vert \left\langle
T_{\sigma }^{\alpha }\left( \bigtriangleup _{I}^{\sigma }f\right)
,\bigtriangleup _{J}^{\omega }g\right\rangle _{\omega }\right\vert \\
&\lesssim &\dsum\limits_{s=\mathbf{\rho }}^{\infty }\ \sum_{N\in \mathbb{Z}%
}\ \sum_{d=N-\varepsilon s-1}^{N}\ \sum_{I\in \Omega \mathcal{D}_{N}}\frac{%
2^{N-s}2^{N\left( n-\alpha \right) }}{2^{d\left( n+1-\alpha \right) }}%
\left\Vert \bigtriangleup _{I}^{\sigma }f\right\Vert _{L^{2}\left( \sigma
\right) }\frac{\sqrt{\left\vert I\right\vert _{\sigma }}\sqrt{\left\vert
3I\right\vert _{\omega }}}{2^{N\left( n-\alpha \right) }} \\
&&\ \ \ \ \ \ \ \ \ \ \ \ \ \ \ \ \ \ \ \ \ \ \ \ \ \ \ \ \ \ \times \sqrt{%
\sum_{\substack{ J\in \Omega \mathcal{D}_{N-s}  \\ J\subset 3I\setminus I%
\text{ and }\limfunc{quasidist}\left( I,J\right) \approx 2^{d}}}\left\Vert
\bigtriangleup _{J}^{\omega }g\right\Vert _{L^{2}\left( \omega \right) }^{2}}
\\
&\lesssim &\dsum\limits_{s=\mathbf{\rho }}^{\infty }\ \sum_{N\in \mathbb{Z}}%
\frac{2^{N-s}2^{N\left( n-\alpha \right) }}{2^{\left( N-\varepsilon s\right)
\left( n+1-\alpha \right) }}\sqrt{A_{2}^{\alpha }}\sum_{I\in \Omega \mathcal{%
D}_{N}}\left\Vert \bigtriangleup _{I}^{\sigma }f\right\Vert _{L^{2}\left(
\sigma \right) }\sqrt{\sum_{\substack{ J\in \Omega \mathcal{D}_{N-s}  \\ %
J\subset 3I\setminus I}}\left\Vert \bigtriangleup _{J}^{\omega }g\right\Vert
_{L^{2}\left( \omega \right) }^{2}} \\
&\lesssim &\dsum\limits_{s=\mathbf{\rho }}^{\infty }2^{-s\left[
1-\varepsilon \left( n+1-\alpha \right) \right] }\sqrt{A_{2}^{\alpha }}%
\left\Vert f\right\Vert _{L^{2}\left( \sigma \right) }\left\Vert
g\right\Vert _{L^{2}\left( \omega \right) }\lesssim \sqrt{A_{2}^{\alpha }}%
\left\Vert f\right\Vert _{L^{2}\left( \sigma \right) }\left\Vert
g\right\Vert _{L^{2}\left( \omega \right) },
\end{eqnarray*}%
where in the third line above we have used $\sum_{d=N-\varepsilon s-1}^{N}%
\frac{1}{2^{d\left( n+1-\alpha \right) }}\approx \frac{1}{2^{\left(
N-\varepsilon s\right) \left( n+1-\alpha \right) }}$, and in the last line $%
\frac{2^{N-s}2^{N\left( n-\alpha \right) }}{2^{\left( N-\varepsilon s\right)
\left( n+1-\alpha \right) }}=2^{-s\left[ 1-\varepsilon \left( n+1-\alpha
\right) \right] }$ followed by Cauchy-Schwarz in $I$ and $N$, using that we
have bounded overlap in the triples of $I$ for $I\in \Omega \mathcal{D}_{N}$%
. More precisely, if we define $f_{k}\equiv \sum_{I\in \Omega \mathcal{D}%
_{k}}\bigtriangleup _{I}^{\sigma }fh_{I}^{\sigma }$ and $g_{k}\equiv
\sum_{I\in \Omega \mathcal{D}_{k}}\bigtriangleup _{J}^{\omega
}gh_{J}^{\omega }$, then we have the orthogonality inequality 
\begin{eqnarray*}
\sum_{N\in \mathbb{Z}}\left\Vert f_{N}\right\Vert _{L^{2}\left( \sigma
\right) }\left\Vert g_{N-s}\right\Vert _{L^{2}\left( \omega \right) } &\leq
&\left( \sum_{N\in \mathbb{Z}}\left\Vert f_{N}\right\Vert _{L^{2}\left(
\sigma \right) }^{2}\right) ^{\frac{1}{2}}\left( \sum_{N\in \mathbb{Z}%
}\left\Vert g_{N-s}\right\Vert _{L^{2}\left( \omega \right) }^{2}\right) ^{%
\frac{1}{2}} \\
&=&\left\Vert f\right\Vert _{L^{2}\left( \sigma \right) }\left\Vert
g\right\Vert _{L^{2}\left( \omega \right) }.
\end{eqnarray*}%
We have assumed that $0<\varepsilon <\frac{1}{n+1-\alpha }$ in the
calculations above, and this completes the proof of Lemma \ref{standard
delta}.
\end{proof}

\section{Corona Decompositions and splittings}

We will use two different corona constructions, namely a Calder\'{o}%
n-Zygmund decomposition and an energy decomposition of NTV type, to reduce
matters to the stopping form, the main part of which is handled by Lacey's
recursion argument. We will then iterate these coronas into a double corona.
We first recall our basic setup. For convenience in notation we will
sometimes suppress the dependence on $\alpha $ in our nonlinear forms, but
will retain it in the operators, Poisson integrals and constants. We will
assume that the good/bad quasicube machinery of Nazarov, Treil and Volberg 
\cite{Vol} is in force here. Let $\Omega \mathcal{D}^{\sigma }=\Omega 
\mathcal{D}^{\omega }$ be an $\left( \mathbf{r},\varepsilon \right) $-good
quasigrid on $\mathbb{R}^{n}$, and let $\left\{ h_{I}^{\sigma ,a}\right\}
_{I\in \Omega \mathcal{D}^{\sigma },\ a\in \Gamma _{n}}$ and $\left\{
h_{J}^{\omega ,b}\right\} _{J\in \Omega \mathcal{D}^{\omega },\ b\in \Gamma
_{n}}$ be corresponding quasiHaar bases as described above, so that%
\begin{equation*}
f=\sum_{I\in \Omega \mathcal{D}^{\sigma }}\bigtriangleup _{I}^{\sigma }f%
\text{ and }g=\sum_{J\in \Omega \mathcal{D}^{\omega }\text{ }}\bigtriangleup
_{J}^{\omega }g\ ,
\end{equation*}%
where the quasiHaar projections $\bigtriangleup _{I}^{\sigma }f$ and $%
\bigtriangleup _{J}^{\omega }g$ vanish if the quasicubes $I$ and $J$ are not
good. Inequality (\ref{two weight}) is equivalent to boundedness of the
bilinear form%
\begin{equation*}
\mathcal{T}^{\alpha }\left( f,g\right) \equiv \left\langle T_{\sigma
}^{\alpha }\left( f\right) ,g\right\rangle _{\omega }=\sum_{I\in \Omega 
\mathcal{D}^{\sigma }\text{ and }J\in \Omega \mathcal{D}^{\omega
}}\left\langle T_{\sigma }^{\alpha }\left( \bigtriangleup _{I}^{\sigma
}f\right) ,\bigtriangleup _{J}^{\omega }g\right\rangle _{\omega }
\end{equation*}%
on $L^{2}\left( \sigma \right) \times L^{2}\left( \omega \right) $, i.e.%
\begin{equation*}
\left\vert \mathcal{T}^{\alpha }\left( f,g\right) \right\vert \leq \mathfrak{%
N}_{T^{\alpha }}\left\Vert f\right\Vert _{L^{2}\left( \sigma \right)
}\left\Vert g\right\Vert _{L^{2}\left( \omega \right) }.
\end{equation*}

\subsection{The Calder\'{o}n-Zygmund corona}

We now introduce a stopping tree $\mathcal{F}$ for the function $f\in
L^{2}\left( \sigma \right) $. Let $\mathcal{F}$ be a collection of Calder%
\'{o}n-Zygmund stopping quasicubes for $f$, and let $\Omega \mathcal{D}%
^{\sigma }=\dbigcup\limits_{F\in \mathcal{F}}\mathcal{C}_{F}$ be the
associated corona decomposition of the dyadic quasigrid $\Omega \mathcal{D}%
^{\sigma }$.

For a quasicube $I\in \Omega \mathcal{D}^{\sigma }$ let $\pi _{\Omega 
\mathcal{D}^{\sigma }}I$ be the $\Omega \mathcal{D}^{\sigma }$-parent of $I$
in the quasigrid $\Omega \mathcal{D}^{\sigma }$, and let $\pi _{\mathcal{F}%
}I $ be the smallest member of $\mathcal{F}$ that contains $I$. For $%
F,F^{\prime }\in \mathcal{F}$, we say that $F^{\prime }$ is an $\mathcal{F}$%
-child of $F$ if $\pi _{\mathcal{F}}\left( \pi _{\Omega \mathcal{D}^{\sigma
}}F^{\prime }\right) =F$ (it could be that $F=\pi _{\Omega \mathcal{D}%
^{\sigma }}F^{\prime }$), and we denote by $\mathfrak{C}_{\mathcal{F}}\left(
F\right) $ the set of $\mathcal{F}$-children of $F$. For $F\in \mathcal{F}$,
define the projection $\mathsf{P}_{\mathcal{C}_{F}}^{\sigma }$ onto the
linear span of the quasiHaar functions $\left\{ h_{I}^{\sigma ,a}\right\}
_{I\in \mathcal{C}_{F},\ a\in \Gamma _{n}}$ by%
\begin{equation*}
\mathsf{P}_{\mathcal{C}_{F}}^{\sigma }f=\sum_{I\in \mathcal{C}%
_{F}}\bigtriangleup _{I}^{\sigma }f=\sum_{I\in \mathcal{C}_{F},\ a\in \Gamma
_{n}}\left\langle f,h_{I}^{\sigma ,a}\right\rangle _{\sigma }h_{I}^{\sigma
,a}.
\end{equation*}%
The standard properties of these projections are%
\begin{equation*}
f=\sum_{F\in \mathcal{F}}\mathsf{P}_{\mathcal{C}_{F}}^{\sigma }f,\ \ \ \ \
\int \left( \mathsf{P}_{\mathcal{C}_{F}}^{\sigma }f\right) \sigma =0,\ \ \ \
\ \left\Vert f\right\Vert _{L^{2}\left( \sigma \right) }^{2}=\sum_{F\in 
\mathcal{F}}\left\Vert \mathsf{P}_{\mathcal{C}_{F}}^{\sigma }f\right\Vert
_{L^{2}\left( \sigma \right) }^{2}.
\end{equation*}

\subsection{The energy corona}

We must also impose a quasienergy corona decomposition as in \cite{NTV3} and 
\cite{LaSaUr2}.

\begin{definition}
\label{def energy corona 3}Given a quasicube $S_{0}$, define $\mathcal{S}%
\left( S_{0}\right) $ to be the maximal subquasicubes $I\subset S_{0}$ such
that%
\begin{equation}
\sum_{J\in \mathcal{M}_{\mathbf{\tau }-\limfunc{deep}}\left( I\right)
}\left( \frac{\mathrm{P}^{\alpha }\left( J,\mathbf{1}_{S_{0}\setminus \gamma
J}\sigma \right) }{\left\vert J\right\vert ^{\frac{1}{n}}}\right)
^{2}\left\Vert \mathsf{P}_{J}^{\limfunc{subgood},\omega }\mathbf{x}%
\right\Vert _{L^{2}\left( \omega \right) }^{2}\geq C_{\limfunc{energy}}\left[
\left( \mathcal{E}_{\alpha }^{\limfunc{deep}}\right) ^{2}+A_{2}^{\alpha }%
\right] \ \left\vert I\right\vert _{\sigma },  \label{def stop 3}
\end{equation}%
where $\mathcal{E}_{\alpha }^{\limfunc{deep}}$ is the constant in the deep
quasienergy condition defined in Definition \ref{energy condition}, and $C_{%
\limfunc{energy}}$ is a sufficiently large positive constant depending only
on $\mathbf{\tau }\geq \mathbf{r},n$ and $\alpha $. Then define the $\sigma $%
-energy stopping quasicubes of $S_{0}$ to be the collection 
\begin{equation*}
\mathcal{S}=\left\{ S_{0}\right\} \cup \dbigcup\limits_{n=0}^{\infty }%
\mathcal{S}_{n}
\end{equation*}%
where $\mathcal{S}_{0}=\mathcal{S}\left( S_{0}\right) $ and $\mathcal{S}%
_{n+1}=\dbigcup\limits_{S\in \mathcal{S}_{n}}\mathcal{S}\left( S\right) $
for $n\geq 0$.
\end{definition}

From the quasienergy condition in Definition \ref{energy condition} we
obtain the $\sigma $-Carleson estimate%
\begin{equation}
\sum_{S\in \mathcal{S}:\ S\subset I}\left\vert S\right\vert _{\sigma }\leq
2\left\vert I\right\vert _{\sigma },\ \ \ \ \ I\in \Omega \mathcal{D}%
^{\sigma }.  \label{sigma Carleson 3}
\end{equation}%
Indeed, using the deep quasienergy condition, the first generation satisfies%
\begin{eqnarray}
&&\ \ \ \ \ \ \ \ \ \ \ \ \ \ \ \ \ \ \ \ \sum_{S\in \mathcal{S}%
_{1}}\left\vert S\right\vert _{\sigma }  \label{first gen} \\
&\leq &\frac{1}{C_{\limfunc{energy}}\left[ \left( \mathcal{E}_{\alpha }^{%
\limfunc{deep}}\right) ^{2}+A_{2}^{\alpha }\right] }\sum_{S\in \mathcal{S}%
_{1}}\sum_{J\in \mathcal{M}_{\mathbf{\tau }-\limfunc{deep}}\left( S\right)
}\left( \frac{\mathrm{P}^{\alpha }\left( J,\mathbf{1}_{S_{0}\setminus \gamma
J}\sigma \right) }{\left\vert J\right\vert ^{\frac{1}{n}}}\right)
^{2}\left\Vert \mathsf{P}_{J}^{\limfunc{subgood},\omega }\mathbf{x}%
\right\Vert _{L^{2}\left( \omega \right) }^{2}  \notag \\
&\leq &\frac{1}{C_{\limfunc{energy}}\left[ \left( \mathcal{E}_{\alpha }^{%
\limfunc{deep}}\right) ^{2}+A_{2}^{\alpha }\right] }\sum_{S\in \mathcal{S}%
_{1}}\sum_{J\in \mathcal{M}_{\mathbf{\tau }-\limfunc{deep}}\left( S\right)
}\left( \frac{\mathrm{P}^{\alpha }\left( J,\mathbf{1}_{S_{0}}\sigma \right) 
}{\left\vert J\right\vert ^{\frac{1}{n}}}\right) ^{2}\left\Vert \mathsf{P}%
_{J}^{\limfunc{subgood},\omega }\mathbf{x}\right\Vert _{L^{2}\left( \omega
\right) }^{2}  \notag \\
&\leq &\frac{C_{\mathbf{\tau },\mathbf{r},n,\alpha }}{C_{\limfunc{energy}}%
\left[ \left( \mathcal{E}_{\alpha }^{\limfunc{deep}}\right)
^{2}+A_{2}^{\alpha }\right] }\sum_{S\in \mathcal{S}_{1}}\sum_{J\in \mathcal{M%
}_{\mathbf{r}-\limfunc{deep}}\left( S\right) }\left( \frac{\mathrm{P}%
^{\alpha }\left( J,\mathbf{1}_{S_{0}}\sigma \right) }{\left\vert
J\right\vert ^{\frac{1}{n}}}\right) ^{2}\left\Vert \mathsf{P}_{J}^{\limfunc{%
subgood},\omega }\mathbf{x}\right\Vert _{L^{2}\left( \omega \right) }^{2} 
\notag \\
&\leq &\frac{C_{\mathbf{\tau },\mathbf{r},n,\alpha }}{C_{\limfunc{energy}}%
\left[ \left( \mathcal{E}_{\alpha }^{\limfunc{deep}}\right)
^{2}+A_{2}^{\alpha }\right] }\left( \mathcal{E}_{\alpha }^{\limfunc{deep}%
\limfunc{plug}}\right) ^{2}\ \left\vert S_{0}\right\vert _{\sigma }=\frac{1}{%
2}\left\vert S_{0}\right\vert _{\sigma }\ ,  \notag
\end{eqnarray}%
provided we take $C_{\limfunc{energy}}=2C_{\mathbf{\tau },\mathbf{r}%
,n,\alpha }\frac{\left( \mathcal{E}_{\alpha }^{\limfunc{deep}\limfunc{plug}%
}\right) ^{2}}{\left( \mathcal{E}_{\alpha }^{\limfunc{deep}}\right)
^{2}+A_{2}^{\alpha }}$. The third inequality above, in which $\mathbf{\tau }$
is replaced by $\mathbf{r}$ (but the goodness parameter $\varepsilon >0$ is
unchanged), follows because if $J_{1}\in \mathcal{M}_{\mathbf{\tau }-%
\limfunc{deep}}\left( S\right) $, then $J_{1}\subset J_{2}$ for a unique $%
J_{2}\in \mathcal{M}_{\mathbf{r}-\limfunc{deep}}\left( S\right) $ and we
have $\ell \left( J_{2}\right) \leq 2^{\mathbf{\tau }-\mathbf{r}}\ell \left(
J_{1}\right) $ from the definitions of $\mathcal{M}_{\mathbf{\tau }-\limfunc{%
deep}}\left( S\right) $ and $\mathcal{M}_{\mathbf{r}-\limfunc{deep}}\left(
S\right) $, hence $\frac{\mathrm{P}^{\alpha }\left( J_{1},\mathbf{1}%
_{S_{0}}\sigma \right) }{\left\vert J_{1}\right\vert ^{\frac{1}{n}}}\leq C_{%
\mathbf{\tau },\mathbf{r},n,\alpha }\frac{\mathrm{P}^{\alpha }\left( J_{2},%
\mathbf{1}_{S_{0}}\sigma \right) }{\left\vert J_{2}\right\vert ^{\frac{1}{n}}%
}$. Subsequent generations satisfy a similar estimate, which then easily
gives (\ref{sigma Carleson 3}). We emphasize that this collection of
stopping times depends only on $S_{0}$ and the weight pair $\left( \sigma
,\omega \right) $, and not on any functions at hand.

Finally, we record the reason for introducing quasienergy stopping times. If 
\begin{equation}
X_{\alpha }\left( \mathcal{C}_{S}\right) ^{2}\equiv \sup_{I\in \mathcal{C}%
_{S}}\frac{1}{\left\vert I\right\vert _{\sigma }}\sum_{J\in \mathcal{M}_{%
\mathbf{\tau }-\limfunc{deep}}\left( I\right) }\left( \frac{\mathrm{P}%
^{\alpha }\left( J,\mathbf{1}_{S\setminus \gamma J}\sigma \right) }{%
\left\vert J\right\vert ^{\frac{1}{n}}}\right) ^{2}\left\Vert \mathsf{P}%
_{J}^{\limfunc{subgood},\omega }\mathbf{x}\right\Vert _{L^{2}\left( \omega
\right) }^{2}  \label{def stopping energy 3}
\end{equation}%
is (the square of) the $\alpha $\emph{-stopping quasienergy} of the weight
pair $\left( \sigma ,\omega \right) $ with respect to the corona $\mathcal{C}%
_{S}$, then we have the \emph{stopping quasienergy bounds}%
\begin{equation}
X_{\alpha }\left( \mathcal{C}_{S}\right) \leq \sqrt{C_{\limfunc{energy}}}%
\sqrt{\left( \mathcal{E}_{\alpha }^{\limfunc{deep}}\right)
^{2}+A_{2}^{\alpha }},\ \ \ \ \ S\in \mathcal{S},
\label{def stopping bounds 3}
\end{equation}%
where $A_{2}^{\alpha }$ and the the deep quasienergy constant $\mathcal{E}%
_{\alpha }^{\limfunc{deep}}$ are controlled by assumption.

\subsection{General stopping data}

It is useful to extend our notion of corona decomposition to more general
stopping data. Our general definition of stopping data will use a positive
constant $C_{0}\geq 4$.

\begin{definition}
\label{general stopping data}Suppose we are given a positive constant $%
C_{0}\geq 4$, a subset $\mathcal{F}$ of the dyadic quasigrid $\Omega 
\mathcal{D}^{\sigma }$ (called the stopping times), and a corresponding
sequence $\alpha _{\mathcal{F}}\equiv \left\{ \alpha _{\mathcal{F}}\left(
F\right) \right\} _{F\in \mathcal{F}}$ of nonnegative numbers $\alpha _{%
\mathcal{F}}\left( F\right) \geq 0$ (called the stopping data). Let $\left( 
\mathcal{F},\prec ,\pi _{\mathcal{F}}\right) $ be the tree structure on $%
\mathcal{F}$ inherited from $\Omega \mathcal{D}^{\sigma }$, and for each $%
F\in \mathcal{F}$ denote by $\mathcal{C}_{F}=\left\{ I\in \Omega \mathcal{D}%
^{\sigma }:\pi _{\mathcal{F}}I=F\right\} $ the corona associated with $F$: 
\begin{equation*}
\mathcal{C}_{F}=\left\{ I\in \Omega \mathcal{D}^{\sigma }:I\subset F\text{
and }I\not\subset F^{\prime }\text{ for any }F^{\prime }\prec F\right\} .
\end{equation*}%
We say the triple $\left( C_{0},\mathcal{F},\alpha _{\mathcal{F}}\right) $
constitutes \emph{stopping data} for a function $f\in L_{loc}^{1}\left(
\sigma \right) $ if

\begin{enumerate}
\item $\mathbb{E}_{I}^{\sigma }\left\vert f\right\vert \leq \alpha _{%
\mathcal{F}}\left( F\right) $ for all $I\in \mathcal{C}_{F}$ and $F\in 
\mathcal{F}$,

\item $\sum_{F^{\prime }\preceq F}\left\vert F^{\prime }\right\vert _{\sigma
}\leq C_{0}\left\vert F\right\vert _{\sigma }$ for all $F\in \mathcal{F}$,

\item $\sum_{F\in \mathcal{F}}\alpha _{\mathcal{F}}\left( F\right)
^{2}\left\vert F\right\vert _{\sigma }\mathbf{\leq }C_{0}^{2}\left\Vert
f\right\Vert _{L^{2}\left( \sigma \right) }^{2}$,

\item $\alpha _{\mathcal{F}}\left( F\right) \leq \alpha _{\mathcal{F}}\left(
F^{\prime }\right) $ whenever $F^{\prime },F\in \mathcal{F}$ with $F^{\prime
}\subset F$.
\end{enumerate}
\end{definition}

\begin{definition}
If $\left( C_{0},\mathcal{F},\alpha _{\mathcal{F}}\right) $ constitutes
(general) \emph{stopping data} for a function $f\in L_{loc}^{1}\left( \sigma
\right) $, we refer to the othogonal decomposition%
\begin{equation*}
f=\sum_{F\in \mathcal{F}}\mathsf{P}_{\mathcal{C}_{F}}^{\sigma }f;\ \ \ \ \ 
\mathsf{P}_{\mathcal{C}_{F}}^{\sigma }f\equiv \sum_{I\in \mathcal{C}%
_{F}}\bigtriangleup _{I}^{\sigma }f,
\end{equation*}%
as the (general) \emph{corona decomposition} of $f$ associated with the
stopping times $\mathcal{F}$.
\end{definition}

Property (1) says that $\alpha _{\mathcal{F}}\left( F\right) $ bounds the
quasiaverages of $f$ in the corona $\mathcal{C}_{F}$, and property (2) says
that the quasicubes at the tops of the coronas satisfy a Carleson condition
relative to the weight $\sigma $. Note that a standard `maximal quasicube'
argument extends the Carleson condition in property (2) to the inequality%
\begin{equation*}
\sum_{F^{\prime }\in \mathcal{F}:\ F^{\prime }\subset A}\left\vert F^{\prime
}\right\vert _{\sigma }\leq C_{0}\left\vert A\right\vert _{\sigma }\text{
for all open sets }A\subset \mathbb{R}^{n}.
\end{equation*}%
Property (3) is the `quasi' orthogonality condition that says the sequence
of functions $\left\{ \alpha _{\mathcal{F}}\left( F\right) \mathbf{1}%
_{F}\right\} _{F\in \mathcal{F}}$ is in the vector-valued space $L^{2}\left(
\ell ^{2};\sigma \right) $, and property (4) says that the control on
quasiaverages is nondecreasing on the stopping tree $\mathcal{F}$. We
emphasize that we are \emph{not} assuming in this definition the stronger
property that there is $C>1$ such that $\alpha _{\mathcal{F}}\left(
F^{\prime }\right) >C\alpha _{\mathcal{F}}\left( F\right) $ whenever $%
F^{\prime },F\in \mathcal{F}$ with $F^{\prime }\subsetneqq F$. Instead, the
properties (2) and (3) substitute for this lack. Of course the stronger
property \emph{does} hold for the familiar \emph{Calder\'{o}n-Zygmund}
stopping data determined by the following requirements for $C>1$,%
\begin{eqnarray*}
\mathbb{E}_{F^{\prime }}^{\sigma }\left\vert f\right\vert &>&C\mathbb{E}%
_{F}^{\sigma }\left\vert f\right\vert \text{ whenever }F^{\prime },F\in 
\mathcal{F}\text{ with }F^{\prime }\subsetneqq F, \\
\mathbb{E}_{I}^{\sigma }\left\vert f\right\vert &\leq &C\mathbb{E}%
_{F}^{\sigma }\left\vert f\right\vert \text{ for }I\in \mathcal{C}_{F},
\end{eqnarray*}%
which are themselves sufficiently strong to automatically force properties
(2) and (3) with $\alpha _{\mathcal{F}}\left( F\right) =\mathbb{E}%
_{F}^{\sigma }\left\vert f\right\vert $.

We have the following useful consequence of (2) and (3) that says the
sequence $\left\{ \alpha _{\mathcal{F}}\left( F\right) \mathbf{1}%
_{F}\right\} _{F\in \mathcal{F}}$ has a \emph{`quasi' orthogonal} property
relative to $f$ with a constant $C_{0}^{\prime }$ depending only on $C_{0}$:%
\begin{equation}
\left\Vert \sum_{F\in \mathcal{F}}\alpha _{\mathcal{F}}\left( F\right) 
\mathbf{1}_{F}\right\Vert _{L^{2}\left( \sigma \right) }^{2}\leq
C_{0}^{\prime }\left\Vert f\right\Vert _{L^{2}\left( \sigma \right) }^{2}.
\label{q orth}
\end{equation}%
Indeed, the Carleson condition (2) implies a geometric decay in levels of
the tree $\mathcal{F}$, namely that there are positive constants $C_{1}$ and 
$\varepsilon $, depending on $C_{0}$, such that if $\mathfrak{C}_{\mathcal{F}%
}^{\left( n\right) }\left( F\right) $ denotes the set of $n^{th}$ generation
children of $F$ in $\mathcal{F}$,%
\begin{equation*}
\sum_{F^{\prime }\in \mathfrak{C}_{\mathcal{F}}^{\left( n\right) }\left(
F\right) :\ }\left\vert F^{\prime }\right\vert _{\sigma }\leq \left(
C_{1}2^{-\varepsilon n}\right) ^{2}\left\vert F\right\vert _{\sigma },\ \ \
\ \ \text{for all }n\geq 0\text{ and }F\in \mathcal{F}.
\end{equation*}%
From this we obtain that%
\begin{eqnarray*}
\sum_{n=0}^{\infty }\sum_{F^{\prime }\in \mathfrak{C}_{\mathcal{F}}^{\left(
n\right) }\left( F\right) :\ }\alpha _{\mathcal{F}}\left( F^{\prime }\right)
\left\vert F^{\prime }\right\vert _{\sigma } &\leq &\sum_{n=0}^{\infty }%
\sqrt{\sum_{F^{\prime }\in \mathfrak{C}_{\mathcal{F}}^{\left( n\right)
}\left( F\right) }\alpha _{\mathcal{F}}\left( F^{\prime }\right)
^{2}\left\vert F^{\prime }\right\vert _{\sigma }}C_{1}2^{-\varepsilon n}%
\sqrt{\left\vert F\right\vert _{\sigma }} \\
&\leq &C_{1}\sqrt{\left\vert F\right\vert _{\sigma }}C_{\varepsilon }\sqrt{%
\sum_{n=0}^{\infty }2^{-\varepsilon n}\sum_{F^{\prime }\in \mathfrak{C}_{%
\mathcal{F}}^{\left( n\right) }\left( F\right) }\alpha _{\mathcal{F}}\left(
F^{\prime }\right) ^{2}\left\vert F^{\prime }\right\vert _{\sigma }},
\end{eqnarray*}%
and hence that$\ $%
\begin{eqnarray*}
&&\sum_{F\in \mathcal{F}}\alpha _{\mathcal{F}}\left( F\right) \left\{
\sum_{n=0}^{\infty }\sum_{F^{\prime }\in \mathfrak{C}_{\mathcal{F}}^{\left(
n\right) }\left( F\right) }\alpha _{\mathcal{F}}\left( F^{\prime }\right)
\left\vert F^{\prime }\right\vert _{\sigma }\right\} \\
&\lesssim &\sum_{F\in \mathcal{F}}\alpha _{\mathcal{F}}\left( F\right) \sqrt{%
\left\vert F\right\vert _{\sigma }}\sqrt{\sum_{n=0}^{\infty }2^{-\varepsilon
n}\sum_{F^{\prime }\in \mathfrak{C}_{\mathcal{F}}^{\left( n\right) }\left(
F\right) }\alpha _{\mathcal{F}}\left( F^{\prime }\right) ^{2}\left\vert
F^{\prime }\right\vert _{\sigma }} \\
&\lesssim &\left( \sum_{F\in \mathcal{F}}\alpha _{\mathcal{F}}\left(
F\right) ^{2}\left\vert F\right\vert _{\sigma }\right) ^{\frac{1}{2}}\left(
\sum_{n=0}^{\infty }2^{-\varepsilon n}\sum_{F\in \mathcal{F}}\sum_{F^{\prime
}\in \mathfrak{C}_{\mathcal{F}}^{\left( n\right) }\left( F\right) }\alpha _{%
\mathcal{F}}\left( F^{\prime }\right) ^{2}\left\vert F^{\prime }\right\vert
_{\sigma }\right) ^{\frac{1}{2}} \\
&\lesssim &\left\Vert f\right\Vert _{L^{2}\left( \sigma \right) }\left(
\sum_{F^{\prime }\in \mathcal{F}}\alpha _{\mathcal{F}}\left( F^{\prime
}\right) ^{2}\left\vert F^{\prime }\right\vert _{\sigma }\right) ^{\frac{1}{2%
}}\lesssim \left\Vert f\right\Vert _{L^{2}\left( \sigma \right) }^{2}.
\end{eqnarray*}%
This proves (\ref{q orth}) since $\left\Vert \sum_{F\in \mathcal{F}}\alpha _{%
\mathcal{F}}\left( F\right) \mathbf{1}_{F}\right\Vert _{L^{2}\left( \sigma
\right) }^{2}$ is dominated by twice the left hand side above.

We will use a construction that permits \emph{iteration} of general corona
decompositions.

\begin{lemma}
\label{iterating coronas}Suppose that $\left( C_{0},\mathcal{F},\alpha _{%
\mathcal{F}}\right) $ constitutes \emph{stopping data} for a function $f\in
L_{loc}^{1}\left( \sigma \right) $, and that for each $F\in \mathcal{F}$, $%
\left( C_{0},\mathcal{K}\left( F\right) ,\alpha _{\mathcal{K}\left( F\right)
}\right) $ constitutes \emph{stopping data} for the corona projection $%
\mathsf{P}_{\mathcal{C}_{F}}^{\sigma }f$, where in addition $F\in \mathcal{K}%
\left( F\right) $. There is a positive constant $C_{1}$, depending only on $%
C_{0}$, such that if%
\begin{eqnarray*}
\mathcal{K}^{\ast }\left( F\right) &\equiv &\left\{ K\in \mathcal{K}\left(
F\right) \cap \mathcal{C}_{F}:\alpha _{\mathcal{K}\left( F\right) }\left(
K\right) \geq \alpha _{\mathcal{F}}\left( F\right) \right\} \\
\mathcal{K} &\equiv &\mathop{\displaystyle \bigcup }\limits_{F\in \mathcal{F}%
}\mathcal{K}^{\ast }\left( F\right) \cup \left\{ F\right\} , \\
\alpha _{\mathcal{K}}\left( K\right) &\equiv &%
\begin{array}{ccc}
\alpha _{\mathcal{K}\left( F\right) }\left( K\right) & \text{ for } & K\in 
\mathcal{K}^{\ast }\left( F\right) \setminus \left\{ F\right\} \\ 
\max \left\{ \alpha _{\mathcal{F}}\left( F\right) ,\alpha _{\mathcal{K}%
\left( F\right) }\left( F\right) \right\} & \text{ for } & K=F%
\end{array}%
,\ \ \ \ \ \text{for }F\in \mathcal{F},
\end{eqnarray*}%
the triple $\left( C_{1},\mathcal{K},\alpha _{\mathcal{K}}\right) $
constitutes \emph{stopping data} for $f$. We refer to the collection of
quasicubes $\mathcal{K}$ as the \emph{iterated} stopping times, and to the
orthogonal decomposition $f=\sum_{K\in \mathcal{K}}P_{\mathcal{C}_{K}^{%
\mathcal{K}}}f$ as the \emph{iterated} corona decomposition of $f$, where 
\begin{equation*}
\mathcal{C}_{K}^{\mathcal{K}}\equiv \left\{ I\in \Omega \mathcal{D}:I\subset
K\text{ and }I\not\subset K^{\prime }\text{ for }K^{\prime }\prec _{\mathcal{%
K}}K\right\} .
\end{equation*}
\end{lemma}

Note that in our definition of $\left( C_{1},\mathcal{K},\alpha _{\mathcal{K}%
}\right) $ we have `discarded' from $\mathcal{K}\left( F\right) $ all of
those $K\in \mathcal{K}\left( F\right) $ that are not in the corona $%
\mathcal{C}_{F}$, and also all of those $K\in \mathcal{K}\left( F\right) $
for which $\alpha _{\mathcal{K}\left( F\right) }\left( K\right) $ is
strictly less than $\alpha _{\mathcal{F}}\left( F\right) $. Then the union
of over $F$ of what remains is our new collection of stopping times. We then
define stopping data $\alpha _{\mathcal{K}}\left( K\right) $ according to
whether or not $K\in \mathcal{F}$: if $K\notin \mathcal{F}$ but $K\in 
\mathcal{C}_{F}$ then $\alpha _{\mathcal{K}}\left( K\right) $ equals $\alpha
_{\mathcal{K}\left( F\right) }\left( K\right) $, while if $K\in \mathcal{F}$%
, then $\alpha _{\mathcal{K}}\left( K\right) $ is the larger of $\alpha _{%
\mathcal{K}\left( F\right) }\left( F\right) $ and $\alpha _{\mathcal{F}%
}\left( K\right) $.

\begin{proof}
The monotonicity property (4) for the triple $\left( C_{1},\mathcal{K}%
,\alpha _{\mathcal{K}}\right) $ is obvious from the construction of $%
\mathcal{K}$ and $\alpha _{\mathcal{K}}\left( K\right) $. To establish
property (1), we must distinguish between the various coronas $\mathcal{C}%
_{K}^{\mathcal{K}}$, $\mathcal{C}_{K}^{\mathcal{K}\left( F\right) }$ and $%
\mathcal{C}_{K}^{\mathcal{F}}$ that could be associated with $K\in \mathcal{K%
}$, when $K$ belongs to any of the stopping trees $\mathcal{K}$, $\mathcal{K}%
\left( F\right) $ or $\mathcal{F}$. Suppose now that $I\in \mathcal{C}_{K}^{%
\mathcal{K}}$ for some $K\in \mathcal{K}$. Then there is a unique $F\in 
\mathcal{F}$ such that $\mathcal{C}_{K}^{\mathcal{K}}\subset \mathcal{C}%
_{K}^{\mathcal{K}\left( F\right) }\subset C_{F}^{\mathcal{F}}$, and so $%
\mathbb{E}_{I}^{\sigma }\left\vert f\right\vert \leq \alpha _{\mathcal{F}%
}\left( F\right) $ by property (1) for the triple $\left( C_{0},\mathcal{F}%
,\alpha _{\mathcal{F}}\right) $. Then $\alpha _{\mathcal{F}}\left( F\right)
\leq \alpha _{\mathcal{K}}\left( K\right) $ follows from the definition of $%
\alpha _{\mathcal{K}}\left( K\right) $, and we have property (1) for the
triple $\left( C_{1},\mathcal{K},\alpha _{\mathcal{K}}\right) $. Property
(2) holds for the triple $\left( C_{1},\mathcal{K},\alpha _{\mathcal{K}%
}\right) $ since if $K\in \mathcal{C}_{F}^{\mathcal{F}}$, then 
\begin{eqnarray*}
\sum_{K^{\prime }\preceq _{\mathcal{K}}K}\left\vert K^{\prime }\right\vert
_{\sigma } &=&\sum_{K^{\prime }\in \mathcal{K}\left( F\right) :\ K^{\prime
}\subset K}\left\vert K^{\prime }\right\vert _{\sigma }+\sum_{F^{\prime
}\prec _{\mathcal{F}}F:\ F^{\prime }\subset K}\sum_{K^{\prime }\in \mathcal{K%
}\left( F^{\prime }\right) }\left\vert K^{\prime }\right\vert _{\sigma } \\
&\leq &C_{0}\left\vert K\right\vert _{\sigma }+\sum_{F^{\prime }\prec _{%
\mathcal{F}}F:\ F^{\prime }\subset K}C_{0}\left\vert F^{\prime }\right\vert
_{\sigma }\leq 2C_{0}^{2}\left\vert K\right\vert _{\sigma }.
\end{eqnarray*}%
Finally, property (3) holds for the triple $\left( C_{1},\mathcal{K},\alpha
_{\mathcal{K}}\right) $ since 
\begin{eqnarray*}
\sum_{K\in \mathcal{K}}\alpha _{\mathcal{K}}\left( K\right) ^{2}\left\vert
K\right\vert _{\sigma } &\leq &\sum_{F\in \mathcal{F}}\sum_{K\in \mathcal{K}%
\left( F\right) }\alpha _{\mathcal{K}\left( F\right) }\left( K\right)
^{2}\left\vert K\right\vert _{\sigma }+\sum_{F\in \mathcal{F}}\alpha _{%
\mathcal{F}}\left( F\right) ^{2}\left\vert F\right\vert _{\sigma } \\
&\leq &\sum_{F\in \mathcal{F}}C_{0}^{2}\left\Vert \mathsf{P}_{\mathcal{C}%
_{F}}^{\sigma }f\right\Vert _{L^{2}\left( \sigma \right)
}^{2}+C_{0}^{2}\left\Vert f\right\Vert _{L^{2}\left( \sigma \right)
}^{2}\leq 2C_{0}^{2}\left\Vert f\right\Vert _{L^{2}\left( \sigma \right)
}^{2}.
\end{eqnarray*}
\end{proof}

\subsection{Doubly iterated coronas and the NTV quasicube size splitting}

Here is a brief schematic diagram of the decompositions, with bounds in $%
\fbox{}$, used in this subsection:%
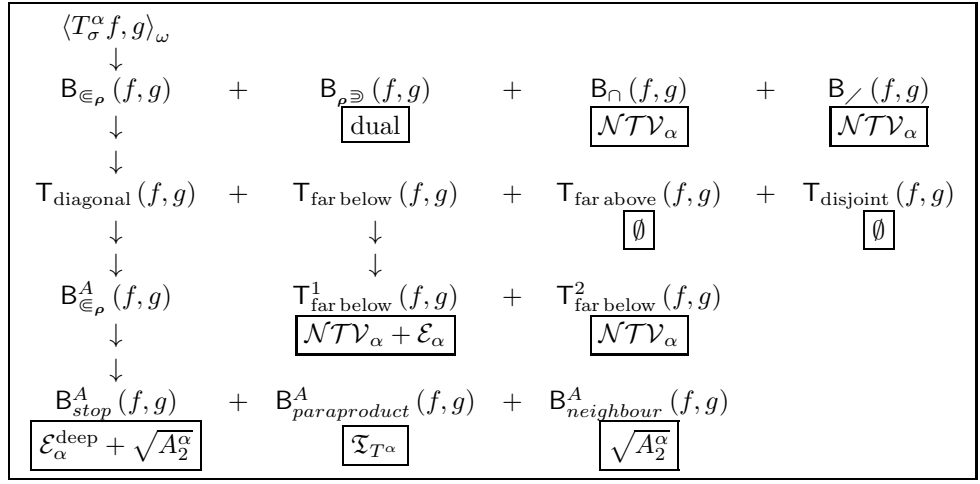
\begin{equation*}
\fbox{$%
\begin{array}{ccccccc}
\left\langle T_{\sigma }^{\alpha }f,g\right\rangle _{\omega } &  &  &  &  & 
&  \\ 
\downarrow &  &  &  &  &  &  \\ 
\mathsf{B}_{\Subset _{\mathbf{\rho }}}\left( f,g\right) & + & \mathsf{B}_{_{%
\mathbf{\rho }}\Supset }\left( f,g\right) & + & \mathsf{B}_{\cap }\left(
f,g\right) & + & \mathsf{B}_{\diagup }\left( f,g\right) \\ 
\downarrow &  & \fbox{dual} &  & \fbox{$\mathcal{NTV}_{\alpha }$} &  & \fbox{%
$\mathcal{NTV}_{\alpha }$} \\ 
\downarrow &  &  &  &  &  &  \\ 
\mathsf{T}_{\limfunc{diagonal}}\left( f,g\right) & + & \mathsf{T}_{\limfunc{%
far}\limfunc{below}}\left( f,g\right) & + & \mathsf{T}_{\limfunc{far}%
\limfunc{above}}\left( f,g\right) & + & \mathsf{T}_{\limfunc{disjoint}%
}\left( f,g\right) \\ 
\downarrow &  & \downarrow &  & \fbox{$\emptyset $} &  & \fbox{$\emptyset $}
\\ 
\downarrow &  & \downarrow &  &  &  &  \\ 
\mathsf{B}_{\Subset _{\mathbf{\rho }}}^{A}\left( f,g\right) &  & \mathsf{T}_{%
\limfunc{far}\limfunc{below}}^{1}\left( f,g\right) & + & \mathsf{T}_{%
\limfunc{far}\limfunc{below}}^{2}\left( f,g\right) &  &  \\ 
\downarrow &  & \fbox{$\mathcal{NTV}_{\alpha }+\mathcal{E}_{\alpha }$} &  & 
\fbox{$\mathcal{NTV}_{\alpha }$} &  &  \\ 
\downarrow &  &  &  &  &  &  \\ 
\mathsf{B}_{stop}^{A}\left( f,g\right) & + & \mathsf{B}_{paraproduct}^{A}%
\left( f,g\right) & + & \mathsf{B}_{neighbour}^{A}\left( f,g\right) &  &  \\ 
\fbox{$\mathcal{E}_{\alpha }^{\limfunc{deep}}+\sqrt{A_{2}^{\alpha }}$} &  & 
\fbox{$\mathfrak{T}_{T^{\alpha }}$} &  & \fbox{$\sqrt{A_{2}^{\alpha }}$} & 
& 
\end{array}%
$}
\end{equation*}

We begin with the NTV \emph{quasicube size splitting} of the inner product $%
\left\langle T_{\sigma }^{\alpha }f,g\right\rangle _{\omega }$ - and later
apply the iterated corona construction - that splits the pairs of quasicubes 
$\left( I,J\right) $ in a simultaneous quasiHaar decomposition of $f$ and $g$
into four groups, namely those pairs that:

\begin{enumerate}
\item are below the size diagonal and $\mathbf{\rho }$-deeply embedded,

\item are above the size diagonal and $\mathbf{\rho }$-deeply embedded,

\item are disjoint, and

\item are of $\mathbf{\rho }$-comparable size.
\end{enumerate}

More precisely we have%
\begin{eqnarray*}
\left\langle T_{\sigma }^{\alpha }f,g\right\rangle _{\omega }
&=&\dsum\limits_{I\in \Omega \mathcal{D}^{\sigma },\ J\in \Omega \mathcal{D}%
^{\omega }}\left\langle T_{\sigma }^{\alpha }\left( \bigtriangleup
_{I}^{\sigma }f\right) ,\left( \bigtriangleup _{I}^{\omega }g\right)
\right\rangle _{\omega } \\
&=&\dsum\limits_{\substack{ I\in \Omega \mathcal{D}^{\sigma },\ J\in \Omega 
\mathcal{D}^{\omega }  \\ J\Subset _{\mathbf{\rho }}I}}\left\langle
T_{\sigma }^{\alpha }\left( \bigtriangleup _{I}^{\sigma }f\right) ,\left(
\bigtriangleup _{J}^{\omega }g\right) \right\rangle _{\omega }+\dsum\limits 
_{\substack{ I\in \Omega \mathcal{D}^{\sigma },\ J\in \Omega \mathcal{D}%
^{\omega }  \\ J_{\mathbf{\rho }}\Supset I}}\left\langle T_{\sigma }^{\alpha
}\left( \bigtriangleup _{I}^{\sigma }f\right) ,\left( \bigtriangleup
_{J}^{\omega }g\right) \right\rangle _{\omega } \\
&&+\dsum\limits_{\substack{ I\in \Omega \mathcal{D}^{\sigma },\ J\in \Omega 
\mathcal{D}^{\omega }  \\ J\cap I=\emptyset }}\left\langle T_{\sigma
}^{\alpha }\left( \bigtriangleup _{I}^{\sigma }f\right) ,\left(
\bigtriangleup _{J}^{\omega }g\right) \right\rangle _{\omega }+\dsum\limits 
_{\substack{ I\in \Omega \mathcal{D}^{\sigma },\ J\in \Omega \mathcal{D}%
^{\omega }  \\ 2^{-\mathbf{\rho }}\leq \ell \left( J\right) \diagup \ell
\left( I\right) \leq 2^{\mathbf{\rho }}}}\left\langle T_{\sigma }^{\alpha
}\left( \bigtriangleup _{I}^{\sigma }f\right) ,\left( \bigtriangleup
_{J}^{\omega }g\right) \right\rangle _{\omega } \\
&=&\mathsf{B}_{\Subset _{\mathbf{\rho }}}\left( f,g\right) +\mathsf{B}_{_{%
\mathbf{\rho }}\Supset }\left( f,g\right) +\mathsf{B}_{\cap }\left(
f,g\right) +\mathsf{B}_{\diagup }\left( f,g\right) .
\end{eqnarray*}%
Lemma \ref{standard delta} in the section on NTV peliminaries show that the 
\emph{disjoint} and \emph{comparable} forms $\mathsf{B}_{\cap }\left(
f,g\right) $ and $\mathsf{B}_{\diagup }\left( f,g\right) $ are both bounded
by the $\mathcal{A}_{2}^{\alpha }$, quasitesting and quasiweak boundedness
property constants. The \emph{below} and \emph{above} forms are clearly
symmetric, so we need only consider the form $\mathsf{B}_{\Subset _{\mathbf{%
\rho }}}\left( f,g\right) $, to which we turn for the remainder of the proof.

In order to bound the below form $\mathsf{B}_{\Subset _{\mathbf{\rho }%
}}\left( f,g\right) $, we will apply two different corona decompositions in
succession to the function $f\in L^{2}\left( \sigma \right) $, gaining
structure with each application; first to a boundedness property for $f$,
and then to a regularizing property of the weight $\sigma $. We first apply
the Calder\'{o}n-Zygmund corona decomposition to the function $f\in
L^{2}\left( \sigma \right) $ obtain%
\begin{equation*}
f=\sum_{F\in \mathcal{F}}\mathsf{P}_{\mathcal{C}_{F}^{\sigma }}^{\sigma }f.
\end{equation*}%
Then for each fixed $F\in \mathcal{F}$, construct the \emph{quasienergy}
corona decomposition $\left\{ \mathcal{C}_{S}^{\sigma }\right\} _{S\in 
\mathcal{S}\left( F\right) }$\ corresponding to the weight pair $\left(
\sigma ,\omega \right) $ with top quasicube $S_{0}=F$, as given in
Definition \ref{def energy corona 3}. At this point we apply Lemma \ref%
{iterating coronas} to obtain iterated stopping times $\mathcal{S}$ and
iterated stopping data $\left\{ \alpha _{\mathcal{S}}\left( S\right)
\right\} _{S\in \mathcal{S}}$. This gives us the following \emph{double
corona decomposition} of $f$,%
\begin{equation}
f=\sum_{F\in \mathcal{F}}\mathsf{P}_{\mathcal{C}_{F}^{\sigma }}^{\sigma
}f=\sum_{F\in \mathcal{F}}\sum_{S\in \mathcal{S}^{\ast }\left( F\right) \cup
\left\{ F\right\} }\mathsf{P}_{\mathcal{C}_{S}^{\sigma }}^{\sigma }\mathsf{P}%
_{\mathcal{C}_{F}^{\sigma }}^{\sigma }f=\sum_{S\in \mathcal{S}}\mathsf{P}_{%
\mathcal{C}_{S}^{\sigma }}^{\sigma }f\equiv \sum_{A\in \mathcal{A}}\mathsf{P}%
_{\mathcal{C}_{A}}^{\sigma }f,  \label{double corona}
\end{equation}%
where $\mathcal{A}\equiv \mathcal{S}\left( \mathcal{F}\right) $ is the
double stopping collection for $f$. We are relabeling the double corona as $%
\mathcal{A}$ here so as to minimize confusion. We now record the main facts
proved above for the double corona.

\begin{lemma}
The data $\mathcal{A}$ and $\left\{ \alpha _{\mathcal{A}}\left( A\right)
\right\} _{A\in \mathcal{A}}$ satisfy properties (1), (2), (3) and (4) in
Definition \ref{general stopping data}.
\end{lemma}

To bound $\mathsf{B}_{\Subset _{\mathbf{\rho }}}\left( f,g\right) $ we fix
the stopping data $\mathcal{A}$ and $\left\{ \alpha _{\mathcal{A}}\left(
A\right) \right\} _{A\in \mathcal{A}}$ constructed above with the double
iterated corona. We now consider the following \emph{canonical splitting} of
the form $\mathsf{B}_{\Subset _{\mathbf{\rho }}}\left( f,g\right) $ that
involves the quasiHaar corona projections $\mathsf{P}_{\mathcal{C}%
_{A}}^{\sigma }$ acting on $f$ and the $\mathbf{\tau }$\emph{-shifted}
quasiHaar corona projections $\mathsf{P}_{\mathcal{C}_{B}^{\mathbf{\tau }-%
\limfunc{shift}}}^{\omega }$ acting on $g$. Here the $\mathbf{\tau }$%
-shifted corona $\mathcal{C}_{B}^{\mathbf{\tau }-\limfunc{shift}}$ is
defined to include only those quasicubes $J\in \mathcal{C}_{B}$ that are 
\emph{not} $\mathbf{\tau }$-nearby $B$, and to include also such quasicubes $%
J$ which in addition \emph{are} $\mathbf{\tau }$-nearby in the children $%
B^{\prime }$ of $B$.

\begin{definition}
\label{def parameters}The parameters $\mathbf{\tau }$ and $\mathbf{\rho }$
are now fixed to satisfy 
\begin{equation*}
\mathbf{\tau }>\mathbf{r}\text{ and }\mathbf{\rho }>\mathbf{r}+\mathbf{\tau }%
,
\end{equation*}%
where $\mathbf{r}$ is the goodness parameter already fixed.
\end{definition}

\begin{definition}
\label{shifted corona}For $B\in \mathcal{A}$ we define%
\begin{equation*}
\mathcal{C}_{B}^{\mathbf{\tau }-\limfunc{shift}}=\left\{ J\in \mathcal{C}%
_{B}:J\Subset _{\mathbf{\tau }}B\right\} \cup \dbigcup\limits_{B^{\prime
}\in \mathfrak{C}_{\mathcal{A}}\left( B\right) }\left\{ J\in \Omega \mathcal{%
D}:J\Subset _{\mathbf{\tau }}B\text{ and }J\text{ \emph{is} }\mathbf{\tau }%
\text{-nearby in }B^{\prime }\right\} .
\end{equation*}
\end{definition}

We will use repeatedly the fact that the $\mathbf{\tau }$-shifted coronas $%
\mathcal{C}_{B}^{\mathbf{\tau }-\limfunc{shift}}$ have overlap bounded by $%
\mathbf{\tau }$:%
\begin{equation}
\sum_{B\in \mathcal{A}}\mathbf{1}_{\mathcal{C}_{B}^{\mathbf{\tau }-\limfunc{%
shift}}}\left( J\right) \leq \mathbf{\tau },\ \ \ \ \ J\in \Omega \mathcal{D}%
.  \label{tau overlap}
\end{equation}%
The forms $\mathsf{B}_{\Subset _{\mathbf{\rho }}}\left( f,g\right) $ are no
longer linear in $f$ and $g$ as the `cut' is determined by the coronas $%
\mathcal{C}_{F}$ and $\mathcal{C}_{G}^{\mathbf{\tau }-\limfunc{shift}}$,
which depend on $f$ as well as the measures $\sigma $ and $\omega $.
However, if the coronas are held fixed, then the forms can be considered
bilinear in $f$ and $g$. It is convenient at this point to introduce the
following shorthand notation:%
\begin{equation*}
\left\langle T_{\sigma }^{\alpha }\left( \mathsf{P}_{\mathcal{C}%
_{A}}^{\sigma }f\right) ,\mathsf{P}_{\mathcal{C}_{B}^{\mathbf{\tau }-%
\limfunc{shift}}}^{\omega }g\right\rangle _{\omega }^{\Subset _{\mathbf{\rho 
}}}\equiv \sum_{\substack{ I\in \mathcal{C}_{A}\text{ and }J\in \mathcal{C}%
_{B}^{\mathbf{\tau }-\limfunc{shift}}  \\ J\Subset _{\mathbf{\rho }}I}}%
\left\langle T_{\sigma }^{\alpha }\left( \bigtriangleup _{I}^{\sigma
}f\right) ,\left( \bigtriangleup _{J}^{\omega }g\right) \right\rangle
_{\omega }\ .
\end{equation*}%
We then have the canonical splitting,%
\begin{eqnarray}
&&\mathsf{B}_{\Subset _{\mathbf{\rho }}}\left( f,g\right)
\label{parallel corona decomp'} \\
&=&\sum_{A,B\in \mathcal{A}}\left\langle T_{\sigma }^{\alpha }\left( \mathsf{%
P}_{\mathcal{C}_{A}}^{\sigma }f\right) ,\mathsf{P}_{\mathcal{C}_{B}^{\mathbf{%
\tau }-\limfunc{shift}}}^{\omega }g\right\rangle _{\omega }^{\Subset _{%
\mathbf{\rho }}}  \notag \\
&=&\sum_{A\in \mathcal{A}}\left\langle T_{\sigma }^{\alpha }\left( \mathsf{P}%
_{\mathcal{C}_{A}}^{\sigma }f\right) ,\mathsf{P}_{\mathcal{C}_{A}^{\mathbf{%
\tau }-\limfunc{shift}}}^{\omega }g\right\rangle _{\omega }^{\Subset _{%
\mathbf{\rho }}}+\sum_{\substack{ A,B\in \mathcal{A}  \\ B\subsetneqq A}}%
\left\langle T_{\sigma }^{\alpha }\left( \mathsf{P}_{\mathcal{C}%
_{A}}^{\sigma }f\right) ,\mathsf{P}_{\mathcal{C}_{B}^{\mathbf{\tau }-%
\limfunc{shift}}}^{\omega }g\right\rangle _{\omega }^{\Subset _{\mathbf{\rho 
}}}  \notag \\
&&+\sum_{\substack{ A,B\in \mathcal{A}  \\ B\supsetneqq A}}\left\langle
T_{\sigma }^{\alpha }\left( \mathsf{P}_{\mathcal{C}_{A}}^{\sigma }f\right) ,%
\mathsf{P}_{\mathcal{C}_{B}^{\mathbf{\tau }-\limfunc{shift}}}^{\omega
}g\right\rangle _{\omega }^{\Subset _{\mathbf{\rho }}}+\sum_{\substack{ %
A,B\in \mathcal{A}  \\ A\cap B=\emptyset }}\left\langle T_{\sigma }^{\alpha
}\left( \mathsf{P}_{\mathcal{C}_{A}}^{\sigma }f\right) ,\mathsf{P}_{\mathcal{%
C}_{B}^{\mathbf{\tau }-\limfunc{shift}}}^{\omega }g\right\rangle _{\omega
}^{\Subset _{\mathbf{\rho }}}  \notag \\
&\equiv &\mathsf{T}_{\limfunc{diagonal}}\left( f,g\right) +\mathsf{T}_{%
\limfunc{far}\limfunc{below}}\left( f,g\right) +\mathsf{T}_{\limfunc{far}%
\limfunc{above}}\left( f,g\right) +\mathsf{T}_{\limfunc{disjoint}}\left(
f,g\right) .  \notag
\end{eqnarray}%
Now the final two terms $\mathsf{T}_{\limfunc{far}\limfunc{above}}\left(
f,g\right) $ and $\mathsf{T}_{\limfunc{disjoint}}\left( f,g\right) $ each
vanish since there are no pairs $\left( I,J\right) \in \mathcal{C}_{A}\times 
\mathcal{C}_{B}^{\mathbf{\tau }-\limfunc{shift}}$ with both (\textbf{i}) $%
J\Subset _{\mathbf{\rho }}I$ and (\textbf{ii}) either $B\subsetneqq A$ or $%
B\cap A=\emptyset $.

The \emph{far below} term $\mathsf{T}_{\limfunc{far}\limfunc{below}}\left(
f,g\right) $ is bounded using the Intertwining Proposition and the control
of functional energy condition by the energy condition given in the next two
sections. Indeed, assuming these two results, we have from $\mathbf{\tau }<%
\mathbf{\rho }$ that%
\begin{eqnarray*}
\mathsf{T}_{\limfunc{far}\limfunc{below}}\left( f,g\right) &=&\sum 
_{\substack{ A,B\in \mathcal{A}  \\ B\subsetneqq A}}\sum_{\substack{ I\in 
\mathcal{C}_{A}\text{ and }J\in \mathcal{C}_{B}^{\mathbf{\tau }-\limfunc{%
shift}}  \\ J\Subset _{\mathbf{\rho }}I}}\left\langle T_{\sigma }^{\alpha
}\left( \bigtriangleup _{I}^{\sigma }f\right) ,\left( \bigtriangleup
_{J}^{\omega }g\right) \right\rangle _{\omega } \\
&=&\sum_{B\in \mathcal{A}}\sum_{A\in \mathcal{A}:\ B\subsetneqq A}\sum 
_{\substack{ I\in \mathcal{C}_{A}\text{ and }J\in \mathcal{C}_{B}^{\mathbf{%
\tau }-\limfunc{shift}}  \\ J\Subset _{\mathbf{\rho }}I}}\left\langle
T_{\sigma }^{\alpha }\left( \bigtriangleup _{I}^{\sigma }f\right) ,\left(
\bigtriangleup _{J}^{\omega }g\right) \right\rangle _{\omega } \\
&=&\sum_{B\in \mathcal{A}}\sum_{A\in \mathcal{A}:\ B\subsetneqq A}\sum_{I\in 
\mathcal{C}_{A}\text{ and }J\in \mathcal{C}_{B}^{\mathbf{\tau }-\limfunc{%
shift}}}\left\langle T_{\sigma }^{\alpha }\left( \bigtriangleup _{I}^{\sigma
}f\right) ,\left( \bigtriangleup _{J}^{\omega }g\right) \right\rangle
_{\omega } \\
&&-\sum_{B\in \mathcal{A}}\sum_{A\in \mathcal{A}:\ B\subsetneqq A}\sum 
_{\substack{ I\in \mathcal{C}_{A}\text{ and }J\in \mathcal{C}_{B}^{\mathbf{%
\tau }-\limfunc{shift}}  \\ J\not\Subset _{\mathbf{\rho }}I}}\left\langle
T_{\sigma }^{\alpha }\left( \bigtriangleup _{I}^{\sigma }f\right) ,\left(
\bigtriangleup _{J}^{\omega }g\right) \right\rangle _{\omega } \\
&=&\mathsf{T}_{\limfunc{far}\limfunc{below}}^{1}\left( f,g\right) -\mathsf{T}%
_{\limfunc{far}\limfunc{below}}^{2}\left( f,g\right) .
\end{eqnarray*}%
Now $\mathsf{T}_{\limfunc{far}\text{\ }\limfunc{below}}^{2}\left( f,g\right) 
$ is bounded by $\mathcal{NTV}_{\alpha }$ by Lemma \ref{standard delta}
since $J$ is good if $\bigtriangleup _{J}^{\omega }g\neq 0$.

The form $\mathsf{T}_{\limfunc{far}\limfunc{below}}^{1}\left( f,g\right) $
can be written as%
\begin{eqnarray*}
\mathsf{T}_{\limfunc{far}\limfunc{below}}^{1}\left( f,g\right) &=&\sum_{B\in 
\mathcal{A}}\sum_{I\in \Omega \mathcal{D}:\ B\subsetneqq I}\left\langle
T_{\sigma }^{\alpha }\left( \bigtriangleup _{I}^{\sigma }f\right)
,g_{B}\right\rangle _{\omega }; \\
\text{where }g_{B} &\equiv &\sum_{J\in \mathcal{C}_{B}^{\mathbf{\tau }-%
\limfunc{shift}}}\bigtriangleup _{J}^{\omega }g\ .
\end{eqnarray*}%
The Intertwining Proposition \ref{strongly adapted} applies to this latter
form and shows that it is bounded by $\mathcal{NTV}_{\alpha }+\mathfrak{F}%
_{\alpha }$. Then Proposition \ref{func ener control} shows that $\mathfrak{F%
}_{\alpha }\lesssim \mathcal{A}_{2}^{\alpha }+\mathcal{E}_{\alpha }$, which
completes the proof that%
\begin{equation}
\left\vert \mathsf{T}_{\limfunc{far}\limfunc{below}}\left( f,g\right)
\right\vert \lesssim \left( \mathcal{NTV}_{\alpha }+\mathcal{E}_{\alpha
}\right) \left\Vert f\right\Vert _{L^{2}\left( \sigma \right) }\left\Vert
g\right\Vert _{L^{2}\left( \omega \right) }\ .  \label{far below bound}
\end{equation}

The boundedness of the diagonal term $\mathsf{T}_{\limfunc{diagonal}}\left(
f,g\right) $ will then be reduced to the forms in the
paraproduct/neighbour/stopping form decomposition of NTV. The stopping form
is then further split into two sublinear forms in (\ref{def split}) below,
where the boundedness of the more difficult of the two is treated by
adapting the stopping time and recursion of M. Lacey \cite{Lac}. More
precisely, to handle the diagonal term $\mathsf{T}_{\limfunc{diagonal}%
}\left( f,g\right) $, it is enough to consider the individual corona pieces 
\begin{equation*}
\mathsf{B}_{\Subset _{\mathbf{\rho }}}^{A}\left( f,g\right) \equiv
\left\langle T_{\sigma }^{\alpha }\left( \mathsf{P}_{\mathcal{C}%
_{A}}^{\sigma }f\right) ,\mathsf{P}_{\mathcal{C}_{A}^{\mathbf{\tau }-%
\limfunc{shift}}}^{\omega }g\right\rangle _{\omega }^{\Subset }\ ,
\end{equation*}%
and to prove the following estimate:%
\begin{equation*}
\left\vert \mathsf{B}_{\Subset _{\mathbf{\rho }}}^{A}\left( f,g\right)
\right\vert \lesssim \left( \mathcal{NTV}_{\alpha }+\mathcal{E}_{\alpha
}\right) \ \left( \alpha _{\mathcal{A}}\left( A\right) \sqrt{\left\vert
A\right\vert _{\sigma }}+\left\Vert \mathsf{P}_{\mathcal{C}_{A}}^{\sigma
}f\right\Vert _{L^{2}\left( \sigma \right) }\right) \ \left\Vert \mathsf{P}_{%
\mathcal{C}_{A}^{\mathbf{\tau }-\limfunc{shift}}}^{\omega }g\right\Vert
_{L^{2}\left( \omega \right) }\ .
\end{equation*}%
Indeed, we then have from Cauchy-Schwarz that%
\begin{eqnarray*}
&&\sum_{A\in \mathcal{A}}\left\vert \mathsf{B}_{\Subset _{\mathbf{\rho }%
}}^{A}\left( f,g\right) \right\vert =\sum_{A\in \mathcal{A}}\left\vert 
\mathsf{B}_{\Subset _{\mathbf{\rho }}}^{A}\left( \mathsf{P}_{\mathcal{C}%
_{A}}^{\sigma }f,\mathsf{P}_{\mathcal{C}_{A}^{\mathbf{\tau }-\limfunc{shift}%
}}^{\omega }g\right) \right\vert \\
&\lesssim &\left( \mathcal{NTV}_{\alpha }+\mathcal{E}_{\alpha }\right) \
\left( \sum_{A\in \mathcal{A}}\alpha _{\mathcal{A}}\left( A\right)
^{2}\left\vert A\right\vert _{\sigma }+\left\Vert \mathsf{P}_{\mathcal{C}%
_{A}}^{\sigma }f\right\Vert _{L^{2}\left( \sigma \right) }^{2}\right) ^{%
\frac{1}{2}}\ \left( \sum_{A\in \mathcal{A}}\left\Vert \mathsf{P}_{\mathcal{C%
}_{A}^{\mathbf{\tau }-\limfunc{shift}}}^{\omega }g\right\Vert _{L^{2}\left(
\omega \right) }^{2}\right) ^{\frac{1}{2}} \\
&\lesssim &\left( \mathcal{NTV}_{\alpha }+\mathcal{E}_{\alpha }\right) \
\left\Vert f\right\Vert _{L^{2}\left( \sigma \right) }\left\Vert
g\right\Vert _{L^{2}\left( \omega \right) }\ ,
\end{eqnarray*}%
where the last line uses `quasi' orthogonality in $f$ and orthogonality in
both $f$ and $g$.

Following arguments in \cite{NTV3}, \cite{Vol} and \cite{LaSaShUr}, we now
use the paraproduct / neighbour / stopping splitting of NTV to reduce
boundedness of $\mathsf{B}_{\Subset _{\mathbf{\rho }}}^{A}\left( f,g\right) $
to boundedness of the associated stopping form 
\begin{equation}
\mathsf{B}_{stop}^{A}\left( f,g\right) \equiv \sum_{I\in \limfunc{supp}%
\widehat{f}}\sum_{J:\ J\Subset _{\mathbf{\rho }}I\text{ and }I_{J}\notin 
\mathcal{A}}\left( \mathbb{E}_{I_{J}}^{\sigma }\bigtriangleup _{I}^{\sigma
}f\right) \ \left\langle T_{\sigma }^{\alpha }\mathbf{1}_{A\setminus
I_{J}},\bigtriangleup _{J}^{\omega }g\right\rangle _{\omega }\ ,
\label{bounded stopping form}
\end{equation}%
where $f$ is supported in the quasicube $A$ and its expectations $\mathbb{E}%
_{I}^{\sigma }\left\vert f\right\vert $ are bounded by $\alpha _{\mathcal{A}%
}\left( A\right) $ for $I\in \mathcal{C}_{A}^{\sigma }$, the quasiHaar
support of $f$ is contained in the corona $\mathcal{C}_{A}^{\sigma }$, and
the quasiHaar support of $g$\ is contained in $\mathcal{C}_{A}^{\mathbf{\tau 
}-\limfunc{shift}}$, and where $I_{J}$ is the $\Omega \mathcal{D}$-child of $%
I$ that contains $J$. Indeed, to see this, we note that $\bigtriangleup
_{I}^{\sigma }f=\mathbf{1}_{I}\bigtriangleup _{I}^{\sigma }f$ and write both%
\begin{eqnarray*}
\mathbf{1}_{I} &=&\mathbf{1}_{I_{J}}+\sum_{\theta \left( I_{J}\right) \in 
\mathfrak{C}_{\Omega \mathcal{D}}\left( I\right) \setminus \left\{
I_{J}\right\} }\mathbf{1}_{\theta \left( I_{J}\right) }\ , \\
\mathbf{1}_{I_{J}} &=&\mathbf{1}_{A}-\mathbf{1}_{A\setminus I_{J}}\ ,
\end{eqnarray*}%
where $\theta \left( I_{J}\right) \in \mathfrak{C}_{\Omega \mathcal{D}%
}\left( I\right) \setminus \left\{ I_{J}\right\} $ ranges over the $2^{n}-1$ 
$\Omega \mathcal{D}$-children of $I$ other than the child $I_{J}$ that
contains $J$. Then we obtain%
\begin{eqnarray*}
\left\langle T_{\sigma }^{\alpha }\bigtriangleup _{I}^{\sigma
}f,\bigtriangleup _{J}^{\omega }g\right\rangle _{\omega } &=&\left\langle
T_{\sigma }^{\alpha }\left( \mathbf{1}_{I_{J}}\bigtriangleup _{I}^{\sigma
}f\right) ,\bigtriangleup _{J}^{\omega }g\right\rangle _{\omega
}+\sum_{\theta \left( I_{J}\right) \in \mathfrak{C}_{\Omega \mathcal{D}%
}\left( I\right) \setminus \left\{ I_{J}\right\} }\left\langle T_{\sigma
}^{\alpha }\left( \mathbf{1}_{\theta \left( I_{J}\right) }\bigtriangleup
_{I}^{\sigma }f\right) ,\bigtriangleup _{J}^{\omega }g\right\rangle _{\omega
} \\
&=&\left( \mathbb{E}_{I_{J}}^{\sigma }\bigtriangleup _{I}^{\sigma }f\right)
\left\langle T_{\sigma }^{\alpha }\left( \mathbf{1}_{I_{J}}\right)
,\bigtriangleup _{J}^{\omega }g\right\rangle _{\omega }+\sum_{\theta \left(
I_{J}\right) \in \mathfrak{C}_{\Omega \mathcal{D}}\left( I\right) \setminus
\left\{ I_{J}\right\} }\left\langle T_{\sigma }^{\alpha }\left( \mathbf{1}%
_{\theta \left( I_{J}\right) }\bigtriangleup _{I}^{\sigma }f\right)
,\bigtriangleup _{J}^{\omega }g\right\rangle _{\omega } \\
&=&\left( \mathbb{E}_{I_{J}}^{\sigma }\bigtriangleup _{I}^{\sigma }f\right)
\left\langle T_{\sigma }^{\alpha }\mathbf{1}_{A},\bigtriangleup _{J}^{\omega
}g\right\rangle _{\omega } \\
&&-\left( \mathbb{E}_{I_{J}}^{\sigma }\bigtriangleup _{I}^{\sigma }f\right)
\left\langle T_{\sigma }^{\alpha }\mathbf{1}_{A\setminus
I_{J}},\bigtriangleup _{J}^{\omega }g\right\rangle _{\omega } \\
&&+\sum_{\theta \left( I_{J}\right) \in \mathfrak{C}_{\Omega \mathcal{D}%
}\left( I\right) \setminus \left\{ I_{J}\right\} }\left\langle T_{\sigma
}^{\alpha }\left( \mathbf{1}_{\theta \left( I_{J}\right) }\bigtriangleup
_{I}^{\sigma }f\right) ,\bigtriangleup _{J}^{\omega }g\right\rangle _{\omega
}\ ,
\end{eqnarray*}%
and the corresponding NTV splitting of $\mathsf{B}_{\Subset _{\mathbf{\rho }%
}}^{A}\left( f,g\right) $:%
\begin{eqnarray*}
\mathsf{B}_{\Subset _{\mathbf{\rho }}}^{A}\left( f,g\right) &=&\left\langle
T_{\sigma }^{\alpha }\left( \mathsf{P}_{\mathcal{C}_{A}}^{\sigma }f\right) ,%
\mathsf{P}_{\mathcal{C}_{A}^{\mathbf{\tau }-\limfunc{shift}}}^{\omega
}g\right\rangle _{\omega }^{\Subset _{\mathbf{\rho }}}=\sum_{\substack{ I\in 
\mathcal{C}_{A}\text{ and }J\in \mathcal{C}_{A}^{\mathbf{\tau }-\limfunc{%
shift}}  \\ J\Subset _{\mathbf{\rho }}I}}\left\langle T_{\sigma }^{\alpha
}\left( \bigtriangleup _{I}^{\sigma }f\right) ,\bigtriangleup _{J}^{\omega
}g\right\rangle _{\omega } \\
&=&\sum_{\substack{ I\in \mathcal{C}_{A}\text{ and }J\in \mathcal{C}_{A}^{%
\mathbf{\tau }-\limfunc{shift}}  \\ J\Subset _{\mathbf{\rho }}I}}\left( 
\mathbb{E}_{I_{J}}^{\sigma }\bigtriangleup _{I}^{\sigma }f\right)
\left\langle T_{\sigma }^{\alpha }\mathbf{1}_{A},\bigtriangleup _{J}^{\omega
}g\right\rangle _{\omega } \\
&&-\sum_{\substack{ I\in \mathcal{C}_{A}\text{ and }J\in \mathcal{C}_{A}^{%
\mathbf{\tau }-\limfunc{shift}}  \\ J\Subset _{\mathbf{\rho }}I}}\left( 
\mathbb{E}_{I_{J}}^{\sigma }\bigtriangleup _{I}^{\sigma }f\right)
\left\langle T_{\sigma }^{\alpha }\mathbf{1}_{A\setminus
I_{J}},\bigtriangleup _{J}^{\omega }g\right\rangle _{\omega } \\
&&+\sum_{\substack{ I\in \mathcal{C}_{A}\text{ and }J\in \mathcal{C}_{A}^{%
\mathbf{\tau }-\limfunc{shift}}  \\ J\Subset _{\mathbf{\rho }}I}}%
\sum_{\theta \left( I_{J}\right) \in \mathfrak{C}_{\Omega \mathcal{D}}\left(
I\right) \setminus \left\{ I_{J}\right\} }\left\langle T_{\sigma }^{\alpha
}\left( \mathbf{1}_{\theta \left( I_{J}\right) }\bigtriangleup _{I}^{\sigma
}f\right) ,\bigtriangleup _{J}^{\omega }g\right\rangle _{\omega } \\
&\equiv &\mathsf{B}_{paraproduct}^{A}\left( f,g\right) -\mathsf{B}%
_{stop}^{A}\left( f,g\right) +\mathsf{B}_{neighbour}^{A}\left( f,g\right) .
\end{eqnarray*}%
The paraproduct form $\mathsf{B}_{paraproduct}^{A}\left( f,g\right) $ is
easily controlled by the testing condition for $T^{\alpha }$. Indeed, we have%
\begin{eqnarray*}
\mathsf{B}_{paraproduct}^{A}\left( f,g\right) &=&\sum_{\substack{ I\in 
\mathcal{C}_{A}\text{ and }J\in \mathcal{C}_{A}^{\mathbf{\tau }-\limfunc{%
shift}}  \\ J\Subset _{\mathbf{\rho }}I}}\left( \mathbb{E}_{I_{J}}^{\sigma
}\bigtriangleup _{I}^{\sigma }f\right) \left\langle T_{\sigma }^{\alpha }%
\mathbf{1}_{A},\bigtriangleup _{J}^{\omega }g\right\rangle _{\omega } \\
&=&\sum_{J\in \mathcal{C}_{A}^{\mathbf{\tau }-\limfunc{shift}}}\left\langle
T_{\sigma }^{\alpha }\mathbf{1}_{A},\bigtriangleup _{J}^{\omega
}g\right\rangle _{\omega }\left\{ \sum_{I\in \mathcal{C}_{A}\text{:\ }%
J\Subset _{\mathbf{\rho }}I}\left( \mathbb{E}_{I_{J}}^{\sigma
}\bigtriangleup _{I}^{\sigma }f\right) \right\} \\
&=&\sum_{J\in \mathcal{C}_{A}^{\mathbf{\tau }-\limfunc{shift}}}\left\langle
T_{\sigma }^{\alpha }\mathbf{1}_{A},\bigtriangleup _{J}^{\omega
}g\right\rangle _{\omega }\left\{ \mathbb{E}_{I^{\natural }\left( J\right)
_{J}}^{\sigma }f-\mathbb{E}_{A}^{\sigma }f\right\} \\
&=&\left\langle T_{\sigma }^{\alpha }\mathbf{1}_{A},\sum_{J\in \mathcal{C}%
_{A}^{\mathbf{\tau }-\limfunc{shift}}}\left\{ \mathbb{E}_{I^{\natural
}\left( J\right) _{J}}^{\sigma }f-\mathbb{E}_{A}^{\sigma }f\right\}
\bigtriangleup _{J}^{\omega }g\right\rangle _{\omega }\ ,
\end{eqnarray*}%
where $I^{\natural }\left( J\right) $ denotes the smallest quasicube $I\in 
\mathcal{C}_{A}$ such that $J\Subset _{\mathbf{\rho }}I$, and of course $%
I^{\natural }\left( J\right) _{J}$ denotes its child containing $J$. We
claim that by construction of the corona we have $I^{\natural }\left(
J\right) _{J}\notin \mathcal{A}$, and so $\left\vert \mathbb{E}_{I^{\natural
}\left( J\right) _{J}}^{\sigma }f\right\vert \lesssim \mathbb{E}_{A}^{\sigma
}\left\vert f\right\vert \leq \alpha _{\mathcal{A}}\left( A\right) $.
Indeed, in our application of the stopping form we have $f=\mathsf{P}_{%
\mathcal{C}_{A}}^{\sigma }f$ and $g=\mathsf{P}_{\mathcal{C}_{A}^{\mathbf{%
\tau }-\limfunc{shift}}}^{\omega }g$, and the definitions of the coronas $%
\mathcal{C}_{A}$ and $\mathcal{C}_{A}^{\mathbf{\tau }-\limfunc{shift}}$
together with $\mathbf{r}<\mathbf{\tau }<\mathbf{\rho }$ imply that $%
I^{\natural }\left( J\right) _{J}\notin \mathcal{A}$ for $J\in \mathcal{C}%
_{A}^{\mathbf{\tau }-\limfunc{shift}}$.

Thus from the orthogonality of the quasiHaar projections $\bigtriangleup
_{J}^{\omega }g$ and the bound on the coefficients $\left\vert \mathbb{E}%
_{I^{\natural }\left( J\right) _{J}}^{\sigma }f-\mathbb{E}_{A}^{\sigma
}f\right\vert \lesssim \alpha _{\mathcal{A}}\left( A\right) $ we have%
\begin{eqnarray*}
\left\vert \mathsf{B}_{paraproduct}^{A}\left( f,g\right) \right\vert
&=&\left\vert \left\langle T_{\sigma }^{\alpha }\mathbf{1}_{A},\sum_{J\in 
\mathcal{C}_{A}^{\mathbf{\tau }-\limfunc{shift}}}\left\{ \mathbb{E}%
_{I^{\natural }\left( J\right) _{J}}^{\sigma }f-\mathbb{E}_{A}^{\sigma
}f\right\} \bigtriangleup _{J}^{\omega }g\right\rangle _{\omega }\right\vert
\\
&\lesssim &\alpha _{\mathcal{A}}\left( A\right) \ \left\Vert \mathbf{1}%
_{A}T_{\sigma }^{\alpha }\mathbf{1}_{A}\right\Vert _{L^{2}\left( \omega
\right) }\ \left\Vert \mathsf{P}_{\mathcal{C}_{A}^{\mathbf{\tau }-\limfunc{%
shift}}}^{\omega }g\right\Vert _{L^{2}\left( \omega \right) } \\
&\leq &\mathfrak{T}_{T^{\alpha }}\ \alpha _{\mathcal{A}}\left( A\right) \ 
\sqrt{\left\vert A\right\vert _{\sigma }}\ \left\Vert \mathsf{P}_{\mathcal{C}%
_{A}^{\mathbf{\tau }-\limfunc{shift}}}^{\omega }g\right\Vert _{L^{2}\left(
\omega \right) },
\end{eqnarray*}%
because $\left\Vert \sum_{J\in \mathcal{C}_{A}^{\mathbf{\tau }-\limfunc{shift%
}}}\lambda _{J}\bigtriangleup _{J}^{\omega }g\right\Vert _{L^{2}\left(
\omega \right) }\leq \left( \sup_{J}\left\vert \lambda _{J}\right\vert
\right) \left\Vert \sum_{J\in \mathcal{C}_{A}^{\mathbf{\tau }-\limfunc{shift}%
}}\bigtriangleup _{J}^{\omega }g\right\Vert _{L^{2}\left( \omega \right) }$.

Next, the neighbour form $\mathsf{B}_{neighbour}^{A}\left( f,g\right) $ is
easily controlled by the $A_{2}^{\alpha }$ condition using the Energy Lemma %
\ref{ener} and the fact that the quasicubes $J$ are good. In particular, the
information encoded in the stopping tree $\mathcal{A}$ plays no role here.
We have%
\begin{equation*}
\mathsf{B}_{neighbour}^{A}\left( f,g\right) =\sum_{\substack{ I\in \mathcal{C%
}_{A}\text{ and }J\in \mathcal{C}_{A}^{\mathbf{\tau }-\limfunc{shift}}  \\ %
J\Subset _{\mathbf{\rho }}I}}\sum_{\theta \left( I_{J}\right) \in \mathfrak{C%
}_{\Omega \mathcal{D}}\left( I\right) \setminus \left\{ I_{J}\right\}
}\left\langle T_{\sigma }^{\alpha }\left( \mathbf{1}_{\theta \left(
I_{J}\right) }\bigtriangleup _{I}^{\sigma }f\right) ,\bigtriangleup
_{J}^{\omega }g\right\rangle _{\omega }.
\end{equation*}%
Recall that $I_{J}$ is the child of $I$ that contains $J$. Fix $\theta
\left( I_{J}\right) \in \mathfrak{C}_{\Omega \mathcal{D}}\left( I\right)
\setminus \left\{ I_{J}\right\} $ momentarily, and an integer $s\geq \mathbf{%
r}$. The inner product to be estimated is 
\begin{equation*}
\left\langle T_{\sigma }^{\alpha }(\mathbf{1}_{\theta \left( I_{J}\right)
}\sigma \Delta _{I}^{\sigma }f),\Delta _{J}^{\omega }g\right\rangle _{\omega
},
\end{equation*}%
i.e. 
\begin{equation*}
\left\langle T_{\sigma }^{\alpha }\left( \mathbf{1}_{\theta \left(
I_{J}\right) }\bigtriangleup _{I}^{\sigma }f\right) ,\bigtriangleup
_{J}^{\omega }g\right\rangle _{\omega }=\mathbb{E}_{\theta \left(
I_{J}\right) }^{\sigma }\Delta _{I}^{\sigma }f\cdot \left\langle T_{\sigma
}^{\alpha }\left( \mathbf{1}_{\theta \left( I_{J}\right) }\right)
,\bigtriangleup _{J}^{\omega }g\right\rangle _{\omega }.
\end{equation*}%
Thus we can write%
\begin{equation}
\mathsf{B}_{neighbour}^{A}\left( f,g\right) =\sum_{\substack{ I\in \mathcal{C%
}_{A}\text{ and }J\in \mathcal{C}_{A}^{\mathbf{\tau }-\limfunc{shift}}  \\ %
J\Subset _{\mathbf{\rho }}I}}\sum_{\theta \left( I_{J}\right) \in \mathfrak{C%
}_{\Omega \mathcal{D}}\left( I\right) \setminus \left\{ I_{J}\right\}
}\left( \mathbb{E}_{\theta \left( I_{J}\right) }^{\sigma }\Delta
_{I}^{\sigma }f\right) \ \left\langle T_{\sigma }^{\alpha }\left( \mathbf{1}%
_{\theta \left( I_{J}\right) }\sigma \right) ,\Delta _{J}^{\omega
}g\right\rangle _{\omega }  \label{neighbour term}
\end{equation}

Now we will use the following fractional analogue of the Poisson inequality
in \cite{Vol}. We remind the reader that there are absolute positive
constants $c,C\,$\ such that $c\left\vert J\right\vert ^{\frac{1}{n}}\leq
\ell \left( J\right) \leq C\left\vert J\right\vert ^{\frac{1}{n}}$ for all
quasicubes $J$, and that we defined the quasidistance $\limfunc{quasidist}%
\left( E,F\right) $ between two sets $E$ and $F$ to be the Euclidean
distance $\limfunc{dist}\left( \Omega ^{-1}E,\Omega ^{-1}F\right) $ between
the preimages under $\Omega $ of the sets $E$ and $F$.

\begin{lemma}
\label{Poisson inequality}Suppose that $J\subset I\subset K$ and that $%
\limfunc{quasidist}\left( J,\partial I\right) >\tfrac{1}{2}\ell \left(
J\right) ^{\varepsilon }\ell \left( I\right) ^{1-\varepsilon }$. Then%
\begin{equation}
\mathrm{P}^{\alpha }(J,\sigma \mathbf{1}_{K\setminus I})\lesssim \left( 
\frac{\ell \left( J\right) }{\ell \left( I\right) }\right) ^{1-\varepsilon
\left( n+1-\alpha \right) }\mathrm{P}^{\alpha }(I,\sigma \mathbf{1}%
_{K\setminus I}).  \label{e.Jsimeq}
\end{equation}
\end{lemma}

\begin{proof}
We have%
\begin{equation*}
\mathrm{P}^{\alpha }\left( J,\sigma \chi _{K\setminus I}\right) \approx
\sum_{k=0}^{\infty }2^{-k}\frac{1}{\left\vert 2^{k}J\right\vert ^{1-\frac{%
\alpha }{n}}}\int_{\left( 2^{k}J\right) \cap \left( K\setminus I\right)
}d\sigma ,
\end{equation*}%
and $\left( 2^{k}J\right) \cap \left( K\setminus I\right) \neq \emptyset $
requires%
\begin{equation*}
\limfunc{quasidist}\left( J,K\setminus I\right) \leq c2^{k}\ell \left(
J\right) ,
\end{equation*}%
for some dimensional constant $c>0$. Let $k_{0}$ be the smallest such $k$.
By our distance assumption we must then have%
\begin{equation*}
\tfrac{1}{2}\ell \left( J\right) ^{\varepsilon }\ell \left( I\right)
^{1-\varepsilon }\leq \limfunc{quasidist}\left( J,\partial I\right) \leq
c2^{k_{0}}\ell \left( J\right) ,
\end{equation*}%
or%
\begin{equation*}
2^{-k_{0}-1}\leq c\left( \frac{\ell \left( J\right) }{\ell \left( I\right) }%
\right) ^{1-\varepsilon }.
\end{equation*}%
Now let $k_{1}$ be defined by $2^{k_{1}}\equiv \frac{\ell \left( I\right) }{%
\ell \left( J\right) }$. Then assuming $k_{1}>k_{0}$ (the case $k_{1}\leq
k_{0}$ is similar) we have%
\begin{eqnarray*}
\mathrm{P}^{\alpha }\left( J,\sigma \chi _{K\setminus I}\right) &\approx
&\left\{ \sum_{k=k_{0}}^{k_{1}}+\sum_{k=k_{1}}^{\infty }\right\} 2^{-k}\frac{%
1}{\left\vert 2^{k}J\right\vert ^{1-\frac{\alpha }{n}}}\int_{\left(
2^{k}J\right) \cap \left( K\setminus I\right) }d\sigma \\
&\lesssim &2^{-k_{0}}\frac{\left\vert I\right\vert ^{1-\frac{\alpha }{n}}}{%
\left\vert 2^{k_{0}}J\right\vert ^{1-\frac{\alpha }{n}}}\left( \frac{1}{%
\left\vert I\right\vert ^{1-\frac{\alpha }{n}}}\int_{\left(
2^{k_{1}}J\right) \cap \left( K\setminus I\right) }d\sigma \right)
+2^{-k_{1}}\mathrm{P}^{\alpha }\left( I,\sigma \chi _{\setminus I}\right) \\
&\lesssim &\left( \frac{\ell \left( J\right) }{\ell \left( I\right) }\right)
^{\left( 1-\varepsilon \right) \left( n+1-\alpha \right) }\left( \frac{\ell
\left( I\right) }{\ell \left( J\right) }\right) ^{n-\alpha }\mathrm{P}%
^{\alpha }\left( I,\sigma \chi _{K\setminus I}\right) +\frac{\ell \left(
J\right) }{\ell \left( I\right) }\mathrm{P}^{\alpha }\left( I,\sigma \chi
_{K\setminus I}\right) ,
\end{eqnarray*}%
which is the inequality (\ref{e.Jsimeq}).
\end{proof}

Now fix $I_{0},I_{\theta }\in \mathfrak{C}_{\Omega \mathcal{D}}\left(
I\right) $ with $I_{0}\neq I_{\theta }$ and assume that $J\Subset _{\mathbf{r%
}}I_{0}$. Let $\frac{\ell \left( J\right) }{\ell \left( I_{0}\right) }%
=2^{-s} $ in the pivotal estimate in the Energy Lemma \ref{ener} with $%
J\subset I_{0}\subset I$ to obtain 
\begin{align*}
\left\vert \langle T_{\sigma }^{\alpha }\left( \mathbf{1}_{I_{\theta
}}\sigma \right) ,\Delta _{J}^{\omega }g\rangle _{\omega }\right\vert &
\lesssim \left\Vert \Delta _{J}^{\omega }g\right\Vert _{L^{2}\left( \omega
\right) }\sqrt{\left\vert J\right\vert _{\omega }}\mathrm{P}^{\alpha }\left(
J,\mathbf{1}_{I_{\theta }}\sigma \right) \\
& \lesssim \left\Vert \Delta _{J}^{\omega }g\right\Vert _{L^{2}\left( \omega
\right) }\sqrt{\left\vert J\right\vert _{\omega }}\cdot 2^{-\left(
1-\varepsilon \left( n+1-\alpha \right) \right) s}\mathrm{P}^{\alpha }\left(
I_{0},\mathbf{1}_{I_{\theta }}\sigma \right) .
\end{align*}%
Here we are using (\ref{e.Jsimeq}), which applies since $J\subset I_{0}$.

In the sum below, we keep the side lengths of the quasicubes $J$ fixed at $%
2^{-s}$ times that of $I$, and of course take $J\subset I_{0}$. We estimate 
\begin{align*}
A(I,I_{0},I_{\theta },s)& \equiv \sum_{J\;:\;2^{s}\ell \left( J\right) =\ell
\left( I\right) :J\subset I_{0}}\left\vert \langle T_{\sigma }^{\alpha
}\left( \mathbf{1}_{I_{\theta }}\sigma \Delta _{I}^{\sigma }f\right) ,\Delta
_{J}^{\omega }g\rangle _{\omega }\right\vert \\
& \leq 2^{-\left( 1-\varepsilon \left( n+1-\alpha \right) \right) s}|\mathbb{%
E}_{I_{\theta }}^{\sigma }\Delta _{I}^{\sigma }f|\ \mathrm{P}^{\alpha
}(I_{0},\mathbf{1}_{I_{\theta }}\sigma )\sum_{J\;:\;2^{s}\ell \left(
J\right) =\ell \left( I\right) :\ J\subset I_{0}}\left\Vert \Delta
_{J}^{\omega }g\right\Vert _{L^{2}\left( \omega \right) }\sqrt{\left\vert
J\right\vert _{\omega }} \\
& \leq 2^{-\left( 1-\varepsilon \left( n+1-\alpha \right) \right) s}|\mathbb{%
E}_{I_{\theta }}^{\sigma }\Delta _{I}^{\sigma }f|\ \mathrm{P}^{\alpha
}(I_{0},\mathbf{1}_{I_{\theta }}\sigma )\sqrt{\left\vert I_{0}\right\vert
_{\omega }}\Lambda (I,I_{0},I_{\theta },s), \\
\Lambda (I,I_{0},I_{\theta },s)^{2}& \equiv \sum_{J\in \mathcal{C}_{A}^{%
\mathbf{\tau }-\limfunc{shift}}:\;2^{s}\ell \left( J\right) =\ell \left(
I\right) :\ J\subset I_{0}}\left\Vert \Delta _{J}^{\omega }g\right\Vert
_{L^{2}\left( \omega \right) }^{2}\,.
\end{align*}%
The last line follows upon using the Cauchy-Schwarz inequality and the fact
that $\Delta _{J}^{\omega }g=0$ if $J\notin \mathcal{C}_{A}^{\mathbf{\tau }-%
\limfunc{shift}}$. We also note that since $2^{s+1}\ell \left( J\right)
=\ell \left( I\right) $, 
\begin{eqnarray}
\sum_{I_{0}\in \mathfrak{C}_{\Omega \mathcal{D}}\left( I\right) }\Lambda
(I,I_{0},I_{\theta },s)^{2} &\equiv &\sum_{J\in \mathcal{C}_{A}^{\mathbf{%
\tau }-\limfunc{shift}}:\;2^{s+1}\ell \left( J\right) =\ell \left( I\right)
:\ J\subset I}\left\Vert \Delta _{J}^{\omega }g\right\Vert _{L^{2}\left(
\omega \right) }^{2}\ ;  \label{g} \\
\sum_{I\in \mathcal{C}_{A}}\sum_{I_{0}\in \mathfrak{C}_{\Omega \mathcal{D}%
}\left( I\right) }\Lambda (I,I_{0},I_{\theta },s)^{2} &\leq &\left\Vert 
\mathsf{P}_{\mathcal{C}_{A}^{\mathbf{\tau }-\limfunc{shift}}}^{\omega
}g\right\Vert _{L^{2}(\omega )}^{2}\ .  \notag
\end{eqnarray}

Using 
\begin{equation}
\left\vert \mathbb{E}_{I_{\theta }}^{\sigma }\Delta _{I}^{\sigma
}f\right\vert \leq \sqrt{\mathbb{E}_{I_{\theta }}^{\sigma }\left\vert \Delta
_{I}^{\sigma }f\right\vert ^{2}}\leq \left\Vert \Delta _{I}^{\sigma
}f\right\Vert _{L^{2}\left( \sigma \right) }\ \left\vert I_{\theta
}\right\vert _{\sigma }^{-\frac{1}{2}},  \label{e.haarAvg}
\end{equation}%
we can thus estimate $A(I,I_{0},I_{\theta },s)$ as follows, in which we use
the $A_{2}^{\alpha }$ hypothesis $\sup_{I}\frac{\left\vert I\right\vert
_{\sigma }\left\vert I\right\vert _{\omega }}{\left\vert I\right\vert
^{2\left( 1-\frac{\alpha }{n}\right) }}=A_{2}^{\alpha }<\infty $: 
\begin{eqnarray*}
A(I,I_{0},I_{\theta },s) &\lesssim &2^{-\left( 1-\varepsilon \left(
n+1-\alpha \right) \right) s}\left\Vert \Delta _{I}^{\sigma }f\right\Vert
_{L^{2}\left( \sigma \right) }\Lambda (I,I_{0},I_{\theta },s)\cdot
\left\vert I_{\theta }\right\vert _{\sigma }^{-\frac{1}{2}}\mathrm{P}%
^{\alpha }(I_{0},\mathbf{1}_{I_{\theta }}\sigma )\sqrt{\left\vert
I_{0}\right\vert _{\omega }} \\
&\lesssim &\sqrt{A_{2}^{\alpha }}2^{-\left( 1-\varepsilon \left( n+1-\alpha
\right) \right) s}\left\Vert \Delta _{I}^{\sigma }f\right\Vert _{L^{2}\left(
\sigma \right) }\Lambda (I,I_{0},I_{\theta },s)\,,
\end{eqnarray*}%
since $\mathrm{P}^{\alpha }(I_{0},\mathbf{1}_{I_{\theta }}\sigma )\lesssim 
\frac{\left\vert I_{\theta }\right\vert _{\sigma }}{\left\vert I_{\theta
}\right\vert ^{1-\frac{\alpha }{n}}}$ shows that 
\begin{equation*}
\left\vert I_{\theta }\right\vert _{\sigma }^{-\frac{1}{2}}\mathrm{P}%
^{\alpha }(I_{0},\mathbf{1}_{I_{\theta }}\sigma )\ \sqrt{\left\vert
I_{0}\right\vert _{\omega }}\lesssim \frac{\sqrt{\left\vert I_{\theta
}\right\vert _{\sigma }}\sqrt{\left\vert I_{0}\right\vert _{\omega }}}{%
\left\vert I\right\vert ^{1-\frac{\alpha }{n}}}\lesssim \sqrt{A_{2}^{\alpha }%
}.
\end{equation*}

An application of Cauchy-Schwarz to the sum over $I$ using (\ref{g}) then
shows that 
\begin{eqnarray*}
&&\sum_{I\in \mathcal{C}_{A}}\sum_{\substack{ I_{0},I_{\theta }\in \mathfrak{%
C}_{\Omega \mathcal{D}}\left( I\right)  \\ I_{0}\neq I_{\theta }}}%
A(I,I_{0},I_{\theta },s) \\
&\lesssim &\sqrt{A_{2}^{\alpha }}2^{-\left( 1-\varepsilon \left( n+1-\alpha
\right) \right) s}\sqrt{\sum_{I\in \mathcal{C}_{A}}\left\Vert \Delta
_{I}^{\sigma }f\right\Vert _{L^{2}\left( \sigma \right) }^{2}}\sqrt{%
\sum_{I\in \mathcal{C}_{A}}\left( \sum_{\substack{ I_{0},I_{\theta }\in 
\mathfrak{C}_{\Omega \mathcal{D}}\left( I\right)  \\ I_{0}\neq I_{\theta }}}%
\Lambda (I,I_{0},I_{\theta },s)\right) ^{2}} \\
&\lesssim &\sqrt{A_{2}^{\alpha }}2^{-\left( 1-\varepsilon \left( n+1-\alpha
\right) \right) s}\sqrt{2^{n}}\lVert \mathsf{P}_{\mathcal{C}_{A}}^{\sigma
}f\rVert _{L^{2}(\sigma )}\sqrt{\sum_{I\in \mathcal{C}_{A}}\left( \sum 
_{\substack{ I_{0}\in \mathfrak{C}_{\Omega \mathcal{D}}\left( I\right)  \\ %
I_{0}\neq I_{\theta }}}\Lambda (I,I_{0},I_{\theta },s)\right) ^{2}} \\
&\lesssim &\sqrt{A_{2}^{\alpha }}2^{-\left( 1-\varepsilon \left( n+1-\alpha
\right) \right) s}\lVert \mathsf{P}_{\mathcal{C}_{A}}^{\sigma }f\rVert
_{L^{2}(\sigma )}\left\Vert \mathsf{P}_{\mathcal{C}_{A}^{\mathbf{\tau }-%
\limfunc{shift}}}^{\omega }g\right\Vert _{L^{2}(\omega )}\,.
\end{eqnarray*}%
This estimate is summable in $s\geq \mathbf{r}$, and so the proof of 
\begin{eqnarray*}
\left\vert \mathsf{B}_{neighbour}^{A}\left( f,g\right) \right\vert
&=&\left\vert \sum_{\substack{ I\in \mathcal{C}_{A}\text{ and }J\in \mathcal{%
C}_{A}^{\mathbf{\tau }-\limfunc{shift}}  \\ J\Subset _{\mathbf{\rho }}I}}%
\sum_{\theta \left( I_{J}\right) \in \mathfrak{C}_{\Omega \mathcal{D}}\left(
I\right) \setminus \left\{ I_{J}\right\} }\left\langle T_{\sigma }^{\alpha
}\left( \mathbf{1}_{\theta \left( I_{J}\right) }\bigtriangleup _{I}^{\sigma
}f\right) ,\bigtriangleup _{J}^{\omega }g\right\rangle _{\omega }\right\vert
\\
&=&\left\vert \sum_{I\in \mathcal{C}_{A}}\sum_{\substack{ I_{0},I_{\theta
}\in \mathfrak{C}_{\Omega \mathcal{D}}\left( I\right)  \\ I_{0}\neq
I_{\theta }}}\sum_{s=\mathbf{r}}^{\infty }A(I,I_{0},I_{\theta },s)\right\vert
\\
&\lesssim &\sqrt{A_{2}^{\alpha }}\left\Vert \mathsf{P}_{\mathcal{C}%
_{A}}^{\sigma }f\right\Vert _{L^{2}(\sigma )}\left\Vert \mathsf{P}_{\mathcal{%
C}_{A}^{\mathbf{\tau }-\limfunc{shift}}}^{\omega }g\right\Vert
_{L^{2}(\omega )}
\end{eqnarray*}%
is complete.

\bigskip

It is to the sublinear form on the left side of (\ref{First inequality})
below, derived from the stopping form $\mathsf{B}_{stop}^{A}\left(
f,g\right) $, that the argument of M. Lacey in \cite{Lac} will be adapted.
This will result in the inequality%
\begin{equation}
\left\vert \mathsf{B}_{stop}^{A}\left( f,g\right) \right\vert \lesssim
\left( \mathcal{E}_{\alpha }^{\limfunc{deep}}+\sqrt{A_{2}^{\alpha }}\right)
\ \left( \alpha _{\mathcal{A}}\left( A\right) \sqrt{\left\vert A\right\vert
_{\sigma }}+\left\Vert f\right\Vert _{L^{2}\left( \sigma \right) }\right) \
\left\Vert g\right\Vert _{L^{2}\left( \omega \right) }\ ,\ \ \ \ \ A\in 
\mathcal{A},  \label{B stop form 3}
\end{equation}%
where the bounded quasiaverages of $f$ in $\mathsf{B}_{stop}^{A}\left(
f,g\right) $ will prove crucial. But first we turn to completing the proof
of the bound (\ref{far below bound}) for the far below form $\mathsf{T}_{%
\limfunc{far}\limfunc{below}}\left( f,g\right) $ using the Intertwining
Proposition and the control of functional energy by the $\mathcal{A}%
_{2}^{\alpha }$ condition and the energy condition $\mathcal{E}_{\alpha }$.

\section{Intertwining proposition}

Here we generalize the Intertwining Proposition (see e.g. \cite{Saw}) to
higher dimensions. The main principle here says that, modulo terms that are
controlled by the functional quasienergy constant $\mathfrak{F}_{\alpha }$
and the quasiNTV constant $\mathcal{NTV}_{\alpha }$ (see below), we can pass
the shifted $\omega $-corona projection $\mathsf{P}_{\mathcal{C}_{B}^{%
\mathbf{\tau }-\limfunc{shift}}}^{\omega }$ through the operator $T^{\alpha
} $ to become the shifted corona projection $\sigma $-corona projection $%
\mathsf{P}_{\mathcal{C}_{B}^{\mathbf{\tau }-\limfunc{shift}}}^{\sigma }$.
More precisely, the idea is that with $T_{\sigma }^{\alpha }f\equiv
T^{\alpha }\left( f\sigma \right) $, the intertwining operator 
\begin{equation*}
\mathsf{P}_{\mathcal{C}_{B}^{\mathbf{\tau }-\limfunc{shift}}}^{\omega }\left[
\mathsf{P}_{\mathcal{C}_{B}^{\mathbf{\tau }-\limfunc{shift}}}^{\omega
}T_{\sigma }^{\alpha }-T_{\sigma }^{\alpha }\mathsf{P}_{\mathcal{C}_{B}^{%
\mathbf{\tau }-\limfunc{shift}}}^{\sigma }\right] \mathsf{P}_{\mathcal{C}%
_{A}}^{\sigma }
\end{equation*}%
is bounded with constant $\mathfrak{F}_{\alpha }+\mathcal{NTV}_{\alpha }$.
In those cases where the coronas $\mathcal{C}_{B}^{\mathbf{\tau }-\limfunc{%
shift}}$ and $\mathcal{C}_{A}$ are (almost) disjoint, the intertwining
operator reduces (essentially) to $\mathsf{P}_{\mathcal{C}_{B}^{\mathbf{\tau 
}-\limfunc{shift}}}^{\omega }T_{\sigma }^{\alpha }\mathsf{P}_{\mathcal{C}%
_{A}}^{\sigma }$, and then combined with the control of the functional
quasienergy constant $\mathfrak{F}_{\alpha }$ by the quasienergy condition
constant $\mathcal{E}_{\alpha }$ and $\mathcal{A}_{2}^{\alpha }+\mathcal{A}%
_{2}^{\alpha ,\ast }$, we obtain the required bound (\ref{far below bound})
for $\mathsf{T}_{\limfunc{far}\limfunc{below}}\left( f,g\right) $ above.

To describe the quantities we use to bound these forms, we need to adapt to
higher dimensions three definitions used for the Hilbert transform that are
relevant to functional energy.

\begin{definition}
\label{sigma carleson n}A collection $\mathcal{F}$ of dyadic quasicubes is $%
\sigma $\emph{-Carleson} if%
\begin{equation*}
\sum_{F\in \mathcal{F}:\ F\subset S}\left\vert F\right\vert _{\sigma }\leq
C_{\mathcal{F}}\left\vert S\right\vert _{\sigma },\ \ \ \ \ S\in \mathcal{F}.
\end{equation*}%
The constant $C_{\mathcal{F}}$ is referred to as the Carleson norm of $%
\mathcal{F}$.
\end{definition}

\begin{definition}
Let $\mathcal{F}$ be a collection of dyadic quasicubes. The good $\tau $%
-shifted corona corresponding to $F$ is defined by%
\begin{equation*}
\mathcal{C}_{F}^{\limfunc{good},\mathbf{\tau }-\limfunc{shift}}\equiv
\left\{ J\in \Omega \mathcal{D}_{\limfunc{good}}^{\omega }:J\Subset _{%
\mathbf{\tau }}F\text{ and }J\not\Subset _{\mathbf{\tau }}F^{\prime }\text{
for any }F^{\prime }\in \mathfrak{C}_{\mathcal{F}}\left( F\right) \right\} .
\end{equation*}
\end{definition}

Note that $\mathcal{C}_{F}^{\limfunc{good},\mathbf{\tau }-\limfunc{shift}}=%
\mathcal{C}_{F}^{\mathbf{\tau }-\limfunc{shift}}\cap \Omega \mathcal{D}_{%
\limfunc{good}}^{\omega }$, and that the collections $\mathcal{C}_{F}^{%
\limfunc{good},\mathbf{\tau }-\limfunc{shift}}$ have bounded overlap $%
\mathbf{\tau }$ since for fixed $J$, there are at most $\mathbf{\tau }$
quasicubes $F\in \mathcal{F}$ with the property that $J$ is good and $%
J\Subset _{\mathbf{\tau }}F$ and $J\not\Subset _{\mathbf{\tau }}F^{\prime }$
for any $F^{\prime }\in \mathfrak{C}_{\mathcal{F}}\left( F\right) $. Here $%
\mathfrak{C}_{\mathcal{F}}\left( F\right) $ denotes the set of $\mathcal{F}$%
-children of $F$. Given any collection $\mathcal{H}\subset \Omega \mathcal{D}
$ of quasicubes, and a dyadic quasicube $J$, we define the corresponding
quasiHaar projection $\mathsf{P}_{\mathcal{H}}^{\omega }$ and its
localization $\mathsf{P}_{\mathcal{H};J}^{\omega }$ to $J$ by%
\begin{equation}
\mathsf{P}_{\mathcal{H}}^{\omega }=\sum_{H\in \mathcal{H}}\bigtriangleup
_{H}^{\omega }\text{ and }\mathsf{P}_{\mathcal{H};J}^{\omega }=\sum_{H\in 
\mathcal{H}:\ H\subset J}\bigtriangleup _{H}^{\omega }\ .
\label{def localization}
\end{equation}

\begin{definition}
\label{functional energy n}Let $\mathfrak{F}_{\alpha }$ be the smallest
constant in the `\textbf{f}unctional quasienergy' inequality below, holding
for all $h\in L^{2}\left( \sigma \right) $ and all $\sigma $-Carleson
collections $\mathcal{F}$ with Carleson norm $C_{\mathcal{F}}$ bounded by a
fixed constant $C$: 
\begin{equation}
\sum_{F\in \mathcal{F}}\sum_{J\in \mathcal{M}_{\mathbf{r}-\limfunc{deep}%
}\left( F\right) }\left( \frac{\mathrm{P}^{\alpha }\left( J,h\sigma \right) 
}{\left\vert J\right\vert ^{\frac{1}{n}}}\right) ^{2}\left\Vert \mathsf{P}_{%
\mathcal{C}_{F}^{\limfunc{good},\mathbf{\tau }-\limfunc{shift}};J}^{\omega }%
\mathbf{x}\right\Vert _{L^{2}\left( \omega \right) }^{2}\leq \mathfrak{F}%
_{\alpha }\lVert h\rVert _{L^{2}\left( \sigma \right) }\,.
\label{e.funcEnergy n}
\end{equation}
\end{definition}

This definition of $\mathfrak{F}_{\alpha }$ depends on the choice of the
fixed constant $C$, but it will be clear from the arguments below that $C$
may be taken to depend only on $n$ and $\alpha $, and we do not compute its
value here. There is a similar definition of the dual constant $\mathfrak{F}%
_{\alpha }^{\ast }$.

\begin{remark}
If in (\ref{e.funcEnergy n}), we take $h=\mathbf{1}_{I}$ and $\mathcal{F}$
to be the trivial Carleson collection $\left\{ I_{r}\right\} _{r=1}^{\infty
} $ where the quasicubes $I_{r}$ are pairwise disjoint in $I$, then we
obtain the deep quasienergy condition in Definition \ref{energy condition},
but with $\mathsf{P}_{\mathcal{C}_{F}^{\limfunc{good},\mathbf{\tau }-%
\limfunc{shift}};J}^{\omega }$ in place of $\mathsf{P}_{J}^{\limfunc{subgood}%
,\omega } $. However, the projection $\mathsf{P}_{J}^{\limfunc{subgood}%
,\omega }$ is larger than $\mathsf{P}_{\mathcal{C}_{F}^{\limfunc{good},%
\mathbf{\tau }-\limfunc{shift}};J}^{\omega }$ by the finite projection $%
\sum_{2^{-\mathbf{\tau }}\ell \left( J\right) \leq \ell \left( J^{\prime
}\right) \leq 2^{-\mathbf{r}}\ell \left( J\right) }\bigtriangleup
_{J^{\prime }}^{\omega }$, and so we just miss obtaining the deep
quasienergy condition as a consequence of the functional quasienergy
condition. Nevertheless, this near miss with $h=\mathbf{1}_{I}$ explains the
terminology `functional' quasienergy.
\end{remark}

We now show that the functional quasienergy inequality (\ref{e.funcEnergy n}%
) suffices to prove an $\alpha $-fractional $n$-dimensional analogue of the
Intertwining Proposition (see e.g. \cite{Saw}). Let $\mathcal{F}$ be any
subset of $\Omega \mathcal{D}$. For any $J\in \Omega \mathcal{D}$, we define 
$\pi _{\mathcal{F}}^{0}J$ to be the smallest $F\in \mathcal{F}$ that
contains $J$. Then for $s\geq 1$, we recursively define $\pi _{\mathcal{F}%
}^{s}J$ to be the smallest $F\in \mathcal{F}$ that \emph{strictly} contains $%
\pi _{\mathcal{F}}^{s-1}J$. This definition satisfies $\pi _{\mathcal{F}%
}^{s+t}J=\pi _{\mathcal{F}}^{s}\pi _{\mathcal{F}}^{t}J$ for all $s,t\geq 0$
and $J\in \Omega \mathcal{D}$. In particular $\pi _{\mathcal{F}}^{s}J=\pi _{%
\mathcal{F}}^{s}F$ where $F=\pi _{\mathcal{F}}^{0}J$. In the special case $%
\mathcal{F}=\Omega \mathcal{D}$ we often suppress the subscript $\mathcal{F}$
and simply write $\pi ^{s}$ for $\pi _{\Omega \mathcal{D}}^{s}$. Finally,
for $F\in \mathcal{F}$, we write $\mathfrak{C}_{\mathcal{F}}\left( F\right)
\equiv \left\{ F^{\prime }\in \mathcal{F}:\pi _{\mathcal{F}}^{1}F^{\prime
}=F\right\} $ for the collection of $\mathcal{F}$-children of $F$. Let 
\begin{equation*}
\mathcal{NTV}_{\alpha }\equiv \sqrt{\mathcal{A}_{2}^{\alpha }}+\mathfrak{T}%
_{\alpha }+\mathcal{WBP}_{\alpha },
\end{equation*}%
where we remind the reader that $\mathfrak{T}_{\alpha }$ and $\mathcal{WBP}%
_{\alpha }$ refer to the quasitesting condition and quasiweak boundedness
property respectively.

\begin{proposition}[The Intertwining Proposition]
\label{strongly adapted}Suppose that $\mathcal{F}$ is $\sigma $-Carleson.
Then%
\begin{equation*}
\left\vert \sum_{F\in \mathcal{F}}\ \sum_{I:\ I\supsetneqq F}\ \left\langle
T_{\sigma }^{\alpha }\bigtriangleup _{I}^{\sigma }f,\mathsf{P}_{\mathcal{C}%
_{F}^{\limfunc{good},\mathbf{\tau }-\limfunc{shift}}}^{\omega
}g\right\rangle _{\omega }\right\vert \lesssim \left( \mathfrak{F}_{\alpha }+%
\mathcal{E}_{\alpha }+\mathcal{NTV}_{\alpha }\right) \ \left\Vert
f\right\Vert _{L^{2}\left( \sigma \right) }\left\Vert g\right\Vert
_{L^{2}\left( \omega \right) }.
\end{equation*}
\end{proposition}

\begin{proof}
We write the left hand side of the display above as%
\begin{equation*}
\sum_{F\in \mathcal{F}}\ \sum_{I:\ I\supsetneqq F}\ \left\langle T_{\sigma
}^{\alpha }\bigtriangleup _{I}^{\sigma }f,g_{F}\right\rangle _{\omega
}=\sum_{F\in \mathcal{F}}\ \left\langle T_{\sigma }^{\alpha }\left(
\sum_{I:\ I\supsetneqq F}\bigtriangleup _{I}^{\sigma }f\right)
,g_{F}\right\rangle _{\omega }\equiv \sum_{F\in \mathcal{F}}\ \left\langle
T_{\sigma }^{\alpha }f_{F},g_{F}\right\rangle _{\omega }\ ,
\end{equation*}%
where%
\begin{equation*}
g_{F}=\mathsf{P}_{\mathcal{C}_{F}^{\limfunc{good},\mathbf{\tau }-\limfunc{%
shift}}}^{\omega }g\text{ and }f_{F}\equiv \sum_{I:\ I\supsetneqq
F}\bigtriangleup _{I}^{\sigma }f\ .
\end{equation*}%
Note that $g_{F}$ is supported in $F$, and that $f_{F}$ is constant on $F$.
We note that the quasicubes $I$ occurring in this sum are linearly and
consecutively ordered by inclusion, along with the quasicubes $F^{\prime
}\in \mathcal{F}$ that contain $F$. More precisely, we can write%
\begin{equation*}
F\equiv F_{0}\subsetneqq F_{1}\subsetneqq F_{2}\subsetneqq ...\subsetneqq
F_{n}\subsetneqq F_{n+1}\subsetneqq ...F_{N}
\end{equation*}%
where $F_{m}=\pi _{\mathcal{F}}^{m}F$ for all $m\geq 1$. We can also write%
\begin{equation*}
F=F_{0}\subsetneqq I_{1}\subsetneqq I_{2}\subsetneqq ...\subsetneqq
I_{k}\subsetneqq I_{k+1}\subsetneqq ...\subsetneqq I_{K}=F_{N}
\end{equation*}%
where $I_{k}=\pi _{\Omega \mathcal{D}}^{k}F$ for all $k\geq 1$. There is a
(unique) subsequence $\left\{ k_{m}\right\} _{m=1}^{N}$ such that%
\begin{equation*}
F_{m}=I_{k_{m}},\ \ \ \ \ 1\leq m\leq N.
\end{equation*}

Define%
\begin{equation*}
f_{F}\left( x\right) =\sum_{\ell =1}^{\infty }\bigtriangleup _{I_{\ell
}}^{\sigma }f\left( x\right) .
\end{equation*}%
Assume now that $k_{m}\leq k<k_{m+1}$. We denote the $2^{n}-1$ siblings of $%
I $ by $\theta \left( I\right) $, $\theta \in \Theta $, i.e. $\left\{ \theta
\left( I\right) \right\} _{\theta \in \Theta }=\mathfrak{C}_{\Omega \mathcal{%
D}}\left( \pi _{\Omega \mathcal{D}}I\right) \setminus \left\{ I\right\} $.
There are two cases to consider here:%
\begin{equation*}
\theta \left( I_{k}\right) \notin \mathcal{F}\text{ and }\theta \left(
I_{k}\right) \in \mathcal{F}.
\end{equation*}%
Suppose first that $\theta \left( I_{k}\right) \notin \mathcal{F}$. Then $%
\theta \left( I_{k}\right) \in \mathcal{C}_{F_{m+1}}^{\sigma }$ and using a
telescoping sum, we compute that for 
\begin{equation*}
x\in \theta \left( I_{k}\right) \subset I_{k+1}\setminus I_{k}\subset
F_{m+1}\setminus F_{m},
\end{equation*}%
we have 
\begin{equation*}
\left\vert f_{F}\left( x\right) \right\vert =\left\vert \sum_{\ell
=k}^{\infty }\bigtriangleup _{I_{\ell }}^{\sigma }f\left( x\right)
\right\vert =\left\vert \mathbb{E}_{\theta \left( I_{k}\right) }^{\sigma }f-%
\mathbb{E}_{I_{K}}^{\sigma }f\right\vert \lesssim \mathbb{E}%
_{F_{m+1}}^{\sigma }\left\vert f\right\vert \ .
\end{equation*}%
On the other hand, if $\theta \left( I_{k}\right) \in \mathcal{F}$, then $%
I_{k+1}\in \mathcal{C}_{F_{m+1}}^{\sigma }$ and we have%
\begin{equation*}
\left\vert f_{F}\left( x\right) -\bigtriangleup _{\theta \left( I_{k}\right)
}^{\sigma }f\left( x\right) \right\vert =\left\vert \sum_{\ell =k+1}^{\infty
}\bigtriangleup _{I_{\ell }}^{\sigma }f\left( x\right) \right\vert
=\left\vert \mathbb{E}_{I_{k+1}}^{\sigma }f-\mathbb{E}_{I_{K}}^{\sigma
}f\right\vert \lesssim \mathbb{E}_{F_{m+1}}^{\sigma }\left\vert f\right\vert
\ .
\end{equation*}%
Now we write%
\begin{eqnarray*}
f_{F} &=&\varphi _{F}+\psi _{F}, \\
\varphi _{F} &\equiv &\sum_{k,\theta :\ \theta \left( I_{k}\right) \in 
\mathcal{F}}^{\infty }\mathbf{1}_{\theta \left( I_{k}\right) }\bigtriangleup
_{I_{k}}^{\sigma }f\text{ and }\psi _{F}=f_{F}-\varphi _{F}\ ; \\
\sum_{F\in \mathcal{F}}\ \left\langle T_{\sigma }^{\alpha
}f_{F},g_{F}\right\rangle _{\omega } &=&\sum_{F\in \mathcal{F}}\
\left\langle T_{\sigma }^{\alpha }\varphi _{F},g_{F}\right\rangle _{\omega
}+\sum_{F\in \mathcal{F}}\ \left\langle T_{\sigma }^{\alpha }\psi
_{F},g_{F}\right\rangle _{\omega }\ ,
\end{eqnarray*}%
and note that both $\varphi _{F}$ and $\psi _{F}$ are constant on $F$. We
can apply (\ref{indicator far}) using $\theta \left( I_{k}\right) \in 
\mathcal{F}$ to the first sum here to obtain%
\begin{eqnarray*}
\left\vert \sum_{F\in \mathcal{F}}\ \left\langle T_{\sigma }^{\alpha
}\varphi _{F},g_{F}\right\rangle _{\omega }\right\vert &\lesssim &\mathcal{%
NTV}_{\alpha }\ \left\Vert \sum_{F\in \mathcal{F}}\varphi _{F}\right\Vert
_{L^{2}\left( \sigma \right) }\left\Vert \sum_{F\in \mathcal{F}%
}g_{F}\right\Vert _{L^{2}\left( \omega \right) }^{2} \\
&\lesssim &\mathcal{NTV}_{\alpha }\ \left\Vert f\right\Vert _{L^{2}\left(
\sigma \right) }\left[ \sum_{F\in \mathcal{F}}\left\Vert g_{F}\right\Vert
_{L^{2}\left( \omega \right) }^{2}\right] ^{\frac{1}{2}}.
\end{eqnarray*}

Turning to the second sum we note that%
\begin{eqnarray*}
\left\vert \psi _{F}\right\vert &\leq &\sum_{m=0}^{N}\left( \mathbb{E}%
_{F_{m+1}}^{\sigma }\left\vert f\right\vert \right) \ \mathbf{1}%
_{F_{m+1}\setminus F_{m}}=\left( \mathbb{E}_{F}^{\sigma }\left\vert
f\right\vert \right) \ \mathbf{1}_{F}+\sum_{m=0}^{N}\left( \mathbb{E}_{\pi _{%
\mathcal{F}}^{m+1}F}^{\sigma }\left\vert f\right\vert \right) \ \mathbf{1}%
_{\pi _{\mathcal{F}}^{m+1}F\setminus \pi _{\mathcal{F}}^{m}F} \\
&=&\left( \mathbb{E}_{F}^{\sigma }\left\vert f\right\vert \right) \ \mathbf{1%
}_{F}+\sum_{F^{\prime }\in \mathcal{F}:\ F\subset F^{\prime }}\left( \mathbb{%
E}_{\pi _{\mathcal{F}}F^{\prime }}^{\sigma }\left\vert f\right\vert \right)
\ \mathbf{1}_{\pi _{\mathcal{F}}F^{\prime }\setminus F^{\prime }} \\
&\leq &\alpha _{\mathcal{F}}\left( F\right) \ \mathbf{1}_{F}+\sum_{F^{\prime
}\in \mathcal{F}:\ F\subset F^{\prime }}\alpha _{\mathcal{F}}\left( \pi _{%
\mathcal{F}}F^{\prime }\right) \ \mathbf{1}_{\pi _{\mathcal{F}}F^{\prime
}\setminus F^{\prime }} \\
&\leq &\alpha _{\mathcal{F}}\left( F\right) \ \mathbf{1}_{F}+\sum_{F^{\prime
}\in \mathcal{F}:\ F\subset F^{\prime }}\alpha _{\mathcal{F}}\left( \pi _{%
\mathcal{F}}F^{\prime }\right) \ \mathbf{1}_{\pi _{\mathcal{F}}F^{\prime }}\ 
\mathbf{1}_{F^{c}} \\
&=&\alpha _{\mathcal{F}}\left( F\right) \ \mathbf{1}_{F}+\Phi \ \mathbf{1}%
_{F^{c}}\ ,\ \ \ \ \ \text{\ for all }F\in \mathcal{F},
\end{eqnarray*}%
where%
\begin{equation*}
\Phi \equiv \sum_{F^{\prime \prime }\in \mathcal{F}}\ \alpha _{\mathcal{F}%
}\left( F^{\prime \prime }\right) \ \mathbf{1}_{F^{\prime \prime }}\ .
\end{equation*}

Now we write%
\begin{equation*}
\sum_{F\in \mathcal{F}}\ \left\langle T_{\sigma }^{\alpha }\psi
_{F},g_{F}\right\rangle _{\omega }=\sum_{F\in \mathcal{F}}\ \left\langle
T_{\sigma }^{\alpha }\left( \mathbf{1}_{F}\psi _{F}\right)
,g_{F}\right\rangle _{\omega }+\sum_{F\in \mathcal{F}}\ \left\langle
T_{\sigma }^{\alpha }\left( \mathbf{1}_{F^{c}}\psi _{F}\right)
,g_{F}\right\rangle _{\omega }\equiv I+II.
\end{equation*}%
Then quasicube testing $\left\vert \left\langle T_{\sigma }^{\alpha }\mathbf{%
1}_{F},g_{F}\right\rangle _{\omega }\right\vert =\left\vert \left\langle 
\mathbf{1}_{F}T_{\sigma }^{\alpha }\mathbf{1}_{F},g_{F}\right\rangle
_{\omega }\right\vert \leq \mathfrak{T}_{T^{\alpha }}\sqrt{\left\vert
F\right\vert _{\sigma }}\left\Vert g_{F}\right\Vert _{L^{2}\left( \omega
\right) }$ and `quasi' orthogonality, together with the fact that $\psi _{F}$
is a constant on $F$ bounded by $\alpha _{\mathcal{F}}\left( F\right) $, give%
\begin{eqnarray*}
&&\left\vert I\right\vert \leq \sum_{F\in \mathcal{F}}\ \left\vert
\left\langle T_{\sigma }^{\alpha }\mathbf{1}_{F}\psi _{F},g_{F}\right\rangle
_{\omega }\right\vert \lesssim \sum_{F\in \mathcal{F}}\ \alpha _{\mathcal{F}%
}\left( F\right) \ \left\vert \left\langle T_{\sigma }^{\alpha }\mathbf{1}%
_{F},g_{F}\right\rangle _{\omega }\right\vert \\
&\lesssim &\sum_{F\in \mathcal{F}}\ \alpha _{\mathcal{F}}\left( F\right) 
\mathcal{NTV}_{\alpha }\sqrt{\left\vert F\right\vert _{\sigma }}\left\Vert
g_{F}\right\Vert _{L^{2}\left( \omega \right) }\lesssim \mathcal{NTV}%
_{\alpha }\left\Vert f\right\Vert _{L^{2}\left( \sigma \right) }\left[
\sum_{F\in \mathcal{F}}\left\Vert g_{F}\right\Vert _{L^{2}\left( \omega
\right) }^{2}\right] ^{\frac{1}{2}}.
\end{eqnarray*}%
Now $\mathbf{1}_{F^{c}}\psi _{F}$ is supported outside $F$, and each $J$ in
the quasiHaar support of $g_{F}$ is $\mathbf{r}$-deeply embedded in $F$,
i.e. $J\Subset _{\mathbf{r}}F$. Thus we can apply the Energy Lemma \ref{ener}
to obtain%
\begin{eqnarray*}
\left\vert II\right\vert &=&\left\vert \sum_{F\in \mathcal{F}}\left\langle
T_{\sigma }^{\alpha }\left( \mathbf{1}_{F^{c}}\psi _{F}\right)
,g_{F}\right\rangle _{\omega }\right\vert \\
&\lesssim &\sum_{F\in \mathcal{F}}\sum_{J\in \mathcal{M}_{\mathbf{r}-%
\limfunc{deep}}\left( F\right) }\frac{\mathrm{P}^{\alpha }\left( J,\mathbf{1}%
_{F^{c}}\Phi \sigma \right) }{\left\vert J\right\vert ^{\frac{1}{n}}}%
\left\Vert \mathsf{P}_{\mathcal{C}_{F}^{\limfunc{good},\mathbf{\tau }-%
\limfunc{shift}};J}^{\omega }\mathbf{x}\right\Vert _{L^{2}\left( \omega
\right) }\left\Vert \mathsf{P}_{J}^{\omega }g_{F}\right\Vert _{L^{2}\left(
\omega \right) } \\
&&+\sum_{F\in \mathcal{F}}\sum_{J\in \mathcal{M}_{\mathbf{r}-\limfunc{deep}%
}\left( F\right) }\frac{\mathrm{P}_{1+\delta ^{\prime }}^{\alpha }\left( J,%
\mathbf{1}_{F^{c}}\Phi \sigma \right) }{\left\vert J\right\vert ^{\frac{1}{n}%
}}\left\Vert \mathsf{P}_{\left( \mathcal{C}_{F}^{\limfunc{good},\mathbf{\tau 
}-\limfunc{shift}}\right) ^{\ast };J}^{\omega }\mathbf{x}\right\Vert
_{L^{2}\left( \omega \right) }\left\Vert \mathsf{P}_{J}^{\omega
}g_{F}\right\Vert _{L^{2}\left( \omega \right) } \\
&\equiv &II_{G}+II_{B}\ .
\end{eqnarray*}

Then\ from Cauchy-Schwarz, the functional quasienergy condition, and $%
\left\Vert \Phi \right\Vert _{L^{2}\left( \sigma \right) }\lesssim
\left\Vert f\right\Vert _{L^{2}\left( \sigma \right) }$ we obtain%
\begin{eqnarray*}
\left\vert II_{G}\right\vert &\leq &\left( \sum_{F\in \mathcal{F}}\sum_{J\in 
\mathcal{M}_{\mathbf{r}-\limfunc{deep}}\left( F\right) }\left( \frac{\mathrm{%
P}^{\alpha }\left( J,\mathbf{1}_{F^{c}}\Phi \sigma \right) }{\left\vert
J\right\vert ^{\frac{1}{n}}}\right) ^{2}\left\Vert \mathsf{P}_{\mathcal{C}%
_{F}^{\limfunc{good},\mathbf{\tau }-\limfunc{shift}};J}^{\omega }\mathbf{x}%
\right\Vert _{L^{2}\left( \omega \right) }^{2}\right) ^{\frac{1}{2}} \\
&&\times \left( \sum_{F\in \mathcal{F}}\sum_{J\in \mathcal{M}_{\mathbf{r}-%
\limfunc{deep}}\left( F\right) }\left\Vert \mathsf{P}_{J}^{\omega
}g_{F}\right\Vert _{L^{2}\left( \omega \right) }^{2}\right) ^{\frac{1}{2}} \\
&\lesssim &\mathfrak{F}_{\alpha }\left\Vert \Phi \right\Vert _{L^{2}\left(
\sigma \right) }\left[ \sum_{F\in \mathcal{F}}\left\Vert g_{F}\right\Vert
_{L^{2}\left( \omega \right) }^{2}\right] ^{\frac{1}{2}}\lesssim \mathfrak{F}%
_{\alpha }\left\Vert f\right\Vert _{L^{2}\left( \sigma \right) }\left\Vert
g\right\Vert _{L^{2}\left( \omega \right) },
\end{eqnarray*}%
by the bounded overlap by $\mathbf{\tau }$ of the shifted coronas $\mathcal{C%
}_{F}^{\limfunc{good},\mathbf{\tau }-\limfunc{shift}}$.

In term $II_{B}$ the projections $\mathsf{P}_{\left( \mathcal{C}_{F}^{%
\limfunc{good},\mathbf{\tau }-\limfunc{shift}}\right) ^{\ast };J}^{\omega }$
are no longer almost orthogonal, and we must instead exploit the decay in
the Poisson integral $\mathrm{P}_{1+\delta ^{\prime }}^{\alpha }$ along with
goodness of the quasicubes $J$. This idea was already used by M. Lacey and
B. Wick in \cite{LaWi} in a similar situation. As a consequence of this
decay we will be able to bound $II_{B}$ \emph{directly} by the quasienergy
condition, without having to invoke the more difficult functional
quasienergy condition. For the decay we compute%
\begin{eqnarray*}
\frac{\mathrm{P}_{1+\delta ^{\prime }}^{\alpha }\left( J,\Phi \sigma \right) 
}{\left\vert J\right\vert ^{\frac{1}{n}}} &=&\int_{F^{c}}\frac{\left\vert
J\right\vert ^{\frac{\delta ^{\prime }}{n}}}{\left\vert y-c_{J}\right\vert
^{n+1+\delta -\alpha }}\Phi \left( y\right) d\sigma \left( y\right) \\
&\leq &\sum_{t=0}^{\infty }\int_{\pi _{\mathcal{F}}^{t+1}F\setminus \pi _{%
\mathcal{F}}^{t}F}\left( \frac{\left\vert J\right\vert ^{\frac{1}{n}}}{%
\limfunc{dist}\left( c_{J},\left( \pi _{\mathcal{F}}^{t}F\right) ^{c}\right) 
}\right) ^{\delta ^{\prime }}\frac{1}{\left\vert y-c_{J}\right\vert
^{n+1-\alpha }}\Phi \left( y\right) d\sigma \left( y\right) \\
&\leq &\sum_{t=0}^{\infty }\left( \frac{\left\vert J\right\vert ^{\frac{1}{n}%
}}{\limfunc{dist}\left( c_{J},\left( \pi _{\mathcal{F}}^{t}F\right)
^{c}\right) }\right) ^{\delta ^{\prime }}\frac{\mathrm{P}^{\alpha }\left( J,%
\mathbf{1}_{\pi _{\mathcal{F}}^{t+1}F\setminus \pi _{\mathcal{F}}^{t}F}\Phi
\sigma \right) }{\left\vert J\right\vert ^{\frac{1}{n}}},
\end{eqnarray*}%
and then use the goodness inequality%
\begin{equation*}
\limfunc{dist}\left( c_{J},\left( \pi _{\mathcal{F}}^{t}F\right) ^{c}\right)
\geq \frac{1}{2}\ell \left( \pi _{\mathcal{F}}^{t}F\right) ^{1-\varepsilon
}\ell \left( J\right) ^{\varepsilon }\geq \frac{1}{2}2^{t\left(
1-\varepsilon \right) }\ell \left( F\right) ^{1-\varepsilon }\ell \left(
J\right) ^{\varepsilon }\geq 2^{t\left( 1-\varepsilon \right) -1}\ell \left(
J\right) ,
\end{equation*}%
to conclude that%
\begin{eqnarray}
\left( \frac{\mathrm{P}_{1+\delta ^{\prime }}^{\alpha }\left( J,\mathbf{1}%
_{F^{c}}\Phi \sigma \right) }{\left\vert J\right\vert ^{\frac{1}{n}}}\right)
^{2} &\lesssim &\left( \sum_{t=0}^{\infty }2^{-t\delta ^{\prime }\left(
1-\varepsilon \right) }\frac{\mathrm{P}^{\alpha }\left( J,\mathbf{1}_{\pi _{%
\mathcal{F}}^{t+1}F\setminus \pi _{\mathcal{F}}^{t}F}\Phi \sigma \right) }{%
\left\vert J\right\vert ^{\frac{1}{n}}}\right) ^{2}  \label{decay in t} \\
&\lesssim &\sum_{t=0}^{\infty }2^{-t\delta ^{\prime }\left( 1-\varepsilon
\right) }\left( \frac{\mathrm{P}^{\alpha }\left( J,\mathbf{1}_{\pi _{%
\mathcal{F}}^{t+1}F\setminus \pi _{\mathcal{F}}^{t}F}\Phi \sigma \right) }{%
\left\vert J\right\vert ^{\frac{1}{n}}}\right) ^{2}.  \notag
\end{eqnarray}%
Now we apply Cauchy-Schwarz to obtain%
\begin{eqnarray*}
II_{B} &=&\sum_{F\in \mathcal{F}}\sum_{J\in \mathcal{M}_{\mathbf{r}-\limfunc{%
deep}}\left( F\right) }\frac{\mathrm{P}_{1+\delta ^{\prime }}^{\alpha
}\left( J,\mathbf{1}_{F^{c}}\Phi \sigma \right) }{\left\vert J\right\vert ^{%
\frac{1}{n}}}\left\Vert \mathsf{P}_{\left( \mathcal{C}_{F}^{\limfunc{good},%
\mathbf{\tau }-\limfunc{shift}}\right) ^{\ast };J}^{\omega }\mathbf{x}%
\right\Vert _{L^{2}\left( \omega \right) }\left\Vert \mathsf{P}_{J}^{\omega
}g_{F}\right\Vert _{L^{2}\left( \omega \right) } \\
&\leq &\left( \sum_{F\in \mathcal{F}}\sum_{J\in \mathcal{M}_{\mathbf{r}-%
\limfunc{deep}}\left( F\right) }\left( \frac{\mathrm{P}_{1+\delta ^{\prime
}}^{\alpha }\left( J,\mathbf{1}_{F^{c}}\Phi \sigma \right) }{\left\vert
J\right\vert ^{\frac{1}{n}}}\right) ^{2}\left\Vert \mathsf{P}_{\left( 
\mathcal{C}_{F}^{\limfunc{good},\mathbf{\tau }-\limfunc{shift}}\right)
^{\ast };J}^{\omega }\mathbf{x}\right\Vert _{L^{2}\left( \omega \right)
}^{2}\right) ^{\frac{1}{2}}\left[ \sum_{F}\left\Vert g_{F}\right\Vert
_{L^{2}\left( \omega \right) }^{2}\right] ^{\frac{1}{2}} \\
&\equiv &\sqrt{II_{\limfunc{energy}}}\left[ \sum_{F}\left\Vert
g_{F}\right\Vert _{L^{2}\left( \omega \right) }^{2}\right] ^{\frac{1}{2}},
\end{eqnarray*}%
and it remains to estimate $II_{\limfunc{energy}}$. From (\ref{decay in t})
and the plugged deep quasienergy condition we have%
\begin{eqnarray*}
&&II_{\limfunc{energy}} \\
&\leq &\sum_{F\in \mathcal{F}}\sum_{J\in \mathcal{M}_{\mathbf{r}-\limfunc{%
deep}}\left( F\right) }\sum_{t=0}^{\infty }2^{-t\delta ^{\prime }\left(
1-\varepsilon \right) }\left( \frac{\mathrm{P}^{\alpha }\left( J,\mathbf{1}%
_{\pi _{\mathcal{F}}^{t+1}F\setminus \pi _{\mathcal{F}}^{t}F}\Phi \sigma
\right) }{\left\vert J\right\vert ^{\frac{1}{n}}}\right) ^{2}\left\Vert 
\mathsf{P}_{\left( \mathcal{C}_{F}^{\limfunc{good},\mathbf{\tau }-\limfunc{%
shift}}\right) ^{\ast };J}^{\omega }\mathbf{x}\right\Vert _{L^{2}\left(
\omega \right) }^{2} \\
&=&\sum_{t=0}^{\infty }2^{-t\delta ^{\prime }\left( 1-\varepsilon \right)
}\sum_{G\in \mathcal{F}}\sum_{F\in \mathfrak{C}_{\mathcal{F}}^{\left(
t+1\right) }\left( G\right) }\sum_{J\in \mathcal{M}_{\mathbf{r}-\limfunc{deep%
}}\left( F\right) }\left( \frac{\mathrm{P}^{\alpha }\left( J,\mathbf{1}%
_{G\setminus \pi _{\mathcal{F}}^{t}F}\Phi \sigma \right) }{\left\vert
J\right\vert ^{\frac{1}{n}}}\right) ^{2}\left\Vert \mathsf{P}_{\left( 
\mathcal{C}_{F}^{\limfunc{good},\mathbf{\tau }-\limfunc{shift}}\right)
^{\ast };J}^{\omega }\mathbf{x}\right\Vert _{L^{2}\left( \omega \right) }^{2}
\\
&\lesssim &\sum_{t=0}^{\infty }2^{-t\delta ^{\prime }\left( 1-\varepsilon
\right) }\sum_{G\in \mathcal{F}}\alpha _{\mathcal{F}}\left( G\right)
^{2}\sum_{F\in \mathfrak{C}_{\mathcal{F}}^{\left( t+1\right) }\left(
G\right) }\sum_{J\in \mathcal{M}_{\mathbf{r}-\limfunc{deep}}\left( F\right)
}\left( \frac{\mathrm{P}^{\alpha }\left( J,\mathbf{1}_{G\setminus \pi _{%
\mathcal{F}}^{t}F}\sigma \right) }{\left\vert J\right\vert ^{\frac{1}{n}}}%
\right) ^{2}\left\Vert \mathsf{P}_{J}^{\limfunc{subgood},\omega }\mathbf{x}%
\right\Vert _{L^{2}\left( \omega \right) }^{2} \\
&\lesssim &\sum_{t=0}^{\infty }2^{-t\delta ^{\prime }\left( 1-\varepsilon
\right) }\sum_{G\in \mathcal{F}}\alpha _{\mathcal{F}}\left( G\right)
^{2}\left( \mathcal{E}_{\alpha }^{2}+A_{2}^{\alpha }\right) \left\vert
G\right\vert _{\sigma }\lesssim \left( \mathcal{E}_{\alpha
}^{2}+A_{2}^{\alpha }\right) \left\Vert f\right\Vert _{L^{2}\left( \sigma
\right) }^{2}.
\end{eqnarray*}

This completes the proof of the Intertwining Proposition \ref{strongly
adapted}.
\end{proof}

\section{Control of functional energy by energy modulo $\mathcal{A}_{2}^{%
\protect\alpha }\label{equiv}$}

Now we show that the functional quasienergy constants $\mathfrak{F}_{\alpha
} $ are controlled by $\mathcal{A}_{2}^{\alpha }$ and both the \emph{deep}
and \emph{refined} quasienergy constants $\mathcal{E}_{\alpha }^{\limfunc{%
deep}}$ and $\mathcal{E}_{\alpha }^{\limfunc{refined}}$\ defined in
Definition \ref{energy condition}. Recall $\left( \mathcal{E}_{\alpha
}\right) ^{2}=\left( \mathcal{E}_{\alpha }^{\limfunc{deep}}\right)
^{2}+\left( \mathcal{E}_{\alpha }^{\limfunc{refined}}\right) ^{2}$.

\begin{proposition}
\label{func ener control}%
\begin{equation*}
\mathfrak{F}_{\alpha }\lesssim \mathcal{E}_{\alpha }+\sqrt{\mathcal{A}%
_{2}^{\alpha }}+\sqrt{\mathcal{A}_{2}^{\alpha ,\ast }}\text{ and }\mathfrak{F%
}_{\alpha }^{\ast }\lesssim \mathcal{E}_{\alpha }^{\ast }+\sqrt{\mathcal{A}%
_{2}^{\alpha }}+\sqrt{\mathcal{A}_{2}^{\alpha ,\ast }}\ .
\end{equation*}
\end{proposition}

To prove this proposition, we fix $\mathcal{F}$ as in (\ref{e.funcEnergy n})
and set 
\begin{equation}
\mu \equiv \sum_{F\in \mathcal{F}}\sum_{J\in \mathcal{M}_{\mathbf{r}-%
\limfunc{deep}}\left( F\right) }\left\Vert \mathsf{P}_{F,J}^{\omega }\frac{x%
}{\left\vert J\right\vert ^{\frac{1}{n}}}\right\Vert _{L^{2}\left( \omega
\right) }^{2}\cdot \delta _{\left( c\left( J\right) ,\ell \left( J\right)
\right) }\ ,  \label{def mu n}
\end{equation}%
where $\mathcal{M}_{\mathbf{r}-\limfunc{deep}}\left( F\right) $ consists of
the maximal $\mathbf{r}$-deeply embedded subquasicubes of $F$, and where $%
\delta _{\left( c\left( J\right) ,\ell \left( J\right) \right) }$ denotes
the Dirac unit mass at the point $\left( c\left( J\right) ,\ell \left(
J\right) \right) $ in the upper half-space $\mathbb{R}_{+}^{n+1}$. Here $J$
is a dyadic quasicube with center $c\left( J\right) $ and side length $\ell
\left( J\right) $. For convenience in notation, we denote for any dyadic
quasicube $J$ the localized projection $\mathsf{P}_{\mathcal{C}_{F}^{%
\limfunc{good},\mathbf{\tau }-\limfunc{shift}};J}^{\omega }$ given in (\ref%
{def localization}) by%
\begin{equation*}
\mathsf{P}_{F,J}^{\omega }\equiv \mathsf{P}_{\mathcal{C}_{F}^{\limfunc{good},%
\mathbf{\tau }-\limfunc{shift}};J}^{\omega }=\sum_{J^{\prime }\subset J:\
J^{\prime }\in \mathcal{C}_{F}^{\limfunc{good},\mathbf{\tau }-\limfunc{shift}%
}}\bigtriangleup _{J^{\prime }}^{\omega }.
\end{equation*}%
We emphasize that the quasicubes $J\in \mathcal{M}_{\mathbf{r}-\limfunc{deep}%
}\left( F\right) $ are not necessarily good, but that the subquasicubes $%
J^{\prime }\subset J$ arising in the projection $\mathsf{P}_{F,J}^{\omega }$
are good. We can replace $x$ by $x-c$ inside the projection for any choice
of $c$ we wish; the projection is unchanged. More generally, $\delta _{q}$
denotes a Dirac unit mass at a point $q$ in the upper half-space $\mathbb{R}%
_{+}^{n+1}$.

We prove the two-weight inequality 
\begin{equation}
\lVert \mathbb{P}^{\alpha }(f\sigma )\rVert _{L^{2}(\mathbb{R}_{+}^{n+1},\mu
)}\lesssim \left( \mathcal{E}_{\alpha }+\sqrt{\mathcal{A}_{2}^{\alpha }}+%
\sqrt{\mathcal{A}_{2}^{\alpha ,\ast }}\right) \lVert f\rVert _{L^{2}\left(
\sigma \right) }\,,  \label{two weight Poisson n}
\end{equation}%
for all nonnegative $f$ in $L^{2}\left( \sigma \right) $, noting that $%
\mathcal{F}$ and $f$ are \emph{not} related here. Above, $\mathbb{P}^{\alpha
}(\cdot )$ denotes the $\alpha $-fractional Poisson extension to the upper
half-space $\mathbb{R}_{+}^{n+1}$,

\begin{equation*}
\mathbb{P}^{\alpha }\nu \left( x,t\right) \equiv \int_{\mathbb{R}^{n}}\frac{t%
}{\left( t^{2}+\left\vert x-y\right\vert ^{2}\right) ^{\frac{n+1-\alpha }{2}}%
}d\nu \left( y\right) ,
\end{equation*}%
so that in particular 
\begin{equation*}
\left\Vert \mathbb{P}^{\alpha }(f\sigma )\right\Vert _{L^{2}(\mathbb{R}%
_{+}^{n+1},\mu )}^{2}=\sum_{F\in \mathcal{F}}\sum_{J\in \mathcal{M}_{\mathbf{%
r}-\limfunc{deep}}\left( F\right) }\mathbb{P}^{\alpha }\left( f\sigma
\right) (c(J),\ell \left( J\right) )^{2}\left\Vert \mathsf{P}_{F,J}^{\omega }%
\frac{x}{\left\vert J\right\vert ^{\frac{1}{n}}}\right\Vert _{L^{2}\left(
\omega \right) }^{2}\,,
\end{equation*}%
and so (\ref{two weight Poisson n}) proves the first line in Proposition \ref%
{func ener control} upon inspecting (\ref{e.funcEnergy n}).

The two-weight inequality for fractional integrals in \cite{Saw3} was
extended to homogeneous spaces in \cite{SaWh}, and the same arguments,
involving only positive operators and their estimates, extend the two-weight
inequality for Poisson integrals to the upper half-space $\mathbb{R}%
_{+}^{n+1}$ where $\mathbb{R}^{n}$ is now given the homogeneous space
product structure corresponding to the quasicubes $\times $ intervals used
here. Using this extended theorem for the two-weight Poisson inequality, we
see that inequality (\ref{two weight Poisson n}) requires checking these two
inequalities for dyadic quasicubes $I\in \Omega \mathcal{D}$ and quasiboxes $%
\widehat{I}=I\times \left[ 0,\ell \left( I\right) \right) $ in the upper
half-space$\mathbb{R}_{+}^{n+1}$: 
\begin{equation}
\int_{\mathbb{R}_{+}^{n+1}}\mathbb{P}^{\alpha }\left( \mathbf{1}_{I}\sigma
\right) \left( x,t\right) ^{2}d\mu \left( x,t\right) \equiv \left\Vert 
\mathbb{P}^{\alpha }\left( \mathbf{1}_{I}\sigma \right) \right\Vert _{L^{2}(%
\widehat{I},\mu )}^{2}\lesssim \left( \mathcal{A}_{2}^{\alpha ,\ast }+%
\mathcal{E}_{\alpha }^{2}\right) \sigma (I)\,,  \label{e.t1 n}
\end{equation}%
\begin{equation}
\int_{\mathbb{R}}[\mathbb{P}^{\alpha \ast }(t\mathbf{1}_{\widehat{I}}\mu
)]^{2}d\sigma (x)\lesssim \left( \mathcal{A}_{2}^{\alpha }+\mathcal{E}%
_{\alpha }\sqrt{\mathcal{A}_{2}^{\alpha }}\right) \int_{\widehat{I}%
}t^{2}d\mu (x,t),  \label{e.t2 n}
\end{equation}%
for all \emph{dyadic} quasicubes $I\in \Omega \mathcal{D}$, and where the
dual Poisson operator is given by 
\begin{equation*}
\mathbb{P}^{\alpha \ast }(t\mathbf{1}_{\widehat{I}}\mu )\left( x\right)
=\int_{\widehat{I}}\frac{t^{2}}{\left( t^{2}+\lvert x-y\rvert ^{2}\right) ^{%
\frac{n+1-\alpha }{2}}}d\mu \left( y,t\right) \,.
\end{equation*}%
It is important to note that we can choose for $\Omega \mathcal{D}$ any
fixed dyadic quasigrid, the compensating point being that the integrations
on the left sides of (\ref{e.t1 n}) and (\ref{e.t2 n}) are taken over the
entire spaces $\mathbb{R}_{+}^{n}$ and $\mathbb{R}^{n}$ respectively.

\begin{remark}
There is a gap in the proof of the Poisson inequality at the top of page 542
in \cite{Saw3}. However, this gap can be fixed as in \cite{SaWh} or \cite%
{LaSaUr1}.
\end{remark}

The following elementary Poisson inequalities will be used extensively.

\begin{lemma}
\label{Poisson inequalities}Suppose that $J,K,I$ are quasicubes satisfying $%
J\subset K\subset 2K\subset I$, and that $\mu $ is a positive measure
supported in $\mathbb{R}^{n}\setminus I$. Then%
\begin{equation*}
\frac{\mathrm{P}^{\alpha }\left( J,\mu \right) }{\left\vert J\right\vert ^{%
\frac{1}{n}}}\lesssim \frac{\mathrm{P}^{\alpha }\left( K,\mu \right) }{%
\left\vert K\right\vert ^{\frac{1}{n}}}\lesssim \frac{\mathrm{P}^{\alpha
}\left( J,\mu \right) }{\left\vert J\right\vert ^{\frac{1}{n}}}.
\end{equation*}
\end{lemma}

\begin{proof}
We have%
\begin{equation*}
\frac{\mathrm{P}^{\alpha }\left( J,\mu \right) }{\left\vert J\right\vert ^{%
\frac{1}{n}}}=\frac{1}{\left\vert J\right\vert ^{\frac{1}{n}}}\int \frac{%
\left\vert J\right\vert ^{\frac{1}{n}}}{\left( \left\vert J\right\vert ^{%
\frac{1}{n}}+\left\vert x-c_{J}\right\vert \right) ^{n+1-\alpha }}d\mu
\left( x\right) ,
\end{equation*}%
where $J\subset K\subset 2K\subset I$ implies that%
\begin{equation*}
\left\vert J\right\vert ^{\frac{1}{n}}+\left\vert x-c_{J}\right\vert \approx
\left\vert K\right\vert ^{\frac{1}{n}}+\left\vert x-c_{K}\right\vert ,\ \ \
\ \ x\in \mathbb{R}^{n}\setminus I.
\end{equation*}
\end{proof}

Now we record the bounded overlap of the projections $\mathsf{P}%
_{F,J}^{\omega }$.

\begin{lemma}
\label{tau ovelap}Suppose $\mathsf{P}_{F,J}^{\omega }$ is as above and fix
any $I_{0}\in \Omega \mathcal{D}$. If $J\in \mathcal{M}_{\mathbf{r}-\limfunc{%
deep}}\left( F\right) $ for some $F\in \mathcal{F}$ with $F\supsetneqq I_{0}$
and $\mathsf{P}_{F,J}^{\omega }\neq 0$, then 
\begin{equation*}
F=\pi _{\mathcal{F}}^{\left( \ell \right) }I_{0}\text{ for some }0\leq \ell
\leq \mathbf{\tau }.
\end{equation*}%
As a consequence we have the bounded overlap,%
\begin{equation*}
\#\left\{ F\in \mathcal{F}:J\subset I_{0}\subsetneqq F\text{ for some }J\in 
\mathcal{M}_{\mathbf{r}-\limfunc{deep}}\left( F\right) \text{ with }\mathsf{P%
}_{F,J}^{\omega }\neq 0\right\} \leq \mathbf{\tau }.
\end{equation*}
\end{lemma}

\begin{proof}
Indeed, if $J^{\prime }\in \mathcal{C}_{\pi _{\mathcal{F}}^{\left( \ell
\right) }I_{0}}^{\limfunc{good},\mathbf{\tau }-\limfunc{shift}}$ for some $%
\ell >\mathbf{\tau }$, then either $J^{\prime }\cap \pi _{\mathcal{F}%
}^{\left( 0\right) }I_{0}=\emptyset $ or $J^{\prime }\supset \pi _{\mathcal{F%
}}^{\left( 0\right) }I_{0}$. Since $J\subset I_{0}\subset \pi _{\mathcal{F}%
}^{\left( 0\right) }I_{0}$, we cannot have $J^{\prime }$ contained in $J$,
and this shows that $\mathsf{P}_{\pi _{\mathcal{F}}^{\left( \ell \right)
}I_{0},J}^{\omega }=0$.
\end{proof}

Finally we record the only place in the proof where the refined quasienergy
condition is used. This lemma will be used in bounding both of the Poisson
testing conditions. Recall that $\mathcal{S}\Omega \mathcal{D}$ consists of
all shifted $\Omega \mathcal{D}$-dyadic quasicubes where $K$ is shifted if
it is a union of $2^{n}$ $\Omega \mathcal{D}$-dyadic quasicubes $K^{\prime }$
with $\ell \left( K^{\prime }\right) =\frac{1}{2}\ell \left( K\right) $.

\begin{lemma}
\label{refined lemma}Let $\Omega \mathcal{D},\mathcal{F}$ and $\left\{ 
\mathsf{P}_{F,J}^{\omega }\right\} _{\substack{ F\in \mathcal{F}  \\ J\in 
\mathcal{M}_{\mathbf{r}-\limfunc{deep}}\left( F\right) }}$ be as above with $%
J,F$ in the dyadic quasigrid $\Omega \mathcal{D}$. For any shifted quasicube 
$I_{0}\in \mathcal{S}\Omega \mathcal{D}$ define%
\begin{equation}
B\left( I_{0}\right) \equiv \sum_{F\in \mathcal{F}:\ F\supsetneqq
I_{0}^{\prime }\text{ for some }I_{0}^{\prime }\in \mathfrak{C}\left(
I_{0}\right) }\sum_{J\in \mathcal{M}_{\mathbf{r}-\limfunc{deep}}\left(
F\right) :\ J\subset I_{0}}\left( \frac{\mathrm{P}^{\alpha }\left( J,\mathbf{%
1}_{I_{0}}\sigma \right) }{\left\vert J\right\vert ^{\frac{1}{n}}}\right)
^{2}\left\Vert \mathsf{P}_{F,J}^{\omega }\mathbf{x}\right\Vert _{L^{2}\left(
\omega \right) }^{2}\ .  \label{term B}
\end{equation}%
Then%
\begin{equation}
B\left( I_{0}\right) \lesssim \mathbf{\tau }\left( \left( \mathcal{E}%
_{\alpha }^{\limfunc{refined}\limfunc{plug}}\right) ^{2}+\left( \mathcal{E}%
_{\alpha }^{\limfunc{deep}\limfunc{plug}}\right) ^{2}\right) \left\vert
I_{0}\right\vert _{\sigma }\lesssim \mathbf{\tau }\left( \left( \mathcal{E}%
_{\alpha }\right) ^{2}+\beta A_{2}^{\alpha }\right) \left\vert
I_{0}\right\vert _{\sigma }\ .  \label{B bound}
\end{equation}
\end{lemma}

\begin{proof}
Suppose first that $I_{0}$ is $\Omega \mathcal{D}$-dyadic. Define 
\begin{equation*}
\Lambda \left( I_{0}\right) \equiv \left\{ J\subset I_{0}:J\in \mathcal{M}_{%
\mathbf{r}-\limfunc{deep}}\left( F\right) \text{ for some }F\supsetneqq I_{0}%
\text{ with }\mathsf{P}_{F,J}^{\omega }\neq 0\right\} .
\end{equation*}%
By Lemma \ref{tau ovelap} we may pigeonhole the quasicubes $J$ in $\Lambda
\left( I_{0}\right) $ as follows:%
\begin{equation*}
\Lambda \left( I_{0}\right) =\dbigcup\limits_{\ell =0}^{\mathbf{\tau }%
}\Lambda _{\ell }\left( I_{0}\right) ;\ \ \ \ \ \Lambda _{\ell }\left(
I_{0}\right) \equiv \left\{ J\subset I_{0}:J\in \mathcal{M}_{\mathbf{r}-%
\limfunc{deep}}\left( \pi _{\mathcal{F}}^{\left( \ell \right) }I_{0}\right)
\right\} .
\end{equation*}%
Now fix $\ell $, and for each $J$ in the pairwise disjoint decomposition $%
\Lambda _{\ell }\left( I_{0}\right) $ of $I_{0}$, note that \emph{either} $J$
must contain some $K\in \mathcal{M}_{\mathbf{r}-\limfunc{deep}}\left(
I_{0}\right) $ or $J\subset K$ for some $K\in \mathcal{M}_{\mathbf{r}-%
\limfunc{deep}}\left( I_{0}\right) $;%
\begin{eqnarray*}
\Lambda _{\ell }\left( I_{0}\right) &=&\Lambda _{\ell }^{\limfunc{big}%
}\left( I_{0}\right) \cup \Lambda _{\ell }^{\limfunc{small}}\left(
I_{0}\right) ; \\
\Lambda _{\ell }^{\limfunc{small}}\left( I_{0}\right) &\equiv &\left\{ J\in
\Lambda _{\ell }\left( I_{0}\right) :J\subset K\text{ for some }K\in 
\mathcal{M}_{\mathbf{r}-\limfunc{deep}}\left( I_{0}\right) \right\} ,
\end{eqnarray*}%
and we make the corresponding decomposition of $B\left( I_{0}\right) $;%
\begin{eqnarray*}
B\left( I_{0}\right) &=&B^{\limfunc{big}}\left( I_{0}\right) +B^{\limfunc{%
small}}\left( I_{0}\right) ; \\
B^{\limfunc{big}/\limfunc{small}}\left( I_{0}\right) &\equiv
&\dsum\limits_{\ell =0}^{\mathbf{\tau }}\sum_{J\in \Lambda _{\ell }^{%
\limfunc{big}/\limfunc{small}}\left( I_{0}\right) }\left( \frac{\mathrm{P}%
^{\alpha }\left( J,\mathbf{1}_{I_{0}}\sigma \right) }{\left\vert
J\right\vert ^{\frac{1}{n}}}\right) ^{2}\sum_{F\in \mathcal{F}:\
F\supsetneqq I_{0}\text{ and }J\in \mathcal{M}_{\mathbf{r}-\limfunc{deep}%
}\left( F\right) }\left\Vert \mathsf{P}_{F,J}^{\omega }\mathbf{x}\right\Vert
_{L^{2}\left( \omega \right) }^{2}\ .
\end{eqnarray*}

Turning first to $B^{\limfunc{small}}\left( I_{0}\right) $, we use the $%
\mathbf{\tau }$-overlap (\ref{tau overlap}) of the projections $\mathsf{P}%
_{F,J}^{\omega }$, together with Lemma \ref{Poisson inequalities}, to obtain%
\begin{eqnarray}
B^{\limfunc{small}}\left( I_{0}\right) &\leq &\mathbf{\tau }%
\dsum\limits_{\ell =0}^{\mathbf{\tau }}\sum_{J\in \Lambda _{\ell }^{\limfunc{%
small}}\left( I_{0}\right) }\left( \frac{\mathrm{P}^{\alpha }\left( J,%
\mathbf{1}_{I_{0}}\sigma \right) }{\left\vert J\right\vert ^{\frac{1}{n}}}%
\right) ^{2}\left\Vert \mathsf{P}_{J}^{\limfunc{good},\omega }\mathbf{x}%
\right\Vert _{L^{2}\left( \omega \right) }^{2}  \label{B small} \\
&\lesssim &\mathbf{\tau }^{2}\left( \mathcal{E}_{\alpha }^{\limfunc{refined}%
\limfunc{plug}}\right) ^{2}\left\vert I_{0}\right\vert _{\sigma }\lesssim 
\mathbf{\tau }^{2}\left[ \left( \mathcal{E}_{\alpha }\right) ^{2}+\beta
A_{2}^{\alpha }\right] \left\vert I_{0}\right\vert _{\sigma }\ ,  \notag
\end{eqnarray}%
where the final estimate follows from (\ref{plug the hole refined}) since
each quasicube $\pi _{\mathcal{F}}^{\left( \ell \right) }I_{0}$ equals $\pi
_{\Omega \mathcal{D}}^{\left( \ell ^{\prime }\right) }I_{0}$ for some $\ell
^{\prime }$ and it is with this $\ell ^{\prime }$ that we apply the plugged
refined quasienergy condition.

This is the only point in the proof of Theorem \ref{T1 theorem} that the
refined quasienergy condition is used. Note that while we only consider here
the parents $\pi _{\mathcal{F}}^{\left( \ell \right) }I_{0}$ for $0\leq \ell
\leq \mathbf{\tau }$ , these $\mathcal{F}$-parents may occur as $\Omega 
\mathcal{D}$-parents $\pi _{\Omega \mathcal{D}}^{\left( \ell ^{\prime
}\right) }I_{0}$ for arbitrarily large $\ell ^{\prime }$, and this accounts
for why we take the supremum over all $\Omega \mathcal{D}$-parents $\pi
_{\Omega \mathcal{D}}^{\left( \ell \right) }I_{0}$ in the definition of the
refined quasienergy constant $\mathcal{E}_{\alpha }^{\limfunc{refined}}$.

Turning now to the more delicate term $B^{\limfunc{big}}\left( I_{0}\right) $%
, we write for $J\in \Lambda _{\ell }^{\limfunc{big}}\left( I_{0}\right) $,%
\begin{eqnarray*}
\left\Vert \mathsf{P}_{J}^{\limfunc{good},\omega }\mathbf{x}\right\Vert
_{L^{2}\left( \omega \right) }^{2} &=&\sum_{J^{\prime }\subset J\text{:\ }%
J^{\prime }\text{ good}}\left\Vert \bigtriangleup _{J^{\prime }}^{\omega }%
\mathbf{x}\right\Vert _{L^{2}\left( \omega \right) }^{2} \\
&=&\sum_{J^{\prime }\in \mathcal{N}_{\mathbf{r}}\left( I\right) \text{:\ }%
J^{\prime }\subset J}\left\Vert \bigtriangleup _{J^{\prime }}^{\omega }%
\mathbf{x}\right\Vert _{L^{2}\left( \omega \right) }^{2}+\sum_{K\in \mathcal{%
M}_{\mathbf{r}-\limfunc{deep}}\left( I_{0}\right) \text{:\ }K\subset
J}\left\Vert \mathsf{P}_{K}^{\limfunc{good},\omega }\mathbf{x}\right\Vert
_{L^{2}\left( \omega \right) }^{2}\ ,
\end{eqnarray*}%
where $\mathcal{N}_{\mathbf{r}}\left( I\right) \equiv \left\{ J^{\prime
}\subset I:\ell \left( J^{\prime }\right) \geq 2^{-\mathbf{r}}\ell \left(
I\right) \right\} $ denotes the set of $r$-near quasicubes in $I$, and then
using the $\mathbf{\tau }$-overlap (\ref{tau overlap}) of the projections $%
\mathsf{P}_{F,J}^{\omega }$, we estimate%
\begin{eqnarray*}
B^{\limfunc{big}}\left( I_{0}\right) &=&\dsum\limits_{\ell =0}^{\mathbf{\tau 
}}\sum_{J\in \Lambda _{\ell }^{\limfunc{big}}\left( I_{0}\right) }\left( 
\frac{\mathrm{P}^{\alpha }\left( J,\mathbf{1}_{I_{0}}\sigma \right) }{%
\left\vert J\right\vert ^{\frac{1}{n}}}\right) ^{2}\sum_{F\in \mathcal{F}:\
F\supsetneqq I_{0}\text{ and }J\in \mathcal{M}_{\mathbf{r}-\limfunc{deep}%
}\left( F\right) }\left\Vert \mathsf{P}_{F,J}^{\omega }\mathbf{x}\right\Vert
_{L^{2}\left( \omega \right) }^{2} \\
&\leq &\mathbf{\tau }\dsum\limits_{\ell =0}^{\mathbf{\tau }}\sum_{J\in
\Lambda _{\ell }^{\limfunc{big}}\left( I_{0}\right) }\left( \frac{\mathrm{P}%
^{\alpha }\left( J,\mathbf{1}_{I_{0}}\sigma \right) }{\left\vert
J\right\vert ^{\frac{1}{n}}}\right) ^{2}\left\Vert \mathsf{P}_{J}^{\limfunc{%
good},\omega }\mathbf{x}\right\Vert _{L^{2}\left( \omega \right) }^{2} \\
&=&\mathbf{\tau }\dsum\limits_{\ell =0}^{\mathbf{\tau }}\sum_{J\in \Lambda
_{\ell }^{\limfunc{big}}\left( I_{0}\right) }\left( \frac{\mathrm{P}^{\alpha
}\left( J,\mathbf{1}_{I_{0}}\sigma \right) }{\left\vert J\right\vert ^{\frac{%
1}{n}}}\right) ^{2}\sum_{J^{\prime }\in \mathcal{N}_{\mathbf{r}}\left(
I_{0}\right) \text{:\ }J^{\prime }\subset J}\left\Vert \bigtriangleup
_{J^{\prime }}^{\omega }\mathbf{x}\right\Vert _{L^{2}\left( \omega \right)
}^{2} \\
&&+\mathbf{\tau }\dsum\limits_{\ell =0}^{\mathbf{\tau }}\sum_{J\in \Lambda
_{\ell }^{\limfunc{big}}\left( I_{0}\right) }\left( \frac{\mathrm{P}^{\alpha
}\left( J,\mathbf{1}_{I_{0}}\sigma \right) }{\left\vert J\right\vert ^{\frac{%
1}{n}}}\right) ^{2}\sum_{K\in \mathcal{M}_{\mathbf{r}-\limfunc{deep}}\left(
I_{0}\right) \text{:\ }K\subset J}\left\Vert \mathsf{P}_{K}^{\limfunc{subgood%
},\omega }\mathbf{x}\right\Vert _{L^{2}\left( \omega \right) L^{2}\left(
\omega \right) }^{2} \\
&\equiv &\mathbf{\tau }B_{1}^{\limfunc{big}}\left( I_{0}\right) +\mathbf{%
\tau }B_{2}^{\limfunc{big}}\left( I_{0}\right) \ .
\end{eqnarray*}%
Now using that the quasicubes $J$ in $\Lambda _{\ell }^{\limfunc{big}}\left(
I_{0}\right) $ are pairwise disjoint, we have%
\begin{eqnarray*}
B_{1}^{\limfunc{big}}\left( I_{0}\right) &\approx &\dsum\limits_{\ell =0}^{%
\mathbf{\tau }}\left( \frac{\mathrm{P}^{\alpha }\left( I_{0},\mathbf{1}%
_{I_{0}}\sigma \right) }{\left\vert I_{0}\right\vert ^{\frac{1}{n}}}\right)
^{2}\sum_{J^{\prime }\in \mathcal{N}_{\mathbf{r}}\left( I_{0}\right)
}\left\Vert \bigtriangleup _{J^{\prime }}^{\omega }\mathbf{x}\right\Vert
_{L^{2}\left( \omega \right) }^{2} \\
&\lesssim &\dsum\limits_{\ell =0}^{\mathbf{\tau }}\left( \frac{\mathrm{P}%
^{\alpha }\left( I_{0},\mathbf{1}_{I_{0}}\sigma \right) }{\left\vert
I_{0}\right\vert ^{\frac{1}{n}}}\right) ^{2}\left\vert I_{0}\right\vert ^{%
\frac{2}{n}}\left\vert I_{0}\right\vert _{\omega }\lesssim \mathbf{\tau }%
A_{2}^{\alpha }\left\vert I_{0}\right\vert _{\sigma }\ .
\end{eqnarray*}%
Using $\mathrm{P}^{\alpha }\left( J,\mathbf{1}_{I_{0}}\sigma \right) =%
\mathrm{P}^{\alpha }\left( J,\mathbf{1}_{J}\sigma \right) +\mathrm{P}%
^{\alpha }\left( J,\mathbf{1}_{I_{0}\setminus J}\sigma \right) $, we have 
\begin{eqnarray*}
B_{2}^{\limfunc{big}}\left( I_{0}\right) &\approx &\dsum\limits_{\ell =0}^{%
\mathbf{\tau }}\sum_{J\in \Lambda _{\ell }^{\limfunc{big}}\left(
I_{0}\right) }\left( \frac{\mathrm{P}^{\alpha }\left( J,\mathbf{1}_{J}\sigma
\right) }{\left\vert J\right\vert ^{\frac{1}{n}}}\right) ^{2}\sum_{K\in 
\mathcal{M}_{\mathbf{r}-\limfunc{deep}}\left( I_{0}\right) \text{:\ }%
K\subset J}\left\Vert \mathsf{P}_{K}^{\limfunc{subgood},\omega }\mathbf{x}%
\right\Vert _{L^{2}\left( \omega \right) }^{2} \\
&&+\dsum\limits_{\ell =0}^{\mathbf{\tau }}\sum_{J\in \Lambda _{\ell }^{%
\limfunc{big}}\left( I_{0}\right) }\left( \frac{\mathrm{P}^{\alpha }\left( J,%
\mathbf{1}_{I_{0}\setminus J}\sigma \right) }{\left\vert J\right\vert ^{%
\frac{1}{n}}}\right) ^{2}\sum_{K\in \mathcal{M}_{\mathbf{r}-\limfunc{deep}%
}\left( I_{0}\right) \text{:\ }K\subset J}\left\Vert \mathsf{P}_{K}^{%
\limfunc{subgood},\omega }\mathbf{x}\right\Vert _{L^{2}\left( \omega \right)
}^{2} \\
&\equiv &B_{3}^{\limfunc{big}}\left( I_{0}\right) +B_{4}^{\limfunc{big}%
}\left( I_{0}\right) \ .
\end{eqnarray*}%
Now since the quasicubes $J$ in $\Lambda _{\ell }^{\limfunc{big}}\left(
I_{0}\right) $ are pairwise disjoint, 
\begin{equation*}
B_{3}^{\limfunc{big}}\left( I_{0}\right) \lesssim \dsum\limits_{\ell =0}^{%
\mathbf{\tau }}\sum_{J\in \Lambda _{\ell }^{\limfunc{big}}\left(
I_{0}\right) }\left( \frac{\left\vert J\right\vert _{\sigma }}{\left(
\left\vert J\right\vert ^{\frac{1}{n}}\right) ^{n+1-\alpha }}\right)
^{2}\left\vert J\right\vert ^{\frac{2}{n}}\left\vert J\right\vert _{\omega
}\lesssim \mathbf{\tau }A_{2}^{\alpha }\left\vert I_{0}\right\vert _{\sigma
}\ ,
\end{equation*}%
and since $\frac{\mathrm{P}^{\alpha }\left( J,\mathbf{1}_{I_{0}\setminus
J}\sigma \right) }{\left\vert J\right\vert ^{\frac{1}{n}}}\lesssim \frac{%
\mathrm{P}^{\alpha }\left( K,\mathbf{1}_{I_{0}\setminus J}\sigma \right) }{%
\left\vert K\right\vert ^{\frac{1}{n}}}$ for $K\subset J$, we have%
\begin{eqnarray*}
B_{4}^{\limfunc{big}}\left( I_{0}\right) &=&\dsum\limits_{\ell =0}^{\mathbf{%
\tau }}\sum_{J\in \Lambda _{\ell }^{\limfunc{big}}\left( I_{0}\right)
}\sum_{K\in \mathcal{M}_{\mathbf{r}-\limfunc{deep}}\left( I_{0}\right) \text{%
:\ }K\subset J}\left( \frac{\mathrm{P}^{\alpha }\left( J,\mathbf{1}%
_{I_{0}\setminus J}\sigma \right) }{\left\vert J\right\vert ^{\frac{1}{n}}}%
\right) ^{2}\left\Vert \mathsf{P}_{K}^{\limfunc{subgood},\omega }\mathbf{x}%
\right\Vert _{L^{2}\left( \omega \right) }^{2} \\
&\lesssim &\dsum\limits_{\ell =0}^{\mathbf{\tau }}\sum_{J\in \Lambda _{\ell
}^{\limfunc{big}}\left( I_{0}\right) }\sum_{K\in \mathcal{M}_{\mathbf{r}-%
\limfunc{deep}}\left( I_{0}\right) \text{:\ }K\subset J}\left( \frac{\mathrm{%
P}^{\alpha }\left( K,\mathbf{1}_{I_{0}\setminus J}\sigma \right) }{%
\left\vert K\right\vert ^{\frac{1}{n}}}\right) ^{2}\left\Vert \mathsf{P}%
_{K}^{\limfunc{subgood},\omega }\mathbf{x}\right\Vert _{L^{2}\left( \omega
\right) }^{2} \\
&\leq &\dsum\limits_{\ell =0}^{\mathbf{\tau }}\sum_{J\in \Lambda _{\ell }^{%
\limfunc{big}}\left( I_{0}\right) }\sum_{K\in \mathcal{M}_{\mathbf{r}-%
\limfunc{deep}}\left( I_{0}\right) \text{:\ }K\subset J}\left( \frac{\mathrm{%
P}^{\alpha }\left( K,\mathbf{1}_{I_{0}\setminus K}\sigma \right) }{%
\left\vert K\right\vert ^{\frac{1}{n}}}\right) ^{2}\left\Vert \mathsf{P}%
_{K}^{\limfunc{subgood},\omega }\mathbf{x}\right\Vert _{L^{2}\left( \omega
\right) }^{2} \\
&\leq &\dsum\limits_{\ell =0}^{\mathbf{\tau }}\sum_{K\in \mathcal{M}_{%
\mathbf{r}-\limfunc{deep}}\left( I_{0}\right) }\left( \frac{\mathrm{P}%
^{\alpha }\left( K,\mathbf{1}_{I_{0}\setminus K}\sigma \right) }{\left\vert
K\right\vert ^{\frac{1}{n}}}\right) ^{2}\left\Vert \mathsf{P}_{K}^{\limfunc{%
subgood},\omega }\mathbf{x}\right\Vert _{L^{2}\left( \omega \right) }^{2} \\
&\lesssim &\mathbf{\tau }\left( \mathcal{E}_{\alpha }^{\limfunc{deep}%
\limfunc{plug}}\right) ^{2}\left\vert I_{0}\right\vert _{\sigma }\lesssim 
\mathbf{\tau }\left( \left( \mathcal{E}_{\alpha }^{\limfunc{deep}}\right)
^{2}+\beta A_{2}^{\alpha }\right) \left\vert I_{0}\right\vert _{\sigma }\ ,
\end{eqnarray*}%
where the final line follows from (\ref{plug the hole deep}).

Now suppose that $I_{0}\in \mathcal{S}\Omega \mathcal{D}$ is a shifted $%
\Omega \mathcal{D}$-dyadic quasicube. Define 
\begin{equation*}
\Lambda \left( I_{0}\right) \equiv \left\{ J\subset I_{0}:J\in \mathcal{M}_{%
\mathbf{r}-\limfunc{deep}}\left( F\right) \text{ for some }F\supsetneqq
I_{0}^{\prime }\text{, }I_{0}^{\prime }\in \mathfrak{C}\left( I_{0}\right) 
\text{ with }\mathsf{P}_{F,J}^{\omega }\neq 0\right\} ,
\end{equation*}%
and pigeonhole the $J^{\prime }s$ in $\Lambda \left( I_{0}\right) $ by
setting $\Lambda \left( I_{0}\right) =\dbigcup\limits_{\ell =0}^{\mathbf{%
\tau }}\Lambda _{\ell }\left( I_{0}\right) $ where $\Lambda _{\ell }\left(
I_{0}\right) $ consists of those $J\subset I_{0}$ such that $J\in \mathcal{M}%
_{\mathbf{r}-\limfunc{deep}}\left( \pi _{\mathcal{F}}^{\left( \ell \right)
}I_{0}^{\prime }\right) $ for some child $I_{0}^{\prime }\in \mathfrak{C}%
\left( I_{0}\right) $. The above argument is now easily adapted to the case
at hand by simply noting that (\ref{B small}) continues to hold in this case
by the definition of the plugged refined energy constant $\mathcal{E}%
_{\alpha }^{\limfunc{refined}\limfunc{plug}}$. This completes the proof of
Lemma \ref{refined lemma}.
\end{proof}

\subsection{The Poisson testing inequality}

Fix $I\in \Omega \mathcal{D}$. We split the integration on the left side of (%
\ref{e.t1 n}) into a local and global piece:%
\begin{equation*}
\int_{\mathbb{R}_{+}^{n+1}}\mathbb{P}^{\alpha }\left( \mathbf{1}_{I}\sigma
\right) ^{2}d\mu =\int_{\widehat{I}}\mathbb{P}^{\alpha }\left( \mathbf{1}%
_{I}\sigma \right) ^{2}d\mu +\int_{\mathbb{R}_{+}^{n+1}\setminus \widehat{I}}%
\mathbb{P}^{\alpha }\left( \mathbf{1}_{I}\sigma \right) ^{2}d\mu \equiv 
\mathbf{Local}\left( I\right) +\mathbf{Global}\left( I\right) .
\end{equation*}%
Here is a brief schematic diagram of the decompositions, with bounds in $%
\fbox{}$, used in this subsection:%
\begin{equation*}
\fbox{$%
\begin{array}{ccc}
\mathbf{Local}\left( I\right) &  &  \\ 
\downarrow &  &  \\ 
\mathbf{Local}^{\limfunc{plug}}\left( I\right) & + & \mathbf{Local}^{%
\limfunc{hole}}\left( I\right) \\ 
\downarrow &  & \fbox{$\left( \mathcal{E}_{\alpha }^{\limfunc{deep}}\right)
^{2}$} \\ 
\downarrow &  &  \\ 
A & + & B \\ 
\fbox{$\left( \mathcal{E}_{\alpha }^{\limfunc{deep}}\right)
^{2}+A_{2}^{\alpha }$} &  & \fbox{$\left( \mathcal{E}_{\alpha }\right)
^{2}+A_{2}^{\alpha }$}%
\end{array}%
$}\text{ and }\fbox{$%
\begin{array}{ccccccc}
\mathbf{Global}\left( I\right) &  &  &  &  &  &  \\ 
\downarrow &  &  &  &  &  &  \\ 
A & + & B & + & C & + & D \\ 
\fbox{$A_{2}^{\alpha }$} &  & \fbox{$A_{2}^{\alpha }$} &  & \fbox{$\mathcal{A%
}_{2}^{\alpha ,\ast }$} &  & \fbox{$\mathcal{A}_{2}^{\alpha ,\ast }$}%
\end{array}%
$}.
\end{equation*}%
We turn first to estimating the local term $\mathbf{Local}\left( I\right) $.

An important consequence of the fact that $I$ and $J$ lie in the same
quasigrid $\Omega \mathcal{D}=\Omega \mathcal{D}^{\omega }$, is that $\left(
c\left( J\right) ,\left\vert J\right\vert \right) \in \widehat{I}$ if and
only if $J\subset I$. Thus we have%
\begin{eqnarray*}
&&\mathbf{Local}\left( I\right) =\int_{\widehat{I}}\mathbb{P}^{\alpha
}\left( \mathbf{1}_{I}\sigma \right) \left( x,t\right) ^{2}d\mu \left(
x,t\right) \\
&=&\sum_{F\in \mathcal{F}}\sum_{J\in \mathcal{M}_{\mathbf{r}-\limfunc{deep}%
}(F):\ J\subset I}\mathbb{P}^{\alpha }\left( \mathbf{1}_{I}\sigma \right)
\left( c_{J},\left\vert J\right\vert ^{\frac{1}{n}}\right) ^{2}\left\Vert 
\mathsf{P}_{F,J}^{\omega }\frac{\mathbf{x}}{\left\vert J\right\vert ^{\frac{1%
}{n}}}\right\Vert _{L^{2}\left( \omega \right) }^{2} \\
&=&\sum_{F\in \mathcal{F}}\sum_{J\in \mathcal{M}_{\mathbf{r}-\limfunc{deep}%
}\left( F\right) :\ J\subset I}\mathrm{P}^{\alpha }\left( J,\mathbf{1}%
_{I}\sigma \right) ^{2}\lVert \mathsf{P}_{F,J}^{\omega }\frac{\mathbf{x}}{%
\left\vert J\right\vert ^{\frac{1}{n}}}\rVert _{L^{2}\left( \omega \right)
}^{2} \\
&=&\mathbf{Local}^{\limfunc{plug}}\left( I\right) +\mathbf{Local}^{\func{hole%
}}\left( I\right) ,
\end{eqnarray*}%
where the `plugged' local sum $\mathbf{Local}^{\limfunc{plug}}\left(
I\right) $ is given by 
\begin{align*}
& \mathbf{Local}^{\limfunc{plug}}\left( I\right) \equiv \sum_{F\in \mathcal{F%
}}\sum_{J\in \mathcal{M}_{\mathbf{r}-\limfunc{deep}}\left( F\right) :\
J\subset I}\left( \frac{\mathrm{P}^{\alpha }\left( J,\mathbf{1}_{F\cap
I}\sigma \right) }{\left\vert J\right\vert ^{\frac{1}{n}}}\right)
^{2}\left\Vert \mathsf{P}_{F,J}^{\omega }\mathbf{x}\right\Vert _{L^{2}\left(
\omega \right) }^{2} \\
& =\left\{ \sum_{F\in \mathcal{F}:\ F\subset I}+\sum_{F\in \mathcal{F}:\
F\supsetneqq I}\right\} \sum_{J\in \mathcal{M}_{\mathbf{r}-\limfunc{deep}%
}\left( F\right) :\ J\subset I}\left( \frac{\mathrm{P}^{\alpha }\left( J,%
\mathbf{1}_{F\cap I}\sigma \right) }{\left\vert J\right\vert ^{\frac{1}{n}}}%
\right) ^{2}\left\Vert \mathsf{P}_{F,J}^{\omega }\mathbf{x}\right\Vert
_{L^{2}\left( \omega \right) }^{2} \\
& =A+B.
\end{align*}%
Then a \emph{trivial} application of the deep quasienergy condition (where
`trivial' means that the outer decomposition is just a single quasicube)
gives 
\begin{eqnarray*}
A &\leq &\sum_{F\in \mathcal{F}:\ F\subset I}\sum_{J\in \mathcal{M}_{\mathbf{%
r}-\limfunc{deep}}\left( F\right) }\left( \frac{\mathrm{P}^{\alpha }\left( J,%
\mathbf{1}_{F}\sigma \right) }{\left\vert J\right\vert ^{\frac{1}{n}}}%
\right) ^{2}\left\Vert \mathsf{P}_{F,J}^{\omega }\mathbf{x}\right\Vert
_{L^{2}\left( \omega \right) }^{2} \\
&\leq &\sum_{F\in \mathcal{F}:\ F\subset I}\left( \mathcal{E}_{\alpha }^{%
\limfunc{deep}\limfunc{plug}}\right) ^{2}\left\vert F\right\vert _{\sigma
}\lesssim \left( \mathcal{E}_{\alpha }^{2}+A_{2}^{\alpha }\right) \left\vert
I\right\vert _{\sigma }\,,
\end{eqnarray*}%
since $\left\Vert \mathsf{P}_{F,J}^{\omega }x\right\Vert _{L^{2}\left(
\omega \right) }^{2}\leq \left\Vert \mathsf{P}_{J}^{\omega }\mathbf{x}%
\right\Vert _{L^{2}\left( \omega \right) }^{2}$, where we recall that the
quasienergy constant $\mathcal{E}_{\alpha }^{\limfunc{deep}\limfunc{plug}}$
is defined in (\ref{plug}). We also used that the stopping quasicubes $%
\mathcal{F}$ satisfy a $\sigma $-Carleson measure estimate, 
\begin{equation*}
\sum_{F\in \mathcal{F}:\ F\subset F_{0}}\left\vert F\right\vert _{\sigma
}\lesssim \left\vert F_{0}\right\vert _{\sigma }.
\end{equation*}%
Lemma \ref{refined lemma} applies with $I_{0}=I$ to the remaining term $B$
to obtain the bound%
\begin{equation*}
B\leq \mathbf{\tau }\left( \left( \mathcal{E}_{\alpha }\right) ^{2}+\beta
A_{2}^{\alpha }\right) \left\vert I\right\vert _{\sigma }\ .
\end{equation*}

It remains then to show the inequality with `holes', where the support of $%
\sigma $ is restricted to the complement of the quasicube $F$.

\begin{lemma}
\label{local hole}We have 
\begin{equation}
\mathbf{Local}^{\func{hole}}\left( I\right) =\sum_{F\in \mathcal{F}:\
F\varsubsetneqq I}\sum_{J\in \mathcal{M}_{\mathbf{r}-\limfunc{deep}}\left(
F\right) }\left( \frac{\mathrm{P}^{\alpha }\left( J,\mathbf{1}_{I\setminus
F}\sigma \right) }{\left\vert J\right\vert ^{\frac{1}{n}}}\right)
^{2}\left\Vert \mathsf{P}_{F,J}^{\omega }x\right\Vert _{L^{2}\left( \omega
\right) }^{2}\lesssim \left( \mathcal{E}_{\alpha }^{\limfunc{deep}}\right)
^{2}\left\vert I\right\vert _{\sigma }\,.  \label{RTS n}
\end{equation}
\end{lemma}

\begin{proof}
Fix $I\in \Omega \mathcal{D}$ and define%
\begin{equation*}
\mathcal{F}_{I}\equiv \left\{ F\in \mathcal{F}:F\varsubsetneqq I\right\}
\cup \left\{ I\right\} ,
\end{equation*}%
and denote by $\pi F$ the parent of $F$ in the tree $\mathcal{F}_{I}$. We
consider the space $\ell _{\digamma }^{2}$ of square summable sequences on
the index set $\digamma $ where 
\begin{equation*}
\digamma \equiv \left\{ \left( F,J\right) :F\in \mathcal{F}_{I}\text{ and }%
J\in \mathcal{M}_{\mathbf{r}-\limfunc{deep}}\left( F\right) \right\}
\end{equation*}%
is the index set of pairs $\left( F,J\right) $ occurring in the sum in (\ref%
{RTS n}). We now take a sequence $a=\left\{ a_{F,J}\right\} _{\left(
F,J\right) \in \digamma }\in \ell _{\digamma }^{2}$ with $a_{F,J}\geq 0$ and
estimate%
\begin{equation*}
S\equiv \sum_{F\in \mathcal{F}_{I}}\sum_{J\in \mathcal{M}_{\mathbf{r}-%
\limfunc{deep}}\left( F\right) }\frac{\mathrm{P}^{\alpha }\left( J,\mathbf{1}%
_{I\setminus F}\sigma \right) }{\left\vert J\right\vert ^{\frac{1}{n}}}%
\left\Vert \mathsf{P}_{F,J}^{\omega }\mathbf{x}\right\Vert _{L^{2}\left(
\omega \right) }a_{F,J}
\end{equation*}%
by 
\begin{eqnarray*}
S &=&\sum_{F\in \mathcal{F}_{I}}\sum_{J\in \mathcal{M}_{\mathbf{r}-\limfunc{%
deep}}\left( F\right) }\sum_{F^{\prime }\in \mathcal{F}:\ F\subset F^{\prime
}\subsetneqq I}\frac{\mathrm{P}^{\alpha }\left( J,\mathbf{1}_{\pi _{\mathcal{%
F}}F^{\prime }\setminus F^{\prime }}\sigma \right) }{\left\vert J\right\vert
^{\frac{1}{n}}}\left\Vert \mathsf{P}_{F,J}^{\omega }\mathbf{x}\right\Vert
_{L^{2}\left( \omega \right) }a_{F,J} \\
&=&\sum_{F^{\prime }\in \mathcal{F}_{I}}\sum_{F\in \mathcal{F}:\ F\subset
F^{\prime }}\sum_{J\in \mathcal{M}_{\mathbf{r}-\limfunc{deep}}\left(
F\right) }\frac{\mathrm{P}^{\alpha }\left( J,\mathbf{1}_{\pi _{\mathcal{F}%
}F^{\prime }\setminus F^{\prime }}\sigma \right) }{\left\vert J\right\vert ^{%
\frac{1}{n}}}\left\Vert \mathsf{P}_{F,J}^{\omega }\mathbf{x}\right\Vert
_{L^{2}\left( \omega \right) }a_{F,J} \\
&=&\sum_{F^{\prime }\in \mathcal{F}_{I}}\sum_{K\in \mathcal{M}_{\mathbf{r}-%
\limfunc{deep}}\left( F^{\prime }\right) }\sum_{F\in \mathcal{F}:\ F\subset
F^{\prime }}\sum_{J\in \mathcal{M}_{\mathbf{r}-\limfunc{deep}}\left(
F\right) :\ J\subset K}\frac{\mathrm{P}^{\alpha }\left( J,\mathbf{1}_{\pi _{%
\mathcal{F}}F^{\prime }\setminus F^{\prime }}\sigma \right) }{\left\vert
J\right\vert ^{\frac{1}{n}}}\left\Vert \mathsf{P}_{F,J}^{\omega }\mathbf{x}%
\right\Vert _{L^{2}\left( \omega \right) }a_{F,J} \\
&\lesssim &\sum_{F^{\prime }\in \mathcal{F}_{I}}\sum_{K\in \mathcal{M}_{%
\mathbf{r}-\limfunc{deep}}\left( F^{\prime }\right) }\frac{\mathrm{P}%
^{\alpha }\left( K,\mathbf{1}_{\pi _{\mathcal{F}}F^{\prime }\setminus
F^{\prime }}\sigma \right) }{\left\vert K\right\vert ^{\frac{1}{n}}}%
\sum_{F\in \mathcal{F}:\ F\subset F^{\prime }}\sum_{J\in \mathcal{M}_{%
\mathbf{r}-\limfunc{deep}}\left( F\right) :\ J\subset K}\left\Vert \mathsf{P}%
_{F,J}^{\omega }\mathbf{x}\right\Vert _{L^{2}\left( \omega \right) }a_{F,J},
\end{eqnarray*}%
by the Poisson inequalities in Lemma \ref{Poisson inequalities}. We now
invoke%
\begin{eqnarray*}
&&\sum_{F\in \mathcal{F}:\ F\subset F^{\prime }}\sum_{J\in \mathcal{M}_{%
\mathbf{r}-\limfunc{deep}}\left( F\right) :\ J\subset K}\left\Vert \mathsf{P}%
_{F,J}^{\omega }\mathbf{x}\right\Vert _{L^{2}\left( \omega \right) }a_{F,J}
\\
&\lesssim &\left( \sum_{F\in \mathcal{F}:\ F\subset F^{\prime }}\sum_{J\in 
\mathcal{M}_{\mathbf{r}-\limfunc{deep}}\left( F\right) :\ J\subset
K}\left\Vert \mathsf{P}_{F,J}^{\omega }\mathbf{x}\right\Vert _{L^{2}\left(
\omega \right) }^{2}\right) ^{\frac{1}{2}} \\
&&\times \left( \sum_{F\in \mathcal{F}:\ F\subset F^{\prime }}\sum_{J\in 
\mathcal{M}_{\mathbf{r}-\limfunc{deep}}\left( F\right) :\ J\subset
K}a_{F,J}^{2}\right) ^{\frac{1}{2}} \\
&\lesssim &\left\Vert \widehat{\mathsf{P}}_{F^{\prime },K}^{\omega }\mathbf{x%
}\right\Vert _{L^{2}\left( \omega \right) }\left\Vert \mathsf{Q}_{F^{\prime
},K}^{\omega }a\right\Vert _{\ell _{\digamma }^{2}}\ ,
\end{eqnarray*}%
where for $K\in \mathcal{M}_{\mathbf{r}-\limfunc{deep}}\left( F^{\prime
}\right) $ and $f\in L^{2}\left( \omega \right) $,%
\begin{equation*}
\widehat{\mathsf{P}}_{F^{\prime },K}^{\omega }\equiv \sum_{F\in \mathcal{F}%
:\ F\subset F^{\prime }}\sum_{J\in \mathcal{M}_{\mathbf{r}-\limfunc{deep}%
}\left( F\right) :\ J\subset K}\mathsf{P}_{F,J}^{\omega },
\end{equation*}%
while for $K\in \mathcal{M}_{\mathbf{r}-\limfunc{deep}}\left( F^{\prime
}\right) $ and $a=\left\{ a_{G,L}\right\} _{\left( G,L\right) \in \digamma
}\in \ell _{\digamma }^{2}$,%
\begin{eqnarray*}
\mathsf{Q}_{F^{\prime },K}^{\omega }a &\equiv &\sum_{F\in \mathcal{F}:\
F\subset F^{\prime }}\sum_{J\in \mathcal{M}_{\mathbf{r}-\limfunc{deep}%
}\left( F\right) :\ J\subset K}\mathsf{Q}_{F,J}^{\omega }a; \\
\mathsf{Q}_{F,J}^{\omega }a &\equiv &\left\{ \mathbf{1}_{\left( F,J\right)
}a_{G,L}\right\} _{\left( G,L\right) \in \digamma }.
\end{eqnarray*}%
Thus $\mathsf{Q}_{F,J}^{\omega }$ acts on the sequence $a$ by projecting
onto the coordinate in $\digamma $ indexed by $\left( F,J\right) $.

Now denote by $d\left( F\right) \equiv d_{\mathcal{F}_{I}}\left( F,I\right) $
the distance from $F$ to the root $I$ in the tree $\mathcal{F}_{I}$. Since
the collection $\mathcal{F}$ satisfies a Carleson condition, we have
geometric decay in generations:%
\begin{equation*}
\sum_{F\in \mathcal{F}_{I}:\ d\left( F\right) =k}\left\vert F\right\vert
_{\sigma }\lesssim 2^{-\delta k}\left\vert I\right\vert _{\sigma }\ ,\ \ \ \
\ k\geq 0.
\end{equation*}%
Thus we can write%
\begin{eqnarray*}
\left\vert S\right\vert &\lesssim &\sum_{F^{\prime }\in \mathcal{F}%
_{I}}\sum_{K\in \mathcal{M}_{\mathbf{r}-\limfunc{deep}}\left( F^{\prime
}\right) }\frac{\mathrm{P}^{\alpha }\left( K,\mathbf{1}_{\pi _{\mathcal{F}%
}F^{\prime }\setminus F^{\prime }}\sigma \right) }{\left\vert K\right\vert ^{%
\frac{1}{n}}}\left\Vert \widehat{\mathsf{P}}_{F^{\prime },K}^{\omega }%
\mathbf{x}\right\Vert _{L^{2}\left( \omega \right) }\left\Vert \mathsf{Q}%
_{F^{\prime },K}^{\omega }a\right\Vert _{\ell _{\digamma }^{2}} \\
&=&\sum_{k=0}^{\infty }\sum_{F^{\prime }\in \mathcal{F}_{I}:\ d\left(
F^{\prime }\right) =k}\sum_{K\in \mathcal{M}_{\mathbf{r}-\limfunc{deep}%
}\left( F^{\prime }\right) }\frac{\mathrm{P}^{\alpha }\left( K,\mathbf{1}%
_{\pi _{\mathcal{F}}F^{\prime }\setminus F^{\prime }}\sigma \right) }{%
\left\vert K\right\vert ^{\frac{1}{n}}}\left\Vert \widehat{\mathsf{P}}%
_{F^{\prime },K}^{\omega }\mathbf{x}\right\Vert _{L^{2}\left( \omega \right)
}\left\Vert \mathsf{Q}_{F^{\prime },K}^{\omega }a\right\Vert _{\ell
_{\digamma }^{2}}\equiv \sum_{k=1}^{\infty }A_{k}.
\end{eqnarray*}%
Now by the $\mathbf{\tau }$-overlap (\ref{tau overlap}) of the projections $%
\mathsf{P}_{F,J}^{\omega }$, we have $\left\Vert \widehat{\mathsf{P}}%
_{F^{\prime },K}^{\omega }\mathbf{x}\right\Vert _{L^{2}\left( \omega \right)
}^{2}\leq \mathbf{\tau }\left\Vert \mathsf{P}_{K}^{\limfunc{good},\omega }%
\mathbf{x}\right\Vert _{L^{2}\left( \omega \right) }^{2}$, and hence by the
deep energy condition,%
\begin{eqnarray*}
A_{k} &\lesssim &\left( \sum_{F^{\prime }\in \mathcal{F}_{I}:\ d\left(
F^{\prime }\right) =k}\sum_{K\in \mathcal{M}_{\mathbf{r}-\limfunc{deep}%
}\left( F^{\prime }\right) }\left( \frac{\mathrm{P}^{\alpha }\left( K,%
\mathbf{1}_{\pi _{\mathcal{F}}F^{\prime }\setminus F^{\prime }}\sigma
\right) }{\left\vert K\right\vert ^{\frac{1}{n}}}\right) ^{2}\left\Vert 
\widehat{\mathsf{P}}_{F^{\prime },K}^{\omega }\mathbf{x}\right\Vert
_{L^{2}\left( \omega \right) }^{2}\right) ^{\frac{1}{2}} \\
&&\times \left( \sum_{F^{\prime }\in \mathcal{F}_{I}:\ d\left( F^{\prime
}\right) =k}\sum_{K\in \mathcal{M}_{\mathbf{r}-\limfunc{deep}}\left(
F^{\prime }\right) }\left\Vert \mathsf{Q}_{F^{\prime },K}^{\omega
}a\right\Vert _{\ell _{\digamma }^{2}}^{2}\right) ^{\frac{1}{2}} \\
&\lesssim &\left( \left( \mathcal{E}_{\alpha }^{\limfunc{deep}}\right)
^{2}\sum_{F^{\prime \prime }\in \mathcal{F}_{I}:\ d\left( F^{\prime \prime
}\right) =k-1}\left\vert F^{\prime \prime }\right\vert _{\sigma }\right) ^{%
\frac{1}{2}}\left\Vert a\right\Vert _{\ell _{\digamma }^{2}} \\
&\lesssim &\mathcal{E}_{\alpha }^{\limfunc{deep}}\left( 2^{-\delta
k}\left\vert I\right\vert _{\sigma }\right) ^{\frac{1}{2}}\left\Vert
a\right\Vert _{\ell _{\digamma }^{2}}\ ,
\end{eqnarray*}%
where we have applied the deep energy condition for each $F^{\prime \prime
}\in \mathcal{F}_{I}$ with $d\left( F^{\prime \prime }\right) =k-1$ to obtain%
\begin{equation}
\sum_{F^{\prime }\in \mathcal{F}_{I}:\ F^{\prime }\subset F^{\prime \prime
}}\sum_{K\in \mathcal{M}_{\mathbf{r}-\limfunc{deep}}\left( F^{\prime
}\right) }\left( \frac{\mathrm{P}^{\alpha }\left( K,\mathbf{1}_{F^{\prime
\prime }\setminus F^{\prime }}\sigma \right) }{\left\vert K\right\vert ^{%
\frac{1}{n}}}\right) ^{2}\left\Vert \widehat{\mathsf{P}}_{F^{\prime
},K}^{\omega }\mathbf{x}\right\Vert _{L^{2}\left( \omega \right) }^{2}\leq
\left( \mathcal{E}_{\alpha }^{\limfunc{deep}}\right) ^{2}\left\vert
F^{\prime \prime }\right\vert _{\sigma }\ .  \label{to obtain}
\end{equation}%
Finally then we obtain%
\begin{equation*}
\left\vert S\right\vert \lesssim \sum_{k=1}^{\infty }\mathcal{E}_{\alpha }^{%
\limfunc{deep}}\left( 2^{-\delta k}\left\vert I\right\vert _{\sigma }\right)
^{\frac{1}{2}}\left\Vert a\right\Vert _{\ell _{\digamma }^{2}}\lesssim 
\mathcal{E}_{\alpha }^{\limfunc{deep}}\sqrt{\left\vert I\right\vert _{\sigma
}}\left\Vert a\right\Vert _{\ell _{\digamma }^{2}}.
\end{equation*}%
By duality of $\ell _{\digamma }^{2}$ we now conclude that%
\begin{equation*}
\sum_{F\in \mathcal{F}_{I}}\sum_{J\in \mathcal{M}_{\mathbf{r}-\limfunc{deep}%
}\left( F\right) }\left( \frac{\mathrm{P}^{\alpha }\left( J,\mathbf{1}%
_{I\setminus F}\sigma \right) }{\left\vert J\right\vert ^{\frac{1}{n}}}%
\right) ^{2}\left\Vert \mathsf{P}_{F,J}^{\omega }\mathbf{x}\right\Vert
_{L^{2}\left( \omega \right) }^{2}\lesssim \left( \mathcal{E}_{\alpha }^{%
\limfunc{deep}}\right) ^{2}\left\vert I\right\vert _{\sigma },
\end{equation*}%
which is (\ref{RTS n}).
\end{proof}

This completes the proof of the estimate $\mathbf{Local}\left( I\right)
\lesssim \left( \left( \mathcal{E}_{\alpha }\right) ^{2}+A_{2}^{\alpha
}\right) \left\vert I\right\vert _{\sigma }$, i.e. for every dyadic
quasicube $L\in \Omega \mathcal{D}$,%
\begin{eqnarray}
\mathbf{Local}\left( L\right) &=&\sum_{F\in \mathcal{F}}\sum_{J\in \mathcal{M%
}_{\mathbf{r}-\limfunc{deep}}\left( F\right) :\ J\subset L}\left\Vert 
\mathsf{P}_{F,J}^{\omega }\mathbf{x}\right\Vert _{L^{2}\left( \omega \right)
}^{2}\left( \frac{\mathrm{P}^{\alpha }\left( J,\mathbf{1}_{L}\sigma \right) 
}{\left\vert J\right\vert ^{\frac{1}{n}}}\right) ^{2}  \label{local} \\
&\lesssim &\left( \left( \mathcal{E}_{\alpha }\right) ^{2}+A_{2}^{\alpha
}\right) \left\vert L\right\vert _{\sigma },\ \ \ L\in \Omega \mathcal{D}. 
\notag
\end{eqnarray}

\subsubsection{The shifted local estimate}

For future use, we prove a strengthening of this estimate to shifted
quasicubes $M\in \mathcal{S}\Omega \mathcal{D}$.

\begin{lemma}
\label{shifted}With notation as above and $M\in \mathcal{S}\Omega \mathcal{D}
$ a shifted quasicube,%
\begin{eqnarray}
\mathbf{Local}\left( M\right) &\equiv &\sum_{F\in \mathcal{F}}\sum_{J\in 
\mathcal{M}_{\mathbf{r}-\limfunc{deep}}\left( F\right) :\ J\subset
M}\left\Vert \mathsf{P}_{F,J}^{\omega }\mathbf{x}\right\Vert _{L^{2}\left(
\omega \right) }^{2}\left( \frac{\mathrm{P}^{\alpha }\left( J,\mathbf{1}%
_{M}\sigma \right) }{\left\vert J\right\vert ^{\frac{1}{n}}}\right) ^{2}
\label{shifted local} \\
&\lesssim &\left( \left( \mathcal{E}_{\alpha }\right) ^{2}+A_{2}^{\alpha
}\right) \left\vert M\right\vert _{\sigma },\ \ \ M\in \mathcal{S}\Omega 
\mathcal{D}.  \notag
\end{eqnarray}
\end{lemma}

\begin{proof}
We prove (\ref{shifted local}) by repeating the above proof of (\ref{local})
and noting the points requiring change. First we decompose 
\begin{equation*}
\mathbf{Local}\left( M\right) =\mathbf{Local}^{\limfunc{plug}}\left(
M\right) +\mathbf{Local}^{\func{hole}}\left( M\right)
\end{equation*}%
as above where%
\begin{align*}
& \mathbf{Local}^{\limfunc{plug}}\left( M\right) \equiv \sum_{F\in \mathcal{F%
}}\sum_{J\in \mathcal{M}_{\mathbf{r}-\limfunc{deep}}\left( F\right) :\
J\subset M}\left( \frac{\mathrm{P}^{\alpha }\left( J,\mathbf{1}_{F\cap
M}\sigma \right) }{\left\vert J\right\vert ^{\frac{1}{n}}}\right)
^{2}\left\Vert \mathsf{P}_{F,J}^{\omega }\mathbf{x}\right\Vert _{L^{2}\left(
\omega \right) }^{2} \\
& =\left\{ \sum_{F\in \mathcal{F}:\ F\subset \text{ some }M^{\prime }\in 
\mathfrak{C}_{\Omega \mathcal{D}}\left( M\right) }+\sum_{F\in \mathcal{F}:\
F\supsetneqq \text{ some }M^{\prime }\in \mathfrak{C}_{\Omega \mathcal{D}%
}\left( M\right) }\right\} \\
& \ \ \ \ \ \ \ \ \ \ \ \ \ \ \ \ \ \ \ \ \ \ \ \ \ \times \sum_{J\in 
\mathcal{M}_{\mathbf{r}-\limfunc{deep}}\left( F\right) :\ J\subset M}\left( 
\frac{\mathrm{P}^{\alpha }\left( J,\mathbf{1}_{F\cap M}\sigma \right) }{%
\left\vert J\right\vert ^{\frac{1}{n}}}\right) ^{2}\left\Vert \mathsf{P}%
_{F,J}^{\omega }\mathbf{x}\right\Vert _{L^{2}\left( \omega \right) }^{2} \\
& =A+B.
\end{align*}%
Term $A$ satisfies%
\begin{equation*}
A\lesssim \left( \left( \mathcal{E}_{\alpha }\right) ^{2}+A_{2}^{\alpha
}\right) \left\vert M\right\vert _{\sigma }\ ,
\end{equation*}%
just as above using $\left\Vert \mathsf{P}_{F,J}^{\omega }x\right\Vert
_{L^{2}\left( \omega \right) }^{2}\leq \left\Vert \mathsf{P}_{J}^{\omega }%
\mathbf{x}\right\Vert _{L^{2}\left( \omega \right) }^{2}$, and the fact that
the stopping quasicubes $\mathcal{F}$ satisfy a $\sigma $-Carleson measure
estimate, 
\begin{equation*}
\sum_{F\in \mathcal{F}:\ F\subset M}\left\vert F\right\vert _{\sigma
}\lesssim \left\vert M\right\vert _{\sigma }.
\end{equation*}%
Term $B$ is handled directly by Lemma \ref{refined lemma} with the shifted
quasicube $I_{0}=M$ to obtain%
\begin{equation*}
B\lesssim \left( \left( \mathcal{E}_{\alpha }\right) ^{2}+A_{2}^{\alpha
}\right) \left\vert M\right\vert _{\sigma }\ .
\end{equation*}%
Finally, to extend Lemma \ref{local hole} to shifted quasicubes $M\in 
\mathcal{S}\Omega \mathcal{D}$, we define%
\begin{equation*}
\mathcal{F}_{M}\equiv \left\{ F\in \mathcal{F}:F\varsubsetneqq M\right\}
\cup \left\{ M\right\} ,
\end{equation*}%
and follow along the proof there with only trivial changes. The analogue of (%
\ref{to obtain}) is now%
\begin{equation*}
\sum_{F^{\prime }\in \mathcal{F}_{M}:\ F^{\prime }\subset F^{\prime \prime
}}\sum_{K\in \mathcal{M}_{\mathbf{r}-\limfunc{deep}}\left( F^{\prime
}\right) }\left( \frac{\mathrm{P}^{\alpha }\left( K,\mathbf{1}_{F^{\prime
\prime }\setminus F^{\prime }}\sigma \right) }{\left\vert K\right\vert ^{%
\frac{1}{n}}}\right) ^{2}\left\Vert \widehat{\mathsf{P}}_{F^{\prime
},K}^{\omega }\mathbf{x}\right\Vert _{L^{2}\left( \omega \right) }^{2}\leq
\left( \mathcal{E}_{\alpha }^{\limfunc{deep}}\right) ^{2}\left\vert
F^{\prime \prime }\right\vert _{\sigma }\ ,
\end{equation*}%
the only change being that $\mathcal{F}_{M}$ now appears in place of $%
\mathcal{F}_{I}$, so that the deep energy condition still applies. We
conclude that%
\begin{equation*}
\mathbf{Local}^{\func{hole}}\left( M\right) \lesssim \left( \left( \mathcal{E%
}_{\alpha }\right) ^{2}+A_{2}^{\alpha }\right) \left\vert M\right\vert
_{\sigma }\ ,
\end{equation*}%
and this completes the proof of the estimate (\ref{shifted local}) in Lemma %
\ref{shifted}.
\end{proof}

\subsubsection{The global estimate}

Now we turn to proving the following estimate for the global part of the
first testing condition \eqref{e.t1 n}:%
\begin{equation*}
\int_{\mathbb{R}_{+}^{n+1}\setminus \widehat{I}}\mathbb{P}^{\alpha }\left( 
\mathbf{1}_{I}\sigma \right) ^{2}d\mu \lesssim \mathcal{A}_{2}^{\alpha ,\ast
}\left\vert I\right\vert _{\sigma }.
\end{equation*}%
We begin by decomposing the integral on the left into four pieces where we
use $F\sim J$ to denote the sum over those $F\in \mathcal{F}$ such that $%
J\in \mathcal{M}_{\mathbf{r}-\limfunc{deep}}\left( F\right) $. As a
particular consequence of Lemma \ref{tau ovelap}, we note that given $J$,
there are at most a fixed number $C$ of $F\in \mathcal{F}$ such that $F\sim
J $. We have:%
\begin{eqnarray*}
&&\int_{\mathbb{R}_{+}^{n+1}\setminus \widehat{I}}\mathbb{P}^{\alpha }\left( 
\mathbf{1}_{I}\sigma \right) ^{2}d\mu =\sum_{J:\ \left( c_{J},\ell \left(
J\right) \right) \in \mathbb{R}_{+}^{n+1}\setminus \widehat{I}}\mathbb{P}%
^{\alpha }\left( \mathbf{1}_{I}\sigma \right) \left( c_{J},\ell \left(
J\right) \right) ^{2}\sum_{\substack{ F\in \mathcal{F}  \\ J\in \mathcal{M}_{%
\mathbf{r}-\limfunc{deep}}\left( F\right) }}\left\Vert \mathsf{P}%
_{F,J}^{\omega }\frac{\mathbf{x}}{\left\vert J\right\vert ^{\frac{1}{n}}}%
\right\Vert _{L^{2}\left( \omega \right) }^{2} \\
&=&\left\{ \sum_{\substack{ J\cap 3I=\emptyset  \\ \ell \left( J\right) \leq
\ell \left( I\right) }}+\sum_{J\subset 3I\setminus I}+\sum_{\substack{ J\cap
I=\emptyset  \\ \ell \left( J\right) >\ell \left( I\right) }}%
+\sum_{J\supsetneqq I}\right\} \mathbb{P}^{\alpha }\left( \mathbf{1}%
_{I}\sigma \right) \left( c_{J},\ell \left( J\right) \right) ^{2}\sum 
_{\substack{ F\in \mathcal{F}:  \\ J\in \mathcal{M}_{\mathbf{r}-\limfunc{deep%
}}\left( F\right) }}\left\Vert \mathsf{P}_{F,J}^{\omega }\frac{\mathbf{x}}{%
\left\vert J\right\vert ^{\frac{1}{n}}}\right\Vert _{L^{2}\left( \omega
\right) }^{2} \\
&=&A+B+C+D.
\end{eqnarray*}

We further decompose term $A$ according to the length of $J$ and its
distance from $I$, and then use Lemma \ref{tau ovelap} with $I_{0}=J$ to
obtain:%
\begin{eqnarray*}
A &\lesssim &\sum_{m=0}^{\infty }\sum_{k=1}^{\infty }\sum_{\substack{ %
J\subset 3^{k+1}I\setminus 3^{k}I  \\ \ell \left( J\right) =2^{-m}\ell
\left( I\right) }}\left( \frac{2^{-m}\left\vert I\right\vert ^{\frac{1}{n}}}{%
\limfunc{quasidist}\left( J,I\right) ^{n+1-\alpha }}\left\vert I\right\vert
_{\sigma }\right) ^{2}\mathbf{\tau }\left\vert J\right\vert _{\omega } \\
&\lesssim &\sum_{m=0}^{\infty }2^{-2m}\sum_{k=1}^{\infty }\frac{\left\vert
I\right\vert ^{\frac{2}{n}}\left\vert I\right\vert _{\sigma }\left\vert
3^{k+1}I\setminus 3^{k}I\right\vert _{\omega }}{\left\vert 3^{k}I\right\vert
^{2\left( 1+\frac{1}{n}-\frac{\alpha }{n}\right) }}\left\vert I\right\vert
_{\sigma } \\
&\lesssim &\sum_{m=0}^{\infty }2^{-2m}\sum_{k=1}^{\infty }3^{-2k}\left\{ 
\frac{\left\vert 3^{k+1}I\right\vert _{\sigma }\left\vert
3^{k+1}I\right\vert _{\omega }}{\left\vert 3^{k}I\right\vert ^{2\left( 1-%
\frac{\alpha }{n}\right) }}\right\} \left\vert I\right\vert _{\sigma
}\lesssim A_{2}^{\alpha }\left\vert I\right\vert _{\sigma }.
\end{eqnarray*}

For term $B$ we let $\mathcal{J}^{\ast }\equiv \dbigcup\limits_{F\in 
\mathcal{F}}\dbigcup\limits_{J\in \mathcal{M}_{r-\limfunc{deep}}\left(
F\right) }\left\{ K\in \mathcal{C}_{F}^{\limfunc{good},\mathbf{\tau }-%
\limfunc{shift}}:K\subset J\right\} $, which is the union of all quasicubes $%
K$ occurring in the projections $\mathsf{P}_{F,J}^{\omega }$. We further
decompose term $B$ according to the length of $J$ and use the fractional
Poisson inequality (\ref{e.Jsimeq}) in Lemma \ref{Poisson inequality} on the
neighbour $I^{\prime }$ of $I$ containing $K$,%
\begin{equation*}
\mathrm{P}^{\alpha }\left( K,\mathbf{1}_{I}\sigma \right) ^{2}\lesssim
\left( \frac{\ell \left( K\right) }{\ell \left( I\right) }\right)
^{2-2\left( n+1-\alpha \right) \varepsilon }\mathrm{P}^{\alpha }\left( I,%
\mathbf{1}_{I}\sigma \right) ^{2},\ \ \ \ \ K\in \mathcal{J}^{\ast
},K\subset 3I\setminus I,
\end{equation*}%
where we have used that $\mathrm{P}^{\alpha }\left( I^{\prime },\mathbf{1}%
_{I}\sigma \right) \approx \mathrm{P}^{\alpha }\left( I,\mathbf{1}_{I}\sigma
\right) $ and that the quasicubes $K\in \mathcal{J}^{\ast }$ are good. We
then obtain from Lemma \ref{tau ovelap} with $I_{0}=J$,%
\begin{eqnarray*}
B &=&\sum_{J\subset 3I\setminus I}\left( \frac{\mathrm{P}^{\alpha }\left( J,%
\mathbf{1}_{I}\sigma \right) }{\left\vert J\right\vert ^{\frac{1}{n}}}%
\right) ^{2}\sum_{\substack{ F\in \mathcal{F}:  \\ J\in \mathcal{M}_{\mathbf{%
r}-\limfunc{deep}}\left( F\right) }}\left\Vert \mathsf{P}_{F,J}^{\omega
}x\right\Vert _{L^{2}\left( \omega \right) }^{2} \\
&\lesssim &\sum_{m=0}^{\infty }\sum_{\substack{ K\subset 3I\setminus I  \\ %
\ell \left( K\right) =2^{-m}\ell \left( I\right) }}\left( 2^{-m}\right)
^{2-2\left( n+1-\alpha \right) \varepsilon }\left( \frac{\left\vert
I\right\vert _{\sigma }}{\left\vert I\right\vert ^{1-\frac{\alpha }{n}}}%
\right) ^{2}\mathbf{\tau \ }\left\vert K\right\vert _{\omega } \\
&\lesssim &\mathbf{\tau }\sum_{m=0}^{\infty }\left( 2^{-m}\right)
^{2-2\left( n+1-\alpha \right) \varepsilon }\frac{\left\vert 3I\right\vert
_{\sigma }\left\vert 3I\right\vert _{\omega }}{\left\vert 3I\right\vert
^{2\left( 1-\frac{\alpha }{n}\right) }}\left\vert I\right\vert _{\sigma
}\lesssim \mathbf{\tau }A_{2}^{\alpha }\left\vert I\right\vert _{\sigma }.
\end{eqnarray*}

For term $C$ we will have to group the quasicubes $J$ into blocks $B_{i}$,
and then exploit Lemma \ref{tau ovelap}. We first split the sum according to
whether or not $I$ intersects the triple of $J$:%
\begin{eqnarray*}
C &\approx &\left\{ \sum_{\substack{ J:\ I\cap 3J=\emptyset  \\ \ell \left(
J\right) >\ell \left( I\right) }}+\sum_{\substack{ J:\ I\subset 3J\setminus
J  \\ \ell \left( J\right) >\ell \left( I\right) }}\right\} \left( \frac{%
\left\vert J\right\vert ^{\frac{1}{n}}}{\left( \left\vert J\right\vert ^{%
\frac{1}{n}}+\limfunc{dist}\left( J,I\right) \right) ^{n+1-\alpha }}%
\left\vert I\right\vert _{\sigma }\right) ^{2}\sum_{\substack{ F\in \mathcal{%
F}:  \\ J\in \mathcal{M}_{\mathbf{r}-\limfunc{deep}}\left( F\right) }}%
\left\Vert \mathsf{P}_{F,J}^{\omega }\frac{\mathbf{x}}{\left\vert
J\right\vert ^{\frac{1}{n}}}\right\Vert _{L^{2}\left( \omega \right) }^{2} \\
&=&C_{1}+C_{2}.
\end{eqnarray*}%
We first consider $C_{1}$. Let $\mathcal{M}$ be the maximal dyadic
quasicubes in $\left\{ Q:3Q\cap I=\emptyset \right\} $, and then let $%
\left\{ B_{i}\right\} _{i=1}^{\infty }$ be an enumeration of those $Q\in 
\mathcal{M}$ whose side length is at least $\ell \left( I\right) $. Now we
further decompose the sum in $C_{1}$ by grouping the quasicubes $J$ into the
Whitney quasicubes $B_{i}$, and then using Lemma \ref{tau ovelap} with $%
I_{0}=J$: 
\begin{eqnarray*}
C_{1} &\leq &\sum_{i=1}^{\infty }\sum_{J:\ J\subset B_{i}}\left( \frac{1}{%
\left( \left\vert J\right\vert ^{\frac{1}{n}}+\limfunc{dist}\left(
J,I\right) \right) ^{n+1-\alpha }}\left\vert I\right\vert _{\sigma }\right)
^{2}\sum_{\substack{ F\in \mathcal{F}:  \\ J\in \mathcal{M}_{\mathbf{r}-%
\limfunc{deep}}\left( F\right) }}\left\Vert \mathsf{P}_{F,J}^{\omega }%
\mathbf{x}\right\Vert _{L^{2}\left( \omega \right) }^{2} \\
&\lesssim &\sum_{i=1}^{\infty }\left( \frac{1}{\left( \left\vert
B_{i}\right\vert ^{\frac{1}{n}}+\limfunc{dist}\left( B_{i},I\right) \right)
^{n+1-\alpha }}\left\vert I\right\vert _{\sigma }\right) ^{2}\sum_{J:\
J\subset B_{i}}\sum_{\substack{ F\in \mathcal{F}:  \\ J\in \mathcal{M}_{%
\mathbf{r}-\limfunc{deep}}\left( F\right) }}\left\Vert \mathsf{P}%
_{F,J}^{\omega }\mathbf{x}\right\Vert _{L^{2}\left( \omega \right) }^{2} \\
&\lesssim &\sum_{i=1}^{\infty }\left( \frac{1}{\left( \left\vert
B_{i}\right\vert ^{\frac{1}{n}}+\limfunc{dist}\left( B_{i},I\right) \right)
^{n+1-\alpha }}\left\vert I\right\vert _{\sigma }\right) ^{2}\sum_{J:\
J\subset B_{i}}\mathbf{\tau \;}\left\vert J\right\vert ^{\frac{2}{n}%
}\left\vert J\right\vert _{\omega } \\
&\lesssim &\sum_{i=1}^{\infty }\left( \frac{1}{\left( \left\vert
B_{i}\right\vert ^{\frac{1}{n}}+\limfunc{dist}\left( B_{i},I\right) \right)
^{n+1-\alpha }}\left\vert I\right\vert _{\sigma }\right) ^{2}\mathbf{\tau \ }%
\left\vert B_{i}\right\vert ^{\frac{2}{n}}\left\vert B_{i}\right\vert
_{\omega } \\
&\lesssim &\mathbf{\tau }\left\{ \sum_{i=1}^{\infty }\frac{\left\vert
B_{i}\right\vert _{\omega }\left\vert I\right\vert _{\sigma }}{\left\vert
B_{i}\right\vert ^{2\left( 1-\frac{\alpha }{n}\right) }}\right\} \left\vert
I\right\vert _{\sigma }\ ,
\end{eqnarray*}%
and 
\begin{eqnarray*}
\sum_{i=1}^{\infty }\frac{\left\vert B_{i}\right\vert _{\omega }\left\vert
I\right\vert _{\sigma }}{\left\vert B_{i}\right\vert ^{2\left( 1-\frac{%
\alpha }{n}\right) }} &=&\frac{\left\vert I\right\vert _{\sigma }}{%
\left\vert I\right\vert ^{1-\frac{\alpha }{n}}}\sum_{i=1}^{\infty }\frac{%
\left\vert I\right\vert ^{1-\frac{\alpha }{n}}}{\left\vert B_{i}\right\vert
^{2\left( 1-\frac{\alpha }{n}\right) }}\left\vert B_{i}\right\vert _{\omega }
\\
&\approx &\frac{\left\vert I\right\vert _{\sigma }}{\left\vert I\right\vert
^{1-\frac{\alpha }{n}}}\sum_{i=1}^{\infty }\int_{B_{i}}\frac{\left\vert
I\right\vert ^{1-\frac{\alpha }{n}}}{\limfunc{quasidist}\left( x,I\right)
^{2\left( n-\alpha \right) }}d\omega \left( x\right) \\
&\approx &\frac{\left\vert I\right\vert _{\sigma }}{\left\vert I\right\vert
^{1-\frac{\alpha }{n}}}\sum_{i=1}^{\infty }\int_{B_{i}}\left( \frac{%
\left\vert I\right\vert ^{\frac{1}{n}}}{\left[ \left\vert I\right\vert ^{%
\frac{1}{n}}+\limfunc{quasidist}\left( x,I\right) \right] ^{2}}\right)
^{n-\alpha }d\omega \left( x\right) \\
&\leq &\frac{\left\vert I\right\vert _{\sigma }}{\left\vert I\right\vert ^{1-%
\frac{\alpha }{n}}}\mathcal{P}^{\alpha }\left( I,\omega \right) \leq 
\mathcal{A}_{2}^{\alpha ,\ast },
\end{eqnarray*}%
we obtain%
\begin{equation*}
C_{1}\lesssim \mathbf{\tau }\mathcal{A}_{2}^{\alpha ,\ast }\left\vert
I\right\vert _{\sigma }
\end{equation*}

Next we turn to estimating term $C_{2}$ where the triple of $J$ contains $I$
but $J$ itself does not. Note that there are at most $2^{n}$ such quasicubes 
$J$ of a given side length, one in each `generalized octant' relative to $I$%
. So with this in mind we sum over the quasicubes $J$ according to their
lengths to obtain%
\begin{eqnarray*}
C_{2} &=&\sum_{m=0}^{\infty }\sum_{\substack{ J:\ I\subset 3J\setminus J  \\ %
\ell \left( J\right) =2^{m}\ell \left( I\right) }}\left( \frac{\left\vert
J\right\vert ^{\frac{1}{n}}}{\left( \left\vert J\right\vert ^{\frac{1}{n}}+%
\limfunc{dist}\left( J,I\right) \right) ^{n+1-\alpha }}\left\vert
I\right\vert _{\sigma }\right) ^{2}\sum_{\substack{ F\in \mathcal{F}:  \\ %
J\in \mathcal{M}_{\mathbf{r}-\limfunc{deep}}\left( F\right) }}\left\Vert 
\mathsf{P}_{F,J}^{\omega }\frac{\mathbf{x}}{\left\vert J\right\vert ^{\frac{1%
}{n}}}\right\Vert _{L^{2}\left( \omega \right) }^{2} \\
&\lesssim &\sum_{m=0}^{\infty }\left( \frac{\left\vert I\right\vert _{\sigma
}}{\left\vert 2^{m}I\right\vert ^{1-\frac{\alpha }{n}}}\right) ^{2}\mathbf{%
\tau \ }\left\vert 3\cdot 2^{m}I\right\vert _{\omega }=\mathbf{\tau }\left\{ 
\frac{\left\vert I\right\vert _{\sigma }}{\left\vert I\right\vert ^{1-\frac{%
\alpha }{n}}}\sum_{m=0}^{\infty }\frac{\left\vert I\right\vert ^{1-\frac{%
\alpha }{n}}\left\vert 3\cdot 2^{m}I\right\vert _{\omega }}{\left\vert
2^{m}I\right\vert ^{2\left( 1-\frac{\alpha }{n}\right) }}\right\} \left\vert
I\right\vert _{\sigma } \\
&\lesssim &\mathbf{\tau }\left\{ \frac{\left\vert I\right\vert _{\sigma }}{%
\left\vert I\right\vert ^{1-\frac{\alpha }{n}}}\mathcal{P}^{\alpha }\left(
I,\omega \right) \right\} \left\vert I\right\vert _{\sigma }\leq \mathbf{%
\tau }\mathcal{A}_{2}^{\alpha ,\ast }\left\vert I\right\vert _{\sigma },
\end{eqnarray*}%
since in analogy with the corresponding estimate above,%
\begin{equation*}
\sum_{m=0}^{\infty }\frac{\left\vert I\right\vert ^{1-\frac{\alpha }{n}%
}\left\vert 3\cdot 2^{n}I\right\vert _{\omega }}{\left\vert
2^{m}I\right\vert ^{2\left( 1-\frac{\alpha }{n}\right) }}=\int
\sum_{m=0}^{\infty }\frac{\left\vert I\right\vert ^{1-\frac{\alpha }{n}}}{%
\left\vert 2^{m}I\right\vert ^{2\left( 1-\frac{\alpha }{n}\right) }}\mathbf{1%
}_{3\cdot 2^{m}I}\left( x\right) \ d\omega \left( x\right) \lesssim \mathcal{%
P}^{\alpha }\left( I,\omega \right) .
\end{equation*}

Finally, we turn to term $D$, which is handled in the same way as term $%
C_{2} $. The quasicubes $J$ occurring here are included in the set of
ancestors $A_{k}\equiv \pi _{\Omega \mathcal{D}}^{\left( k\right) }I$ of $I$%
, $1\leq k<\infty $. We thus have from Lemma \ref{tau ovelap} again,%
\begin{eqnarray*}
D &=&\sum_{k=1}^{\infty }\mathbb{P}^{\alpha }\left( \mathbf{1}_{I}\sigma
\right) \left( c\left( A_{k}\right) ,\left\vert A_{k}\right\vert ^{\frac{1}{n%
}}\right) ^{2}\sum_{\substack{ F\in \mathcal{F}:  \\ A_{k}\in \mathcal{M}_{%
\mathbf{r}-\limfunc{deep}}\left( F\right) }}\left\Vert \mathsf{P}%
_{F,A_{k}}^{\omega }\frac{\mathbf{x}}{\lvert A_{k}\rvert ^{\frac{1}{n}}}%
\right\Vert _{L^{2}\left( \omega \right) }^{2} \\
&\lesssim &\sum_{k=1}^{\infty }\left( \frac{\left\vert I\right\vert _{\sigma
}\left\vert A_{k}\right\vert ^{\frac{1}{n}}}{\left\vert A_{k}\right\vert ^{1+%
\frac{1-\alpha }{n}}}\right) ^{2}\mathbf{\tau \;}\left\vert A_{k}\right\vert
_{\omega }=\mathbf{\tau }\left\{ \frac{\left\vert I\right\vert _{\sigma }}{%
\left\vert I\right\vert ^{1-\frac{\alpha }{n}}}\sum_{k=1}^{\infty }\frac{%
\left\vert I\right\vert ^{1-\frac{\alpha }{n}}}{\left\vert A_{k}\right\vert
^{2\left( 1-\frac{\alpha }{n}\right) }}\left\vert A_{k}\right\vert _{\omega
}\right\} \left\vert I\right\vert _{\sigma } \\
&\lesssim &\left\{ \frac{\left\vert I\right\vert _{\sigma }}{\left\vert
I\right\vert ^{1-\frac{\alpha }{n}}}\mathcal{P}^{\alpha }\left( I,\omega
\right) \right\} \left\vert I\right\vert _{\sigma }\lesssim \mathcal{A}%
_{2}^{\alpha ,\ast }\left\vert I\right\vert _{\sigma },
\end{eqnarray*}%
since%
\begin{eqnarray*}
\sum_{k=1}^{\infty }\frac{\left\vert I\right\vert ^{1-\frac{\alpha }{n}}}{%
\left\vert A_{k}\right\vert ^{2\left( 1-\frac{\alpha }{n}\right) }}%
\left\vert A_{k}\right\vert _{\omega } &=&\int \sum_{k=1}^{\infty }\frac{%
\left\vert I\right\vert ^{1-\frac{\alpha }{n}}}{\left\vert A_{k}\right\vert
^{2\left( 1-\frac{\alpha }{n}\right) }}\mathbf{1}_{A_{k}\left( x\right)
}d\omega \left( x\right) \\
&=&\int \sum_{k=1}^{\infty }\frac{1}{2^{2\left( 1-\frac{\alpha }{n}\right) k}%
}\frac{\left\vert I\right\vert ^{1-\frac{\alpha }{n}}}{\left\vert
I\right\vert ^{2\left( 1-\frac{\alpha }{n}\right) }}\mathbf{1}_{A_{k}\left(
x\right) }d\omega \left( x\right) \\
&\lesssim &\int \left( \frac{\left\vert I\right\vert ^{\frac{1}{n}}}{\left[
\left\vert I\right\vert ^{\frac{1}{n}}+\limfunc{quasidist}\left( x,I\right) %
\right] ^{2}}\right) ^{n-\alpha }d\omega \left( x\right) =\mathcal{P}%
^{\alpha }\left( I,\omega \right) .
\end{eqnarray*}

\subsection{The dual Poisson testing inequality}

Again we split the integration on the left side of (\ref{e.t2 n}) into local
and global parts:%
\begin{equation}
\int_{\mathbb{R}}[\mathbb{P}^{\alpha \ast }(t\mathbf{1}_{\widehat{I}}\mu
)]^{2}\sigma =\int_{I}[\mathbb{P}^{\alpha \ast }(t\mathbf{1}_{\widehat{I}%
}\mu )]^{2}\sigma +\int_{\mathbb{R}\setminus I}[\mathbb{P}^{\alpha \ast }(t%
\mathbf{1}_{\widehat{I}}\mu )]^{2}\sigma \equiv \mathbf{Local}\left(
I\right) +\mathbf{Global}\left( I\right) .  \label{loc and glo}
\end{equation}%
Here is a brief schematic diagram of the decompositions, with bounds in $%
\fbox{}$, used in this subsection:%
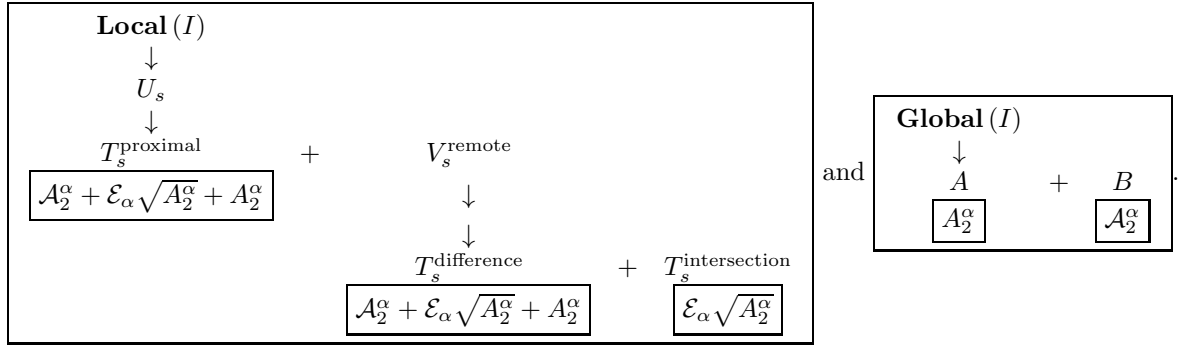
\begin{equation*}
\fbox{$%
\begin{array}{ccccc}
\mathbf{Local}\left( I\right) &  &  &  &  \\ 
\downarrow &  &  &  &  \\ 
U_{s} &  &  &  &  \\ 
\downarrow &  &  &  &  \\ 
T_{s}^{\limfunc{proximal}} & + & V_{s}^{\limfunc{remote}} &  &  \\ 
\fbox{$\mathcal{A}_{2}^{\alpha }+\mathcal{E}_{\alpha }\sqrt{A_{2}^{\alpha }}%
+A_{2}^{\alpha }$} &  & \downarrow &  &  \\ 
&  & \downarrow &  &  \\ 
&  & T_{s}^{\limfunc{difference}} & + & T_{s}^{\limfunc{intersection}} \\ 
&  & \fbox{$\mathcal{A}_{2}^{\alpha }+\mathcal{E}_{\alpha }\sqrt{%
A_{2}^{\alpha }}+A_{2}^{\alpha }$} &  & \fbox{$\mathcal{E}_{\alpha }\sqrt{%
A_{2}^{\alpha }}$}%
\end{array}%
$}\text{ and }\fbox{$%
\begin{array}{ccc}
\mathbf{Global}\left( I\right) &  &  \\ 
\downarrow &  &  \\ 
A & + & B \\ 
\fbox{$A_{2}^{\alpha }$} &  & \fbox{$\mathcal{A}_{2}^{\alpha }$}%
\end{array}%
$}.
\end{equation*}%
We begin with the local part $\mathbf{Local}\left( I\right) $. Note that the
right hand side of (\ref{e.t2 n}) is 
\begin{equation}
\int_{\widehat{I}}t^{2}d\mu =\sum_{F\in \mathcal{F}}\sum_{\substack{ J\in 
\mathcal{M}_{\mathbf{r}-\limfunc{deep}}\left( F\right)  \\ J\subset I}}%
\lVert \mathsf{P}_{F,J}^{\omega }\mathbf{x}\rVert _{L^{2}\left( \omega
\right) }^{2}\,,  \label{mu I hat}
\end{equation}%
We now compute 
\begin{equation}
\mathbb{P}^{\alpha \ast }\left( t\mathbf{1}_{\widehat{I}}\mu \right) \left(
y\right) =\sum_{F\in \mathcal{F}}\sum_{\substack{ J\in \mathcal{M}_{\mathbf{r%
}-\limfunc{deep}}\left( F\right)  \\ J\subset I}}\frac{\lVert \mathsf{P}%
_{F,J}^{\omega }\mathbf{x}\rVert _{L^{2}\left( \omega \right) }^{2}}{\left(
\left\vert J\right\vert ^{\frac{1}{n}}+\left\vert y-c_{J}\right\vert \right)
^{n+1-\alpha }},  \label{PI hat}
\end{equation}%
and then expand the square and integrate to obtain that the local term $%
\mathbf{Local}\left( I\right) $ is 
\begin{equation*}
\sum_{\substack{ F\in \mathcal{F}  \\ J\in \mathcal{M}_{\mathbf{r}-\limfunc{%
deep}}\left( F\right)  \\ J\subset I}}\sum_{\substack{ F^{\prime }\in 
\mathcal{F}:  \\ J^{\prime }\in \mathcal{M}_{\mathbf{r}-\limfunc{deep}%
}\left( F^{\prime }\right)  \\ J^{\prime }\subset I}}\int_{I}\frac{%
\left\Vert \mathsf{P}_{F,J}^{\omega }\mathbf{x}\right\Vert _{L^{2}\left(
\omega \right) }^{2}}{\left( \left\vert J\right\vert ^{\frac{1}{n}%
}+\left\vert y-c_{J}\right\vert \right) ^{n+1-\alpha }}\frac{\left\Vert 
\mathsf{P}_{F^{\prime },J^{\prime }}^{\omega }\mathbf{x}\right\Vert
_{L^{2}\left( \omega \right) }^{2}}{\left( \left\vert J^{\prime }\right\vert
^{\frac{1}{n}}+\left\vert y-c_{J^{\prime }}\right\vert \right) ^{n+1-\alpha }%
}d\sigma \left( y\right) .
\end{equation*}%
By symmetry we may assume that $\ell \left( J^{\prime }\right) \leq \ell
\left( J\right) $. We fix an integer $s$, and consider those quasicubes $J$
and $J^{\prime }$ with $\ell \left( J^{\prime }\right) =2^{-s}\ell \left(
J\right) $. For fixed $s$ we will control the expression 
\begin{eqnarray*}
U_{s} &\equiv &\sum_{\substack{ F,F^{\prime }\in \mathcal{F}}}\sum 
_{\substack{ J\in \mathcal{M}_{\mathbf{r}-\limfunc{deep}}\left( F\right) ,\
J^{\prime }\in \mathcal{M}_{\mathbf{r}-\limfunc{deep}}\left( F^{\prime
}\right)  \\ J,J^{\prime }\subset I,\ \ell \left( J^{\prime }\right)
=2^{-s}\ell \left( J\right) }} \\
&&\times \int_{I}\frac{\left\Vert \mathsf{P}_{F,J}^{\omega }\mathbf{x}%
\right\Vert _{L^{2}\left( \omega \right) }^{2}}{\left( \left\vert
J\right\vert ^{\frac{1}{n}}+\left\vert y-c_{J}\right\vert \right)
^{n+1-\alpha }}\frac{\left\Vert \mathsf{P}_{F^{\prime },J^{\prime }}^{\omega
}\mathbf{x}\right\Vert _{L^{2}\left( \omega \right) }^{2}}{\left( \left\vert
J^{\prime }\right\vert ^{\frac{1}{n}}+\left\vert y-c_{J^{\prime
}}\right\vert \right) ^{n+1-\alpha }}d\sigma \left( y\right) ,
\end{eqnarray*}%
by proving that%
\begin{equation}
U_{s}\lesssim 2^{-\varepsilon s}\left( \mathcal{A}_{2}^{\alpha }+\mathcal{E}%
_{\alpha }\sqrt{A_{2}^{\alpha }}\right) .  \label{Us bound}
\end{equation}%
With this accomplished, we can sum in $s\geq 0$ to control the local term $%
\mathbf{Local}\left( I\right) $.

Our first decomposition is to write%
\begin{equation}
U_{s}=T_{s}^{\limfunc{proximal}}+V_{s}^{\limfunc{remote}}\ ,
\label{initial decomp}
\end{equation}%
where we fix $\varepsilon >0$ to be chosen later ($\varepsilon =\frac{1}{2n}$
works), and in the `proximal' term $T_{s}^{\limfunc{proximal}}$ we restrict
the summation over pairs of quasicubes $J,J^{\prime }$ to those satisfying $%
\left\vert c\left( J\right) -c\left( J^{\prime }\right) \right\vert _{%
\limfunc{quasi}}<2^{s\varepsilon }\ell \left( J\right) $; while in the
`remote' term $V_{s}^{\limfunc{remote}}$ we restrict the summation over
pairs of quasicubes $J,J^{\prime }$ to those satisfying the opposite
inequality $\left\vert c\left( J\right) -c\left( J^{\prime }\right)
\right\vert _{\limfunc{quasi}}\geq 2^{s\varepsilon }\ell \left( J\right) $.
Then we further decompose 
\begin{equation*}
V_{s}^{\limfunc{remote}}=T_{s}^{\limfunc{difference}}+T_{s}^{\limfunc{%
intersection}},
\end{equation*}%
where in the `difference' term $T_{s}^{\limfunc{difference}}$ we restict
integration in $y$ to the difference $I\setminus B\left( J,J^{\prime
}\right) $ of $I$ and 
\begin{equation*}
B\left( J,J^{\prime }\right) \equiv B\left( c_{J},\frac{1}{2}\left\vert
c_{J}-c_{J^{\prime }}\right\vert _{\limfunc{quasi}}\right) ,
\end{equation*}%
the ball centered at $c_{J}$ with radius $\frac{1}{2}\left\vert
c_{J}-c_{J^{\prime }}\right\vert _{\limfunc{quasi}}$; while in the
`intersection' term $T_{s}^{\limfunc{intersection}}$ we restict integration
in $y$ to the intersection $I\cap B\left( J,J^{\prime }\right) $ of $I$ with
the ball $B\left( J,J^{\prime }\right) $; i.e. 
\begin{eqnarray*}
T_{s}^{\limfunc{intersection}} &\equiv &\sum_{\substack{ F,F^{\prime }\in 
\mathcal{F}}}\sum_{\substack{ J\in \mathcal{M}_{\mathbf{r}-\limfunc{deep}%
}\left( F\right) ,\ J^{\prime }\in \mathcal{M}_{\mathbf{r}-\limfunc{deep}%
}\left( F^{\prime }\right)  \\ J,J^{\prime }\subset I,\ \ell \left(
J^{\prime }\right) =2^{-s}\ell \left( J\right)  \\ \left\vert c\left(
J\right) -c\left( J^{\prime }\right) \right\vert _{\limfunc{quasi}}\geq
2^{s\left( 1+\varepsilon \right) }\ell \left( J^{\prime }\right) }} \\
&&\times \int_{I\cap B\left( J,J^{\prime }\right) }\frac{\left\Vert \mathsf{P%
}_{F,J}^{\omega }\mathbf{x}\right\Vert _{L^{2}\left( \omega \right) }^{2}}{%
\left( \left\vert J\right\vert ^{\frac{1}{n}}+\left\vert y-c_{J}\right\vert
\right) ^{n+1-\alpha }}\frac{\left\Vert \mathsf{P}_{F^{\prime },J^{\prime
}}^{\omega }\mathbf{x}\right\Vert _{L^{2}\left( \omega \right) }^{2}}{\left(
\left\vert J^{\prime }\right\vert ^{\frac{1}{n}}+\left\vert y-c_{J^{\prime
}}\right\vert \right) ^{n+1-\alpha }}d\sigma \left( y\right) ,
\end{eqnarray*}%
We will exploit the restriction of integration to $I\cap B\left( J,J^{\prime
}\right) $, together with the condition 
\begin{equation*}
\left\vert c_{J}-c_{J^{\prime }}\right\vert _{\limfunc{quasi}}\geq
2^{s\left( 1+\varepsilon \right) }\ell \left( J^{\prime }\right)
=2^{s\varepsilon }\ell \left( J\right) ,
\end{equation*}%
in establishing (\ref{important}) below, which will then give an estimate
for the term $T_{s}^{\limfunc{intersection}}$ using an argument dual to that
used for the other terms $T_{s}^{\limfunc{proximal}}$ and $T_{s}^{\limfunc{%
difference}}$. We now turn to estimating the proximal and difference terms.

\subsubsection{The proximal and difference terms}

We have using (\ref{mu I hat}) that%
\begin{align*}
T_{s}^{\limfunc{proximal}}& \equiv \sum_{\substack{ F,F^{\prime }\in 
\mathcal{F}}}\sum_{\substack{ J\in \mathcal{M}_{\mathbf{r}-\limfunc{deep}%
}\left( F\right) ,\ J^{\prime }\in \mathcal{M}_{\mathbf{r}-\limfunc{deep}%
}\left( F^{\prime }\right)  \\ J,J^{\prime }\subset I,\ \ell \left(
J^{\prime }\right) =2^{-s}\ell \left( J\right) \text{ and }\left\vert
c_{J}-c_{J^{\prime }}\right\vert _{\limfunc{quasi}}<2^{s\varepsilon }\ell
\left( J\right) }} \\
& \times \int_{I}\frac{\left\Vert \mathsf{P}_{F,J}^{\omega }\mathbf{x}%
\right\Vert _{L^{2}\left( \omega \right) }^{2}}{\left( \left\vert
J\right\vert ^{\frac{1}{n}}+\left\vert y-c_{J}\right\vert \right)
^{n+1-\alpha }}\frac{\left\Vert \mathsf{P}_{F^{\prime },J^{\prime }}^{\omega
}\mathbf{x}\right\Vert _{L^{2}\left( \omega \right) }^{2}}{\left( \left\vert
J^{\prime }\right\vert ^{\frac{1}{n}}+\left\vert y-c_{J^{\prime
}}\right\vert \right) ^{n+1-\alpha }}d\sigma \left( y\right) \\
& \leq M_{s}^{\limfunc{proximal}}\sum_{F\in \mathcal{F}}\sum_{\substack{ %
J\in \mathcal{M}_{\mathbf{r}-\limfunc{deep}}\left( F\right)  \\ J\subset I}}%
\lVert \mathsf{P}_{F,J}^{\omega }z\rVert _{\omega }^{2}=M_{s}^{\limfunc{%
proximal}}\int_{\widehat{I}}t^{2}d\mu ,
\end{align*}%
where%
\begin{align*}
M_{s}^{\limfunc{proximal}}& \equiv \sup_{F\in \mathcal{F}}\sup_{\substack{ %
J\in \mathcal{M}_{\mathbf{r}-\limfunc{deep}}\left( F\right) }}A_{s}^{%
\limfunc{proximal}}\left( J\right) ; \\
A_{s}^{\limfunc{proximal}}\left( J\right) & \equiv \sum_{F^{\prime }\in 
\mathcal{F}}\sum_{\substack{ J^{\prime }\in \mathcal{M}_{\mathbf{r}-\limfunc{%
deep}}\left( F^{\prime }\right)  \\ J^{\prime }\subset I,\ \ell \left(
J^{\prime }\right) =2^{-s}\ell \left( J\right) \text{ and }\left\vert
c_{J}-c_{J^{\prime }}\right\vert _{\limfunc{quasi}}<2^{s\varepsilon }\ell
\left( J\right) }}\int_{I}S_{\left( J^{\prime },J\right) }^{F^{\prime
}}\left( y\right) d\sigma \left( y\right) ; \\
S_{\left( J^{\prime },J\right) }^{F^{\prime }}\left( x\right) & \equiv \frac{%
1}{\left( \left\vert J\right\vert ^{\frac{1}{n}}+\left\vert
y-c_{J}\right\vert \right) ^{n+1-\alpha }}\frac{\left\Vert \mathsf{P}%
_{F^{\prime },J^{\prime }}^{\omega }\mathbf{x}\right\Vert _{L^{2}\left(
\omega \right) }^{2}}{\left( \left\vert J^{\prime }\right\vert ^{\frac{1}{n}%
}+\left\vert y-c_{J^{\prime }}\right\vert \right) ^{n+1-\alpha }},
\end{align*}%
and similarly%
\begin{align*}
T_{s}^{\limfunc{difference}}& \equiv \sum_{\substack{ F,F^{\prime }\in 
\mathcal{F}}}\sum_{\substack{ J\in \mathcal{M}_{\mathbf{r}-\limfunc{deep}%
}\left( F\right) ,\ J^{\prime }\in \mathcal{M}_{\mathbf{r}-\limfunc{deep}%
}\left( F^{\prime }\right)  \\ J,J^{\prime }\subset I,\ \ell \left(
J^{\prime }\right) =2^{-s}\ell \left( J\right) \text{ and }\left\vert
c_{J}-c_{J^{\prime }}\right\vert _{\limfunc{quasi}}\geq 2^{s\varepsilon
}\ell \left( J\right) }} \\
& \times \int_{I\setminus B\left( J,J^{\prime }\right) }\frac{\left\Vert 
\mathsf{P}_{F,J}^{\omega }\mathbf{x}\right\Vert _{L^{2}\left( \omega \right)
}^{2}}{\left( \left\vert J\right\vert ^{\frac{1}{n}}+\left\vert
y-c_{J}\right\vert \right) ^{n+1-\alpha }}\frac{\left\Vert \mathsf{P}%
_{F^{\prime },J^{\prime }}^{\omega }\mathbf{x}\right\Vert _{L^{2}\left(
\omega \right) }^{2}}{\left( \left\vert J^{\prime }\right\vert ^{\frac{1}{n}%
}+\left\vert y-c_{J^{\prime }}\right\vert \right) ^{n+1-\alpha }}d\sigma
\left( y\right) \\
& \leq M_{s}^{\limfunc{difference}}\sum_{F\in \mathcal{F}}\sum_{\substack{ %
J\in \mathcal{M}_{\mathbf{r}-\limfunc{deep}}\left( F\right)  \\ J\subset I}}%
\lVert \mathsf{P}_{F,J}^{\omega }z\rVert _{\omega }^{2}=M_{s}^{\limfunc{%
difference}}\int_{\widehat{I}}t^{2}d\mu ,
\end{align*}%
where%
\begin{eqnarray*}
M_{s}^{\limfunc{difference}} &\equiv &\sup_{F\in \mathcal{F}}\sup_{\substack{
J\in \mathcal{M}_{\mathbf{r}-\limfunc{deep}}\left( F\right) }}A_{s}^{\func{%
difference}}\left( J\right) ; \\
A_{s}^{\limfunc{difference}}\left( J\right) &\equiv &\sum_{F^{\prime }\in 
\mathcal{F}}\sum_{\substack{ J^{\prime }\in \mathcal{M}_{\mathbf{r}-\limfunc{%
deep}}\left( F^{\prime }\right)  \\ J^{\prime }\subset I,\ \ell \left(
J^{\prime }\right) =2^{-s}\ell \left( J\right) \text{ and }\left\vert
c_{J}-c_{J^{\prime }}\right\vert _{\limfunc{quasi}}\geq 2^{s\varepsilon
}\ell \left( J\right) }}\int_{I\setminus B\left( J,J^{\prime }\right)
}S_{\left( J^{\prime },J\right) }^{F^{\prime }}\left( y\right) d\sigma
\left( y\right) .
\end{eqnarray*}%
The restriction of integration in $A_{s}^{\limfunc{difference}}$ to $%
I\setminus B\left( J,J^{\prime }\right) $ will be used to establish (\ref%
{vanishing close}) below.

\begin{notation}
Since the quasicubes$F,J,F^{\prime },J^{\prime }$ that arise in all of the
sums here satisfy 
\begin{equation*}
J\in \mathcal{M}_{\mathbf{r}-\limfunc{deep}}\left( F\right) ,J^{\prime }\in 
\mathcal{M}_{\mathbf{r}-\limfunc{deep}}\left( F^{\prime }\right) \text{ and }%
\ell \left( J^{\prime }\right) =2^{-s}\ell \left( J\right) ,
\end{equation*}%
we will often employ the notation $\overset{\ast }{\sum }$ to remind the
reader that, as applicable, these three conditions are in force even when
they are\ not explictly mentioned.
\end{notation}

Now fix $J$ as in $M_{s}^{\limfunc{proximal}}$ respectively $M_{s}^{\limfunc{%
difference}}$, and decompose the sum over $J^{\prime }$ in $A_{s}^{\limfunc{%
proximal}}\left( J\right) $ respectively $A_{s}^{\limfunc{difference}}\left(
J\right) $ by%
\begin{eqnarray*}
A_{s}^{\limfunc{proximal}}\left( J\right) &=&\sum_{F^{\prime }\in \mathcal{F}%
}\sum_{\substack{ J^{\prime }\in \mathcal{M}_{\mathbf{r}-\limfunc{deep}%
}\left( F^{\prime }\right)  \\ J^{\prime }\subset I,\ \ell \left( J^{\prime
}\right) =2^{-s}\ell \left( J\right) \text{ and }\left\vert
c_{J}-c_{J^{\prime }}\right\vert _{\limfunc{quasi}}<2^{s\varepsilon }\ell
\left( J\right) }}\int_{I}S_{\left( J^{\prime },J\right) }^{F^{\prime
}}\left( y\right) d\sigma \left( y\right) \\
&=&\sum_{F^{\prime }\in \mathcal{F}}\overset{\ast }{\sum_{\substack{ %
c_{J^{\prime }}\in 2J  \\ \left\vert c_{J}-c_{J^{\prime }}\right\vert _{%
\limfunc{quasi}}<2^{s\varepsilon }\ell \left( J\right) }}}\int_{I}S_{\left(
J^{\prime },J\right) }^{F^{\prime }}\left( y\right) d\sigma \left( y\right)
+\sum_{F^{\prime }\in \mathcal{F}}\sum_{\ell =1}^{\infty }\overset{\ast }{%
\sum_{\substack{ c_{J^{\prime }}\in 2^{\ell +1}J\setminus 2^{\ell }J  \\ %
\left\vert c_{J}-c_{J^{\prime }}\right\vert _{\limfunc{quasi}%
}<2^{s\varepsilon }\ell \left( J\right) }}}\int_{I}S_{\left( J^{\prime
},J\right) }^{F^{\prime }}\left( y\right) d\sigma \left( y\right) \\
&\equiv &\sum_{\ell =0}^{\infty }A_{s}^{\limfunc{proximal},\ell }\left(
J\right) ,
\end{eqnarray*}%
respectively%
\begin{eqnarray*}
A_{s}^{\limfunc{difference}}\left( J\right) &=&\sum_{F^{\prime }\in \mathcal{%
F}}\sum_{\substack{ J^{\prime }\in \mathcal{M}_{\mathbf{r}-\limfunc{deep}%
}\left( F^{\prime }\right)  \\ J^{\prime }\subset I,\ \ell \left( J^{\prime
}\right) =2^{-s}\ell \left( J\right) \text{ and }\left\vert
c_{J}-c_{J^{\prime }}\right\vert _{\limfunc{quasi}}\geq 2^{s\varepsilon
}\ell \left( J\right) }}\int_{I\setminus B\left( J,J^{\prime }\right)
}S_{\left( J^{\prime },J\right) }^{F^{\prime }}\left( y\right) d\sigma
\left( y\right) \\
&=&\sum_{F^{\prime }\in \mathcal{F}}\overset{\ast }{\sum_{\substack{ %
c_{J^{\prime }}\in 2J  \\ \left\vert c_{J}-c_{J^{\prime }}\right\vert _{%
\limfunc{quasi}}\geq 2^{s\varepsilon }\ell \left( J\right) }}}%
\int_{I\setminus B\left( J,J^{\prime }\right) }S_{\left( J^{\prime
},J\right) }^{F^{\prime }}\left( y\right) d\sigma \left( y\right) \\
&&+\sum_{\ell =1}^{\infty }\sum_{F^{\prime }\in \mathcal{F}}\overset{\ast }{%
\sum_{\substack{ c_{J^{\prime }}\in 2^{\ell +1}J\setminus 2^{\ell }J  \\ %
\left\vert c_{J}-c_{J^{\prime }}\right\vert _{\limfunc{quasi}}\geq
2^{s\varepsilon }\ell \left( J\right) }}}\int_{I\setminus B\left(
J,J^{\prime }\right) }S_{\left( J^{\prime },J\right) }^{F^{\prime }}\left(
y\right) d\sigma \left( y\right) \\
&\equiv &\sum_{\ell =0}^{\infty }A_{s}^{\limfunc{difference},\ell }\left(
J\right) .
\end{eqnarray*}%
Let $m$ be the smallest integer for which 
\begin{equation}
2^{-m}\sqrt{n}\leq \frac{1}{3}.  \label{smallest m}
\end{equation}%
Now decompose the integrals over $I$ in $A_{s}^{\limfunc{proximal},\ell
}\left( J\right) $ by%
\begin{eqnarray*}
A_{s}^{\limfunc{proximal},0}\left( J\right) &=&\sum_{F^{\prime }\in \mathcal{%
F}}\overset{\ast }{\sum_{\substack{ c_{J^{\prime }}\in 2J  \\ \left\vert
c_{J}-c_{J^{\prime }}\right\vert _{\limfunc{quasi}}<2^{s\varepsilon }\ell
\left( J\right) }}}\int_{I\setminus 4J}S_{\left( J^{\prime },J\right)
}^{F^{\prime }}\left( y\right) d\sigma \left( y\right) \\
&&+\sum_{F^{\prime }\in \mathcal{F}}\overset{\ast }{\sum_{\substack{ %
c_{J^{\prime }}\in 2J  \\ \left\vert c_{J}-c_{J^{\prime }}\right\vert _{%
\limfunc{quasi}}<2^{s\varepsilon }\ell \left( J\right) }}}\int_{I\cap
4J}S_{\left( J^{\prime },J\right) }^{F^{\prime }}\left( y\right) d\sigma
\left( y\right) \\
&\equiv &A_{s,far}^{\limfunc{proximal},0}\left( J\right) +A_{s,near}^{%
\limfunc{proximal},0}\left( J\right) , \\
A_{s}^{\limfunc{proximal},\ell }\left( J\right) &=&\sum_{F^{\prime }\in 
\mathcal{F}}\overset{\ast }{\sum_{\substack{ c_{J^{\prime }}\in 2^{\ell
+1}J\setminus 2^{\ell }J  \\ \left\vert c_{J}-c_{J^{\prime }}\right\vert _{%
\limfunc{quasi}}<2^{s\varepsilon }\ell \left( J\right) }}}\int_{I\setminus
2^{\ell +2}J}S_{\left( J^{\prime },J\right) }^{F^{\prime }}\left( y\right)
d\sigma \left( y\right) \\
&&+\sum_{F^{\prime }\in \mathcal{F}}\overset{\ast }{\sum_{\substack{ %
c_{J^{\prime }}\in 2^{\ell +1}J\setminus 2^{\ell }J  \\ \left\vert
c_{J}-c_{J^{\prime }}\right\vert _{\limfunc{quasi}}<2^{s\varepsilon }\ell
\left( J\right) }}}\int_{I\cap \left( 2^{\ell +2}J\setminus 2^{\ell
-m}J\right) }S_{\left( J^{\prime },J\right) }^{F^{\prime }}\left( y\right)
d\sigma \left( y\right) \\
&&+\sum_{F^{\prime }\in \mathcal{F}}\overset{\ast }{\sum_{\substack{ %
c_{J^{\prime }}\in 2^{\ell +1}J\setminus 2^{\ell }J  \\ \left\vert
c_{J}-c_{J^{\prime }}\right\vert _{\limfunc{quasi}}<2^{s\varepsilon }\ell
\left( J\right) }}}\int_{I\cap 2^{\ell -m}J}S_{\left( J^{\prime },J\right)
}^{F^{\prime }}\left( y\right) d\sigma \left( y\right) \\
&\equiv &A_{s,far}^{\limfunc{proximal},\ell }\left( J\right) +A_{s,near}^{%
\limfunc{proximal},\ell }\left( J\right) +A_{s,close}^{\limfunc{proximal}%
,\ell }\left( J\right) ,\ \ \ \ \ \ell \geq 1.
\end{eqnarray*}%
Similarly we decompose the integrals over $I^{\ast }\equiv I\setminus
B\left( J,J^{\prime }\right) $ in $A_{s}^{\limfunc{difference},\ell }\left(
J\right) $ by%
\begin{eqnarray*}
A_{s}^{\limfunc{difference},0}\left( J\right) &=&\sum_{F^{\prime }\in 
\mathcal{F}}\overset{\ast }{\sum_{\substack{ c_{J^{\prime }}\in 2J  \\ %
\left\vert c_{J}-c_{J^{\prime }}\right\vert _{\limfunc{quasi}}\geq
2^{s\varepsilon }\ell \left( J\right) }}}\int_{I^{\ast }\setminus
4J}S_{\left( J^{\prime },J\right) }^{F^{\prime }}\left( y\right) d\sigma
\left( y\right) \\
&&+\sum_{F^{\prime }\in \mathcal{F}}\overset{\ast }{\sum_{\substack{ %
c_{J^{\prime }}\in 2J  \\ \left\vert c_{J}-c_{J^{\prime }}\right\vert _{%
\limfunc{quasi}}\geq 2^{s\varepsilon }\ell \left( J\right) }}}\int_{I^{\ast
}\cap 4J}S_{\left( J^{\prime },J\right) }^{F^{\prime }}\left( y\right)
d\sigma \left( y\right) \\
&\equiv &A_{s,far}^{\limfunc{difference},0}\left( J\right) +A_{s,near}^{%
\limfunc{difference},0}\left( J\right) , \\
A_{s}^{\limfunc{difference},\ell }\left( J\right) &=&\sum_{F^{\prime }\in 
\mathcal{F}}\overset{\ast }{\sum_{\substack{ c_{J^{\prime }}\in 2^{\ell
+1}J\setminus 2^{\ell }J  \\ \left\vert c_{J}-c_{J^{\prime }}\right\vert _{%
\limfunc{quasi}}\geq 2^{s\varepsilon }\ell \left( J\right) }}}\int_{I^{\ast
}\setminus 2^{\ell +2}J}S_{\left( J^{\prime },J\right) }^{F^{\prime }}\left(
y\right) d\sigma \left( y\right) \\
&&+\sum_{F^{\prime }\in \mathcal{F}}\overset{\ast }{\sum_{\substack{ %
c_{J^{\prime }}\in 2^{\ell +1}J\setminus 2^{\ell }J  \\ \left\vert
c_{J}-c_{J^{\prime }}\right\vert _{\limfunc{quasi}}\geq 2^{s\varepsilon
}\ell \left( J\right) }}}\int_{I^{\ast }\cap \left( 2^{\ell +2}J\setminus
2^{\ell -m}J\right) }S_{\left( J^{\prime },J\right) }^{F^{\prime }}\left(
y\right) d\sigma \left( y\right) \\
&&+\sum_{F^{\prime }\in \mathcal{F}}\overset{\ast }{\sum_{\substack{ %
c_{J^{\prime }}\in 2^{\ell +1}J\setminus 2^{\ell }J  \\ \left\vert
c_{J}-c_{J^{\prime }}\right\vert _{\limfunc{quasi}}\geq 2^{s\varepsilon
}\ell \left( J\right) }}}\int_{I^{\ast }\cap 2^{\ell -m}J}S_{\left(
J^{\prime },J\right) }^{F^{\prime }}\left( y\right) d\sigma \left( y\right)
\\
&\equiv &A_{s,far}^{\limfunc{difference},\ell }\left( J\right) +A_{s,near}^{%
\limfunc{difference},\ell }\left( J\right) +A_{s,close}^{\limfunc{difference}%
,\ell }\left( J\right) ,\ \ \ \ \ \ell \geq 1.
\end{eqnarray*}%
We now note the important point that the close terms $A_{s,close}^{\limfunc{%
proximal},\ell }\left( J\right) $ and $A_{s,close}^{\limfunc{difference}%
,\ell }\left( J\right) $ both $\emph{vanish}$ for $\ell >\varepsilon s$
because of the decomposition (\ref{initial decomp}):%
\begin{equation}
A_{s,close}^{\limfunc{proximal},\ell }\left( J\right) =A_{s,close}^{\limfunc{%
difference},\ell }\left( J\right) =0,\ \ \ \ \ \ell >1+\varepsilon s.
\label{vanishing close}
\end{equation}%
Indeed, if $c_{J^{\prime }}\in 2^{\ell +1}J\setminus 2^{\ell }J$, then we
have%
\begin{equation}
\frac{1}{2}2^{\ell }\ell \left( J\right) \leq \left\vert c_{J}-c_{J^{\prime
}}\right\vert _{\limfunc{quasi}},  \label{distJJ'}
\end{equation}%
and if $\ell >1+\varepsilon s$, then%
\begin{equation*}
\left\vert c_{J}-c_{J^{\prime }}\right\vert _{\limfunc{quasi}}\geq
2^{\varepsilon s}\ell \left( J\right) =2^{\left( 1+\varepsilon \right)
s}\ell \left( J^{\prime }\right) .
\end{equation*}%
It now follows from the definition of $V_{s}$ and $T_{s}$ in (\ref{initial
decomp}), that $A_{s,close}^{\limfunc{proximal},\ell }\left( J\right) =0$,
and so we are left to consider the term $A_{s,close}^{\limfunc{difference}%
,\ell }\left( J\right) $, where the integration is taken over the set $%
I\setminus B\left( J,J^{\prime }\right) $. But we are also restricted in $%
A_{s,close}^{\limfunc{difference},\ell }\left( J\right) $ to integrating
over the quasicube $2^{\ell -m}J$, which is contained in $B\left(
J,J^{\prime }\right) $ by (\ref{distJJ'}). Indeed, the smallest\ ball
centered at $c\left( J\right) $ that contains $2^{\ell -m}J$ has radius $%
\sqrt{n}\frac{1}{2}2^{\ell -m}\ell \left( J\right) $, which by (\ref%
{smallest m}) and (\ref{distJJ'}) is at most $\frac{1}{4}2^{\ell }\ell
\left( J\right) \leq \frac{1}{2}\left\vert c_{J}-c_{J^{\prime }}\right\vert
_{\limfunc{quasi}}$, the radius of $B\left( J,J^{\prime }\right) $. Thus the
range of integration in the term $A_{s,close}^{\limfunc{difference},\ell
}\left( J\right) $ is the empty set, and so $A_{s,close}^{\limfunc{difference%
},\ell }\left( J\right) =0$ as well as $A_{s,close}^{\limfunc{proximal},\ell
}\left( J\right) =0$. This proves (\ref{vanishing close}).

Thus from now on in this subsubsection, we may replace $I\setminus B\left(
J,J^{\prime }\right) $ by $I$ since all the terms are positive, and we treat 
$T_{s}^{\limfunc{proximal}}$ and $T_{s}^{\limfunc{difference}}$ in the same
way now that we know the terms $A_{s,close}^{\limfunc{proximal},\ell }\left(
J\right) $ and $A_{s,close}^{\limfunc{difference},\ell }\left( J\right) $
both vanish for $\ell >1+\varepsilon s$. Thus we will suppress the
superscripts $\limfunc{proximal}$ and $\limfunc{difference}$ in the $far$, $%
near$ and $close$ decomposition of $A_{s,close}^{\limfunc{proximal},\ell
}\left( J\right) $ and $A_{s,close}^{\limfunc{difference},\ell }\left(
J\right) $, and we will also suppress the conditions $\left\vert
c_{J}-c_{J^{\prime }}\right\vert _{\limfunc{quasi}}<2^{s\varepsilon }\ell
\left( J\right) $ and $\left\vert c_{J}-c_{J^{\prime }}\right\vert _{%
\limfunc{quasi}}\geq 2^{s\varepsilon }\ell \left( J\right) $ in the proximal
and difference terms since they no longer play a role. Using the bounded
overlap of the shifted coronas $\mathcal{C}_{F}^{\limfunc{good},\mathbf{\tau 
}-\limfunc{shift}}$, we have $\sum_{F^{\prime }\in \mathcal{F}}\left\Vert 
\mathsf{P}_{F^{\prime },J^{\prime }}^{\omega }\mathbf{x}\right\Vert
_{L^{2}\left( \omega \right) }^{2}\lesssim \mathbf{\tau }\left\vert
J^{\prime }\right\vert ^{\frac{2}{n}}\left\vert J^{\prime }\right\vert
_{\omega }$ and so%
\begin{eqnarray*}
A_{s,far}^{0}\left( J\right) &=&\sum_{F^{\prime }\in \mathcal{F}}\overset{%
\ast }{\sum_{c_{J^{\prime }}\in 2J}}\int_{I\setminus \left( 3J\right)
}S_{\left( J^{\prime },J\right) }^{F^{\prime }}\left( y\right) d\sigma
\left( y\right) \\
&\lesssim &\mathbf{\tau }\sum_{c_{J^{\prime }}\in 2J}\int_{I\setminus \left(
3J\right) }\frac{\left\vert J^{\prime }\right\vert ^{\frac{2}{n}}\left\vert
J^{\prime }\right\vert _{\omega }}{\left( \left\vert J\right\vert ^{\frac{1}{%
n}}+\left\vert y-c_{J}\right\vert \right) ^{2\left( n+1-\alpha \right) }}%
d\sigma \left( y\right) \\
&=&\mathbf{\tau }2^{-2s}\left( \sum_{c_{J^{\prime }}\in 2J}\left\vert
J^{\prime }\right\vert _{\omega }\right) \int_{I\setminus \left( 3J\right) }%
\frac{\left\vert J\right\vert ^{\frac{2}{n}}}{\left( \left\vert J\right\vert
^{\frac{1}{n}}+\left\vert y-c_{J}\right\vert \right) ^{2\left( n+1-\alpha
\right) }}d\sigma \left( y\right) ,
\end{eqnarray*}%
which is dominated by%
\begin{eqnarray*}
&&\mathbf{\tau }2^{-2s}\left\vert 3J\right\vert _{\omega }\int_{I\setminus
\left( 3J\right) }\frac{1}{\left( \left\vert J\right\vert ^{\frac{1}{n}%
}+\left\vert y-c_{J}\right\vert \right) ^{2\left( n-\alpha \right) }}d\sigma
\left( y\right) \\
&\approx &\mathbf{\tau }2^{-2s}\frac{\left\vert 3J\right\vert _{\omega }}{%
\left\vert 4J\right\vert ^{1-\frac{\alpha }{n}}}\int_{I\setminus \left(
3J\right) }\left( \frac{\left\vert J\right\vert ^{\frac{1}{n}}}{\left(
\left\vert J\right\vert ^{\frac{1}{n}}+\left\vert y-c_{J}\right\vert \right)
^{2}}\right) ^{n-\alpha }d\sigma \left( y\right) \\
&\lesssim &\mathbf{\tau }2^{-2s}\frac{\left\vert 3J\right\vert _{\omega }}{%
\left\vert 3J\right\vert ^{1-\frac{\alpha }{n}}}\mathcal{P}^{\alpha }\left(
3J,\sigma \right) \lesssim \mathbf{\tau }2^{-2s}\mathcal{A}_{2}^{\alpha }\ .
\end{eqnarray*}%
To estimate the near term $A_{s,near}^{0}\left( J\right) $, we initially
keep the energy $\left\Vert \mathsf{P}_{F^{\prime },J^{\prime }}^{\omega
}z\right\Vert _{L^{2}\left( \omega \right) }^{2}$ and write 
\begin{eqnarray*}
A_{s,near}^{0}\left( J\right) &=&\sum_{F^{\prime }\in \mathcal{F}}\overset{%
\ast }{\sum_{c_{J^{\prime }}\in 2J}}\int_{I\cap \left( 3J\right) }S_{\left(
J^{\prime },J\right) }^{F^{\prime }}\left( y\right) d\sigma \left( y\right)
\\
&\approx &\sum_{F^{\prime }\in \mathcal{F}}\overset{\ast }{%
\sum_{c_{J^{\prime }}\in 2J}}\int_{I\cap \left( 3J\right) }\frac{1}{%
\left\vert J\right\vert ^{\frac{1}{n}\left( n+1-\alpha \right) }}\frac{%
\left\Vert \mathsf{P}_{F^{\prime },J^{\prime }}^{\omega }\mathbf{x}%
\right\Vert _{L^{2}\left( \omega \right) }^{2}}{\left( \left\vert J^{\prime
}\right\vert ^{\frac{1}{n}}+\left\vert y-c_{J^{\prime }}\right\vert \right)
^{n+1-\alpha }}d\sigma \left( y\right) \\
&=&\sum_{F^{\prime }\in \mathcal{F}}\frac{1}{\left\vert J\right\vert ^{\frac{%
1}{n}\left( n+1-\alpha \right) }}\overset{\ast }{\sum_{c_{J^{\prime }}\in 2J}%
}\left\Vert \mathsf{P}_{F^{\prime },J^{\prime }}^{\omega }\mathbf{x}%
\right\Vert _{L^{2}\left( \omega \right) }^{2}\int_{I\cap \left( 3J\right) }%
\frac{1}{\left( \left\vert J^{\prime }\right\vert ^{\frac{1}{n}}+\left\vert
y-c_{J^{\prime }}\right\vert \right) ^{n+1-\alpha }}d\sigma \left( y\right)
\\
&=&\sum_{F^{\prime }\in \mathcal{F}}\frac{1}{\left\vert J\right\vert ^{\frac{%
1}{n}\left( n+1-\alpha \right) }}\overset{\ast }{\sum_{c_{J^{\prime }}\in 2J}%
}\left\Vert \mathsf{P}_{F^{\prime },J^{\prime }}^{\omega }\mathbf{x}%
\right\Vert _{L^{2}\left( \omega \right) }^{2}\frac{\mathrm{P}^{\alpha
}\left( J^{\prime },\mathbf{1}_{I\cap \left( 3J\right) }\sigma \right) }{%
\left\vert J^{\prime }\right\vert ^{\frac{1}{n}}}.
\end{eqnarray*}%
Now by Cauchy-Schwarz and the shifted local estimate (\ref{shifted local})
in Lemma \ref{shifted}, this is dominated by%
\begin{eqnarray*}
&&\frac{1}{\left\vert J\right\vert ^{\frac{1}{n}\left( n+1-\alpha \right) }}%
\left( \sum_{F^{\prime }\in \mathcal{F}}\overset{\ast }{\sum_{c\left(
J^{\prime }\right) \in 2J\text{ and }J^{\prime }\subset I}}\left\Vert 
\mathsf{P}_{F^{\prime },J^{\prime }}^{\omega }\mathbf{x}\right\Vert
_{L^{2}\left( \omega \right) }^{2}\right) ^{\frac{1}{2}} \\
&&\ \ \ \ \ \ \ \ \ \ \times \left( \sum_{F^{\prime }\in \mathcal{F}}\overset%
{\ast }{\sum_{c_{J^{\prime }}\in 2J\text{ and }J^{\prime }\subset I}}%
\left\Vert \mathsf{P}_{F^{\prime },J^{\prime }}^{\omega }\mathbf{x}%
\right\Vert _{L^{2}\left( \omega \right) }^{2}\left( \frac{\mathrm{P}%
^{\alpha }\left( J^{\prime },\mathbf{1}_{I\cap \left( 4J\right) }\sigma
\right) }{\left\vert J^{\prime }\right\vert ^{\frac{1}{n}}}\right)
^{2}\right) ^{\frac{1}{2}} \\
&\lesssim &\frac{1}{\left\vert J\right\vert ^{\frac{1}{n}\left( n+1-\alpha
\right) }}\left( \mathbf{\tau }\sum_{c_{J^{\prime }}\in 2J}\left\vert
J^{\prime }\right\vert ^{\frac{2}{n}}\left\vert J^{\prime }\right\vert
_{\omega }\right) ^{\frac{1}{2}}\mathcal{E}_{\alpha }\sqrt{\mathbf{\tau }%
\left\vert CJ\right\vert _{\sigma }} \\
&\lesssim &\mathbf{\tau }\frac{2^{-s}\left\vert J\right\vert ^{\frac{1}{n}}}{%
\left\vert J\right\vert ^{\frac{1}{n}\left( n+1-\alpha \right) }}\sqrt{%
\left\vert CJ\right\vert _{\omega }}\mathcal{E}_{\alpha }\sqrt{\left\vert
CJ\right\vert _{\sigma }}\lesssim \mathbf{\tau }2^{-s}\mathcal{E}_{\alpha }%
\sqrt{\frac{\left\vert CJ\right\vert _{\omega }}{\left\vert J\right\vert ^{%
\frac{1}{n}\left( n-\alpha \right) }}\frac{\left\vert CJ\right\vert _{\sigma
}}{\left\vert J\right\vert ^{\frac{1}{n}\left( n-\alpha \right) }}}\lesssim 
\mathbf{\tau }2^{-s}\mathcal{E}_{\alpha }\sqrt{A_{2}^{\alpha }}\ .
\end{eqnarray*}%
Here the shifted local estimate (\ref{shifted local}) applies to the
expression%
\begin{equation*}
\sum_{F^{\prime }\in \mathcal{F}}\overset{\ast }{\sum_{c_{J^{\prime }}\in 2J%
\text{ and }J^{\prime }\subset I}}\left\Vert \mathsf{P}_{F^{\prime
},J^{\prime }}^{\omega }\mathbf{x}\right\Vert _{L^{2}\left( \omega \right)
}^{2}\left( \frac{\mathrm{P}^{\alpha }\left( J^{\prime },\mathbf{1}_{I\cap
\left( 4J\right) }\sigma \right) }{\left\vert J^{\prime }\right\vert ^{\frac{%
1}{n}}}\right) ^{2}\ ,
\end{equation*}%
with $M=\widehat{J}$, where $\widehat{J}\in \mathcal{S}\Omega \mathcal{D}$
is a shifted quasicube satisfying $\dbigcup\limits_{c_{J^{\prime }}\in
2J}J^{\prime }\subset 4J\subset \widehat{J}$ and $\ell \left( \widehat{J}%
\right) \leq C\ell \left( J\right) $.

For $\ell \geq 1$, we can estimate the far term

\begin{eqnarray*}
A_{s,far}^{\ell }\left( J\right) &=&\sum_{F^{\prime }\in \mathcal{F}}\overset%
{\ast }{\sum_{c_{J^{\prime }}\in \left( 2^{\ell +1}J\right) \setminus \left(
2^{\ell }J\right) }}\int_{I\setminus \left( 2^{\ell +2}J\right) }S_{\left(
J^{\prime },J\right) }^{F^{\prime }}\left( y\right) d\sigma \left( y\right)
\\
&\lesssim &\mathbf{\tau }\sum_{c_{J^{\prime }}\in \left( 2^{\ell +1}J\right)
\setminus \left( 2^{\ell }J\right) }\int_{I\setminus \left( 2^{\ell
+2}J\right) }\frac{\left\vert J^{\prime }\right\vert ^{\frac{2}{n}%
}\left\vert J^{\prime }\right\vert _{\omega }}{\left( \left\vert
J\right\vert ^{\frac{1}{n}}+\left\vert y-c_{J}\right\vert \right) ^{2\left(
n+1-\alpha \right) }}d\sigma \left( y\right) \\
&=&\mathbf{\tau }2^{-2s}\left( \sum_{c_{J^{\prime }}\in \left( 2^{\ell
+1}J\right) }\left\vert J^{\prime }\right\vert _{\omega }\right)
\int_{I\setminus \left( 2^{\ell +2}J\right) }\frac{\left\vert J\right\vert ^{%
\frac{2}{n}}}{\left( \left\vert J\right\vert ^{\frac{1}{n}}+\left\vert
y-c_{J}\right\vert \right) ^{2\left( n+1-\alpha \right) }}d\sigma \left(
y\right) \\
&\approx &\mathbf{\tau }2^{-2s}2^{-\ell \frac{2}{n}}\left(
\sum_{c_{J^{\prime }}\in \left( 2^{\ell +1}J\right) }\left\vert J^{\prime
}\right\vert _{\omega }\right) \int_{I\setminus \left( 2^{\ell +2}J\right) }%
\frac{\left\vert 2^{\ell }J\right\vert ^{\frac{2}{n}}}{\left( \left\vert
2^{\ell }J\right\vert ^{\frac{1}{n}}+\left\vert y-c_{2^{\ell }J}\right\vert
\right) ^{2\left( n+1-\alpha \right) }}d\sigma \left( y\right) ,
\end{eqnarray*}%
which is at most%
\begin{eqnarray*}
&&\mathbf{\tau }2^{-2s}2^{-\ell \frac{2}{n}}\left\vert 2^{\ell
+2}J\right\vert _{\omega }\int_{I\setminus \left( 2^{\ell +2}J\right) }\frac{%
1}{\left( \left\vert 2^{\ell }J\right\vert ^{\frac{1}{n}}+\left\vert
y-c_{2^{\ell }J}\right\vert \right) ^{2\left( n-\alpha \right) }}d\sigma
\left( y\right) \\
&\approx &\mathbf{\tau }2^{-2s}2^{-\ell \frac{2}{n}}\frac{\left\vert 3^{\ell
+2}J\right\vert _{\omega }}{\left\vert 3^{\ell }J\right\vert ^{1-\frac{%
\alpha }{n}}}\int_{I\setminus \left( 3^{\ell +2}J\right) }\left( \frac{%
\left\vert 2^{\ell }J\right\vert ^{\frac{1}{n}}}{\left( \left\vert 2^{\ell
}J\right\vert ^{\frac{1}{n}}+\left\vert y-c_{2^{\ell }J}\right\vert \right)
^{2}}\right) ^{n-\alpha }d\sigma \left( y\right) \\
&\lesssim &\mathbf{\tau }2^{-2s}2^{-\ell \frac{2}{n}}\left\{ \frac{%
\left\vert 2^{\ell +2}J\right\vert _{\omega }}{\left\vert 2^{\ell
}J\right\vert ^{1-\frac{\alpha }{n}}}\mathcal{P}^{\alpha }\left( 2^{\ell
+2}J,\sigma \right) \right\} \lesssim \mathbf{\tau }2^{-2s}2^{-\ell \frac{2}{%
n}}\mathcal{A}_{2}^{\alpha }\ .
\end{eqnarray*}

The near term $A_{s,near}^{\ell }\left( J\right) $ is

\begin{eqnarray*}
&&\sum_{F^{\prime }\in \mathcal{F}}\overset{\ast }{\sum_{c_{J^{\prime }}\in
2^{\ell +1}J\setminus 2^{\ell }J}}\int_{I\cap \left( 2^{\ell +2}J\setminus
2^{\ell -m}J\right) }S_{\left( J^{\prime },J\right) }^{F^{\prime }}\left(
y\right) d\sigma \left( y\right) \\
&\approx &\sum_{F^{\prime }\in \mathcal{F}}\overset{\ast }{%
\sum_{c_{J^{\prime }}\in 2^{\ell +1}J\setminus 2^{\ell }J}}\int_{I\cap
\left( 2^{\ell +2}J\setminus 2^{\ell -m}J\right) }\frac{1}{\left\vert
2^{\ell \left( 1-\varepsilon \right) }J\right\vert ^{\frac{1}{n}\left(
n+1-\alpha \right) }}\frac{\left\Vert \mathsf{P}_{F^{\prime },J^{\prime
}}^{\omega }\mathbf{x}\right\Vert _{L^{2}\left( \omega \right) }^{2}}{\left(
\left\vert J^{\prime }\right\vert ^{\frac{1}{n}}+\left\vert y-c_{J^{\prime
}}\right\vert \right) ^{n+1-\alpha }}d\sigma \left( y\right) \\
&=&\frac{1}{\left\vert 2^{\ell -1}J\right\vert ^{\frac{1}{n}\left(
n+1-\alpha \right) }}\sum_{F^{\prime }\in \mathcal{F}}\overset{\ast }{%
\sum_{c_{J^{\prime }}\in 2^{\ell +1}J\setminus 2^{\ell }J}}\left\Vert 
\mathsf{P}_{F^{\prime },J^{\prime }}^{\omega }\mathbf{x}\right\Vert
_{L^{2}\left( \omega \right) }^{2}\int_{I\cap \left( 2^{\ell +2}J\setminus
2^{\ell -m}J\right) }\frac{1}{\left( \left\vert J^{\prime }\right\vert ^{%
\frac{1}{n}}+\left\vert y-c_{J^{\prime }}\right\vert \right) ^{n+1-\alpha }}%
d\sigma \left( y\right) ,
\end{eqnarray*}%
and is dominated by%
\begin{eqnarray*}
&&\frac{1}{\left\vert 2^{\ell -m}J\right\vert ^{\frac{1}{n}\left( n+1-\alpha
\right) }}\sum_{F^{\prime }\in \mathcal{F}}\overset{\ast }{%
\sum_{c_{J^{\prime }}\in 2^{\ell +1}J\setminus 2^{\ell }J}}\left\Vert 
\mathsf{P}_{F^{\prime },J^{\prime }}^{\omega }\mathbf{x}\right\Vert
_{L^{2}\left( \omega \right) }^{2}\frac{\mathrm{P}^{\alpha }\left( J^{\prime
},\mathbf{1}_{I\cap \left( 2^{\ell +2}J\right) }\sigma \right) }{\left\vert
J^{\prime }\right\vert ^{\frac{1}{n}}} \\
&\leq &\frac{1}{\left\vert 2^{\ell -m}J\right\vert ^{\frac{1}{n}\left(
n+1-\alpha \right) }}\left( \sum_{F^{\prime }\in \mathcal{F}}\overset{\ast }{%
\sum_{c_{J^{\prime }}\in 2^{\ell +1}J\setminus 2^{\ell }J}}\left\Vert 
\mathsf{P}_{F^{\prime },J^{\prime }}^{\omega }\mathbf{x}\right\Vert
_{L^{2}\left( \omega \right) }^{2}\right) ^{\frac{1}{2}} \\
&&\times \left( \sum_{F^{\prime }\in \mathcal{F}}\overset{\ast }{%
\sum_{c_{J^{\prime }}\in 2^{\ell +1}J\setminus 2^{\ell }J}}\left\Vert 
\mathsf{P}_{F^{\prime },J^{\prime }}^{\omega }\mathbf{x}\right\Vert
_{L^{2}\left( \omega \right) }^{2}\left( \frac{\mathrm{P}^{\alpha }\left(
J^{\prime },\mathbf{1}_{I\cap \left( 2^{\ell +2}J\right) }\sigma \right) }{%
\left\vert J^{\prime }\right\vert ^{\frac{1}{n}}}\right) ^{2}\right) ^{\frac{%
1}{2}}.
\end{eqnarray*}%
This can now be estimated by $\mathcal{E}_{\alpha }$ using $\sum_{F^{\prime
}\in \mathcal{F}}\left\Vert \mathsf{P}_{F^{\prime },J^{\prime }}^{\omega
}z\right\Vert _{L^{2}\left( \omega \right) }^{2}\leq \mathbf{\tau }%
\left\vert J^{\prime }\right\vert ^{\frac{2}{n}}\left\vert J^{\prime
}\right\vert _{\omega }$ and the shifted local estimate (\ref{shifted local}%
) in Lemma \ref{shifted} to get%
\begin{eqnarray*}
A_{s,near}^{\ell }\left( J\right) &\lesssim &2^{-s}2^{-\frac{\ell }{n}}\frac{%
\left\vert 2^{\ell }J\right\vert ^{\frac{1}{n}}}{\left\vert 2^{\ell
-m}J\right\vert ^{\frac{1}{n}\left( n+1-\alpha \right) }}\sqrt{\left\vert
2^{\ell +3}J\right\vert _{\omega }}\mathcal{E}_{\alpha }\sqrt{\left\vert
C2^{\ell }J\right\vert _{\sigma }} \\
&\lesssim &2^{-s}2^{-\frac{\ell }{n}}\mathcal{E}_{\alpha }\sqrt{\frac{%
\left\vert C2^{\ell }J\right\vert _{\omega }}{\left\vert C2^{\ell
}J\right\vert ^{1-\frac{\alpha }{n}}}\frac{\left\vert C2^{\ell }J\right\vert
_{\sigma }}{\left\vert C2^{\ell }J\right\vert ^{1-\frac{\alpha }{n}}}} \\
&\lesssim &2^{-s}2^{-\frac{\ell }{n}}\mathcal{E}_{\alpha }\sqrt{%
A_{2}^{\alpha }}\ .
\end{eqnarray*}%
Here the shifted local estimate (\ref{shifted local}) applies to the
expression%
\begin{equation*}
\sum_{F^{\prime }\in \mathcal{F}}\overset{\ast }{\sum_{c_{J^{\prime }}\in
2^{\ell +1}J\setminus 2^{\ell }J}}\left\Vert \mathsf{P}_{F^{\prime
},J^{\prime }}^{\omega }\mathbf{x}\right\Vert _{L^{2}\left( \omega \right)
}^{2}\left( \frac{\mathrm{P}^{\alpha }\left( J^{\prime },\mathbf{1}_{I\cap
\left( 2^{\ell +2}J\right) }\sigma \right) }{\left\vert J^{\prime
}\right\vert ^{\frac{1}{n}}}\right) ^{2}
\end{equation*}%
with $M=\widehat{J}$, where $\widehat{J}\in \mathcal{S}\Omega \mathcal{D}$
is a shifted quasicube satisfying $\dbigcup\limits_{c_{J^{\prime }}\in
2^{\ell +1}J\setminus 2^{\ell }J}J^{\prime }\subset \widehat{J}$ and $\ell
\left( \widehat{J}\right) \leq C2^{\ell }\ell \left( J\right) $. We are also
using here that $m\approx 1+\frac{1}{2}\log _{2}n$ in the definition of $%
A_{s,near}^{\ell }\left( J\right) $ is harmless. These estimates are
summable in both $s$ and $\ell $.

Now we turn to the terms $A_{s,close}^{\ell }\left( J\right) $, and recall
from (\ref{vanishing close}) that $A_{s,close}^{\ell }\left( J\right) =0$ if 
$\ell >1+\varepsilon s$. So we now suppose that $\ell \leq 1+\varepsilon s$.
We have, with $m$ as in (\ref{smallest m}),%
\begin{eqnarray*}
&&A_{s,close}^{\ell }\left( J\right) \\
&=&\sum_{F^{\prime }\in \mathcal{F}}\overset{\ast }{\sum_{c_{J^{\prime }}\in
2^{\ell +1}J\setminus 2^{\ell }J}}\int_{I\cap \left( 2^{\ell -m}J\right)
}S_{\left( J^{\prime },J\right) }^{F^{\prime }}\left( y\right) d\sigma
\left( y\right) \\
&\approx &\sum_{F^{\prime }\in \mathcal{F}}\overset{\ast }{%
\sum_{c_{J^{\prime }}\in 2^{\ell +1}J\setminus 2^{\ell }J}}\int_{I\cap
\left( 2^{\ell -m}J\right) }\frac{1}{\left( \left\vert J\right\vert ^{\frac{1%
}{n}}+\left\vert y-c_{J}\right\vert \right) ^{n+1-\alpha }}\frac{\left\Vert 
\mathsf{P}_{F^{\prime },J^{\prime }}^{\omega }\mathbf{x}\right\Vert
_{L^{2}\left( \omega \right) }^{2}}{\left\vert 2^{\ell }J\right\vert ^{\frac{%
1}{n}\left( n+1-\alpha \right) }}d\sigma \left( y\right) \\
&\approx &\left( \sum_{F^{\prime }\in \mathcal{F}}\overset{\ast }{%
\sum_{c_{J^{\prime }}\in 2^{\ell +1}J\setminus 2^{\ell }J}}\left\Vert 
\mathsf{P}_{F^{\prime },J^{\prime }}^{\omega }\mathbf{x}\right\Vert
_{L^{2}\left( \omega \right) }^{2}\right) \frac{1}{\left\vert 2^{\ell
}J\right\vert ^{\frac{1}{n}\left( n+1-\alpha \right) }}\int_{I\cap \left(
2^{\ell -m}J\right) }\frac{1}{\left( \left\vert J\right\vert ^{\frac{1}{n}%
}+\left\vert y-c_{J}\right\vert \right) ^{n+1-\alpha }}d\sigma \left(
y\right) .
\end{eqnarray*}%
Now we use the inequality $\sum_{F^{\prime }\in \mathcal{F}}\left\Vert 
\mathsf{P}_{F^{\prime },J^{\prime }}^{\omega }z\right\Vert _{L^{2}\left(
\omega \right) }^{2}\leq \mathbf{\tau }\left\vert J^{\prime }\right\vert ^{%
\frac{2}{n}}\left\vert J^{\prime }\right\vert _{\omega }$ to get the
relatively crude estimate%
\begin{eqnarray*}
A_{s,close}^{\ell }\left( J\right) &\lesssim &\mathbf{\tau }%
2^{-2s}\left\vert J\right\vert ^{\frac{2}{n}}\left\vert 2^{\ell
+1}J\right\vert _{\omega }\frac{1}{\left\vert 2^{\ell }J\right\vert ^{\frac{1%
}{n}\left( n+1-\alpha \right) }}\int_{I\cap \left( 2^{\ell -m}J\right) }%
\frac{1}{\left( \left\vert J\right\vert ^{\frac{1}{n}}+\left\vert
y-c_{J}\right\vert \right) ^{n+1-\alpha }}d\sigma \left( y\right) \\
&\lesssim &\mathbf{\tau }2^{-2s}\left\vert J\right\vert ^{\frac{2}{n}}\frac{%
\left\vert 2^{\ell +1}J\right\vert _{\omega }}{\left\vert 2^{\ell
}J\right\vert ^{\frac{1}{n}\left( n+1-\alpha \right) }}\frac{\left\vert
2^{\ell -m}J\right\vert _{\sigma }}{\left\vert J\right\vert ^{\frac{1}{n}%
\left( n+1-\alpha \right) }}\lesssim 2^{-2s}\frac{\left\vert 2^{\ell
+1}J\right\vert _{\omega }}{\left\vert 2^{\ell +1}J\right\vert ^{1-\frac{%
\alpha }{n}}}\frac{\left\vert 2^{\ell +1}J\right\vert _{\sigma }}{\left\vert
2^{\ell +1}J\right\vert ^{1-\frac{\alpha }{n}}}2^{\ell \left( n-1-\alpha
\right) } \\
&\lesssim &\mathbf{\tau }2^{-2s}2^{\ell \left( n-1-\alpha \right)
}A_{2}^{\alpha }\lesssim \mathbf{\tau }2^{-s}A_{2}^{\alpha }
\end{eqnarray*}%
provided that $\ell \leq \frac{s}{n}$. But we are assuming $\ell \leq
1+\varepsilon s$ here and so we obtain a suitable estimate for $%
A_{s,close}^{\ell }\left( J\right) $ provided we choose $0<\varepsilon <%
\frac{1}{n}$.

\begin{remark}
We cannot simply sum the estimate 
\begin{equation*}
A_{s,close}^{\ell }\left( J\right) \lesssim 2^{-2s}\left\vert J\right\vert ^{%
\frac{2}{n}}\left\vert 2^{\ell +1}J\right\vert _{\omega }\frac{1}{\left\vert
2^{\ell }J\right\vert ^{\frac{1}{n}\left( n+1-\alpha \right) }}\frac{\mathrm{%
P}^{\alpha }\left( J,\mathbf{1}_{2^{\ell -1}J}\sigma \right) }{\left\vert
J\right\vert ^{\frac{1}{n}}},
\end{equation*}%
over all $\ell \geq 1$ to get%
\begin{equation*}
\sum_{\ell }A_{s,close}^{\ell }\left( J\right) \lesssim 2^{-2s}\mathrm{P}%
^{\alpha }\left( J,\sigma \right) \sum_{\ell }\frac{\left\vert J\right\vert
^{\frac{1}{n}}}{\left\vert 2^{\ell }J\right\vert ^{\frac{1}{n}\left(
n+1-\alpha \right) }}\left\vert 2^{\ell +1}J\right\vert _{\omega }\lesssim
2^{-2s}\mathrm{P}^{\alpha }\left( J,\sigma \right) \mathrm{P}^{\alpha
}\left( J,\omega \right) ,
\end{equation*}%
since we only have control of the product $\mathrm{P}\left( J,\sigma \right) 
\mathrm{P}\left( J,\omega \right) $ in dimension $n=1$, where the two
Poisson kernels $\mathrm{P}$ and $\mathcal{P}$\ coincide, and the two-tailed 
$\mathcal{A}_{2}$ condition is known to hold.
\end{remark}

The above estimates prove%
\begin{equation*}
T_{s}^{\limfunc{proximal}}+T_{s}^{\limfunc{difference}}\lesssim 2^{-s}\left( 
\mathcal{A}_{2}^{\alpha }+\mathcal{E}_{\alpha }\sqrt{A_{2}^{\alpha }}\right)
.
\end{equation*}

\subsubsection{The intersection term}

Now we return to the term,%
\begin{eqnarray*}
T_{s}^{\limfunc{intersection}} &\equiv &\sum_{\substack{ F,F^{\prime }\in 
\mathcal{F}}}\sum_{\substack{ J\in \mathcal{M}_{\mathbf{r}-\limfunc{deep}%
}\left( F\right) ,\ J^{\prime }\in \mathcal{M}_{\mathbf{r}-\limfunc{deep}%
}\left( F^{\prime }\right)  \\ J,J^{\prime }\subset I,\ \ell \left(
J^{\prime }\right) =2^{-s}\ell \left( J\right)  \\ \left\vert c\left(
J\right) -c\left( J^{\prime }\right) \right\vert _{\limfunc{quasi}}\geq
2^{s\left( 1+\varepsilon \right) }\ell \left( J^{\prime }\right) }} \\
&&\times \int_{I\cap B\left( J,J^{\prime }\right) }\frac{\left\Vert \mathsf{P%
}_{F,J}^{\omega }\mathbf{x}\right\Vert _{L^{2}\left( \omega \right) }^{2}}{%
\left( \left\vert J\right\vert ^{\frac{1}{n}}+\left\vert y-c_{J}\right\vert
\right) ^{n+1-\alpha }}\frac{\left\Vert \mathsf{P}_{F^{\prime },J^{\prime
}}^{\omega }\mathbf{x}\right\Vert _{L^{2}\left( \omega \right) }^{2}}{\left(
\left\vert J^{\prime }\right\vert ^{\frac{1}{n}}+\left\vert y-c_{J^{\prime
}}\right\vert \right) ^{n+1-\alpha }}d\sigma \left( y\right) .
\end{eqnarray*}%
It will suffice to show that $T_{s}^{\limfunc{intersection}}$ satisfies the
estimate,%
\begin{equation*}
T_{s}^{\limfunc{intersection}}\lesssim 2^{-s\varepsilon }\mathcal{E}_{\alpha
}\sqrt{A_{2}^{\alpha }}\sum_{F\in \mathcal{F}}\sum_{\substack{ J\in \mathcal{%
M}_{\mathbf{r}-\limfunc{deep}}\left( F\right)  \\ J\subset I}}\lVert \mathsf{%
P}_{F,J}^{\omega }\mathbf{x}\rVert _{L^{2}\left( \omega \right)
}^{2}=2^{-s\varepsilon }\mathcal{E}_{\alpha }\sqrt{A_{2}^{\alpha }}\int_{%
\widehat{I}}t^{2}d\mu \ .
\end{equation*}%
Using $B\left( J,J^{\prime }\right) =B\left( c_{J},\frac{1}{2}\left\vert
c_{J}-c_{J^{\prime }}\right\vert _{\limfunc{quasi}}\right) $, we can write
(suppressing some notation for clarity),%
\begin{eqnarray*}
T_{s}^{\limfunc{intersection}} &=&\sum_{F,F^{\prime }}\sum_{\substack{ %
J,J^{\prime }}}\int_{I\cap B\left( J,J^{\prime }\right) }\frac{\left\Vert 
\mathsf{P}_{F,J}^{\omega }\mathbf{x}\right\Vert _{L^{2}\left( \omega \right)
}^{2}}{\left( \left\vert J\right\vert ^{\frac{1}{n}}+\left\vert
y-c_{J}\right\vert \right) ^{n+1-\alpha }}\frac{\left\Vert \mathsf{P}%
_{F^{\prime },J^{\prime }}^{\omega }\mathbf{x}\right\Vert _{L^{2}\left(
\omega \right) }^{2}}{\left( \left\vert J^{\prime }\right\vert ^{\frac{1}{n}%
}+\left\vert y-c_{J^{\prime }}\right\vert \right) ^{n+1-\alpha }}d\sigma
\left( y\right) \\
&\approx &\sum_{F,F^{\prime }}\sum_{\substack{ J,J^{\prime }}}\left\Vert 
\mathsf{P}_{F,J}^{\omega }\mathbf{x}\right\Vert _{L^{2}\left( \omega \right)
}^{2}\left\Vert \mathsf{P}_{F^{\prime },J^{\prime }}^{\omega }\mathbf{x}%
\right\Vert _{L^{2}\left( \omega \right) }^{2}\frac{1}{\left\vert
c_{J}-c_{J^{\prime }}\right\vert ^{n+1-\alpha }} \\
&&\times \int_{I\cap B\left( J,J^{\prime }\right) }\frac{1}{\left(
\left\vert J\right\vert ^{\frac{1}{n}}+\left\vert y-c_{J}\right\vert \right)
^{n+1-\alpha }}d\sigma \left( y\right) \\
&\approx &\sum_{F,F^{\prime }}\sum_{\substack{ J,J^{\prime }}}\left\Vert 
\mathsf{P}_{F,J}^{\omega }\mathbf{x}\right\Vert _{L^{2}\left( \omega \right)
}^{2}\left\Vert \mathsf{P}_{F^{\prime },J^{\prime }}^{\omega }\mathbf{x}%
\right\Vert _{L^{2}\left( \omega \right) }^{2}\frac{1}{\left\vert
c_{J}-c_{J^{\prime }}\right\vert ^{n+1-\alpha }}\frac{\mathrm{P}^{\alpha
}\left( J,\mathbf{1}_{I\cap B\left( J,J^{\prime }\right) }\sigma \right) }{%
\left\vert J\right\vert ^{\frac{1}{n}}} \\
&\leq &\sum_{F^{\prime }}\sum_{\substack{ J^{\prime }}}\left\Vert \mathsf{P}%
_{F^{\prime },J^{\prime }}^{\omega }\mathbf{x}\right\Vert _{L^{2}\left(
\omega \right) }^{2}\sum_{F}\sum_{\substack{ J}}\frac{1}{\left\vert
c_{J}-c_{J^{\prime }}\right\vert ^{n+1-\alpha }}\left\Vert \mathsf{P}%
_{F,J}^{\omega }\mathbf{x}\right\Vert _{L^{2}\left( \omega \right) }^{2}%
\frac{\mathrm{P}^{\alpha }\left( J,\mathbf{1}_{I\cap B\left( J,J^{\prime
}\right) }\sigma \right) }{\left\vert J\right\vert ^{\frac{1}{n}}},
\end{eqnarray*}%
and it remains to show that for each $J^{\prime }$,%
\begin{equation*}
S_{s}\left( J^{\prime }\right) \equiv \sum_{F}\overset{\ast }{\sum 
_{\substack{ J:\ \left\vert c\left( J\right) -c\left( J^{\prime }\right)
\right\vert _{\limfunc{quasi}}\geq 2^{s\left( 1+\varepsilon \right) }\ell
\left( J^{\prime }\right) }}}\frac{\left\Vert \mathsf{P}_{F,J}^{\omega }%
\mathbf{x}\right\Vert _{L^{2}\left( \omega \right) }^{2}}{\left\vert
c_{J}-c_{J^{\prime }}\right\vert ^{n+1-\alpha }}\frac{\mathrm{P}^{\alpha
}\left( J,\mathbf{1}_{I\cap B\left( J,J^{\prime }\right) }\sigma \right) }{%
\left\vert J\right\vert ^{\frac{1}{n}}}\lesssim 2^{-\varepsilon s}\mathcal{E}%
_{\alpha }\sqrt{A_{2}^{\alpha }}\ .
\end{equation*}%
We write%
\begin{eqnarray*}
S_{s}\left( J^{\prime }\right) &\approx &\sum_{k\geq s\left( 1+\varepsilon
\right) -m}\frac{1}{\left( 2^{k}\left\vert J^{\prime }\right\vert ^{\frac{1}{%
n}}\right) ^{n+1-\alpha }}\sum_{F}\overset{\ast }{\sum_{J:\ \left\vert
c_{J}-c_{J^{\prime }}\right\vert _{\limfunc{quasi}}\approx 2^{k}\ell \left(
J^{\prime }\right) }}\left\Vert \mathsf{P}_{F,J}^{\omega }\mathbf{x}%
\right\Vert _{L^{2}\left( \omega \right) }^{2}\frac{\mathrm{P}^{\alpha
}\left( J,\mathbf{1}_{I\cap B\left( J,J^{\prime }\right) }\sigma \right) }{%
\left\vert J\right\vert ^{\frac{1}{n}}} \\
&\equiv &\sum_{k\geq s\left( 1+\varepsilon \right) -m}\frac{1}{\left(
2^{k}\left\vert J^{\prime }\right\vert ^{\frac{1}{n}}\right) ^{n+1-\alpha }}%
S_{s}^{k}\left( J^{\prime }\right) \ ,
\end{eqnarray*}%
where by $\left\vert c_{J}-c_{J^{\prime }}\right\vert _{\limfunc{quasi}%
}\approx 2^{k}\ell \left( J^{\prime }\right) $ we mean $2^{k}\ell \left(
J^{\prime }\right) \leq \left\vert c_{J}-c_{J^{\prime }}\right\vert _{%
\limfunc{quasi}}\leq 2^{k+1}\ell \left( J^{\prime }\right) $. Here $m$ is as
in (\ref{smallest m}), and we are using the inequality, 
\begin{equation}
k+m\geq \left( 1+\varepsilon \right) s.  \label{important}
\end{equation}%
Indeed, in the term $V_{s}$ we have $\left\vert c_{J}-c_{J^{\prime
}}\right\vert _{\limfunc{quasi}}\geq 2^{\left( 1+\varepsilon \right) s}\ell
\left( J^{\prime }\right) $, and combined with $\left\vert
c_{J}-c_{J^{\prime }}\right\vert _{\limfunc{quasi}}\leq \sqrt{n}2^{k}\ell
\left( J^{\prime }\right) $, we obtain (\ref{important}).

Now we apply Cauchy-Schwarz and the shifted local estimate (\ref{shifted
local}) in Lemma \ref{shifted} to get%
\begin{eqnarray*}
S_{s}^{k}\left( J^{\prime }\right) &\leq &\left( \sum_{F}\overset{\ast }{%
\sum_{J:\ \left\vert c_{J}-c_{J^{\prime }}\right\vert _{\limfunc{quasi}%
}\approx 2^{k}\ell \left( J^{\prime }\right) }}\left\Vert \mathsf{P}%
_{F,J}^{\omega }\mathbf{x}\right\Vert _{L^{2}\left( \omega \right)
}^{2}\right) ^{\frac{1}{2}} \\
&&\times \left( \sum_{F}\overset{\ast }{\sum_{J:\ \left\vert
c_{J}-c_{J^{\prime }}\right\vert _{\limfunc{quasi}}\approx 2^{k}\ell \left(
J^{\prime }\right) }}\left\Vert \mathsf{P}_{F,J}^{\omega }\mathbf{x}%
\right\Vert _{L^{2}\left( \omega \right) }^{2}\left( \frac{\mathrm{P}%
^{\alpha }\left( J,\mathbf{1}_{I\cap B\left( J,J^{\prime }\right) }\sigma
\right) }{\left\vert J\right\vert ^{\frac{1}{n}}}\right) ^{2}\right) ^{\frac{%
1}{2}} \\
&\lesssim &\left( \mathbf{\tau }\overset{\ast }{\sum_{J:\ \left\vert
c_{J}-c_{J^{\prime }}\right\vert _{\limfunc{quasi}}\approx 2^{k}\ell \left(
J^{\prime }\right) }}\left\vert J\right\vert ^{\frac{2}{n}}\left\vert
J\right\vert _{\omega }\right) ^{\frac{1}{2}}\left( \mathbf{\tau }\mathcal{E}%
_{\alpha }^{2}\left\vert C2^{k}J^{\prime }\right\vert _{\sigma }\right) ^{%
\frac{1}{2}} \\
&\lesssim &\mathbf{\tau }\mathcal{E}_{\alpha }2^{s}\left\vert J^{\prime
}\right\vert ^{\frac{1}{n}}\sqrt{\left\vert C2^{k}J^{\prime }\right\vert
_{\omega }}\sqrt{\left\vert C2^{k}J^{\prime }\right\vert _{\sigma }}\lesssim 
\mathbf{\tau }\mathcal{E}_{\alpha }\sqrt{A_{2}^{\alpha }}2^{s}\left\vert
J^{\prime }\right\vert ^{\frac{1}{n}}\left\vert 2^{k}J^{\prime }\right\vert
^{1-\frac{\alpha }{n}} \\
&=&\mathbf{\tau }\mathcal{E}_{\alpha }\sqrt{A_{2}^{\alpha }}2^{s}2^{k\left(
n-\alpha \right) }\left\vert J^{\prime }\right\vert ^{\frac{1}{n}\left(
n+1-\alpha \right) },
\end{eqnarray*}%
provided 
\begin{equation*}
B\left( J,J^{\prime }\right) \subset C2^{k}J^{\prime }.
\end{equation*}%
But this follows from $\left\vert c_{J}-c_{J^{\prime }}\right\vert _{%
\limfunc{quasi}}\approx 2^{k}\ell \left( J^{\prime }\right) $ and (\ref%
{important}), which shows in particular that $k\geq s+c$. Here the shifted
local estimate (\ref{shifted local}) applies to the expression%
\begin{equation*}
\sum_{F}\overset{\ast }{\sum_{J:\ \left\vert c_{J}-c_{J^{\prime
}}\right\vert _{\limfunc{quasi}}\approx 2^{k}\ell \left( J^{\prime }\right) }%
}\left\Vert \mathsf{P}_{F,J}^{\omega }\mathbf{x}\right\Vert _{L^{2}\left(
\omega \right) }^{2}\left( \frac{\mathrm{P}^{\alpha }\left( J,\mathbf{1}%
_{I\cap B\left( J,J^{\prime }\right) }\sigma \right) }{\left\vert
J\right\vert ^{\frac{1}{n}}}\right) ^{2}
\end{equation*}%
with $I_{0}=\widehat{J}$, where $\widehat{J}\in \mathcal{S}\Omega \mathcal{D}
$ is a shifted quasicube satisfying $\dbigcup\limits_{J:\ \left\vert
c_{J}-c_{J^{\prime }}\right\vert _{\limfunc{quasi}}\approx 2^{k}\ell \left(
J^{\prime }\right) }J\subset \widehat{J}$ and $\ell \left( \widehat{J}%
\right) \leq C2^{k}\ell \left( J^{\prime }\right) $.

Then we have%
\begin{eqnarray*}
S_{s}\left( J^{\prime }\right) &=&\sum_{k\geq \left( 1+\varepsilon \right)
s-m}\frac{1}{\left( 2^{k}\left\vert J^{\prime }\right\vert ^{\frac{1}{n}%
}\right) ^{n+1-\alpha }}S_{s}^{k}\left( J^{\prime }\right) \\
&\lesssim &\mathbf{\tau }\mathcal{E}_{\alpha }\sqrt{A_{2}^{\alpha }}%
\sum_{k\geq \left( 1+\varepsilon \right) s-m}\frac{1}{\left( 2^{k}\left\vert
J^{\prime }\right\vert ^{\frac{1}{n}}\right) ^{n+1-\alpha }}2^{s}2^{k\left(
n-\alpha \right) }\left\vert J^{\prime }\right\vert ^{\frac{1}{n}\left(
n+1-\alpha \right) } \\
&\lesssim &\mathbf{\tau }\mathcal{E}_{\alpha }\sqrt{A_{2}^{\alpha }}%
\sum_{k\geq \left( 1+\varepsilon \right) s-m}2^{s-k}\lesssim \mathbf{\tau }%
2^{-\varepsilon s}\mathcal{E}_{\alpha }\sqrt{A_{2}^{\alpha }},
\end{eqnarray*}%
which is summable in $s$. This completes the proof of (\ref{Us bound}), and
hence of the estimate for the local part $\mathbf{Local}\left( I\right) $ in
(\ref{loc and glo}) of the second testing condition (\ref{e.t2 n}).

\subsubsection{The global estimate}

It remains to prove the following estimate for the global part $\mathbf{%
Global}$ in (\ref{loc and glo}) of the second testing condition (\ref{e.t2 n}%
):%
\begin{equation*}
\int_{\mathbb{R}\setminus I}[\mathbb{P}^{\alpha \ast }(t\mathbf{1}_{\widehat{%
I}}\mu )]^{2}\sigma \lesssim \mathcal{A}_{2}^{\alpha }\left\vert
I\right\vert _{\sigma }.
\end{equation*}%
We decompose the integral on the left into two pieces:%
\begin{equation*}
\int_{\mathbb{R}\setminus I}[\mathbb{P}^{\alpha \ast }(t\mathbf{1}_{\widehat{%
I}}\mu )]^{2}\sigma =\int_{\mathbb{R}\setminus 3I}[\mathbb{P}^{\alpha \ast
}(t\mathbf{1}_{\widehat{I}}\mu )]^{2}\sigma +\int_{3I\setminus I}[\mathbb{P}%
^{\alpha \ast }(t\mathbf{1}_{\widehat{I}}\mu )]^{2}\sigma =A+B.
\end{equation*}%
We further decompose term $A$ in annuli and use (\ref{PI hat}) to obtain%
\begin{eqnarray*}
A &=&\sum_{m=1}^{\infty }\int_{3^{m+1}I\setminus 3^{m}I}[\mathbb{P}^{\alpha
\ast }(t\mathbf{1}_{\widehat{I}}\mu )]^{2}\sigma \\
&=&\sum_{m=1}^{\infty }\int_{3^{m+1}I\setminus 3^{m}I}\left[ \sum_{F\in 
\mathcal{F}}\sum_{\substack{ J\in \mathcal{M}_{\mathbf{r}-\limfunc{deep}%
}\left( F\right)  \\ J\subset I}}\frac{\lVert \mathsf{P}_{F,J}^{\omega }%
\mathbf{x}\rVert _{L^{2}\left( \omega \right) }^{2}}{\left( \left\vert
J\right\vert ^{\frac{1`}{n}}+\left\vert y-c_{J}\right\vert \right)
^{n+1-\alpha }}\right] ^{2}d\sigma \left( y\right) \\
&\lesssim &\sum_{m=1}^{\infty }\int_{3^{m+1}I\setminus 3^{m}I}\left[
\sum_{F\in \mathcal{F}}\sum_{\substack{ J\in \mathcal{M}_{\mathbf{r}-%
\limfunc{deep}}\left( F\right)  \\ J\subset I}}\lVert \mathsf{P}%
_{F,J}^{\omega }\mathbf{x}\rVert _{L^{2}\left( \omega \right) }^{2}\right]
^{2}\frac{1}{\left( 3^{m}\left\vert I\right\vert ^{\frac{1}{n}}\right)
^{2\left( n+1-\alpha \right) }}d\sigma \left( y\right) .
\end{eqnarray*}%
Now use (\ref{mu I hat}) to get%
\begin{equation*}
\int_{\widehat{I}}t^{2}d\mu =\sum_{F\in \mathcal{F}}\sum_{\substack{ J\in 
\mathcal{M}_{\mathbf{r}-\limfunc{deep}}\left( F\right)  \\ J\subset I}}%
\lVert \mathsf{P}_{F,J}^{\omega }\mathbf{x}\rVert _{L^{2}\left( \omega
\right) }^{2}\lesssim \mathbf{\tau \ }\lVert \mathbf{1}_{I}\left(
x-c_{I}\right) \rVert _{L^{2}\left( \omega \right) }^{2}\lesssim \left\vert
I\right\vert ^{\frac{2}{n}}\left\vert I\right\vert _{\omega }
\end{equation*}%
to obtain that%
\begin{eqnarray*}
A &\lesssim &\sum_{m=1}^{\infty }\int_{3^{m+1}I\setminus 3^{m}I}\left[ \int_{%
\widehat{I}}t^{2}d\mu \right] \left[ \left\vert I\right\vert ^{\frac{2}{n}%
}\left\vert I\right\vert _{\omega }\right] \frac{1}{\left( 3^{m}\left\vert
I\right\vert ^{\frac{1}{n}}\right) ^{2\left( n+1-\alpha \right) }}d\sigma
\left( y\right) \\
&\lesssim &\left\{ \sum_{m=1}^{\infty }3^{-2m}\frac{\left\vert
3^{m+1}I\right\vert _{\omega }\left\vert 3^{m+1}I\right\vert _{\sigma }}{%
\left\vert 3^{m+1}I\right\vert ^{2\left( 1-\frac{\alpha }{n}\right) }}%
\right\} \left[ \int_{\widehat{I}}t^{2}d\mu \right] \lesssim A_{2}^{\alpha
}\int_{\widehat{I}}t^{2}d\mu .
\end{eqnarray*}

Finally, we estimate term $B$ by using (\ref{PI hat}) to write%
\begin{equation*}
B=\int_{3I\setminus I}\left[ \sum_{F\in \mathcal{F}}\sum_{\substack{ J\in 
\mathcal{M}_{\mathbf{r}-\limfunc{deep}}\left( F\right)  \\ J\subset I}}\frac{%
\lVert \mathsf{P}_{F,J}^{\omega }\mathbf{x}\rVert _{L^{2}\left( \omega
\right) }^{2}}{\left( \left\vert J\right\vert ^{\frac{1}{n}}+\left\vert
y-c_{J}\right\vert \right) ^{n+1-\alpha }}\right] ^{2}d\sigma \left(
y\right) ,
\end{equation*}%
and then expanding the square and calculating as in the proof of the local
part given earlier to obtain the bound $\mathcal{A}_{2}^{\alpha }$. The
details are similar, but easier in that the energy condition is not needed,
and they are left to the reader.

\section{The stopping form}

In this section we adapt the argument of M. Lacey in \cite{Lac} to apply in
the setting of a general $\alpha $-fractional Calder\'{o}n-Zygmund operator $%
T^{\alpha }$ in $\mathbb{R}^{n}$ using the Monotonicity Lemma \ref{mono} and
our quasienergy condition in Definition \ref{energy condition}. We will
prove the bound (\ref{B stop form 3}) for the stopping form%
\begin{eqnarray}
\mathsf{B}_{\limfunc{stop}}^{A}\left( f,g\right) &\equiv &\sum_{\substack{ %
I\in \mathcal{C}_{A}\text{ and }J\in \mathcal{C}_{A}^{\mathbf{\tau }-%
\limfunc{shift}}  \\ J\Subset _{\mathbf{\rho }}I_{J}}}\left( \mathbb{E}%
_{I_{J}}^{\sigma }\bigtriangleup _{I}^{\sigma }f\right) \left\langle
T_{\sigma }^{\alpha }\mathbf{1}_{A\setminus I_{J}},\bigtriangleup
_{J}^{\omega }g\right\rangle _{\omega }  \label{dummy} \\
&=&\sum_{\substack{ I:\ \pi I\in \mathcal{C}_{A}\text{ and }J\in \mathcal{C}%
_{A}^{\mathbf{\tau }-\limfunc{shift}}  \\ J\Subset _{\mathbf{\rho }}I}}%
\left( \mathbb{E}_{I}^{\sigma }\bigtriangleup _{\pi I}^{\sigma }f\right)
\left\langle T_{\sigma }^{\alpha }\mathbf{1}_{A\setminus I},\bigtriangleup
_{J}^{\omega }g\right\rangle _{\omega },  \notag
\end{eqnarray}%
where we have made the `change of dummy variable' $I_{J}\rightarrow I$ for
convenience in notation (recall that the child of $I$ that contains $J$ is
denoted $I_{J}$).

However, the Monotonicity Lemma of Lacey and Wick has an additional term on
the right hand side, and our quasienergy condition is not a direct
generalization of the one-dimensional energy condition. These differences in
higher dimension result in changes and complications that must be tracked
throughout the argument. In particular, we find it necessary to separate the
interaction of the two terms on the right side of the Monotonicity Lemma by
splitting the stopping form into the two corresponding sublinear forms in (%
\ref{def split}) below. Recall that for $A\in \mathcal{A}$ the \emph{shifted}
corona is given in Definition \ref{shifted corona} by 
\begin{equation*}
\mathcal{C}_{A}^{\mathbf{\tau }-\limfunc{shift}}=\left\{ J\in \mathcal{C}%
_{A}:J\Subset _{\mathbf{\tau }}A\right\} \cup \dbigcup\limits_{A^{\prime
}\in \mathfrak{C}_{\mathcal{A}}\left( A\right) }\left\{ J\in \Omega \mathcal{%
D}:J\Subset _{\mathbf{\tau }}A\text{ and }J\text{ is }\mathbf{\tau }\text{%
-nearby in }A^{\prime }\right\} ,
\end{equation*}%
and in particular the $\mathbf{1}$-shifted corona is given by $\mathcal{C}%
_{A}^{\mathbf{1}-\limfunc{shift}}=\left( \mathcal{C}_{A}\setminus \left\{
A\right\} \right) \cup \mathfrak{C}_{\mathcal{A}}\left( A\right) $.

\begin{definition}
Suppose that $A\in \mathcal{A}$ and that $\mathcal{P}\subset \mathcal{C}%
_{A}^{\mathbf{1}-\limfunc{shift}}\times \mathcal{C}_{A}^{\mathbf{\tau }-%
\limfunc{shift}}$. We say that the collection of pairs $\mathcal{P}$ is $A$%
\emph{-admissible} if
\end{definition}

\begin{itemize}
\item (good and $\left( \mathbf{\rho -1}\right) $-deeply embedded) $J$ is
good and $J\Subset _{\mathbf{\rho -1}}I\varsubsetneqq A$ for every $\left(
I,J\right) \in \mathcal{P},$

\item (tree-connected in the first component) if $I_{1}\subset I_{2}$ and
both $\left( I_{1},J\right) \in \mathcal{P}$ and $\left( I_{2},J\right) \in 
\mathcal{P}$, then $\left( I,J\right) \in \mathcal{P}$ for every $I$ in the
geodesic $\left[ I_{1},I_{2}\right] =\left\{ I\in \Omega \mathcal{D}%
:I_{1}\subset I\subset I_{2}\right\} $.
\end{itemize}

However, since $\left( I,J\right) \in \mathcal{P}$ implies both $J\in 
\mathcal{C}_{A}^{\mathbf{\tau }-\limfunc{shift}}$ and $J\Subset _{\mathbf{%
\rho -1}}I$, the assumption $\mathbf{\rho >\tau }$ in Definition \ref{def
parameters} shows that $I$ is in the corona $\mathcal{C}_{A}$, and hence we
may replace $\mathcal{C}_{A}^{\mathbf{1}-\limfunc{shift}}$ with the
restricted corona $\mathcal{C}_{A}^{\prime }\equiv \mathcal{C}_{A}\setminus
\left\{ A\right\} $ in the above definition of $A$\emph{-admissible}. The
basic example of an admissible collection of pairs is obtained from the
pairs of quasicubes summed in the stopping form $\mathsf{B}_{stop}^{A}\left(
f,g\right) $ in (\ref{dummy}), which occurs in (\ref{B stop form 3}) above; 
\begin{equation}
\mathcal{P}^{A}\equiv \left\{ \left( I,J\right) :I\in \mathcal{C}%
_{A}^{\prime }\text{ and }J\in \mathcal{C}_{A}^{\mathbf{\tau }-\limfunc{shift%
}}\text{ where}\ J\text{ is }\mathbf{\tau }\text{-good and\ }J\Subset _{%
\mathbf{\rho -1}}I\right\} .  \label{initial P}
\end{equation}%
Recall also that $J$ is $\mathbf{\tau }$-good if $J\in \Omega \mathcal{D}%
_{\left( \mathbf{r},\varepsilon \right) -\limfunc{good}}^{\mathbf{\tau }}$
as in (\ref{extended good grid}), i.e. if $J$ and its children and its $\ell 
$-parents up to level $\mathbf{\tau }$ are all good. Recall that the
quasiHaar support of $g$ is contained in the collection of $\mathbf{\tau }$%
-good quasicubes.

\begin{definition}
Suppose that $A\in \mathcal{A}$ and that $\mathcal{P}$ is an $A$\emph{%
-admissible} collection of pairs. Define the associated \emph{stopping} form 
$\mathsf{B}_{\limfunc{stop}}^{A,\mathcal{P}}$ by%
\begin{equation*}
\mathsf{B}_{\limfunc{stop}}^{A,\mathcal{P}}\left( f,g\right) \equiv
\sum_{\left( I,J\right) \in \mathcal{P}}\left( \mathbb{E}_{I}^{\sigma
}\bigtriangleup _{\pi I}^{\sigma }f\right) \ \left\langle T_{\sigma
}^{\alpha }\mathbf{1}_{A\setminus I},\bigtriangleup _{J}^{\omega
}g\right\rangle _{\omega }\ ,
\end{equation*}%
where we may of course further restrict $I$ to $\pi I\in \limfunc{supp}%
\widehat{f}$ if we wish.
\end{definition}

Given an $A$-admissible collection $\mathcal{P}$ of pairs define the reduced
collection $\mathcal{P}^{\limfunc{red}}$ as follows. For each fixed $J$ let $%
I_{J}^{\limfunc{red}}$ be the largest good quasicube $I$ such that $\left(
I,J\right) \in \mathcal{P}$. Then set%
\begin{equation*}
\mathcal{P}^{\limfunc{red}}\equiv \left\{ \left( I,J\right) \in \mathcal{P}%
:I\subset I_{J}^{\limfunc{red}}\right\} .
\end{equation*}%
Clearly $\mathcal{P}^{\limfunc{red}}$ is $A$-admissible. Now recall our
assumption that the quasiHaar support of $f$ is contained in the set of $%
\mathbf{\tau }$-good quasicubes, which in particular requires that their
children are all good as well. This assumption has the important implication
that $\mathsf{B}_{\limfunc{stop}}^{A,\mathcal{P}}\left( f,g\right) =\mathsf{B%
}_{\limfunc{stop}}^{A,\mathcal{P}^{\limfunc{red}}}\left( f,g\right) $.
Indeed, if $\left( I,J\right) \in \mathcal{P}\setminus \mathcal{P}^{\limfunc{%
red}}$ then $\pi I\not\in \limfunc{supp}\widehat{f}$ and so $\mathbb{E}%
_{I}^{\sigma }\bigtriangleup _{\pi I}^{\sigma }f=0$. Thus for the purpose of
bounding the stopping form, we may assume that the following additional
property holds for any $A$-admissible collection of pairs $\mathcal{P}$:

\begin{itemize}
\item if $\left( I,J\right) \in \mathcal{P}$ is maximal in the sense that $%
I\supset I^{\prime }$ for all $I^{\prime }$ satisfying $\left( I^{\prime
},J\right) \in \mathcal{P}$, then $I$ is good.
\end{itemize}

Note that there is an asymmetry in our definition of $\mathcal{P}^{\limfunc{%
red}}$ here, namely that the second components $J$ are required to be $%
\mathbf{\tau }$-good, while the maximal first components $I$ are required to
be good. Of course the treatment of the dual stopping forms will use the
reversed requirements, and this accounts for our symmetric restrictions
imposed on the quasiHaar supports of $f$ and $g$ at the outset of the proof.

\begin{definition}
\label{def reduced}We say that an admissible collection $\mathcal{P}$ is 
\emph{reduced} if $\mathcal{P}=\mathcal{P}^{\limfunc{red}}$, so that the
additional property above holds.
\end{definition}

Note that $\mathsf{B}_{\limfunc{stop}}^{A,\mathcal{P}}\left( f,g\right) =%
\mathsf{B}_{\limfunc{stop}}^{A,\mathcal{P}}\left( \mathsf{P}_{\mathcal{C}%
_{A}}^{\sigma }f,\mathsf{P}_{\mathcal{C}_{A}^{\mathbf{\tau }-\limfunc{shift}%
}}^{\omega }g\right) $. Recall that the deep quasienergy condition constant $%
\mathcal{E}_{\alpha }^{\limfunc{deep}}$ is given by%
\begin{equation*}
\left( \mathcal{E}_{\alpha }^{\limfunc{deep}}\right) ^{2}\equiv \sup_{I=\dot{%
\cup}I_{r}}\frac{1}{\left\vert I\right\vert _{\sigma }}\sum_{r=1}^{\infty
}\sum_{J\in \mathcal{M}_{\mathbf{r}-\limfunc{deep}}\left( I_{r}\right)
}\left( \frac{\mathrm{P}^{\alpha }\left( J,\mathbf{1}_{I\setminus \gamma
J}\sigma \right) }{\left\vert J\right\vert ^{\frac{1}{n}}}\right)
^{2}\left\Vert \mathsf{P}_{J}^{\limfunc{subgood},\omega }\mathbf{x}%
\right\Vert _{L^{2}\left( \omega \right) }^{2}\ .
\end{equation*}

\begin{proposition}
\label{stopping bound}Suppose that $A\in \mathcal{A}$ and that $\mathcal{P}$
is an $A$\emph{-admissible} collection of pairs. Then the stopping form $%
\mathsf{B}_{\limfunc{stop}}^{A,\mathcal{P}}$ satisfies the bound%
\begin{equation}
\left\vert \mathsf{B}_{\limfunc{stop}}^{A,\mathcal{P}}\left( f,g\right)
\right\vert \lesssim \left( \mathcal{E}_{\alpha }^{\limfunc{deep}}+\sqrt{%
A_{2}^{\alpha }}\right) \left( \left\Vert f\right\Vert _{L^{2}\left( \sigma
\right) }+\alpha _{\mathcal{A}}\left( f\right) \sqrt{\left\vert A\right\vert
_{\sigma }}\right) \left\Vert g\right\Vert _{L^{2}\left( \omega \right) }\ .
\label{B stop form}
\end{equation}
\end{proposition}

With this proposition in hand, we can complete the proof of (\ref{B stop
form 3}), and hence of Theorem \ref{T1 theorem}, by summing over the
stopping quasicubes $A\in \mathcal{A}$ with the choice $\mathcal{P}^{A}$ of $%
A$-admissible pairs for each $A$:%
\begin{eqnarray*}
&&\sum_{A\in \mathcal{A}}\left\vert \mathsf{B}_{\limfunc{stop}}^{A,\mathcal{P%
}^{A}}\left( f,g\right) \right\vert \\
&\lesssim &\sum_{A\in \mathcal{A}}\left( \mathcal{E}_{\alpha }^{\limfunc{deep%
}}+\sqrt{A_{2}^{\alpha }}\right) \left( \left\Vert \mathsf{P}_{\mathcal{C}%
_{A}}f\right\Vert _{L^{2}\left( \sigma \right) }+\alpha _{\mathcal{A}}\left(
f\right) \sqrt{\left\vert A\right\vert _{\sigma }}\right) \left\Vert \mathsf{%
P}_{\mathcal{C}_{A}^{\mathbf{\tau }-\limfunc{shift}}}g\right\Vert
_{L^{2}\left( \omega \right) } \\
&\lesssim &\left( \mathcal{E}_{\alpha }^{\limfunc{deep}}+\sqrt{A_{2}^{\alpha
}}\right) \left( \sum_{A\in \mathcal{A}}\left( \left\Vert \mathsf{P}_{%
\mathcal{C}_{A}}f\right\Vert _{L^{2}\left( \sigma \right) }^{2}+\alpha _{%
\mathcal{A}}\left( f\right) ^{2}\left\vert A\right\vert _{\sigma }\right)
\right) ^{\frac{1}{2}}\left( \sum_{A\in \mathcal{A}}\left\Vert \mathsf{P}_{%
\mathcal{C}_{A}^{\mathbf{\tau }-\limfunc{shift}}}g\right\Vert _{L^{2}\left(
\omega \right) }^{2}\right) ^{\frac{1}{2}} \\
&\lesssim &\left( \mathcal{E}_{\alpha }^{\limfunc{deep}}+\sqrt{A_{2}^{\alpha
}}\right) \left\Vert f\right\Vert _{L^{2}\left( \sigma \right) }\left\Vert
g\right\Vert _{L^{2}\left( \omega \right) }\ ,
\end{eqnarray*}%
by orthogonality $\sum_{A\in \mathcal{A}}\left\Vert \mathsf{P}_{\mathcal{C}%
_{A}}f\right\Vert _{L^{2}\left( \sigma \right) }^{2}\leq \left\Vert
f\right\Vert _{L^{2}\left( \sigma \right) }^{2}$ in corona projections $%
\mathcal{C}_{A}^{\sigma }$, `quasi' orthogonality $\sum_{A\in \mathcal{A}%
}\alpha _{\mathcal{A}}\left( f\right) ^{2}\left\vert A\right\vert _{\sigma
}\lesssim \left\Vert f\right\Vert _{L^{2}\left( \sigma \right) }^{2}$ in the
stopping quasicubes $\mathcal{A}$, and by the bounded overlap of the shifted
coronas $\mathcal{C}_{A}^{\mathbf{\tau }-\limfunc{shift}}$: 
\begin{equation*}
\sum_{A\in \mathcal{A}}\mathbf{1}_{\mathcal{C}_{A}^{\mathbf{\tau }-\limfunc{%
shift}}}\leq \mathbf{\tau 1}_{\Omega \mathcal{D}}.
\end{equation*}

To prove Proposition \ref{stopping bound}, we begin by letting $\Pi _{2}%
\mathcal{P}$ consist of the second components of the pairs in $\mathcal{P}$
and writing 
\begin{eqnarray*}
\mathsf{B}_{\limfunc{stop}}^{A,\mathcal{P}}\left( f,g\right) &=&\sum_{J\in
\Pi _{2}\mathcal{P}}\left\langle T_{\sigma }^{\alpha }\varphi _{J}^{\mathcal{%
P}},\bigtriangleup _{J}^{\omega }g\right\rangle _{\omega }; \\
\text{where }\varphi _{J}^{\mathcal{P}} &\equiv &\sum_{I\in \mathcal{C}%
_{A}^{\prime }:\ \left( I,J\right) \in \mathcal{P}}\mathbb{E}_{I}^{\sigma
}\left( \bigtriangleup _{\pi I}^{\sigma }f\right) \ \mathbf{1}_{A\setminus
I}\ .
\end{eqnarray*}%
By the tree-connected property of $\mathcal{P}$, and the telescoping
property of martingale differences, together with the bound $\alpha _{%
\mathcal{A}}\left( A\right) $ on the quasiaverages of $f$ in the corona $%
\mathcal{C}_{A}$, we have%
\begin{equation}
\left\vert \varphi _{J}^{\mathcal{P}}\right\vert \lesssim \alpha _{\mathcal{A%
}}\left( A\right) 1_{A\setminus I_{\mathcal{P}}\left( J\right) },
\label{phi bound}
\end{equation}%
where $I_{\mathcal{P}}\left( J\right) \equiv \dbigcap \left\{ I:\left(
I,J\right) \in \mathcal{P}\right\} $ is the smallest quasicube $I$ for which 
$\left( I,J\right) \in \mathcal{P}$. Another important property of these
functions is the sublinearity:%
\begin{equation}
\left\vert \varphi _{J}^{\mathcal{P}}\right\vert \leq \left\vert \varphi
_{J}^{\mathcal{P}_{1}}\right\vert +\left\vert \varphi _{J}^{\mathcal{P}%
_{2}}\right\vert ,\ \ \ \ \ \mathcal{P}=\mathcal{P}_{1}\dot{\cup}\mathcal{P}%
_{2}\ .  \label{phi sublinear}
\end{equation}%
Now apply the Monotonicity Lemma \ref{mono} to the inner product $%
\left\langle T_{\sigma }^{\alpha }\varphi _{J},\bigtriangleup _{J}^{\omega
}g\right\rangle _{\omega }$ to obtain%
\begin{eqnarray*}
\left\vert \left\langle T_{\sigma }^{\alpha }\varphi _{J},\bigtriangleup
_{J}^{\omega }g\right\rangle _{\omega }\right\vert &\lesssim &\frac{\mathrm{P%
}^{\alpha }\left( J,\left\vert \varphi _{J}\right\vert \mathbf{1}%
_{A\setminus I_{\mathcal{P}}\left( J\right) }\sigma \right) }{\left\vert
J\right\vert ^{\frac{1}{n}}}\left\Vert \bigtriangleup _{J}^{\omega }\mathbf{x%
}\right\Vert _{L^{2}\left( \omega \right) }\left\Vert \bigtriangleup
_{J}^{\omega }g\right\Vert _{L^{2}\left( \omega \right) } \\
&&+\frac{\mathrm{P}_{1+\delta }^{\alpha }\left( J,\left\vert \varphi
_{J}\right\vert \mathbf{1}_{A\setminus I_{\mathcal{P}}\left( J\right)
}\sigma \right) }{\left\vert J\right\vert ^{\frac{1}{n}}}\left\Vert \mathsf{P%
}_{J}^{\omega }\mathbf{x}\right\Vert _{L^{2}\left( \omega \right)
}\left\Vert \bigtriangleup _{J}^{\omega }g\right\Vert _{L^{2}\left( \omega
\right) }.
\end{eqnarray*}%
Thus we have%
\begin{eqnarray}
&&  \label{def split} \\
\left\vert \mathsf{B}_{\limfunc{stop}}^{A,\mathcal{P}}\left( f,g\right)
\right\vert &\leq &\sum_{J\in \Pi _{2}\mathcal{P}}\frac{\mathrm{P}%
_{1}^{\alpha }\left( J,\left\vert \varphi _{J}\right\vert \mathbf{1}%
_{A\setminus I_{\mathcal{P}}\left( J\right) }\sigma \right) }{\left\vert
J\right\vert ^{\frac{1}{n}}}\left\Vert \bigtriangleup _{J}^{\omega }\mathbf{x%
}\right\Vert _{L^{2}\left( \omega \right) }\left\Vert \bigtriangleup
_{J}^{\omega }g\right\Vert _{L^{2}\left( \omega \right) }  \notag \\
&&+\sum_{J\in \Pi _{2}\mathcal{P}}\frac{\mathrm{P}_{1+\delta }^{\alpha
}\left( J,\left\vert \varphi _{J}\right\vert \mathbf{1}_{A\setminus I_{%
\mathcal{P}}\left( J\right) }\sigma \right) }{\left\vert J\right\vert ^{%
\frac{1}{n}}}\left\Vert \mathsf{P}_{J}^{\omega }\mathbf{x}\right\Vert
_{L^{2}\left( \omega \right) }\left\Vert \bigtriangleup _{J}^{\omega
}g\right\Vert _{L^{2}\left( \omega \right) }  \notag \\
&\equiv &\left\vert \mathsf{B}\right\vert _{\limfunc{stop},1,\bigtriangleup
^{\omega }}^{A,\mathcal{P}}\left( f,g\right) +\left\vert \mathsf{B}%
\right\vert _{\limfunc{stop},1+\delta ,\mathsf{P}^{\omega }}^{A,\mathcal{P}%
}\left( f,g\right) ,  \notag
\end{eqnarray}%
where we have dominated the stopping form by two sublinear stopping forms
that involve the Poisson integrals of order $1$ and $1+\delta $
respectively, and where the smaller Poisson integral $\mathrm{P}_{1+\delta
}^{\alpha }$ is multiplied by the larger projection $\left\Vert \mathsf{P}%
_{J}^{\omega }\mathbf{x}\right\Vert _{L^{2}\left( \omega \right) }$. This
splitting turns out to be successful in separating the two energy terms from
the right hand side of the Energy Lemma, because of the two properties (\ref%
{phi bound}) and (\ref{phi sublinear}) above. It remains to show the two
inequalities:%
\begin{equation}
\left\vert \mathsf{B}\right\vert _{\limfunc{stop},1,\bigtriangleup ^{\omega
}}^{A,\mathcal{P}}\left( f,g\right) \lesssim \left( \mathcal{E}_{\alpha }^{%
\limfunc{deep}}+\sqrt{A_{2}^{\alpha }}\right) \alpha _{\mathcal{A}}\left(
A\right) \sqrt{\left\vert A\right\vert _{\sigma }}\left\Vert g\right\Vert
_{L^{2}\left( \omega \right) },  \label{First inequality}
\end{equation}%
for $f\in L^{2}\left( \sigma \right) $ satisfying where $\mathbb{E}%
_{I}^{\sigma }\left\vert f\right\vert \leq \alpha _{\mathcal{A}}\left(
A\right) $ for all $I\in \mathcal{C}_{A}$; and%
\begin{equation}
\left\vert \mathsf{B}\right\vert _{\limfunc{stop},1+\delta ,\mathsf{P}%
^{\omega }}^{A,\mathcal{P}}\left( f,g\right) \lesssim \left( \mathcal{E}%
_{\alpha }^{\limfunc{deep}}+\sqrt{A_{2}^{\alpha }}\right) \left\Vert
f\right\Vert _{L^{2}\left( \sigma \right) }\left\Vert g\right\Vert
_{L^{2}\left( \omega \right) },  \label{Second inequality}
\end{equation}%
where we only need the case $\mathcal{P}=\mathcal{P}^{A}$ in this latter
inequality as there is no recursion involved in treating this second
sublinear form. We consider first the easier inequality (\ref{Second
inequality}) that does not require recursion. In the subsequent subsections
we will control the more difficult inequality (\ref{First inequality}) by
adapting the stopping time and recursion of M. Lacey to the sublinear form $%
\left\vert \mathsf{B}\right\vert _{\limfunc{stop},1,\bigtriangleup ^{\omega
}}^{A,\mathcal{P}}\left( f,g\right) $.

\subsection{The second inequality}

Now we turn to proving (\ref{Second inequality}), i.e.%
\begin{equation*}
\left\vert \mathsf{B}\right\vert _{\limfunc{stop},1+\delta ,\mathsf{P}%
^{\omega }}^{A,\mathcal{P}}\left( f,g\right) \lesssim \left( \mathcal{E}%
_{\alpha }^{\limfunc{deep}}+\sqrt{A_{2}^{\alpha }}\right) \left\Vert
f\right\Vert _{L^{2}\left( \sigma \right) }\left\Vert g\right\Vert
_{L^{2}\left( \omega \right) },
\end{equation*}%
where since 
\begin{equation*}
\left\vert \varphi _{J}\right\vert =\left\vert \sum_{I\in \mathcal{C}%
_{A}^{\prime }:\ \left( I,J\right) \in \mathcal{P}}\mathbb{E}_{I}^{\sigma
}\left( \bigtriangleup _{\pi I}^{\sigma }f\right) \ \mathbf{1}_{A\setminus
I}\right\vert \leq \sum_{I\in \mathcal{C}_{A}^{\prime }:\ \left( I,J\right)
\in \mathcal{P}}\left\vert \mathbb{E}_{I}^{\sigma }\left( \bigtriangleup
_{\pi I}^{\sigma }f\right) \ \mathbf{1}_{A\setminus I}\right\vert ,
\end{equation*}%
the sublinear form $\left\vert \mathsf{B}\right\vert _{\limfunc{stop}%
,1+\delta ,\mathsf{P}^{\omega }}^{A,\mathcal{P}}$ can be dominated and then
decomposed by pigeonholing the ratio of side lengths of $J$ and $I$:%
\begin{eqnarray*}
\left\vert \mathsf{B}\right\vert _{\limfunc{stop},1+\delta ,\mathsf{P}%
^{\omega }}^{A,\mathcal{P}}\left( f,g\right) &=&\sum_{J\in \Pi _{2}\mathcal{P%
}}\frac{\mathrm{P}_{1+\delta }^{\alpha }\left( J,\left\vert \varphi
_{J}\right\vert \mathbf{1}_{A\setminus I_{\mathcal{P}}\left( J\right)
}\sigma \right) }{\left\vert J\right\vert ^{\frac{1}{n}}}\left\Vert \mathsf{P%
}_{J}^{\omega }\mathbf{x}\right\Vert _{L^{2}\left( \omega \right)
}\left\Vert \bigtriangleup _{J}^{\omega }g\right\Vert _{L^{2}\left( \omega
\right) } \\
&\leq &\sum_{\left( I,J\right) \in \mathcal{P}}\frac{\mathrm{P}_{1+\delta
}^{\alpha }\left( J,\left\vert \mathbb{E}_{I}^{\sigma }\left( \bigtriangleup
_{\pi I}^{\sigma }f\right) \right\vert \mathbf{1}_{A\setminus I}\sigma
\right) }{\left\vert J\right\vert ^{\frac{1}{n}}}\left\Vert \mathsf{P}%
_{J}^{\omega }\mathbf{x}\right\Vert _{L^{2}\left( \omega \right) }\left\Vert
\bigtriangleup _{J}^{\omega }g\right\Vert _{L^{2}\left( \omega \right) } \\
&\equiv &\sum_{s=0}^{\infty }\left\vert \mathsf{B}\right\vert _{\limfunc{stop%
},1+\delta ,\mathsf{P}^{\omega }}^{A,\mathcal{P};s}\left( f,g\right) ; \\
\left\vert \mathsf{B}\right\vert _{\limfunc{stop},1+\delta ,\mathsf{P}%
^{\omega }}^{A,\mathcal{P};s}\left( f,g\right) &\equiv &\sum_{\substack{ %
\left( I,J\right) \in \mathcal{P}  \\ \ell \left( J\right) =2^{-s}\ell
\left( I\right) }}\frac{\mathrm{P}_{1+\delta }^{\alpha }\left( J,\left\vert 
\mathbb{E}_{I}^{\sigma }\left( \bigtriangleup _{\pi I}^{\sigma }f\right)
\right\vert \mathbf{1}_{A\setminus I}\sigma \right) }{\left\vert
J\right\vert ^{\frac{1}{n}}}\left\Vert \mathsf{P}_{J}^{\omega }\mathbf{x}%
\right\Vert _{L^{2}\left( \omega \right) }\left\Vert \bigtriangleup
_{J}^{\omega }g\right\Vert _{L^{2}\left( \omega \right) }.
\end{eqnarray*}%
Here we have the \emph{entire} projection $\mathsf{P}_{J}^{\omega }\mathbf{x}
$ onto all of the dyadic subquasicubes of $J$, but this is offset by the
smaller Poisson integral $\mathrm{P}_{1+\delta }^{\alpha }$. We will now
adapt the argument for the stopping term starting on page 42 of \cite%
{LaSaUr2}, where the geometric gain from the assumed Energy Hypothesis there
will be replaced by a geometric gain from the smaller Poisson integral $%
\mathrm{P}_{1+\delta }^{\alpha }$ used here.

First, we exploit the additional decay in the Poisson integral $\mathrm{P}%
_{1+\delta }^{\alpha }$ as follows. Suppose that $\left( I,J\right) \in 
\mathcal{P}$ with $\ell \left( J\right) =2^{-s}\ell \left( I\right) $. We
then compute%
\begin{eqnarray*}
\frac{\mathrm{P}_{1+\delta }^{\alpha }\left( J,\mathbf{1}_{A\setminus
I}\sigma \right) }{\left\vert J\right\vert ^{\frac{1}{n}}} &\approx
&\int_{A\setminus I}\frac{\left\vert J\right\vert ^{\frac{\delta }{n}}}{%
\left\vert y-c_{J}\right\vert ^{n+1+\delta -\alpha }}d\sigma \left( y\right)
\\
&\leq &\int_{A\setminus I}\left( \frac{\left\vert J\right\vert ^{\frac{1}{n}}%
}{\limfunc{quasidist}\left( c_{J},I^{c}\right) }\right) ^{\delta }\frac{1}{%
\left\vert y-c_{J}\right\vert ^{n+1-\alpha }}d\sigma \left( y\right) \\
&\lesssim &\left( \frac{\left\vert J\right\vert ^{\frac{1}{n}}}{\limfunc{%
quasidist}\left( c_{J},I^{c}\right) }\right) ^{\delta }\frac{\mathrm{P}%
^{\alpha }\left( J,\mathbf{1}_{A\setminus I}\sigma \right) }{\left\vert
J\right\vert ^{\frac{1}{n}}},
\end{eqnarray*}%
and use the goodness inequality,%
\begin{equation*}
\limfunc{quasidist}\left( c_{J},I^{c}\right) \geq \frac{1}{2}\ell \left(
I\right) ^{1-\varepsilon }\ell \left( J\right) ^{\varepsilon }\geq \frac{1}{2%
}2^{s\left( 1-\varepsilon \right) }\ell \left( J\right) ,
\end{equation*}%
to conclude that%
\begin{equation}
\left( \frac{\mathrm{P}_{1+\delta }^{\alpha }\left( J,\mathbf{1}_{A\setminus
I}\sigma \right) }{\left\vert J\right\vert ^{\frac{1}{n}}}\right) \lesssim
2^{-s\delta \left( 1-\varepsilon \right) }\frac{\mathrm{P}^{\alpha }\left( J,%
\mathbf{1}_{A\setminus I}\sigma \right) }{\left\vert J\right\vert ^{\frac{1}{%
n}}}.  \label{Poisson decay}
\end{equation}

We next claim that for $s\geq 0$ an integer,%
\begin{eqnarray*}
\left\vert \mathsf{B}\right\vert _{\limfunc{stop},1+\delta ,\mathsf{P}%
^{\omega }}^{A,\mathcal{P};s}\left( f,g\right) &=&\sum_{\substack{ \left(
I,J\right) \in \mathcal{P}  \\ \ell \left( J\right) =2^{-s}\ell \left(
I\right) }}\frac{\mathrm{P}_{1+\delta }^{\alpha }\left( J,\left\vert \mathbb{%
E}_{I}^{\sigma }\left( \bigtriangleup _{\pi I}^{\sigma }f\right) \right\vert 
\mathbf{1}_{A\setminus I}\sigma \right) }{\left\vert J\right\vert ^{\frac{1}{%
n}}}\left\Vert \mathsf{P}_{J}^{\omega }\mathbf{x}\right\Vert _{L^{2}\left(
\omega \right) }\left\Vert \bigtriangleup _{J}^{\omega }g\right\Vert
_{L^{2}\left( \omega \right) } \\
&\lesssim &2^{-s\delta \left( 1-\varepsilon \right) }\ \left( \mathcal{E}%
_{\alpha }^{\limfunc{deep}}+\sqrt{A_{2}^{\alpha }}\right) \ \left\Vert
f\right\Vert _{L^{2}\left( \sigma \right) }\ \left\Vert g\right\Vert
_{L^{2}\left( \omega \right) }\,,
\end{eqnarray*}%
from which (\ref{Second inequality}) follows upon summing in $s\geq 0$. Now
using both%
\begin{eqnarray*}
\left\vert \mathbb{E}_{I}^{\sigma }\left( \bigtriangleup _{\pi I}^{\sigma
}f\right) \right\vert &=&\frac{1}{\left\vert I\right\vert _{\sigma }}%
\int_{I}\left\vert \bigtriangleup _{\pi I}^{\sigma }f\right\vert d\sigma
\leq \left\Vert \bigtriangleup _{\pi I}^{\sigma }f\right\Vert _{L^{2}\left(
\sigma \right) }\frac{1}{\sqrt{\left\vert I\right\vert _{\sigma }}}, \\
2^{n}\left\Vert f\right\Vert _{L^{2}(\sigma )}^{2} &=&\sum_{I\in \Omega 
\mathcal{D}}\left\Vert \bigtriangleup _{\pi I}^{\sigma }f\right\Vert
_{L^{2}\left( \sigma \right) }^{2}\ ,
\end{eqnarray*}%
we apply Cauchy-Schwarz in the $I$ variable above to see that 
\begin{eqnarray*}
&&\left[ \left\vert \mathsf{B}\right\vert _{\limfunc{stop},1+\delta ,\mathsf{%
P}^{\omega }}^{A,\mathcal{P};s}\left( f,g\right) \right] ^{2}\lesssim
\left\Vert f\right\Vert _{L^{2}(\sigma )}^{2} \\
&&\times \left[ \sum_{I\in \mathcal{C}_{A}^{\prime }}\left( \frac{1}{\sqrt{%
\left\vert I\right\vert _{\sigma }}}\sum_{\substack{ J:\ \left( I,J\right)
\in \mathcal{P}  \\ \ell \left( J\right) =2^{-s}\ell \left( I\right) }}\frac{%
\mathrm{P}_{1+\delta }^{\alpha }\left( J,\mathbf{1}_{A\setminus I}\sigma
\right) }{\left\vert J\right\vert ^{\frac{1}{n}}}\left\Vert \mathsf{P}%
_{J}^{\omega }\mathbf{x}\right\Vert _{L^{2}\left( \omega \right) }\left\Vert
\bigtriangleup _{J}^{\omega }g\right\Vert _{L^{2}\left( \omega \right)
}\right) ^{2}\right] ^{\frac{1}{2}}.
\end{eqnarray*}%
We can then estimate the sum inside the square brackets by%
\begin{equation*}
\sum_{I\in \mathcal{C}_{A}^{\prime }}\left\{ \sum_{\substack{ J:\ \left(
I,J\right) \in \mathcal{P}  \\ \ell \left( J\right) =2^{-s}\ell \left(
I\right) }}\left\Vert \bigtriangleup _{J}^{\omega }g\right\Vert
_{L^{2}\left( \omega \right) }^{2}\right\} \sum_{\substack{ J:\ \left(
I,J\right) \in \mathcal{P}  \\ \ell \left( J\right) =2^{-s}\ell \left(
I\right) }}\frac{1}{\left\vert I\right\vert _{\sigma }}\left( \frac{\mathrm{P%
}_{1+\delta }^{\alpha }\left( J,\mathbf{1}_{A\setminus I}\sigma \right) }{%
\left\vert J\right\vert ^{\frac{1}{n}}}\right) ^{2}\left\Vert \mathsf{P}%
_{J}^{\omega }\mathbf{x}\right\Vert _{L^{2}\left( \omega \right)
}^{2}\lesssim \left\Vert g\right\Vert _{L^{2}\left( \omega \right)
}^{2}A\left( s\right) ^{2},
\end{equation*}%
where%
\begin{equation*}
A\left( s\right) ^{2}\equiv \sup_{I\in \mathcal{C}_{A}^{\prime }}\sum 
_{\substack{ J:\ \left( I,J\right) \in \mathcal{P}  \\ \ell \left( J\right)
=2^{-s}\ell \left( I\right) }}\frac{1}{\left\vert I\right\vert _{\sigma }}%
\left( \frac{\mathrm{P}_{1+\delta }^{\alpha }\left( J,\mathbf{1}_{A\setminus
I}\sigma \right) }{\left\vert J\right\vert ^{\frac{1}{n}}}\right)
^{2}\left\Vert \mathsf{P}_{J}^{\omega }\mathbf{x}\right\Vert _{L^{2}\left(
\omega \right) }^{2}\,.
\end{equation*}%
Finally then we turn to the analysis of the supremum in last display. From
the Poisson decay (\ref{Poisson decay}) we have 
\begin{eqnarray*}
A\left( s\right) ^{2} &\lesssim &\sup_{I\in \mathcal{C}_{A}^{\prime }}\frac{1%
}{\left\vert I\right\vert _{\sigma }}2^{-2s\delta \left( 1-\varepsilon
\right) }\sum_{\substack{ J:\ \left( I,J\right) \in \mathcal{P}  \\ \ell
\left( J\right) =2^{-s}\ell \left( I\right) }}\left( \frac{\mathrm{P}%
^{\alpha }\left( J,\mathbf{1}_{A\setminus I}\sigma \right) }{\left\vert
J\right\vert ^{\frac{1}{n}}}\right) ^{2}\left\Vert \mathsf{P}_{J}^{\omega
}x\right\Vert _{L^{2}\left( \omega \right) }^{2} \\
&\lesssim &\sup_{I\in \mathcal{C}_{A}^{\prime }}\frac{1}{\left\vert
I\right\vert _{\sigma }}2^{-2s\delta \left( 1-\varepsilon \right)
}\sum_{K\in \mathcal{M}_{\mathbf{r}-\limfunc{deep}}\left( I\right) }\left( 
\frac{\mathrm{P}^{\alpha }\left( K,\mathbf{1}_{A\setminus I}\sigma \right) }{%
\left\vert K\right\vert ^{\frac{1}{n}}}\right) ^{2}\sum_{\substack{ J\subset
K:\ \left( I,J\right) \in \mathcal{P}  \\ \ell \left( J\right) =2^{-s}\ell
\left( I\right) }}\left\Vert \mathsf{P}_{J}^{\omega }x\right\Vert
_{L^{2}\left( \omega \right) }^{2} \\
&\lesssim &2^{-2s\delta \left( 1-\varepsilon \right) }\left[ \left( \mathcal{%
E}_{\alpha }^{\limfunc{deep}}\right) ^{2}+A_{2}^{\alpha }\right] \,,
\end{eqnarray*}%
where the last inequality is the one for which the definition of quasienergy
stopping quasicubes was designed. Indeed, from Definition~\ref{def energy
corona 3}, as $(I,J)\in \mathcal{P}$, we have that $I$ is \emph{not} a
stopping quasicube in $\mathcal{A}$, and hence that (\ref{def stop 3}) \emph{%
fails} to hold, delivering the estimate above since $J\Subset _{\mathbf{\rho
-1}}I$ good must be contained in some $K\in \mathcal{M}_{\mathbf{r}-\limfunc{%
deep}}\left( I\right) $, and since $\frac{\mathrm{P}^{\alpha }\left( J,%
\mathbf{1}_{A\setminus I}\sigma \right) }{\left\vert J\right\vert ^{\frac{1}{%
n}}}\approx \frac{\mathrm{P}^{\alpha }\left( K,\mathbf{1}_{A\setminus
I}\sigma \right) }{\left\vert K\right\vert ^{\frac{1}{n}}}$. The terms $%
\left\Vert \mathsf{P}_{J}^{\omega }x\right\Vert _{L^{2}\left( \omega \right)
}^{2}$ are additive since the $J^{\prime }s$ are pigeonholed by $\ell \left(
J\right) =2^{-s}\ell \left( I\right) $.

\subsection{The first inequality and the recursion of M. Lacey}

Now we turn to proving the more difficult inequality (\ref{First inequality}%
). Recall that in dimension $n=1$ the energy condition%
\begin{equation*}
\sum_{n=1}^{\infty }\left\vert J_{n}\right\vert _{\omega }\mathsf{E}\left(
J_{n},\omega \right) ^{2}\mathrm{P}\left( J_{n},\mathbf{1}_{I}\sigma \right)
^{2}\lesssim \left( \mathcal{NTV}\right) \ \left\vert I\right\vert _{\sigma
},\ \ \ \ \ \overset{\cdot }{\mathop{\displaystyle \bigcup }}_{n=1}^{\infty
}J_{n}\subset I,
\end{equation*}%
could not be used in the NTV argument, because the set functional $%
J\rightarrow \left\vert J\right\vert _{\omega }\mathsf{E}\left( J,\omega
\right) ^{2}$ failed to be superadditive. On the other hand, the pivotal
condition of NTV,%
\begin{equation*}
\sum_{n=1}^{\infty }\left\vert J_{n}\right\vert _{\omega }\mathrm{P}\left(
J_{n},\mathbf{1}_{I}\sigma \right) ^{2}\lesssim \left\vert I\right\vert
_{\sigma },\ \ \ \ \ \overset{\cdot }{\mathop{\displaystyle \bigcup }}%
_{n=1}^{\infty }J_{n}\subset I,
\end{equation*}%
succeeded in the NTV argument because the set functional $J\rightarrow
\left\vert J\right\vert _{\omega }$ is trivially superadditive, indeed
additive. The final piece of the argument needed to prove the NTV conjecture
was found by M. Lacey in \cite{Lac}, and amounts to first replacing the
additivity of the functional $J\rightarrow \left\vert J\right\vert _{\omega
} $ with the additivity of the projection functional $\mathcal{H}\rightarrow
\left\Vert \mathsf{P}_{\mathcal{H}}^{\omega }x\right\Vert _{L^{2}\left(
\omega \right) }^{2}$ defined on subsets $\mathcal{H}$ of the dyadic
quasigrid $\Omega \mathcal{D}$. Then a stopping time argument relative to
this more subtle functional, together with a clever recursion, constitute
the main new ingredients in Lacey's argument \cite{Lac}.

To begin the extension to a more general Calder\'{o}n-Zygmund operator $%
T^{\alpha }$, we also recall the stopping quasienergy generalized to higher
dimensions by 
\begin{equation*}
\mathbf{X}^{\alpha }\left( \mathcal{C}_{A}\right) ^{2}\equiv \sup_{I\in 
\mathcal{C}_{A}}\frac{1}{\left\vert I\right\vert _{\sigma }}\sum_{J\in 
\mathcal{M}_{\mathbf{r}-\limfunc{deep}}\left( I\right) }\left( \frac{\mathrm{%
P}^{\alpha }\left( J,\mathbf{1}_{A\setminus \gamma J}\sigma \right) }{%
\left\vert J\right\vert ^{\frac{1}{n}}}\right) ^{2}\left\Vert \mathsf{P}%
_{J}^{\limfunc{subgood},\omega }\mathbf{x}\right\Vert _{L^{2}\left( \omega
\right) }^{2}\ ,
\end{equation*}%
where $\mathcal{M}_{\mathbf{r}-\limfunc{deep}}\left( I\right) $ is the set
of maximal $\mathbf{r}$-deeply embedded subquasicubes of $I$ where $\mathbf{r%
}$ is the goodness parameter. What now follows is an adaptation to our deep
quasienergy condition and the sublinear form $\left\vert \mathsf{B}%
\right\vert _{\limfunc{stop},1,\bigtriangleup ^{\omega }}^{A,\mathcal{P}}$
of the arguments of M. Lacey in \cite{Lac}. We have the following Poisson
inequality for quasicubes $B\subset A\subset I$:%
\begin{eqnarray}
\frac{\mathrm{P}^{\alpha }\left( A,\mathbf{1}_{I\setminus A}\sigma \right) }{%
\left\vert A\right\vert ^{\frac{1}{n}}} &\approx &\int_{I\setminus A}\frac{1%
}{\left( \left\vert y-c_{A}\right\vert \right) ^{n+1-\alpha }}d\sigma \left(
y\right)  \label{BAI} \\
&\lesssim &\int_{I\setminus A}\frac{1}{\left( \left\vert y-c_{B}\right\vert
\right) ^{n+1-\alpha }}d\sigma \left( y\right) \approx \frac{\mathrm{P}%
^{\alpha }\left( B,\mathbf{1}_{I\setminus A}\sigma \right) }{\left\vert
B\right\vert ^{\frac{1}{n}}}.  \notag
\end{eqnarray}

\subsection{The stopping energy}

Fix $A\in \mathcal{A}$. We will use a `decoupled' modification of the
stopping energy $\mathbf{X}\left( \mathcal{C}_{A}\right) $. Suppose that $%
\mathcal{P}$ is an $A$-admissible collection of pairs of quasicubes in the
product set $\Omega \mathcal{D}\times \Omega \mathcal{D}_{\limfunc{good}}$
of pairs of dyadic quasicubes in $\mathbb{R}^{n}$ with second component
good. For an admissible collection $\mathcal{P}$ let $\Pi _{1}\mathcal{P}$
and $\Pi _{2}\mathcal{P}$ be the quasicubes in the first and second
components of the pairs in $\mathcal{P}$ respectively, let $\Pi \mathcal{P}%
\equiv \Pi _{1}\mathcal{P}\cup \Pi _{2}\mathcal{P}$, and for $K\in \Pi 
\mathcal{P}$ define the $\mathbf{\tau }$-$\limfunc{deep}$ projection of $%
\mathcal{P}$ relative to $K$ by 
\begin{equation*}
\Pi _{2}^{K,\mathbf{\tau }-\limfunc{deep}}\mathcal{P}\equiv \left\{ J\in \Pi
_{2}\mathcal{P}:\ J\Subset _{\mathbf{\tau }}K\right\} .
\end{equation*}%
Now the quasicubes $J$ in$\ \Pi _{2}\mathcal{P}$ are of course \emph{always}
good, but this is \emph{not} the case for quasicubes $I$ in $\Pi _{1}%
\mathcal{P}$. Indeed, the collection $\mathcal{P}$ is tree-connected in the
first component, and it is clear that there can be many \emph{bad}
quasicubes in a connected geodesic in the tree $\Omega \mathcal{D}$. But the
quasiHaar support of $f$ is contained in \emph{good} quasicubes $I$, and we
have also assumed that the children of these quasicubes $I$ are good as
well. As a consequence we may always assume that our $A$-admissible
collections $\mathcal{P}$ are reduced in the sense of Definition \ref{def
reduced}. Thus we will use as our `size testing collection' of quasicubes
for $\mathcal{P}$ the collection 
\begin{equation*}
\Pi ^{\limfunc{goodbelow}}\mathcal{P}\equiv \left\{ K^{\prime }\in \Omega 
\mathcal{D}:K^{\prime }\text{ is }\limfunc{good}\text{ and }K^{\prime
}\subset K\text{ for some }K\in \Pi \mathcal{P}\right\} ,
\end{equation*}%
which consists of all the good subquasicubes of any quasicube in $\Pi 
\mathcal{P}$. Note that the maximal quasicubes in $\Pi \mathcal{P}=\Pi 
\mathcal{P}^{\limfunc{red}}$ are already good themselves, and so we have the
important property that%
\begin{equation}
\left( I,J\right) \in \mathcal{P}=\mathcal{P}^{\limfunc{red}}\text{ implies }%
I\subset K\text{ for some quasicube }K\in \Pi ^{\limfunc{goodbelow}}\mathcal{%
P}.  \label{inclusive}
\end{equation}%
Now define the `size functional' $\mathcal{S}_{\limfunc{size}}^{\alpha
,A}\left( \mathcal{P}\right) $ of $\mathcal{P}$ as follows. Recall that a
projection $\mathsf{P}_{\mathcal{H}}^{\omega }$ on $\mathbf{x}$ satisfies 
\begin{equation*}
\left\Vert \mathsf{P}_{\mathcal{H}}^{\omega }\mathbf{x}\right\Vert
_{L^{2}\left( \omega \right) }^{2}=\sum_{J\in \mathcal{H}}\left\Vert
\bigtriangleup _{J}^{\omega }\mathbf{x}\right\Vert _{L^{2}\left( \omega
\right) }^{2}.
\end{equation*}

\begin{definition}
\label{def ext size}If $\mathcal{P}$ is $A$-admissible, define%
\begin{equation}
\mathcal{S}_{\limfunc{size}}^{\alpha ,A}\left( \mathcal{P}\right) ^{2}\equiv
\sup_{K\in \Pi ^{\limfunc{goodbelow}}\mathcal{P}}\frac{1}{\left\vert
K\right\vert _{\sigma }}\left( \frac{\mathrm{P}^{\alpha }\left( K,\mathbf{1}%
_{A\setminus K}\sigma \right) }{\left\vert K\right\vert ^{\frac{1}{n}}}%
\right) ^{2}\left\Vert \mathsf{P}_{\Pi _{2}^{K,\mathbf{\tau }-\limfunc{deep}}%
\mathcal{P}}^{\omega }\mathbf{x}\right\Vert _{L^{2}\left( \omega \right)
}^{2}.  \label{def P stop energy 3}
\end{equation}
\end{definition}

We should remark that that the quasicubes $K$ in $\Pi ^{\limfunc{goodbelow}}%
\mathcal{P}$ that fail to contain any $\mathbf{\tau }$-parents of quasicubes
from $\Pi _{2}\mathcal{P}$ will not contribute to the size functional since $%
\Pi _{2}^{K,\mathbf{\tau }-\limfunc{deep}}\mathcal{P}$ is empty in this
case. We note three essential properties of this definition of size
functional:

\begin{enumerate}
\item \textbf{Monotonicity} of size: $\mathcal{S}_{\limfunc{size}}^{\alpha
,A}\left( \mathcal{P}\right) \leq \mathcal{S}_{\limfunc{size}}^{\alpha
,A}\left( \mathcal{Q}\right) $ if $\mathcal{P}\subset \mathcal{Q}$,

\item \textbf{Goodness} of testing quasicubes: $\Pi ^{\limfunc{goodbelow}}%
\mathcal{P}\subset \Omega \mathcal{D}_{\limfunc{good}}$,

\item \textbf{Control} by deep quasienergy condition: $\mathcal{S}_{\limfunc{%
size}}^{\alpha ,A}\left( \mathcal{P}\right) \lesssim \mathcal{E}_{\alpha }^{%
\limfunc{deep}}+\sqrt{A_{2}^{\alpha }}$.
\end{enumerate}

The monotonicity property follows from $\Pi ^{\limfunc{goodbelow}}\mathcal{P}%
\subset \Pi ^{\limfunc{goodbelow}}\mathcal{Q}$ and $\Pi _{2}^{K,\mathbf{\tau 
}-\limfunc{deep}}\mathcal{P}\subset \Pi _{2}^{K,\mathbf{\tau }-\limfunc{deep}%
}\mathcal{Q}$, and the goodness property follows from the definition of $\Pi
^{\limfunc{goodbelow}}\mathcal{P}$. The control property is contained in the
next lemma, which uses the stopping quasienergy control for the form $%
\mathsf{B}_{stop}^{A}\left( f,g\right) $ associated with $A$.

\begin{lemma}
\label{energy control}If $\mathcal{P}^{A}$ is as in (\ref{initial P}) and $%
\mathcal{P}\subset \mathcal{P}^{A}$, then 
\begin{equation*}
\mathcal{S}_{\limfunc{size}}^{\alpha ,A}\left( \mathcal{P}\right) \leq 
\mathbf{X}_{\alpha }\left( \mathcal{C}_{A}\right) \lesssim \mathcal{E}%
_{\alpha }^{\limfunc{deep}}+\sqrt{A_{2}^{\alpha }}\ .
\end{equation*}
\end{lemma}

\begin{proof}
Suppose that $K\in \Pi ^{\limfunc{goodbelow}}\mathcal{P}$. To prove the
first inequality in the statement we note that%
\begin{eqnarray*}
&&\frac{1}{\left\vert K\right\vert _{\sigma }}\left( \frac{\mathrm{P}%
^{\alpha }\left( K,\mathbf{1}_{A\setminus K}\sigma \right) }{\left\vert
K\right\vert ^{\frac{1}{n}}}\right) ^{2}\left\Vert \mathsf{P}_{\left( \Pi
_{2}^{K,\mathbf{\tau }-\limfunc{deep}}\mathcal{P}\right) ^{\ast }}^{\omega }%
\mathbf{x}\right\Vert _{L^{2}\left( \omega \right) }^{2} \\
&\leq &\frac{1}{\left\vert K\right\vert _{\sigma }}\left( \frac{\mathrm{P}%
^{\alpha }\left( K,\mathbf{1}_{A\setminus K}\sigma \right) }{\left\vert
K\right\vert ^{\frac{1}{n}}}\right) ^{2}\sum_{J\in \mathcal{M}_{\mathbf{r}-%
\limfunc{deep}}\left( K\right) }\left\Vert \mathsf{P}_{J}^{\limfunc{subgood}%
,\omega }\mathbf{x}\right\Vert _{L^{2}\left( \omega \right) }^{2} \\
&\lesssim &\frac{1}{\left\vert K\right\vert _{\sigma }}\sum_{J\in \mathcal{M}%
_{\mathbf{r}-\limfunc{deep}}\left( K\right) }\left( \frac{\mathrm{P}^{\alpha
}\left( J,\mathbf{1}_{A\setminus K}\sigma \right) }{\left\vert J\right\vert
^{\frac{1}{n}}}\right) ^{2}\left\Vert \mathsf{P}_{J}^{\limfunc{subgood}%
,\omega }\mathbf{x}\right\Vert _{L^{2}\left( \omega \right) }^{2} \\
&\lesssim &\frac{1}{\left\vert K\right\vert _{\sigma }}\sum_{J\in \mathcal{M}%
_{\mathbf{r}-\limfunc{deep}}\left( K\right) }\left( \frac{\mathrm{P}^{\alpha
}\left( J,\mathbf{1}_{A\setminus \gamma J}\sigma \right) }{\left\vert
J\right\vert ^{\frac{1}{n}}}\right) ^{2}\left\Vert \mathsf{P}_{J}^{\limfunc{%
subgood},\omega }\mathbf{x}\right\Vert _{L^{2}\left( \omega \right)
}^{2}\leq \mathbf{X}_{\alpha }\left( \mathcal{C}_{A}\right) ^{2},
\end{eqnarray*}%
where the first inequality above follows since every $J^{\prime }\in \Pi
_{2}^{K,\mathbf{\tau }-\limfunc{deep}}\mathcal{P}$ is contained in some $%
J\in \mathcal{M}_{\mathbf{r}-\limfunc{deep}}\left( I\right) $, the second
inequality follows from (\ref{BAI}) with $J\subset K\subset A$, and then the
third inequality follows since $J\Subset _{\mathbf{r}}I$ implies $\gamma
J\subset I$ by (\ref{bounded overlap}), and finally since $\Pi _{2}^{K,%
\mathbf{\tau }-\limfunc{deep}}\mathcal{P}=\emptyset $ if $K\subset A$ and $%
K\notin \mathcal{C}_{A}$ by (\ref{later use}) below. The second inequality
in the statement of the lemma follows from (\ref{def stopping bounds 3}).
\end{proof}

The following useful fact is needed above and will be used later as well:%
\begin{equation}
K\subset A\text{ and }K\notin \mathcal{C}_{A}\Longrightarrow \Pi _{2}^{K,%
\mathbf{\tau }-\limfunc{deep}}\mathcal{P}=\emptyset \ .  \label{later use}
\end{equation}%
To see this, suppose that $K\in \mathcal{C}_{A}^{\mathbf{\tau }-\limfunc{%
shift}}\setminus \mathcal{C}_{A}$. Then $K\subset A^{\prime }$ for some $%
A^{\prime }\in \mathfrak{C}_{\mathcal{A}}\left( A\right) $, and so if there
is $J\in \Pi _{2}^{K,\mathbf{\tau }-\limfunc{deep}}\mathcal{P}$, then $\ell
\left( J\right) \leq 2^{-\mathbf{\tau }}\ell \left( K\right) \leq 2^{-%
\mathbf{\tau }}\ell \left( A^{\prime }\right) $, which implies that $J\notin 
\mathcal{C}_{A}^{\mathbf{\tau }-\limfunc{shift}}$, which contradicts $\Pi
_{2}^{K,\mathbf{\tau }-\limfunc{deep}}\mathcal{P}\subset \mathcal{C}_{A}^{%
\mathbf{\tau }-\limfunc{shift}}$.

We remind the reader again that $c\left\vert J\right\vert ^{\frac{1}{n}}\leq
\ell \left( J\right) \leq C\left\vert J\right\vert ^{\frac{1}{n}}$ for any
quasicube $J$, and that we will generally use $\left\vert J\right\vert ^{%
\frac{1}{n}}$ in the Poisson integrals and estimates, but will usually use $%
\ell \left( J\right) $ when defining collections of quasicubes. Now define
an atomic measure $\omega _{\mathcal{P}}$ in the upper half space $\mathbb{R}%
_{+}^{n+1}$ by%
\begin{equation*}
\omega _{\mathcal{P}}\equiv \sum_{J\in \Pi _{2}\mathcal{P}}\left\Vert
\bigtriangleup _{J}^{\omega }\mathbf{x}\right\Vert _{L^{2}\left( \omega
\right) }^{2}\ \delta _{\left( c_{J},\ell \left( J\right) \right) }.
\end{equation*}%
Define the tent $\mathbf{T}\left( K\right) $ over a quasicube $K=\Omega L$
to be $\Omega \left( \mathbf{T}\left( L\right) \right) $ where $\mathbf{T}%
\left( L\right) $ is the convex hull of the $n$-cube $L\times \left\{
0\right\} $ and the point $\left( c_{L},\ell \left( L\right) \right) \in 
\mathbb{R}_{+}^{n+1}$. Define the $\mathbf{\tau }$-$\limfunc{deep}$ tent $%
\mathbf{T}^{\mathbf{\tau }-\limfunc{deep}}\left( K\right) $ over a quasicube 
$K$ to be the restriction of the tent $\mathbf{T}\left( K\right) $ to those
points at depth $\tau $ or more below $K$:%
\begin{equation*}
\mathbf{T}^{\mathbf{\tau }-\limfunc{deep}}\left( K\right) \equiv \left\{
\left( y,t\right) \in \mathbf{T}\left( K\right) :t\leq 2^{-\tau }\ell \left(
K\right) \right\} .
\end{equation*}%
We can now rewrite the size functional (\ref{def P stop energy 3}) of $%
\mathcal{P}$ as%
\begin{equation}
\mathcal{S}_{\limfunc{size}}^{\alpha ,A}\left( \mathcal{P}\right) ^{2}\equiv
\sup_{K\in \Pi ^{\limfunc{goodbelow}}\mathcal{P}}\frac{1}{\left\vert
K\right\vert _{\sigma }}\left( \frac{\mathrm{P}^{\alpha }\left( K,\mathbf{1}%
_{A\setminus K}\sigma \right) }{\left\vert K\right\vert ^{\frac{1}{n}}}%
\right) ^{2}\omega _{\mathcal{P}}\left( \mathbf{T}^{\mathbf{\tau }-\limfunc{%
deep}}\left( K\right) \right) .  \label{def P stop energy' 3}
\end{equation}%
It will be convenient to write%
\begin{equation*}
\Psi ^{\alpha }\left( K;\mathcal{P}\right) ^{2}\equiv \left( \frac{\mathrm{P}%
^{\alpha }\left( K,\mathbf{1}_{A\setminus K}\sigma \right) }{\left\vert
K\right\vert ^{\frac{1}{n}}}\right) ^{2}\omega _{\mathcal{P}}\left( \mathbf{T%
}^{\mathbf{\tau }-\limfunc{deep}}\left( K\right) \right) ,
\end{equation*}%
so that we have simply%
\begin{equation*}
\mathcal{S}_{\limfunc{size}}^{\alpha ,A}\left( \mathcal{P}\right)
^{2}=\sup_{K\in \Pi ^{\limfunc{goodbelow}}\mathcal{P}}\frac{\Psi ^{\alpha
}\left( K;\mathcal{P}\right) ^{2}}{\left\vert K\right\vert _{\sigma }}.
\end{equation*}

\begin{remark}
The functional $\omega _{\mathcal{P}}\left( \mathbf{T}^{\mathbf{\tau }-%
\limfunc{deep}}\left( K\right) \right) $ is increasing in $K$, while the
functional $\frac{\mathrm{P}^{\alpha }\left( K,\mathbf{1}_{A\setminus
K}\sigma \right) }{\left\vert K\right\vert ^{\frac{1}{n}}}$ is `almost
decreasing' in $K$: if $K_{0}\subset K$ then%
\begin{eqnarray*}
\frac{\mathrm{P}^{\alpha }\left( K,\mathbf{1}_{A\setminus K}\sigma \right) }{%
\left\vert K\right\vert ^{\frac{1}{n}}} &=&\int_{A\setminus K}\frac{d\sigma
\left( y\right) }{\left( \left\vert K\right\vert ^{\frac{1}{n}}+\left\vert
y-c_{K}\right\vert \right) ^{n+1-\alpha }} \\
&\lesssim &\int_{A\setminus K}\frac{\left( \sqrt{n}\right) ^{n+1-\alpha
}d\sigma \left( y\right) }{\left( \left\vert K_{0}\right\vert ^{\frac{1}{n}%
}+\left\vert y-c_{K_{0}}\right\vert \right) ^{n+1-\alpha }} \\
&\leq &C_{n,\alpha }\int_{A\setminus K_{0}}\frac{d\sigma \left( y\right) }{%
\left( \left\vert K_{0}\right\vert ^{\frac{1}{n}}+\left\vert
y-c_{K_{0}}\right\vert \right) ^{n+1-\alpha }}=C_{n,\alpha }\frac{\mathrm{P}%
^{\alpha }\left( K_{0},\mathbf{1}_{A\setminus K_{0}}\sigma \right) }{%
\left\vert K_{0}\right\vert ^{\frac{1}{n}}},
\end{eqnarray*}%
since $\left\vert K_{0}\right\vert ^{\frac{1}{n}}+\left\vert
y-c_{K_{0}}\right\vert \leq \left\vert K\right\vert ^{\frac{1}{n}%
}+\left\vert y-c_{K}\right\vert +\frac{1}{2}\limfunc{diam}\left( K\right) $
for $y\in A\setminus K$.
\end{remark}

\subsection{The recursion}

Recall that if $\mathcal{P}$ is an admissible collection for a dyadic
quasicube $A$, the corresponding sublinear form in (\ref{First inequality})
is given in (\ref{def split}) by%
\begin{eqnarray*}
\left\vert \mathsf{B}\right\vert _{\limfunc{stop},1,\bigtriangleup ^{\omega
}}^{A,\mathcal{P}}\left( f,g\right) &\equiv &\sum_{J\in \Pi _{2}\mathcal{P}}%
\frac{\mathrm{P}^{\alpha }\left( J,\left\vert \varphi _{J}^{\mathcal{P}%
}\right\vert \mathbf{1}_{A\setminus I_{\mathcal{P}}\left( J\right) }\sigma
\right) }{\left\vert J\right\vert ^{\frac{1}{n}}}\left\Vert \bigtriangleup
_{J}^{\omega }\mathbf{x}\right\Vert _{L^{2}\left( \omega \right) }\left\Vert
\bigtriangleup _{J}^{\omega }g\right\Vert _{L^{2}\left( \omega \right) }; \\
\text{where }\varphi _{J}^{\mathcal{P}} &\equiv &\sum_{I\in \mathcal{C}%
_{A}^{\prime }:\ \left( I,J\right) \in \mathcal{P}}\mathbb{E}_{I}^{\sigma
}\left( \bigtriangleup _{\pi I}^{\sigma }f\right) \ \mathbf{1}_{A\setminus
I}\ .
\end{eqnarray*}%
In the notation for $\left\vert \mathsf{B}\right\vert _{\limfunc{stop}%
,1,\bigtriangleup ^{\omega }}^{A,\mathcal{P}}$, we are omitting dependence
on the parameter $\alpha $, and to avoid clutter, we will often do so from
now on when the dependence on $\alpha $ is inconsequential. Following Lacey 
\cite{Lac}, we now claim the following proposition, from which we obtain (%
\ref{First inequality}) as a corollary below. Motivated by the conclusion of
Proposition \ref{stopping bound}, we define the \emph{restricted} norm $%
\mathfrak{N}_{\limfunc{stop},1,\bigtriangleup }^{A,\mathcal{P}}$ of the
sublinear form $\left\vert \mathsf{B}\right\vert _{\limfunc{stop}%
,1,\bigtriangleup ^{\omega }}^{A,\mathcal{P}}$ to be the best constant $%
\mathfrak{N}_{\limfunc{stop},1,\bigtriangleup }^{A,\mathcal{P}}$ in the
inequality%
\begin{equation*}
\left\vert \mathsf{B}\right\vert _{\limfunc{stop},1,\bigtriangleup ^{\omega
}}^{A,\mathcal{P}}\left( f,g\right) \leq \mathfrak{N}_{\limfunc{stop}%
,1,\bigtriangleup }^{A,\mathcal{P}}\left( \alpha _{\mathcal{A}}\left(
A\right) \sqrt{\left\vert A\right\vert _{\sigma }}+\left\Vert f\right\Vert
_{L^{2}\left( \sigma \right) }\right) \left\Vert g\right\Vert _{L^{2}\left(
\omega \right) },
\end{equation*}%
where $f\in L^{2}\left( \sigma \right) $ satisfies $\mathbb{E}_{I}^{\sigma
}\left\vert f\right\vert \leq \alpha _{\mathcal{A}}\left( A\right) $ for all 
$I\in \mathcal{C}_{A}^{\limfunc{good}}$.

\begin{proposition}
\label{bottom up 3}(This is a variant for sublinear forms of the Size Lemma
in Lacey \cite{Lac}) Suppose $\varepsilon >0$. Let $\mathcal{P}$ be an \emph{%
admissible} collection of pairs for a dyadic quasicube $A$. Then we can
decompose $\mathcal{P}$ into two disjoint collections $\mathcal{P}=\mathcal{P%
}^{big}\dot{\cup}\mathcal{P}^{small}$, and further decompose $\mathcal{P}%
^{small}$ into pairwise disjoint collections $\mathcal{P}_{1}^{small},%
\mathcal{P}_{2}^{small}...\mathcal{P}_{\ell }^{small}...$ i.e.%
\begin{equation*}
\mathcal{P}=\mathcal{P}^{big}\dot{\cup}\left( \overset{\cdot }{\dbigcup }%
_{\ell =1}^{\infty }\mathcal{P}_{\ell }^{small}\right) \ ,
\end{equation*}%
such that the collections $\mathcal{P}^{big}$ and $\mathcal{P}_{\ell
}^{small}$ are admissible and satisfy 
\begin{equation}
\sup_{\ell \geq 1}\mathcal{S}_{\limfunc{size}}^{\alpha ,A}\left( \mathcal{P}%
_{\ell }^{small}\right) ^{2}\leq \varepsilon \mathcal{S}_{\limfunc{size}%
}^{\alpha ,A}\left( \mathcal{P}\right) ^{2},  \label{small 3}
\end{equation}%
and 
\begin{equation}
\mathfrak{N}_{\limfunc{stop},1,\bigtriangleup }^{A,\mathcal{P}}\leq
C_{\varepsilon }\mathcal{S}_{\limfunc{size}}^{\alpha ,A}\left( \mathcal{P}%
\right) +\sqrt{n\mathbf{\tau }}\sup_{\ell \geq 1}\mathfrak{N}_{\limfunc{stop}%
,1,\bigtriangleup }^{A,\mathcal{P}_{\ell }^{small}}\ .  \label{big 3}
\end{equation}
\end{proposition}

\begin{corollary}
The sublinear stopping form inequality (\ref{First inequality}) holds.
\end{corollary}

\begin{proof}[Proof of the Corollary]
Set $\mathcal{Q}^{0}=\mathcal{P}^{A}$. Apply Proposition \ref{bottom up 3}
to obtain a subdecomposition $\left\{ \mathcal{Q}_{\ell }^{1}\right\} _{\ell
=1}^{\infty }$ of $\mathcal{Q}^{0}$ such that%
\begin{eqnarray*}
&&\mathfrak{N}_{\limfunc{stop},1,\bigtriangleup }^{A,\mathcal{Q}^{0}}\leq
C_{\varepsilon }\mathcal{S}_{\limfunc{size}}^{\alpha ,A}\left( \mathcal{Q}%
^{0}\right) +\sqrt{n\mathbf{\tau }}\sup_{\ell \geq 1}\mathfrak{N}_{\limfunc{%
stop},1,\bigtriangleup }^{A,\mathcal{Q}_{\ell }^{1}}\ , \\
&&\sup_{\ell \geq 1}\mathcal{S}_{\limfunc{size}}^{\alpha ,A}\left( \mathcal{Q%
}_{\ell }^{1}\right) \leq \varepsilon \mathcal{S}_{\limfunc{size}}^{\alpha
,A}\left( \mathcal{Q}^{0}\right) .
\end{eqnarray*}%
Now apply Proposition \ref{bottom up 3} to each $\mathcal{Q}_{\ell }^{1}$ to
obtain a subdecomposition $\left\{ \mathcal{Q}_{\ell ,k}^{2}\right\}
_{k=1}^{\infty }$ of $\mathcal{Q}_{\ell }^{1}$ such that%
\begin{eqnarray*}
\mathfrak{N}_{\limfunc{stop},1,\bigtriangleup }^{A,\mathcal{Q}_{\ell }^{1}}
&\leq &C_{\varepsilon }\mathcal{S}_{\limfunc{size}}^{\alpha ,A}\left( 
\mathcal{Q}_{\ell }^{1}\right) +\sqrt{n\mathbf{\tau }}\sup_{k\geq 1}%
\mathfrak{N}_{\limfunc{stop},1,\bigtriangleup }^{A,\mathcal{Q}_{\ell
,k}^{2}}\ , \\
&&\sup_{k\geq 1}\mathcal{S}_{\limfunc{size}}^{\alpha ,A}\left( \mathcal{Q}%
_{\ell ,k}^{2}\right) \leq \varepsilon \mathcal{S}_{\limfunc{size}}^{\alpha
,A}\left( \mathcal{Q}_{\ell }^{1}\right) .
\end{eqnarray*}%
Altogether we have%
\begin{eqnarray*}
\mathfrak{N}_{\limfunc{stop},1,\bigtriangleup }^{A,\mathcal{Q}^{0}} &\leq
&C_{\varepsilon }\mathcal{S}_{\limfunc{size}}^{\alpha ,A}\left( \mathcal{Q}%
^{0}\right) +\sqrt{n\mathbf{\tau }}\sup_{\ell \geq 1}\left\{ C_{\varepsilon }%
\mathcal{S}_{\limfunc{size}}^{\alpha ,A}\left( \mathcal{Q}_{\ell
}^{1}\right) +\sqrt{n\mathbf{\tau }}\sup_{k\geq 1}\mathfrak{N}_{\limfunc{stop%
},1,\bigtriangleup }^{A,\mathcal{Q}_{\ell ,k}^{2}}\right\} \\
&=&C_{\varepsilon }\left\{ \mathcal{S}_{\limfunc{size}}^{\alpha ,A}\left( 
\mathcal{Q}^{0}\right) +\sqrt{n\mathbf{\tau }}\varepsilon \mathcal{S}_{%
\limfunc{size}}^{\alpha ,A}\left( \mathcal{Q}^{0}\right) \right\} +\left( n%
\mathbf{\tau }\right) \sup_{\ell ,k\geq 1}\mathfrak{N}_{\limfunc{stop}%
,1,\bigtriangleup }^{A,\mathcal{Q}_{\ell ,k}^{2}}\ .
\end{eqnarray*}%
Then with $\zeta \equiv \sqrt{n\mathbf{\tau }}$, we obtain by induction for
every $N\in \mathbb{N}$,%
\begin{eqnarray*}
\mathfrak{N}_{\limfunc{stop},1,\bigtriangleup }^{A,\mathcal{Q}^{0}} &\leq
&C_{\varepsilon }\left\{ \mathcal{S}_{\limfunc{size}}^{\alpha ,A}\left( 
\mathcal{Q}^{0}\right) +\zeta \varepsilon \mathcal{S}_{\limfunc{size}%
}^{\alpha ,A}\left( \mathcal{Q}^{0}\right) +...\zeta ^{N}\varepsilon ^{N}%
\mathcal{S}_{\limfunc{size}}^{\alpha ,A}\left( \mathcal{Q}^{0}\right)
\right\} \\
&&+\zeta ^{N+1}\sup_{m\in \mathbb{N}^{N+1}}\mathfrak{N}_{\limfunc{stop}%
,1,\bigtriangleup }^{A,\mathcal{Q}_{m}^{N+1}}\ .
\end{eqnarray*}%
Now we may assume the collection $\mathcal{Q}^{0}=\mathcal{P}^{A}$ of pairs
is finite (simply truncate the corona $\mathcal{C}_{A}$ and obtain bounds
independent of the truncation) and so $\sup_{m\in \mathbb{N}^{N+1}}\mathfrak{%
N}_{\limfunc{stop},1,\bigtriangleup }^{A,\mathcal{Q}_{m}^{N+1}}=0$ for $N$
large enough. Then we obtain (\ref{First inequality}) if we choose $%
0<\varepsilon <\frac{1}{1+\zeta }$ and apply Lemma \ref{energy control}.
\end{proof}

\begin{proof}[Proof of Proposition \protect\ref{bottom up 3}]
Recall that the `size testing collection' of quasicubes $\Pi ^{\limfunc{%
goodbelow}}\mathcal{P}$ is the collection of all \emph{good} subquasicubes
of a quasicube in $\Pi \mathcal{P}$. We may assume that $\mathcal{P}$ is a
finite collection. Begin by defining the collection $\mathcal{L}_{0}$ to
consist of the \emph{minimal} dyadic quasicubes $K$ in $\Pi ^{\limfunc{%
goodbelow}}\mathcal{P}$ such that%
\begin{equation*}
\frac{\Psi ^{\alpha }\left( K;\mathcal{P}\right) ^{2}}{\left\vert
K\right\vert _{\sigma }}\geq \varepsilon \mathcal{S}_{\limfunc{size}%
}^{\alpha ,A}\left( \mathcal{P}\right) ^{2}.
\end{equation*}%
where we recall that%
\begin{equation*}
\Psi ^{\alpha }\left( K;\mathcal{P}\right) ^{2}\equiv \left( \frac{\mathrm{P}%
^{\alpha }\left( K,\mathbf{1}_{A\setminus K}\sigma \right) }{\left\vert
K\right\vert ^{\frac{1}{n}}}\right) ^{2}\omega _{\mathcal{P}}\left( \mathbf{T%
}^{\mathbf{\tau }-\limfunc{deep}}\left( K\right) \right) .
\end{equation*}%
Note that such minimal quasicubes exist when $0<\varepsilon <1$ because $%
\mathcal{S}_{\limfunc{size}}^{\alpha ,A}\left( \mathcal{P}\right) ^{2}$ is
the supremum over $K\in \Pi ^{\limfunc{goodbelow}}\mathcal{P}$ of $\frac{%
\Psi ^{\alpha }\left( K;\mathcal{P}\right) ^{2}}{\left\vert K\right\vert
_{\sigma }}$. A key property of the the minimality requirement is that%
\begin{equation}
\frac{\Psi ^{\alpha }\left( K^{\prime };\mathcal{P}\right) ^{2}}{\left\vert
K^{\prime }\right\vert _{\sigma }}<\varepsilon \mathcal{S}_{\limfunc{size}%
}^{\alpha ,A}\left( \mathcal{P}\right) ^{2},  \label{key property 3}
\end{equation}%
for all $K^{\prime }\in \Pi ^{\limfunc{goodbelow}}\mathcal{P}$ with $%
K^{\prime }\varsubsetneqq K$ and $K\in \mathcal{L}_{0}$.

We now perform a stopping time argument `from the bottom up' with respect to
the atomic measure $\omega _{\mathcal{P}}$ in the upper half space. This
construction of a stopping time `from the bottom up' is one of two key
innovations in Lacey's argument \cite{Lac}, the other being the recursion
described in Proposition \ref{bottom up 3}.

We refer to $\mathcal{L}_{0}$ as the initial or level $0$ generation of
stopping times. Choose $\rho =1+\varepsilon $. We then recursively define a
sequence of generations $\left\{ \mathcal{L}_{m}\right\} _{m=0}^{\infty }$
by letting $\mathcal{L}_{m}$ consist of the \emph{minimal} dyadic quasicubes 
$L$ in $\Pi ^{\limfunc{goodbelow}}\mathcal{P}$ that contain a quasicube from
some previous level $\mathcal{L}_{\ell }$, $\ell <m$, such that%
\begin{equation}
\omega _{\mathcal{P}}\left( \mathbf{T}^{\mathbf{\tau }-\limfunc{deep}}\left(
L\right) \right) \geq \rho \omega _{\mathcal{P}}\left(
\dbigcup\limits_{L^{\prime }\in \dbigcup\limits_{\ell =0}^{m-1}\mathcal{L}%
_{\ell }:\ L^{\prime }\subset L}\mathbf{T}^{\mathbf{\tau }-\limfunc{deep}%
}\left( L^{\prime }\right) \right) .  \label{up stopping condition}
\end{equation}%
Since $\mathcal{P}$ is finite this recursion stops at some level $M$. We
then let $\mathcal{L}_{M+1}$ consist of all the maximal quasicubes in $\Pi ^{%
\limfunc{goodbelow}}\mathcal{P}$ that are not already in some $\mathcal{L}%
_{m}$. Thus $\mathcal{L}_{M+1}$ will contain either none, some, or all of
the maximal quasicubes in $\Pi ^{\limfunc{goodbelow}}\mathcal{P}$. We do not
of course have (\ref{up stopping condition}) for $A^{\prime }\in \mathcal{L}%
_{M+1}$ in this case, but we do have that (\ref{up stopping condition})
fails for subquasicubes $K$ of $A^{\prime }\in \mathcal{L}_{M+1}$ that are
not contained in any other $L\in \mathcal{L}_{m}$, and this is sufficient
for the arguments below.

We now define the collections $\mathcal{P}^{small}$ and $\mathcal{P}^{big}$.
The collection $\mathcal{P}^{big}$ will consist of those pairs $\left(
I,J\right) \in \mathcal{P}$ for which there is $L\in
\dbigcup\limits_{m=0}^{M+1}\mathcal{L}_{m}$ with $J\Subset _{\tau }L\subset
I $, and $\mathcal{P}^{small}$ will consist of the remaining pairs. But a
considerable amount of further analysis is required to prove the conclusion
of the proposition. First, let $\mathcal{L}\equiv \dbigcup\limits_{m=0}^{M+1}%
\mathcal{L}_{m}$ be the tree of stopping quasienergy quasicubes defined
above. By our construction above, the maximal elements in $\mathcal{L}$ are
the maximal quasicubes in $\Pi ^{\limfunc{goodbelow}}\mathcal{P}$. For $L\in 
\mathcal{L}$, denote by $\mathcal{C}_{L}$ the \emph{corona} associated with $%
L$ in the tree $\mathcal{L}$,%
\begin{equation*}
\mathcal{C}_{L}\equiv \left\{ K\in \Omega \mathcal{D}:K\subset L\text{ and
there is no }L^{\prime }\in \mathcal{L}\text{ with }K\subset L^{\prime
}\subsetneqq L\right\} ,
\end{equation*}%
and define the \emph{shifted} corona by%
\begin{equation*}
\mathcal{C}_{L}^{\mathbf{\tau }-\limfunc{shift}}\equiv \left\{ K\in \mathcal{%
C}_{L}:K\Subset _{\mathbf{\tau }}L\right\} \cup \dbigcup\limits_{L^{\prime
}\in \mathfrak{C}_{\mathcal{L}}\left( L\right) }\left\{ K\in \Omega \mathcal{%
D}:K\Subset _{\mathbf{\tau }}L\text{ and }K\text{ is }\mathbf{\tau }\text{%
-nearby in }L^{\prime }\right\} .
\end{equation*}%
Now the parameter $m$ in $\mathcal{L}_{m}$ refers to the level at which the
stopping construction was performed, but for \thinspace $L\in \mathcal{L}%
_{m} $, the corona children $L^{\prime }$ of $L$ are \emph{not} all
necessarily in $\mathcal{L}_{m-1}$, but may be in $\mathcal{L}_{m-t}$ for $t$
large. Thus we need to introduce the notion of geometric depth $d$ in the
tree $\mathcal{L}$ by defining%
\begin{eqnarray*}
\mathcal{G}_{0} &\equiv &\left\{ L\in \mathcal{L}:L\text{ is maximal}%
\right\} , \\
\mathcal{G}_{1} &\equiv &\left\{ L\in \mathcal{L}:L\text{ is maximal wrt }%
L\subsetneqq L_{0}\text{ for some }L_{0}\in \mathcal{G}_{0}\right\} , \\
&&\vdots \\
\mathcal{G}_{d+1} &\equiv &\left\{ L\in \mathcal{L}:L\text{ is maximal wrt }%
L\subsetneqq L_{d}\text{ for some }L_{d}\in \mathcal{G}_{d}\right\} , \\
&&\vdots
\end{eqnarray*}%
We refer to $\mathcal{G}_{d}$ as the $d^{th}$ generation of quasicubes in
the tree $\mathcal{L}$, and say that the quasicubes in $\mathcal{G}_{d}$ are
at depth $d$ in the tree $\mathcal{L}$. Thus the quasicubes in $\mathcal{G}%
_{d}$ are the stopping quasicubes in $\mathcal{L}$ that are $d$ levels in
the \emph{geometric} sense below the top level.

Then for $L\in \mathcal{G}_{d}$ and $t\geq 0$ define%
\begin{equation*}
\mathcal{P}_{L,t}\equiv \left\{ \left( I,J\right) \in \mathcal{P}:I\in 
\mathcal{C}_{L}\text{ and }J\in \mathcal{C}_{L^{\prime }}^{\mathbf{\tau }-%
\limfunc{shift}}\text{ for some }L^{\prime }\in \mathcal{G}_{d+t}\text{ with 
}L^{\prime }\subset L\right\} .
\end{equation*}%
In particular, $\left( I,J\right) \in \mathcal{P}_{L,t}$ implies that $I$ is
in the corona $\mathcal{C}_{L}$, and that $J$ is in a shifted corona $%
\mathcal{C}_{L^{\prime }}^{\mathbf{\tau }-\limfunc{shift}}$ that is $t$
levels of generation \emph{below} $\mathcal{C}_{L}$. We emphasize the
distinction `generation' as this refers to the depth rather than the level
of stopping construction. For $t=0$ we further decompose $\mathcal{P}_{L,0}$
as%
\begin{eqnarray*}
\mathcal{P}_{L,0} &=&\mathcal{P}_{L,0}^{small}\dot{\cup}\mathcal{P}%
_{L,0}^{big}; \\
\mathcal{P}_{L,0}^{small} &\equiv &\left\{ \left( I,J\right) \in \mathcal{P}%
_{L,0}:I\neq L\right\} , \\
\mathcal{P}_{L,0}^{big} &\equiv &\left\{ \left( I,J\right) \in \mathcal{P}%
_{L,0}:I=L\right\} ,
\end{eqnarray*}%
with one exception: if $L\in \mathcal{L}_{M+1}$ we set $\mathcal{P}%
_{L,0}^{small}\equiv \mathcal{P}_{L,0}$ since in this case $L$ fails to
satisfy (\ref{up stopping condition}) as pointed out above. Then we set%
\begin{eqnarray*}
\mathcal{P}^{big} &\equiv &\left\{ \dbigcup\limits_{L\in \mathcal{L}}%
\mathcal{P}_{L,0}^{big}\right\} \dbigcup \left\{ \dbigcup\limits_{t\geq
1}\dbigcup\limits_{L\in \mathcal{L}}\mathcal{P}_{L,t}\right\} ; \\
\left\{ \mathcal{P}_{\ell }^{small}\right\} _{\ell =0}^{\infty } &\equiv
&\left\{ \mathcal{P}_{L,0}^{small}\right\} _{L\in \mathcal{L}},\ \ \ \ \ 
\text{after relabelling}.
\end{eqnarray*}%
It is important to note that by (\ref{inclusive}), every pair $\left(
I,J\right) \in \mathcal{P}$ will be included in either $\mathcal{P}^{small}$
or $\mathcal{P}^{big}$. Now we turn to proving the inequalities (\ref{small
3}) and (\ref{big 3}).

To prove the inequality (\ref{small 3}), it suffices with the above
relabelling to prove the following claim:%
\begin{equation}
\mathcal{S}_{\limfunc{size}}^{\alpha ,A}\left( \mathcal{P}%
_{L,0}^{small}\right) ^{2}\leq \left( \rho -1\right) \mathcal{S}_{\limfunc{%
size}}^{\alpha ,A}\left( \mathcal{P}\right) ^{2},\ \ \ \ \ L\in \mathcal{L}.
\label{small claim' 3}
\end{equation}%
To see (\ref{small claim' 3}), suppose first that $L\notin \mathcal{L}_{M+1}$%
. In the case that $L\in \mathcal{L}_{0}$ is an initial generation
quasicube, then from (\ref{key property 3}) we obtain that%
\begin{eqnarray*}
&&\mathcal{S}_{\limfunc{size}}^{\alpha ,A}\left( \mathcal{P}%
_{L,0}^{small}\right) ^{2} \\
&\leq &\sup_{K^{\prime }\in \Pi ^{\limfunc{goodbelow}}\mathcal{P}:\
K^{\prime }\varsubsetneqq L}\frac{1}{\left\vert K^{\prime }\right\vert
_{\sigma }}\left( \frac{\mathrm{P}^{\alpha }\left( K^{\prime },\mathbf{1}%
_{A\setminus K^{\prime }}\sigma \right) }{\left\vert K^{\prime }\right\vert
^{\frac{1}{n}}}\right) ^{2}\omega _{\mathcal{P}}\left( \mathbf{T}^{\mathbf{%
\tau }-\limfunc{deep}}\left( K^{\prime }\right) \right) \leq \varepsilon 
\mathcal{S}_{\limfunc{size}}^{\alpha ,A}\left( \mathcal{P}\right) ^{2}.
\end{eqnarray*}%
Now suppose that $L\not\in \mathcal{L}_{0}$ and also that $L\notin \mathcal{L%
}_{M+1}$. Pick a pair $\left( I,J\right) \in \mathcal{P}_{L,0}^{small}$.
Then $I$ is in the restricted corona $\mathcal{C}_{L}^{\prime }$ and $J$ is
in the $\mathbf{\tau }$\emph{-shifted} corona $\mathcal{C}_{L}^{\mathbf{\tau 
}-\limfunc{shift}}$. Since $\mathcal{P}_{L,0}^{small}$ is a finite
collection, the definition of $\mathcal{S}_{\limfunc{size}}^{\alpha
,A}\left( \mathcal{P}_{L,0}^{small}\right) $ shows that there is a quasicube 
$K\in \Pi ^{\limfunc{goodbelow}}\mathcal{P}_{L,0}^{small}$ so that%
\begin{equation*}
\mathcal{S}_{\limfunc{size}}^{\alpha ,A}\left( \mathcal{P}%
_{L,0}^{small}\right) ^{2}=\frac{1}{\left\vert K\right\vert _{\sigma }}%
\left( \frac{\mathrm{P}^{\alpha }\left( K,\mathbf{1}_{A\setminus K}\sigma
\right) }{\left\vert K\right\vert ^{\frac{1}{n}}}\right) ^{2}\omega _{%
\mathcal{P}}\left( \mathbf{T}^{\mathbf{\tau }-\limfunc{deep}}\left( K\right)
\right) .
\end{equation*}%
Now define 
\begin{equation*}
t^{\prime }=t^{\prime }\left( K\right) \equiv \max \left\{ s:\text{there is }%
L^{\prime }\in \mathcal{L}_{s}\text{ with }L^{\prime }\subset K\right\} .
\end{equation*}%
First, suppose that $t^{\prime }=0$ so that $K$ does not contain any $%
L^{\prime }\in \mathcal{L}$. Then it follows from our construction at level $%
\ell =0$ that%
\begin{equation*}
\frac{1}{\left\vert K\right\vert _{\sigma }}\left( \frac{\mathrm{P}^{\alpha
}\left( K,\mathbf{1}_{A\setminus K}\sigma \right) }{\left\vert K\right\vert
^{\frac{1}{n}}}\right) ^{2}\omega _{\mathcal{P}}\left( \mathbf{T}^{\mathbf{%
\tau }-\limfunc{deep}}\left( K\right) \right) <\varepsilon \mathcal{S}_{%
\limfunc{size}}^{\alpha ,A}\left( \mathcal{P}\right) ^{2},
\end{equation*}%
and hence from $\rho =1+\varepsilon $ we obtain 
\begin{equation*}
\mathcal{S}_{\limfunc{size}}^{\alpha ,A}\left( \mathcal{P}%
_{L,0}^{small}\right) ^{2}<\varepsilon \mathcal{S}_{\limfunc{size}}^{\alpha
,A}\left( \mathcal{P}\right) ^{2}=\left( \rho -1\right) \mathcal{S}_{%
\limfunc{size}}^{\alpha ,A}\left( \mathcal{P}\right) ^{2}.
\end{equation*}%
Now suppose that $t^{\prime }\geq 1$. Then $K$ fails the stopping condition (%
\ref{up stopping condition}) with $m=t^{\prime }+1$, and so%
\begin{equation*}
\omega _{\mathcal{P}}\left( \mathbf{T}^{\mathbf{\tau }-\limfunc{deep}}\left(
K\right) \right) <\rho \omega _{\mathcal{P}}\left(
\dbigcup\limits_{L^{\prime }\in \dbigcup\limits_{\ell =0}^{t^{\prime }}%
\mathcal{L}_{\ell }:\ L^{\prime }\subset K}\mathbf{T}^{\mathbf{\tau }-%
\limfunc{deep}}\left( L^{\prime }\right) \right) .
\end{equation*}%
Now we use the crucial fact that $\omega _{\mathcal{P}}$ is \emph{additive}
and finite to obtain from this that%
\begin{eqnarray}
&&\omega _{\mathcal{P}}\left( \mathbf{T}^{\mathbf{\tau }-\limfunc{deep}%
}\left( K\right) \setminus \dbigcup\limits_{L^{\prime }\in
\dbigcup\limits_{\ell =0}^{t^{\prime }}\mathcal{L}_{\ell }:\ L^{\prime
}\subset K}\mathbf{T}^{\mathbf{\tau }-\limfunc{deep}}\left( L^{\prime
}\right) \right)  \label{additive} \\
&=&\omega _{\mathcal{P}}\left( \mathbf{T}^{\mathbf{\tau }-\limfunc{deep}%
}\left( K\right) \right) -\omega _{\mathcal{P}}\left(
\dbigcup\limits_{L^{\prime }\in \dbigcup\limits_{\ell =0}^{t^{\prime }}%
\mathcal{L}_{\ell }:\ L^{\prime }\subset K}\mathbf{T}^{\mathbf{\tau }-%
\limfunc{deep}}\left( L^{\prime }\right) \right)  \notag \\
&\leq &\left( \rho -1\right) \omega _{\mathcal{P}}\left(
\dbigcup\limits_{L^{\prime }\in \dbigcup\limits_{\ell =0}^{t^{\prime }}%
\mathcal{L}_{\ell }:\ L^{\prime }\subset K}\mathbf{T}^{\mathbf{\tau }-%
\limfunc{deep}}\left( L^{\prime }\right) \right) .  \notag
\end{eqnarray}%
Thus using 
\begin{equation*}
\omega _{\mathcal{P}_{L,0}^{small}}\left( \mathbf{T}^{\mathbf{\tau }-%
\limfunc{deep}}\left( K\right) \right) \leq \omega _{\mathcal{P}}\left( 
\mathbf{T}^{\mathbf{\tau }-\limfunc{deep}}\left( K\right) \setminus
\dbigcup\limits_{L^{\prime }\in \dbigcup\limits_{\ell =0}^{t^{\prime }}%
\mathcal{L}_{\ell }:\ L^{\prime }\subset K}\mathbf{T}^{\mathbf{\tau }-%
\limfunc{deep}}\left( L^{\prime }\right) \right) ,
\end{equation*}%
and (\ref{additive}) we have%
\begin{eqnarray*}
&&\mathcal{S}_{\limfunc{size}}^{\alpha ,A}\left( \mathcal{P}%
_{L,0}^{small}\right) ^{2} \\
&\leq &\sup_{K\in \Pi ^{\limfunc{goodbelow}}\mathcal{P}_{L,0}^{small}}\frac{1%
}{\left\vert K\right\vert _{\sigma }}\left( \frac{\mathrm{P}^{\alpha }\left(
K,\mathbf{1}_{A\setminus K}\sigma \right) }{\left\vert K\right\vert ^{\frac{1%
}{n}}}\right) ^{2}\omega _{\mathcal{P}}\left( \mathbf{T}^{\mathbf{\tau }-%
\limfunc{deep}}\left( K\right) \setminus \dbigcup\limits_{L^{\prime }\in
\dbigcup\limits_{\ell =0}^{t^{\prime }}\mathcal{L}_{\ell }:\ L^{\prime
}\subset K}\mathbf{T}^{\mathbf{\tau }-\limfunc{deep}}\left( L^{\prime
}\right) \right) \\
&\leq &\left( \rho -1\right) \sup_{K\in \Pi ^{\limfunc{goodbelow}}\mathcal{P}%
_{L,0}^{small}}\frac{1}{\left\vert K\right\vert _{\sigma }}\left( \frac{%
\mathrm{P}^{\alpha }\left( K,\mathbf{1}_{A\setminus K}\sigma \right) }{%
\left\vert K\right\vert ^{\frac{1}{n}}}\right) ^{2}\omega _{\mathcal{P}%
}\left( \dbigcup\limits_{L^{\prime }\in \dbigcup\limits_{\ell =0}^{t^{\prime
}}\mathcal{L}_{\ell }:\ L^{\prime }\subset K}\mathbf{T}^{\mathbf{\tau }-%
\limfunc{deep}}\left( L^{\prime }\right) \right) .
\end{eqnarray*}%
and we can continue with 
\begin{eqnarray*}
&&\mathcal{S}_{\limfunc{size}}^{\alpha ,A}\left( \mathcal{P}%
_{L,0}^{small}\right) \\
&\leq &\left( \rho -1\right) \sup_{K\in \Pi ^{\limfunc{goodbelow}}\mathcal{P}%
}\frac{1}{\left\vert K\right\vert _{\sigma }}\left( \frac{\mathrm{P}^{\alpha
}\left( K,\mathbf{1}_{A\setminus K}\sigma \right) }{\left\vert K\right\vert
^{\frac{1}{n}}}\right) ^{2}\omega _{\mathcal{P}}\left( \mathbf{T}^{\mathbf{%
\tau }-\limfunc{deep}}\left( K\right) \right) \\
&\leq &\left( \rho -1\right) \mathcal{S}_{\limfunc{size}}^{\alpha ,A}\left( 
\mathcal{P}\right) ^{2}.
\end{eqnarray*}

In the remaining case where $L\in \mathcal{L}_{M+1}$ we can include $L$ as a
testing quasicube $K$ and the same reasoning applies. This completes the
proof of (\ref{small claim' 3}).

To prove the other inequality (\ref{big 3}), we need a lemma to bound the
norm of certain `straddled' stopping forms by the size functional $\mathcal{S%
}_{\limfunc{size}}^{\alpha ,A}$, and another lemma to bound sums of
`mutually orthogonal' stopping forms. We interrupt the proof to turn to
these matters.
\end{proof}

\subsubsection{The Straddling Lemma}

Given an admissible collection of pairs $\mathcal{Q}$ for $A$, and a
subpartition $\mathcal{S}\subset \Pi ^{\limfunc{goodbelow}}\mathcal{Q}$ of
pairwise disjoint quasicubes in $A$, we say that $\mathcal{Q}$ $\mathbf{\tau 
}$\emph{-straddles} $\mathcal{S}$ if for every pair $\left( I,J\right) \in 
\mathcal{Q}$ there is $S\in \mathcal{S}\cap \left[ J,I\right] $ where $\left[
J,I\right] $ denotes the geodesic in the dyadic tree $\Omega \mathcal{D}$
that connects $J$ to $I$, and moreover that $J\Subset _{\mathbf{\tau }}S$.
Denote by $\mathcal{N}_{\mathbf{\rho }-\mathbf{\tau }}^{\limfunc{good}%
}\left( S\right) $ the finite collection of quasicubes that are both good
and $\left( \mathbf{\rho }-\mathbf{\tau }\right) $-nearby in $S$. For any
good dyadic quasicube $S\in \Omega \mathcal{D}_{\limfunc{good}}$, we will
also need the collection $\mathcal{W}^{\limfunc{good}}\left( S\right) $ of
maximal \emph{good} subquasicubes $I$ of $S$ whose triples $3I$ are
contained in $S$.

\begin{lemma}
\label{straddle 3}Let $\mathcal{S}$ be a subpartition of $A$, and suppose
that $\mathcal{Q}$ is an admissible collection of pairs for $A$ such that $%
\mathcal{S}\subset \Pi ^{\limfunc{goodbelow}}\mathcal{Q}$, and such that $%
\mathcal{Q}$ $\mathbf{\tau }$-straddles $\mathcal{S}$. Then we have the
sublinear form bound%
\begin{equation*}
\mathfrak{N}_{\limfunc{stop},1,\bigtriangleup }^{A,\mathcal{Q}}\leq C_{%
\mathbf{r},\mathbf{\tau },\mathbf{\rho }}\sup_{S\in \mathcal{S}}\mathcal{S}_{%
\limfunc{size}}^{\alpha ,A;S}\left( \mathcal{Q}\right) \leq C_{\mathbf{r},%
\mathbf{\tau },\mathbf{\rho }}\mathcal{S}_{\limfunc{size}}^{\alpha ,A}\left( 
\mathcal{Q}\right) ,
\end{equation*}%
where $\mathcal{S}_{\limfunc{size}}^{\alpha ,A;S}$ is an $S$-localized
version of $\mathcal{S}_{\limfunc{size}}^{\alpha ,A}$ with an $S$-hole given
by%
\begin{equation}
\mathcal{S}_{\limfunc{size}}^{\alpha ,A;S}\left( \mathcal{Q}\right)
^{2}\equiv \sup_{K\in \mathcal{N}_{\mathbf{\rho }-\mathbf{\tau }}^{\limfunc{%
good}}\left( S\right) \cup \mathcal{W}^{\limfunc{good}}\left( S\right) }%
\frac{1}{\left\vert K\right\vert _{\sigma }}\left( \frac{\mathrm{P}^{\alpha
}\left( K,\mathbf{1}_{A\setminus S}\sigma \right) }{\left\vert K\right\vert
^{\frac{1}{n}}}\right) ^{2}\omega _{\mathcal{Q}}\left( \mathbf{T}^{\mathbf{%
\tau }-\limfunc{deep}}\left( K\right) \right) .  \label{localized size}
\end{equation}
\end{lemma}

\begin{proof}
For $S\in S$ let $\mathcal{Q}^{S}\equiv \left\{ \left( I,J\right) \in 
\mathcal{Q}:J\Subset _{\mathbf{\tau }}S\subset I\right\} $. We begin by
using that $\mathcal{Q}$ $\mathbf{\tau }$-straddles $\mathcal{S}$, together
with the sublinearity property (\ref{phi sublinear}) of $\varphi _{J}^{%
\mathcal{Q}}$, to write%
\begin{eqnarray*}
\left\vert \mathsf{B}\right\vert _{\limfunc{stop},1,\bigtriangleup }^{A,%
\mathcal{Q}}\left( f,g\right) &=&\sum_{J\in \Pi _{2}\mathcal{P}}\frac{%
\mathrm{P}^{\alpha }\left( J,\left\vert \varphi _{J}^{\mathcal{Q}%
}\right\vert \mathbf{1}_{A\setminus I_{\mathcal{Q}}\left( J\right) }\sigma
\right) }{\left\vert J\right\vert ^{\frac{1}{n}}}\left\Vert \bigtriangleup
_{J}^{\omega }\mathbf{x}\right\Vert _{L^{2}\left( \omega \right) }\left\Vert
\bigtriangleup _{J}^{\omega }g\right\Vert _{L^{2}\left( \omega \right) } \\
&\leq &\sum_{S\in \mathcal{S}}\sum_{J\in \Pi _{2}^{S,\mathbf{\tau }-\limfunc{%
deep}}\mathcal{Q}}\frac{\mathrm{P}^{\alpha }\left( J,\left\vert \varphi
_{J}^{\mathcal{Q}^{S}}\right\vert \mathbf{1}_{A\setminus I_{\mathcal{Q}%
}\left( J\right) }\sigma \right) }{\left\vert J\right\vert ^{\frac{1}{n}}}%
\left\Vert \bigtriangleup _{J}^{\omega }\mathbf{x}\right\Vert _{L^{2}\left(
\omega \right) }\left\Vert \bigtriangleup _{J}^{\omega }g\right\Vert
_{L^{2}\left( \omega \right) }; \\
\text{where }\varphi _{J}^{\mathcal{Q}^{S}} &\equiv &\sum_{I\in \Pi _{1}%
\mathcal{Q}^{S}:\mathcal{\ }\left( I,J\right) \in \mathcal{Q}^{S}}\mathbb{E}%
_{I}^{\sigma }\left( \bigtriangleup _{\pi I}^{\sigma }f\right) \ \mathbf{1}%
_{A\setminus I}\ .
\end{eqnarray*}%
At this point, with $S$ fixed for the moment, we consider separately the
finitely many cases $\ell \left( J\right) =2^{-s}\ell \left( S\right) $
where $s\geq \mathbf{\rho }$ and where $\mathbf{\tau }\leq s<\mathbf{\rho }$%
. More precisely, we pigeonhole the side length of $J\in \Pi _{2}\mathcal{Q}%
^{S}=\Pi _{2}^{S,\tau -\limfunc{deep}}\mathcal{Q}$ by%
\begin{eqnarray*}
\mathcal{Q}_{\ast }^{S} &\equiv &\left\{ \left( I,J\right) \in \mathcal{Q}%
^{S}:J\in \Pi _{2}\mathcal{Q}^{S}\text{ and }\ell \left( J\right) \leq 2^{-%
\mathbf{\rho }}\ell \left( S\right) \right\} , \\
\mathcal{Q}_{s}^{S} &\equiv &\left\{ \left( I,J\right) \in \mathcal{Q}%
^{S}:J\in \Pi _{2}\mathcal{Q}^{S}\text{ and }\ell \left( J\right)
=2^{-s}\ell \left( S\right) \right\} ,\ \ \ \ \ \mathbf{\tau }\leq s<\mathbf{%
\rho }.
\end{eqnarray*}%
Then we have%
\begin{eqnarray*}
\Pi _{2}\mathcal{Q}_{\ast }^{S} &\equiv &\left\{ J\in \Pi _{2}\mathcal{Q}%
^{S}:\ell \left( J\right) \leq 2^{-\mathbf{\rho }}\ell \left( S\right)
\right\} , \\
\Pi _{2}\mathcal{Q}_{s}^{S} &\equiv &\left\{ J\in \Pi _{2}\mathcal{Q}%
^{S}:\ell \left( J\right) =2^{-s}\ell \left( S\right) \right\} ,\ \ \ \ \ 
\mathbf{\tau }\leq s<\mathbf{\rho },
\end{eqnarray*}%
and we make the corresponding decomposition for the sublinear form%
\begin{eqnarray*}
\left\vert \mathsf{B}\right\vert _{\limfunc{stop},1,\bigtriangleup }^{A,%
\mathcal{Q}}\left( f,g\right) &=&\left\vert \mathsf{B}\right\vert _{\limfunc{%
stop},1,\bigtriangleup }^{A,\mathcal{Q}_{\ast }}\left( f,g\right)
+\dsum\limits_{\mathbf{\tau }\leq s<\mathbf{\rho }}\left\vert \mathsf{B}%
\right\vert _{\limfunc{stop},1,\bigtriangleup }^{A,\mathcal{Q}_{s}}\left(
f,g\right) \\
&\equiv &\sum_{S\in \mathcal{S}}\sum_{J\in \Pi _{2}\mathcal{Q}_{\ast }^{S}}%
\frac{\mathrm{P}^{\alpha }\left( J,\left\vert \varphi _{J}^{\mathcal{Q}%
_{\ast }^{S}}\right\vert \mathbf{1}_{A\setminus I_{\mathcal{Q}_{\ast
}}\left( J\right) }\sigma \right) }{\left\vert J\right\vert ^{\frac{1}{n}}}%
\left\Vert \bigtriangleup _{J}^{\omega }\mathbf{x}\right\Vert _{L^{2}\left(
\omega \right) }\left\Vert \bigtriangleup _{J}^{\omega }g\right\Vert
_{L^{2}\left( \omega \right) } \\
&&+\dsum\limits_{\mathbf{\tau }\leq s<\mathbf{\rho }}\sum_{S\in \mathcal{S}%
}\sum_{J\in \Pi _{2}\mathcal{Q}_{s}^{S}}\frac{\mathrm{P}^{\alpha }\left(
J,\left\vert \varphi _{J}^{\mathcal{Q}_{s}^{S}}\right\vert \mathbf{1}%
_{A\setminus I_{\mathcal{Q}_{S}}\left( J\right) }\sigma \right) }{\left\vert
J\right\vert ^{\frac{1}{n}}}\left\Vert \bigtriangleup _{J}^{\omega }\mathbf{x%
}\right\Vert _{L^{2}\left( \omega \right) }\left\Vert \bigtriangleup
_{J}^{\omega }g\right\Vert _{L^{2}\left( \omega \right) }\ .
\end{eqnarray*}%
By the tree-connected property of $\mathcal{Q}$, and the telescoping
property of martingale differences, together with the bound $\alpha _{%
\mathcal{A}}\left( A\right) $ on the quasiaverages of $f$ in the corona $%
\mathcal{C}_{A}$, we have%
\begin{equation}
\left\vert \varphi _{J}^{\mathcal{Q}_{\ast }^{S}}\right\vert ,\left\vert
\varphi _{J}^{\mathcal{Q}_{s}^{S}}\right\vert \lesssim \alpha _{\mathcal{A}%
}\left( A\right) 1_{A\setminus I_{\mathcal{Q}^{S}}\left( J\right) },
\label{bfi 3}
\end{equation}%
where $I_{\mathcal{Q}^{S}}\left( J\right) \equiv \dbigcap \left\{ I:\left(
I,J\right) \in \mathcal{Q}^{S}\right\} $ is the smallest quasicube $I$ for
which $\left( I,J\right) \in \mathcal{Q}^{S}$.

\bigskip

\textbf{Case} for $\left\vert \mathsf{B}\right\vert _{\limfunc{stop}%
,1,\bigtriangleup }^{A,\mathcal{Q}_{s}^{S}}\left( f,g\right) $ when $\mathbf{%
\tau }\leq s\leq \mathbf{\rho }$: Now is a crucial definition that permits
us to bound the form by the size functional with a large hole. Let 
\begin{equation*}
\mathcal{C}_{s}^{S}\equiv \pi ^{\mathbf{\tau }}\left( \Pi _{2}\mathcal{Q}%
_{s}^{S}\right)
\end{equation*}%
be the collection of $\mathbf{\tau }$-parents of quasicubes in $\Pi _{2}%
\mathcal{Q}_{s}^{S}$, and denote by $\mathcal{M}_{s}^{S}$ the set of \emph{%
maximal} quasicubes in the collection $\mathcal{C}_{s}^{S}$. We have that
the quasicubes in $\mathcal{M}_{s}^{S}$ are good by our assumption that the
quasiHaar support of $g$ is contained in the $\mathbf{\tau }$-good quasigrid 
$\Omega \mathcal{D}_{\left( \mathbf{r},\varepsilon \right) -\limfunc{good}}^{%
\mathbf{\tau }}$, and so $\mathcal{M}_{s}^{S}\subset \mathcal{N}_{\mathbf{%
\rho }-\mathbf{\tau }}\left( S\right) $. Here is the first of two key
inclusions:%
\begin{equation}
J\Subset _{\mathbf{\tau }}K\subset S\text{ if }K\in \mathcal{M}_{s}^{S}\text{
is the unique quasicube containing }J.  \label{first key}
\end{equation}

Let $I_{s}\equiv \pi ^{\mathbf{\rho }-s}S$ so that for each $J$ in $\Pi _{2}%
\mathcal{Q}_{s}^{S}$ we have the second key inclusion%
\begin{equation}
\pi ^{\mathbf{\rho }}J=I_{s}\subset I_{\mathcal{Q}^{S}}\left( J\right) .
\label{second key}
\end{equation}%
Now\ each $K\in \mathcal{M}_{s}^{S}$ is also $\left( \mathbf{\rho }-\mathbf{%
\tau }\right) $-deeply embedded in $I_{s}$ if $\mathbf{\rho }\geq \mathbf{r}+%
\mathbf{\tau }$, so that in particular, $3K\subset I_{s}$. This and (\ref%
{second key}) have the consequence that the following Poisson inequalities
hold:%
\begin{equation*}
\frac{\mathrm{P}^{\alpha }\left( J,\mathbf{1}_{A\setminus I_{\mathcal{Q}%
^{S}}\left( J\right) }\sigma \right) }{\left\vert J\right\vert ^{\frac{m}{n}}%
}\lesssim \frac{\mathrm{P}^{\alpha }\left( J,\mathbf{1}_{A\setminus
I_{s}}\sigma \right) }{\left\vert J\right\vert ^{\frac{m}{n}}}\lesssim \frac{%
\mathrm{P}^{\alpha }\left( K,\mathbf{1}_{A\setminus I_{s}}\sigma \right) }{%
\left\vert K\right\vert ^{\frac{m}{n}}}\lesssim \frac{\mathrm{P}^{\alpha
}\left( K,\mathbf{1}_{A\setminus S}\sigma \right) }{\left\vert K\right\vert
^{\frac{m}{n}}}.
\end{equation*}%
Let $\Pi _{2}\mathcal{Q}_{s}^{S}\left( K\right) \equiv \left\{ J\in \Pi _{2}%
\mathcal{Q}_{s}^{S}:J\subset K\right\} $. Let%
\begin{eqnarray*}
\left[ \Pi _{2}\mathcal{Q}_{s}^{S}\right] _{\ell } &\equiv &\left\{ J\in \Pi
_{2}\mathcal{Q}_{s}^{S}:\ell \left( J^{\prime }\right) =2^{-\ell }\ell
\left( K\right) \right\} , \\
\left[ \Pi _{2}\mathcal{Q}_{s}^{S}\right] _{\ell }^{\ast } &\equiv &\left\{
J^{\prime }:J^{\prime }\subset J\in \Pi _{2}\mathcal{Q}_{s}^{S}:\ell \left(
J^{\prime }\right) =2^{-\ell }\ell \left( K\right) \right\} .
\end{eqnarray*}%
Now set $\mathcal{Q}_{s}\equiv \dbigcup\limits_{S\in \mathcal{S}}\mathcal{Q}%
_{s}^{S}$. We apply (\ref{bfi 3}) and Cauchy-Schwarz in $J$ to bound $%
\left\vert \mathsf{B}\right\vert _{\limfunc{stop},1,\bigtriangleup }^{A,%
\mathcal{Q}_{s}}\left( f,g\right) $ by 
\begin{equation*}
\alpha _{\mathcal{A}}\left( A\right) \sum_{S\in \mathcal{S}}\sum_{K\in 
\mathcal{M}_{s}^{S}}\left( \frac{\mathrm{P}^{\alpha }\left( K,\mathbf{1}%
_{A\setminus S}\sigma \right) }{\left\vert K\right\vert ^{\frac{1}{n}}}%
\right) \left\Vert \mathsf{P}_{\Pi _{2}^{S,\mathbf{\tau }-\limfunc{deep}}%
\mathcal{Q}_{s};K}^{\omega }\mathbf{x}\right\Vert _{L^{2}\left( \omega
\right) }\left\Vert \mathsf{P}_{\Pi _{2}^{S,\mathbf{\tau }-\limfunc{deep}}%
\mathcal{Q}_{s};K}^{\omega }g\right\Vert _{L^{2}\left( \omega \right) },
\end{equation*}%
where the localized projections $\mathsf{P}_{\Pi _{2}^{S,\mathbf{\tau }-%
\limfunc{deep}}\mathcal{Q}_{s};K}^{\omega }$ are defined in (\ref{def
localization}) above.

Thus using Cauchy-Schwarz in $K$ we have that $\left\vert \mathsf{B}_{%
\limfunc{stop},1,\bigtriangleup }^{A,\mathcal{Q}_{s}}\left( f,g\right)
\right\vert $ is bounded by%
\begin{eqnarray*}
&&\alpha _{\mathcal{A}}\left( A\right) \sum_{S\in \mathcal{S}}\sum_{K\in 
\mathcal{M}_{s}^{S}}\sqrt{\left\vert K\right\vert _{\sigma }} \\
&&\times \frac{1}{\sqrt{\left\vert K\right\vert _{\sigma }}}\left( \frac{%
\mathrm{P}^{\alpha }\left( K,\mathbf{1}_{A\setminus S}\sigma \right) }{%
\left\vert K\right\vert ^{\frac{1}{n}}}\right) \left\Vert \mathsf{P}_{\Pi
_{2}\mathcal{Q}_{s}^{S}\left( K\right) }^{\omega }\mathbf{x}\right\Vert
_{L^{2}\left( \omega \right) }\left\Vert \mathsf{P}_{\Pi _{2}\mathcal{Q}%
_{s}^{S}\left( K\right) }^{\omega }g\right\Vert _{L^{2}\left( \omega \right)
} \\
&\leq &\alpha _{\mathcal{A}}\left( A\right) \sup_{S\in \mathcal{S}}\mathcal{S%
}_{\limfunc{size}}^{\alpha ,A;S}\left( \mathcal{Q}\right) \left( \sum_{S\in 
\mathcal{S}}\sum_{K\in \mathcal{N}_{\mathbf{\rho }-\mathbf{\tau }}\left(
S\right) }\left\vert K\right\vert _{\sigma }\right) ^{\frac{1}{2}}\left\Vert
g\right\Vert _{L^{2}\left( \omega \right) } \\
&\leq &\sup_{S\in \mathcal{S}}\mathcal{S}_{\limfunc{size}}^{\alpha
,A;S}\left( \mathcal{Q}\right) \alpha _{\mathcal{A}}\left( A\right) \sqrt{%
\left\vert A\right\vert _{\sigma }}\left\Vert g\right\Vert _{L^{2}\left(
\omega \right) },
\end{eqnarray*}%
since $J\Subset _{\mathbf{\tau }}M\subset K$ by (\ref{first key}), since $%
\mathcal{M}_{s}^{S}\subset \mathcal{N}_{\mathbf{\rho }-\mathbf{\tau }}\left(
S\right) $, and since the collection of quasicubes $\dbigcup\limits_{S\in 
\mathcal{S}}\mathcal{M}_{s}^{S}$ is pairwise disjoint in $A$.

\textbf{Case} for $\left\vert \mathsf{B}\right\vert _{\limfunc{stop}%
,1,\bigtriangleup }^{A,\mathcal{Q}_{\ast }}\left( f,g\right) $: This time we
let $\mathcal{C}_{\ast }^{S}\equiv \pi ^{\mathbf{\tau }}\left( \Pi _{2}%
\mathcal{Q}_{\ast }^{S}\right) $ and denote by $\mathcal{M}_{\ast }^{S}$ the
set of \emph{maximal} quasicubes in the collection $\mathcal{C}_{\ast }^{S}$%
. We have the two key inclusions,%
\begin{equation*}
J\Subset _{\mathbf{\tau }}M\Subset _{\mathbf{\rho }-\mathbf{\tau }}S\text{
if }M\in \mathcal{M}_{\ast }^{S}\text{ is the unique quasicube containing }J,
\end{equation*}%
and%
\begin{equation*}
\pi ^{\mathbf{\rho }}J\subset S\subset I_{\mathcal{Q}}\left( J\right) .
\end{equation*}%
Moreover there is $K\in \mathcal{W}^{\limfunc{good}}\left( S\right) $ that
contains $M$. Thus $3K\subset S$ and we have 
\begin{equation*}
\frac{\mathrm{P}^{\alpha }\left( J,\mathbf{1}_{A\setminus S}\sigma \right) }{%
\left\vert J\right\vert ^{\frac{1}{n}}}\lesssim \frac{\mathrm{P}^{\alpha
}\left( K,\mathbf{1}_{A\setminus S}\sigma \right) }{\left\vert K\right\vert
^{\frac{1}{n}}},
\end{equation*}%
and $\left\vert \varphi _{J}\right\vert \lesssim \alpha _{\mathcal{A}}\left(
A\right) 1_{A\setminus S}$. Now set $\mathcal{Q}_{\ast }\equiv
\dbigcup\limits_{S\in \mathcal{S}}\mathcal{Q}_{\ast }^{S}$. Arguing as
above, but with $\mathcal{W}^{\limfunc{good}}\left( S\right) $ in place of $%
\mathcal{N}_{\mathbf{\rho }-\mathbf{\tau }}\left( S\right) $, and using $%
J\Subset _{\mathbf{\rho }}I_{\mathcal{Q}}\left( J\right) $, we can bound $%
\left\vert \mathsf{B}\right\vert _{\limfunc{stop},1,\bigtriangleup }^{A,%
\mathcal{Q}_{\ast }}\left( f,g\right) $ by%
\begin{eqnarray*}
&&\alpha _{\mathcal{A}}\left( A\right) \sum_{S\in \mathcal{S}}\sum_{K\in 
\mathcal{W}^{\limfunc{good}}\left( S\right) }\sqrt{\left\vert K\right\vert
_{\sigma }} \\
&&\times \frac{1}{\sqrt{\left\vert K\right\vert _{\sigma }}}\left( \frac{%
\mathrm{P}^{\alpha }\left( K,\mathbf{1}_{A\setminus S}\sigma \right) }{%
\left\vert K\right\vert ^{\frac{1}{n}}}\right) \left\Vert \mathsf{P}_{\Pi
_{2}\mathcal{Q}_{s}^{S}\left( K\right) }^{\omega }\mathbf{x}\right\Vert
_{L^{2}\left( \omega \right) }\left\Vert \mathsf{P}_{\Pi _{2}\mathcal{Q}%
_{\ast }^{S}\left( K\right) }^{\omega }g\right\Vert _{L^{2}\left( \omega
\right) } \\
&\leq &\alpha _{\mathcal{A}}\left( A\right) \sup_{S\in \mathcal{S}}\mathcal{S%
}_{\limfunc{size}}^{\alpha ,A;S}\left( \mathcal{Q}\right) \left( \sum_{S\in 
\mathcal{S}}\sum_{K\in \mathcal{W}^{\limfunc{good}}\left( S\right)
}\left\vert K\right\vert _{\sigma }\right) ^{\frac{1}{2}}\left\Vert
g\right\Vert _{L^{2}\left( \omega \right) } \\
&\leq &\sup_{S\in \mathcal{S}}\mathcal{S}_{\limfunc{size}}^{\alpha
,A;S}\left( \mathcal{Q}\right) \alpha _{\mathcal{A}}\left( A\right) \sqrt{%
\left\vert A\right\vert _{\sigma }}\left\Vert g\right\Vert _{L^{2}\left(
\omega \right) }.
\end{eqnarray*}%
We now sum these bounds in $s$ and $\ast $ and use $\sup_{S\in \mathcal{S}}%
\mathcal{S}_{\limfunc{size}}^{\alpha ,A;S}\left( \mathcal{Q}\right) \leq 
\mathcal{S}_{\limfunc{size}}^{\alpha ,A}\left( \mathcal{Q}\right) $ to
complete the proof of Lemma \ref{straddle 3}.
\end{proof}

\subsubsection{The Orthogonality Lemma}

Given a set $\left\{ \mathcal{Q}_{m}\right\} _{m=0}^{\infty }$ of admissible
collections for $A$, we say that the collections $\mathcal{Q}_{m}$ are \emph{%
mutually orthogonal}, if each collection $\mathcal{Q}_{m}$ satisfies%
\begin{equation*}
\mathcal{Q}_{m}\subset \dbigcup\limits_{j=0}^{\infty }\left\{ \mathcal{A}%
_{m,j}\times \mathcal{B}_{m,j}\right\} \ ,
\end{equation*}%
where the sets $\left\{ \mathcal{A}_{m,j}\right\} _{m,j}$ and $\left\{ 
\mathcal{B}_{m,j}\right\} _{m,j}$ each have bounded overlap on the dyadic
quasigrid $\Omega \mathcal{D}$: 
\begin{equation*}
\sum_{m,j=0}^{\infty }\mathbf{1}_{\mathcal{A}_{m,j}}\leq A\mathbf{1}_{\Omega 
\mathcal{D}}\text{ and }\sum_{m,j=0}^{\infty }\mathbf{1}_{\mathcal{B}%
_{m,j}}\leq B\mathbf{1}_{\Omega \mathcal{D}}.
\end{equation*}

\begin{lemma}
\label{mut orth}Suppose that $\left\{ \mathcal{Q}_{m}\right\} _{m=0}^{\infty
}$ is a set of admissible collections for $A$ that are \emph{mutually
orthogonal}. Then if $\mathcal{Q}\equiv \dbigcup\limits_{m=0}^{\infty }%
\mathcal{Q}_{m}$, the sublinear stopping form $\left\vert \mathsf{B}%
\right\vert _{\limfunc{stop},1,\bigtriangleup }^{A,\mathcal{Q}}\left(
f,g\right) $ has its restricted norm $\mathfrak{N}_{\limfunc{stop}%
,1,\bigtriangleup }^{A,\mathcal{Q}}$ controlled by the \emph{supremum} of
the restricted norms $\mathfrak{N}_{\limfunc{stop},1,\bigtriangleup }^{A,%
\mathcal{Q}_{m}}$: 
\begin{equation*}
\mathfrak{N}_{\limfunc{stop},1,\bigtriangleup }^{A,\mathcal{Q}}\leq \sqrt{nAB%
}\sup_{m\geq 0}\mathfrak{N}_{\limfunc{stop},1,\bigtriangleup }^{A,\mathcal{Q}%
_{m}}.
\end{equation*}
\end{lemma}

\begin{proof}
If $\mathsf{P}_{m}^{\sigma }=\dsum\limits_{j\geq 0}\dsum\limits_{I\in 
\mathcal{A}_{m,j}}\bigtriangleup _{\pi I}^{\sigma }$ (note the parent $\pi I$
in the projection $\bigtriangleup _{\pi I}^{\sigma }$ because of our `change
of dummy variable' in (\ref{dummy})) and $\mathsf{P}_{m}^{\omega
}=\dsum\limits_{j\geq 0}\dsum\limits_{J\in \mathcal{B}_{m,j}}\bigtriangleup
_{J}^{\omega }$, then we have%
\begin{equation*}
\mathsf{B}_{\limfunc{stop}}^{A,\mathcal{Q}_{m}}\left( f,g\right) =\mathsf{B}%
_{\limfunc{stop}}^{A,\mathcal{Q}_{m}}\left( \mathsf{P}_{m}^{\sigma }f,%
\mathsf{P}_{m}^{\omega }g\right) ,
\end{equation*}%
and%
\begin{eqnarray*}
\sum_{m\geq 0}\left\Vert \mathsf{P}_{m}^{\sigma }f\right\Vert _{L^{2}\left(
\sigma \right) }^{2} &\leq &\sum_{m\geq 0}\sum_{j\geq 0}\left\Vert \mathsf{P}%
_{\mathcal{A}_{m,j}}^{\sigma }f\right\Vert _{L^{2}\left( \sigma \right)
}^{2}\leq An\left\Vert f\right\Vert _{L^{2}\left( \sigma \right) }^{2}, \\
\sum_{m\geq 0}\left\Vert \mathsf{P}_{m}^{\omega }g\right\Vert _{L^{2}\left(
\sigma \right) }^{2} &\leq &\sum_{m\geq 0}\sum_{j\geq 0}\left\Vert \mathsf{P}%
_{\mathcal{B}_{m,j}}^{\omega }g\right\Vert _{L^{2}\left( \omega \right)
}^{2}\leq B\left\Vert g\right\Vert _{L^{2}\left( \omega \right) }^{2}\ .
\end{eqnarray*}%
The sublinear inequality (\ref{phi sublinear}) and Cauchy-Schwarz now give%
\begin{eqnarray*}
\left\vert \mathsf{B}\right\vert _{\limfunc{stop},1,\bigtriangleup }^{A,%
\mathcal{Q}}\left( f,g\right) &\leq &\sum_{m\geq 0}\left\vert \mathsf{B}%
\right\vert _{\limfunc{stop},1,\bigtriangleup }^{A,\mathcal{Q}_{m}}\left(
f,g\right) \leq \sum_{m\geq 0}\mathfrak{N}_{stop}^{A,\mathcal{Q}%
_{m}}\left\Vert \mathsf{P}_{m}^{\sigma }f\right\Vert _{L^{2}\left( \sigma
\right) }\left\Vert \mathsf{P}_{m}^{\omega }g\right\Vert _{L^{2}\left(
\sigma \right) } \\
&\leq &\left( \sup_{m\geq 0}\mathfrak{N}_{\limfunc{stop},1,\bigtriangleup
}^{A,\mathcal{Q}_{m}}\right) \sqrt{\sum_{m\geq 0}\left\Vert \mathsf{P}%
_{m}^{\sigma }f\right\Vert _{L^{2}\left( \sigma \right) }^{2}}\sqrt{%
\sum_{m\geq 0}\left\Vert \mathsf{P}_{m}^{\omega }g\right\Vert _{L^{2}\left(
\sigma \right) }^{2}} \\
&\leq &\left( \sup_{m\geq 0}\mathfrak{N}_{\limfunc{stop},1,\bigtriangleup
}^{A,\mathcal{Q}_{m}}\right) \sqrt{nAB}\sqrt{n}\left\Vert f\right\Vert
_{L^{2}\left( \sigma \right) }\left\Vert g\right\Vert _{L^{2}\left( \omega
\right) }.
\end{eqnarray*}
\end{proof}

\subsubsection{Completion of the proof}

Now we return to the proof of inequality (\ref{big 3}) in Proposition \ref%
{bottom up 3}.

\begin{proof}[Proof of (\protect\ref{big 3})]
Recall that 
\begin{eqnarray*}
\mathcal{P}^{big} &=&\left\{ \dbigcup\limits_{L\in \mathcal{L}}\mathcal{P}%
_{L,0}^{big}\right\} \dbigcup \left\{ \dbigcup\limits_{t\geq
1}\dbigcup\limits_{L\in \mathcal{L}}\mathcal{P}_{L,t}\right\} \equiv 
\mathcal{Q}_{0}^{big}\dbigcup \mathcal{Q}_{1}^{big}; \\
\mathcal{Q}_{0}^{big} &\equiv &\dbigcup\limits_{L\in \mathcal{L}}\mathcal{P}%
_{L,0}^{big}\ ,\ \ \ \ \ \mathcal{Q}_{1}^{big}\equiv \dbigcup\limits_{t\geq
1}\mathcal{P}_{t}^{big},\ \ \ \ \ \mathcal{P}_{t}^{big}\equiv
\dbigcup\limits_{L\in \mathcal{L}}\mathcal{P}_{L,t}.
\end{eqnarray*}%
We first consider the collection $\mathcal{Q}_{0}^{big}=\dbigcup\limits_{L%
\in \mathcal{L}}\mathcal{P}_{L,0}^{big}$, and claim that%
\begin{equation}
\mathfrak{N}_{\limfunc{stop},1,\bigtriangleup }^{A,\mathcal{P}%
_{L,0}^{big}}\leq C\mathcal{S}_{\limfunc{size}}^{\alpha ,A}\left( \mathcal{P}%
_{L,0}^{big}\right) \leq C\mathcal{S}_{\limfunc{size}}^{\alpha ,A}\left( 
\mathcal{P}\right) ,\ \ \ \ \ L\in \mathcal{L}.  \label{big t 3}
\end{equation}%
To see this we note that $\mathcal{P}_{L,0}^{big}$ $\mathbf{\tau }$%
-straddles the trivial collection $\left\{ L\right\} $ consisting of a
single quasicube, since the pairs $\left( I,J\right) $ that arise in $%
\mathcal{P}_{L,0}^{big}$ have $I=L$ and $J$ in the shifted corona $\mathcal{C%
}_{I}^{\mathbf{\tau }-\limfunc{shift}}$. Thus we can apply Lemma \ref%
{straddle 3} with $\mathcal{Q}=\mathcal{P}_{L,0}^{big}$ and $\mathcal{S}%
=\left\{ L\right\} $ to obtain (\ref{big t 3}).

Next, we observe that the collections $\mathcal{P}_{L,0}^{big}$ are \emph{%
mutually orthogonal}, namely 
\begin{eqnarray*}
\mathcal{P}_{L,0}^{big} &\subset &\mathcal{C}_{L}\times \mathcal{C}_{L}^{%
\mathbf{\tau }-\limfunc{shift}}\ , \\
\dsum\limits_{L\in \mathcal{L}}\mathbf{1}_{\mathcal{C}_{L}} &\leq &1\text{
and }\dsum\limits_{L\in \mathcal{L}}\mathbf{1}_{\mathcal{C}_{L}^{\mathbf{%
\tau }-\limfunc{shift}}}\leq \mathbf{\tau }.
\end{eqnarray*}%
Thus the Orthogonality Lemma \ref{mut orth} shows that%
\begin{equation*}
\mathfrak{N}_{\limfunc{stop},1,\bigtriangleup }^{A,\mathcal{Q}%
_{0}^{big}}\leq \sqrt{n\mathbf{\tau }}\sup_{L\in \mathcal{L}}\mathfrak{N}_{%
\limfunc{stop},1,\bigtriangleup }^{A,\mathcal{P}_{L,0}^{big}}\leq \sqrt{n%
\mathbf{\tau }}C\mathcal{S}_{\limfunc{size}}^{\alpha ,A}\left( \mathcal{P}%
\right) .
\end{equation*}

Now we turn to the collection%
\begin{eqnarray*}
\mathcal{Q}_{1}^{big} &=&\dbigcup\limits_{t\geq 1}\dbigcup\limits_{L\in 
\mathcal{L}}\mathcal{P}_{L,t}=\dbigcup\limits_{t\geq 1}\mathcal{P}_{t}^{big};
\\
\mathcal{P}_{t}^{big} &\equiv &\dbigcup\limits_{L\in \mathcal{L}}\mathcal{P}%
_{L,t}\ ,\ \ \ \ \ t\geq 0.
\end{eqnarray*}%
We claim that%
\begin{equation}
\mathfrak{N}_{\limfunc{stop},1,\bigtriangleup }^{A,\mathcal{P}%
_{t}^{big}}\leq C\rho ^{-\frac{t}{2}}\mathcal{S}_{\limfunc{size}}^{\alpha
,A}\left( \mathcal{P}\right) ,\ \ \ \ \ t\geq 1.  \label{S big t 3}
\end{equation}%
Note that with this claim established, we have%
\begin{equation*}
\mathfrak{N}_{\limfunc{stop},1,\bigtriangleup }^{A,\mathcal{P}^{big}}\leq 
\mathfrak{N}_{\limfunc{stop},1,\bigtriangleup }^{A,\mathcal{Q}_{0}^{big}}+%
\mathfrak{N}_{\limfunc{stop},1,\bigtriangleup }^{A,\mathcal{Q}%
_{1}^{big}}\leq \mathfrak{N}_{\limfunc{stop},1,\bigtriangleup }^{A,\mathcal{Q%
}_{0}^{big}}+\sum_{t=1}^{\infty }\mathfrak{N}_{\limfunc{stop}%
,1,\bigtriangleup }^{A,\mathcal{P}_{t}^{big}}\leq C_{\mathbf{\rho }}\mathcal{%
S}_{\limfunc{size}}^{\alpha ,A}\left( \mathcal{P}\right) ,
\end{equation*}%
which proves (\ref{big 3}) if we apply the Orthogonal Lemma \ref{mut orth}
to the set of collections $\left\{ \mathcal{P}_{L,0}^{small}\right\} _{L\in 
\mathcal{L}}$, which is mutually orthogonal since $\mathcal{P}%
_{L,0}^{small}\subset \mathcal{C}_{L}^{\prime }\times \mathcal{C}_{L}^{%
\mathbf{\tau }-\limfunc{shift}}$ . With this the proof of Proposition \ref%
{bottom up 3} is now complete since $\rho =1+\varepsilon $. Thus it remains
only to show that (\ref{S big t 3}) holds.

The cases $1\leq t\leq \mathbf{r}+1$ can be handled with relative ease since
decay in $t$ is not needed there. Indeed, $\mathcal{P}_{L,t}$ $\mathbf{\tau }
$-straddles the collection $\mathfrak{C}_{\mathcal{L}}\left( L\right) $ of $%
\mathcal{L}$-children of $L$, and so the Straddling Lemma applies to give%
\begin{equation*}
\mathfrak{N}_{\limfunc{stop},1,\bigtriangleup }^{A,\mathcal{P}_{L,t}}\leq C%
\mathcal{S}_{\limfunc{size}}^{\alpha ,A}\left( \mathcal{P}_{L,t}\right) \leq
C\mathcal{S}_{\limfunc{size}}^{\alpha ,A}\left( \mathcal{P}\right) ,
\end{equation*}%
and then the Orthogonality Lemma \ref{mut orth} applies to give%
\begin{equation*}
\mathfrak{N}_{\limfunc{stop},1,\bigtriangleup }^{A,\mathcal{P}%
_{t}^{big}}\leq \sqrt{n\mathbf{\tau }}\sup_{L\in \mathcal{L}}\mathfrak{N}_{%
\limfunc{stop},1,\bigtriangleup }^{A,\mathcal{P}_{L,t}}\leq C\sqrt{n\mathbf{%
\tau }}\mathcal{S}_{\limfunc{size}}^{\alpha ,A}\left( \mathcal{P}\right) ,
\end{equation*}%
since $\left\{ \mathcal{P}_{L,t}\right\} _{L\in \mathcal{L}}$ is mutually
orthogonal as $\mathcal{P}_{L,t}\subset \mathcal{C}_{L}\times \mathcal{C}%
_{L^{\prime }}^{\mathbf{\tau }-\limfunc{shift}}$ with $L\in \mathcal{G}_{d}$
and $L^{\prime }\in \mathcal{G}_{d+t}$ for depth $d=d\left( L\right) $.

Now we consider the case $t\geq \mathbf{r}+2$, where it is essential to
obtain decay in $t$. We again apply Lemma \ref{straddle 3} to $\mathcal{P}%
_{L,t}$ with $\mathcal{S}=\mathfrak{C}_{\mathcal{L}}\left( L\right) $, but
this time we must use the stronger localized bounds $\mathcal{S}_{\limfunc{%
size}}^{\alpha ,A;S}$ with an $S$-hole, that give%
\begin{equation}
\mathfrak{N}_{\limfunc{stop},1,\bigtriangleup }^{A,\mathcal{P}_{L,t}}\leq
C\sup_{S\in \mathfrak{C}_{\mathcal{L}}\left( L\right) }\mathcal{S}_{\limfunc{%
size}}^{\alpha ,A;S}\left( \mathcal{P}_{L,t}\right) ,\ \ \ \ \ t\geq 0.
\label{t,n 3}
\end{equation}%
Fix $L\in \mathcal{G}_{d}$. Now we note that if $J\in \Pi _{2}^{L,\mathbf{%
\tau }-\limfunc{deep}}\mathcal{P}_{L,t}$ then $J$ belongs to the $\mathbf{%
\tau }$-shifted corona $\mathcal{C}_{L^{d+t}}^{\mathbf{\tau }-\limfunc{shift}%
}$ for some quasicube $L^{d+t}\in \mathcal{G}_{d+t}$. Then $\pi ^{\mathbf{%
\tau }}J$ is $\mathbf{\tau }$ levels above $J$, hence in the corona $%
\mathcal{C}_{L^{d+t}}$. This quasicube $L^{d+t}$ lies in some child $S\in 
\mathcal{S}=\mathfrak{C}_{\mathcal{L}}\left( L\right) $. So fix $S\in 
\mathcal{S}$ and a quasicube $L^{d+t}\in \mathcal{G}_{d+t}$ that is
contained in $S$ with $t\geq \mathbf{r}+2$. Now the quasicubes $K$ that
arise in the supremum defining $\mathcal{S}_{\limfunc{size}}^{\alpha
,A;S}\left( \mathcal{P}_{L,t}\right) $ in (\ref{localized size}) belong to
either $\mathcal{N}_{\mathbf{\rho }-\mathbf{\tau }}\left( S\right) $ or $%
\mathcal{W}^{\limfunc{good}}\left( S\right) $. We will consider these two
cases separately.

So first suppose that $K\in \mathcal{N}_{\mathbf{\rho }-\mathbf{\tau }%
}\left( S\right) $. A simple induction on levels yields%
\begin{eqnarray*}
\omega _{\mathcal{P}_{L,t}}\left( \mathbf{T}^{\mathbf{\tau }-\limfunc{deep}%
}\left( K\right) \right) &=&\sum_{\substack{ J\in \Pi _{2}^{S,\mathbf{\tau }-%
\limfunc{deep}}\mathcal{P}_{L,t}  \\ J\subset K}}\left\Vert \bigtriangleup
_{J}^{\omega }\mathbf{x}\right\Vert _{L^{2}\left( \omega \right) }^{2} \\
&\leq &\omega _{\mathcal{P}}\left( \dbigcup\limits_{L^{d+t}\in \mathcal{G}%
_{d+t}:\ L^{d+t}\subset K}\mathbf{T}^{\mathbf{\tau }-\limfunc{deep}}\left(
L^{d+t}\right) \right) \\
&\leq &\frac{1}{\rho }\omega _{\mathcal{P}}\left(
\dbigcup\limits_{L^{d+t-1}\in \mathcal{G}_{d+t-1}:\ L^{d+t-1}\subset K}%
\mathbf{T}^{\mathbf{\tau }-\limfunc{deep}}\left( L^{d+t-1}\right) \right) \\
&&\vdots \\
&\lesssim &\rho ^{-\left( t-\mathbf{\rho }-\mathbf{\tau }\right) }\omega _{%
\mathcal{P}}\left( \mathbf{T}^{\mathbf{\tau }-\limfunc{deep}}\left( K\right)
\right) ,\ \ \ \ \ t\geq \mathbf{\rho }-\mathbf{\tau +}2.
\end{eqnarray*}%
Thus we have%
\begin{eqnarray*}
&&\frac{1}{\left\vert K\right\vert _{\sigma }}\left( \frac{\mathrm{P}%
^{\alpha }\left( K,\mathbf{1}_{A\setminus S}\sigma \right) }{\left\vert
K\right\vert ^{\frac{1}{n}}}\right) ^{2}\omega _{\mathcal{P}_{L,t}}\left( 
\mathbf{T}^{\mathbf{\tau }-\limfunc{deep}}\left( K\right) \right) \\
&\lesssim &\rho ^{-t}\frac{1}{\left\vert K\right\vert _{\sigma }}\left( 
\frac{\mathrm{P}^{\alpha }\left( K,\mathbf{1}_{A\setminus S}\sigma \right) }{%
\left\vert K\right\vert ^{\frac{1}{n}}}\right) ^{2}\omega _{\mathcal{P}%
}\left( \mathbf{T}^{\mathbf{\tau }-\limfunc{deep}}\left( K\right) \right)
\lesssim \rho ^{-t}\mathcal{S}_{\limfunc{size}}^{\alpha ,A}\left( \mathcal{P}%
\right) ^{2}.
\end{eqnarray*}

Now suppose that $K\in \mathcal{W}^{\limfunc{good}}\left( S\right) $ and
that $J\in \Pi _{2}^{S,\mathbf{\tau }-\limfunc{deep}}\mathcal{P}_{L,t}$ and $%
J\subset K$. There is a unique quasicube $L^{d+\mathbf{r}+1}\in \mathcal{G}%
_{d+\mathbf{r}+1}$ such that $J\subset L^{d+\mathbf{r}+1}\subset S$. Now $%
L^{d+\mathbf{r}+1}$ is good so $L^{d+\mathbf{r}+1}\Subset _{\mathbf{r}}S$.
Thus in particular $3L^{d+\mathbf{r}+1}\subset S$ so that $L^{d+\mathbf{r}%
+1}\subset K$. The above simple induction applies here to give%
\begin{eqnarray*}
\sum_{\substack{ J\in \Pi _{2}^{S,\mathbf{\tau }-\limfunc{deep}}\mathcal{P}%
_{L,t}  \\ J\subset L^{d+\mathbf{r}+1}}}\left\Vert \bigtriangleup
_{J}^{\omega }\mathbf{x}\right\Vert _{L^{2}\left( \omega \right) }^{2} &\leq
&\omega _{\mathcal{P}}\left( \dbigcup\limits_{L^{d+t}\in \mathcal{G}_{d+t}:\
L^{m-t}\subset L^{d+\mathbf{r}+1}}\mathbf{T}^{\mathbf{\tau }-\limfunc{deep}%
}\left( L^{d+t}\right) \right) \\
&\lesssim &\rho ^{-\left( t-1-\mathbf{r}\right) }\omega _{\mathcal{P}}\left( 
\mathbf{T}^{\mathbf{\tau }-\limfunc{deep}}\left( L^{d+\mathbf{r}+1}\right)
\right) ,\ \ \ \ \ t\geq \mathbf{r}+2.
\end{eqnarray*}%
Thus we have,%
\begin{eqnarray*}
&&\left( \frac{\mathrm{P}^{\alpha }\left( K,\mathbf{1}_{A\setminus S}\sigma
\right) }{\left\vert K\right\vert ^{\frac{1}{n}}}\right) ^{2}\sum_{\substack{
J\in \Pi _{2}^{K,\mathbf{\tau }-\limfunc{deep}}\mathcal{P}_{L,t}  \\ %
J\subset K}}\left\Vert \bigtriangleup _{J}^{\omega }\mathbf{x}\right\Vert
_{L^{2}\left( \omega \right) }^{2} \\
&\leq &C\left( \frac{\mathrm{P}^{\alpha }\left( K,\mathbf{1}_{A\setminus
S}\sigma \right) }{\left\vert K\right\vert ^{\frac{1}{n}}}\right) ^{2}\rho
^{-\left( t-1-\mathbf{r}\right) }\sum_{\substack{ L^{d+\mathbf{r}+1}\in 
\mathcal{G}_{d+\mathbf{r}+1}  \\ L^{d+\mathbf{r}+1}\subset K}}\omega _{%
\mathcal{P}}\left( \mathbf{T}^{\mathbf{\tau }-\limfunc{deep}}\left( L^{d+%
\mathbf{r}+1}\right) \right) \\
&\leq &C\rho ^{-\left( t-1-\mathbf{r}\right) }\left( \frac{\mathrm{P}%
^{\alpha }\left( K,\mathbf{1}_{A\setminus S}\sigma \right) }{\left\vert
K\right\vert ^{\frac{1}{n}}}\right) ^{2}\omega _{\mathcal{P}}\left( \mathbf{T%
}^{\mathbf{\tau }-\limfunc{deep}}\left( K\right) \right) \leq C\rho
^{-\left( t-1-\mathbf{r}\right) }\mathcal{S}_{\limfunc{size}}^{\alpha
,A}\left( \mathcal{P}\right) ^{2}.
\end{eqnarray*}

So altogether we conclude that%
\begin{eqnarray*}
&&\sup_{S\in \mathfrak{C}_{\mathcal{L}}\left( L\right) }\mathcal{S}_{%
\limfunc{size}}^{\alpha ,A;S}\left( \mathcal{P}_{L,t}\right) ^{2} \\
&=&\sup_{S\in \mathfrak{C}_{\mathcal{L}}\left( L\right) }\sup_{K\in \mathcal{%
N}_{\mathbf{\rho }-\mathbf{\tau }}\left( S\right) \cup \mathcal{W}^{\limfunc{%
good}}\left( S\right) }\frac{1}{\left\vert K\right\vert _{\sigma }}\left( 
\frac{\mathrm{P}^{\alpha }\left( K,\mathbf{1}_{A\setminus K}\sigma \right) }{%
\left\vert K\right\vert ^{\frac{1}{n}}}\right) ^{2}\sum_{\substack{ J\in \Pi
_{2}^{K,\mathbf{\tau }-\limfunc{deep}}\mathcal{P}_{L,t}  \\ J\subset K}}%
\left\Vert \mathsf{P}_{J}^{\omega }\mathbf{x}\right\Vert _{L^{2}\left(
\omega \right) }^{2} \\
&\leq &C_{\mathbf{r},\mathbf{\tau },\mathbf{\rho }}\rho ^{-t}\mathcal{S}_{%
\limfunc{size}}^{\alpha ,A}\left( \mathcal{P}\right) ^{2},
\end{eqnarray*}%
and combined with (\ref{t,n 3}) this gives (\ref{S big t 3}). As we pointed
out above, this completes the proof of Proposition \ref{bottom up 3}, hence
of Proposition \ref{stopping bound}, and finally of Theorem \ref{T1 theorem}.
\end{proof}

\end{document}